\documentclass[a4paper,reqno,12pt]{amsart} 
\usepackage{subcaption}
\usepackage{amssymb, mathtools, mathrsfs, xcolor, comment, hyperref}
\usepackage{microtype}
\usepackage{graphicx}
\definecolor{steel}{RGB}{70,130,180}
\definecolor{overleafedit}{RGB}{230,0,230}
\definecolor{overleafreply}{RGB}{150,75,0}
\usepackage[foot]{amsaddr}
\usepackage[alphabetic, abbrev, msc-links, nobysame]{amsrefs}
\usepackage[normalem]{ulem}
\calclayout
\mathtoolsset{showonlyrefs}
\numberwithin{equation}{section}
\allowdisplaybreaks

\theoremstyle{plain}
\newtheorem{thm}{Theorem}[section]
\newtheorem{prop}[thm]{Proposition}
\newtheorem{lemma}[thm]{Lemma}
\newtheorem{corollary}[thm]{Corollary}

\theoremstyle{remark}
\newtheorem{remark}[thm]{Remark}

\newtheorem{example}[thm]{Example}

\DeclareMathOperator*{\esssup}{ess\,sup}
\newcommand{\R}{\mathbb{R}}
\newcommand{\Z}{\mathbb{Z}}
\newcommand{\N}{\mathbb{N}}

\newcommand{\E}{{\mathbb E}} 
\newcommand{\X}{{\mathbb X}}  
\newcommand{\bbP}{{\mathbb P}}  
\newcounter{constant}\newcommand{\C}{\refstepcounter{constant}\ensuremath{C_{\theconstant}}}\newcommand{\Cr}[1]{\ensuremath{C_{\ref{#1}}}}
\newcounter{c}\newcommand{\counter}{\refstepcounter{c}{\thec}}
\newcommand{\eq}{\refstepcounter{equation}(\theequation)}

\DeclarePairedDelimiter{\Parentheses}{(}{)} \newcommand{\pr}{\Parentheses*}
\DeclarePairedDelimiter{\Brackets}{[}{]} \newcommand{\brk}{\Brackets*}
\DeclarePairedDelimiter{\Braces}{\{}{\}} \newcommand{\brc}{\Braces*}
 
\DeclarePairedDelimiter{\Absolute}{\lvert}{\rvert} \newcommand{\abs}{\Absolute*}
\DeclarePairedDelimiter{\Norm}{\lVert}{\rVert} \newcommand{\norm}{\Norm*}
\newcounter{step}\newcommand{\firststep}{\setcounter{step}{0}\refstepcounter{step}\emph{Step \thestep. }} \newcommand{\nextstep}{\refstepcounter{step}\emph{Step \thestep. }}

\begin{document}
	\title[SRDE on MMS]{Comparison principles for stochastic reaction-diffusion equations on metric measure spaces}
	\author[L.~Fan]{Louis Wai-Tong Fan} \email[L.~Fan]{louisfan@unc.edu}
	\author[Z.~Sun]{Zhenyao Sun} \email[Z.~Sun]{zhenyao.sun@gmail.com}
	\author[J.~Yang]{Johnny (Yifan) Yang} \email[J.~Yang]{jyang436@asu.edu}
	\address[L.~Fan]{School of Data and Information Sciences, The University of North Carolina at Chapel Hill} 
	\address[Z.~Sun]{School of Mathematics and Statistics \\ Beijing Institute of Technology \\ Beijing 100081 \\ China} 
	\address[J.~Yang]{School of Mathematics and Statistical Sciences, Arizona State University}
	\subjclass[2020]{TBD}
	\begin{abstract}
		We study  parabolic stochastic partial differential equations on metric measure spaces $(\X, d,m)$ of the form
		\begin{align}
			\partial_t u(t,x) = \mathcal{L}^* u(t,x) + b(t,x,u(t,x)) + \sigma(t,x,u(t,x)) \dot{W}(t,x),\quad t>0,\, x \in \X,
		\end{align}
		where $\mathcal{L}$ is the generator of a Markov process which possesses transition densities, and $\dot{W}$  is a Gaussian noise that is white in time and possibly with spatial correlation. We assume the coefficients $b$ and $\sigma$ are Lipschitz and satisfy the linear growth condition. We formulate general and checkable assumptions on $(\X,\mathcal{L},\dot{W})$ that ensure existence and uniqueness of probabilistically strong, continuous, tempered  mild solutions. We then prove comparison principles (including a strong comparison principle) and  strict positivity, relative to initial conditions. 
        Our framework allows non-symmetric heat kernels and includes diffusion-type and stable-type examples, such as metric graphs and fractal spaces with sub-Gaussian heat kernel estimates.
	\end{abstract}
	\maketitle

\tableofcontents
	
	\section{Introduction}
	\subsection{Motivation}
	In this article, we study parabolic-type stochastic partial differential equations (SPDEs) on general metric measure spaces. Let $(\X,d)$ be a locally compact separable metric space, and let $m$ be a Radon measure on $\X$ with full support. 
	We study the equations of the form
	\begin{align}\label{eq:spde_intro}
		\partial_t u(t,x) = \mathcal L^* u(t,x) + b(t,x,u(t,x)) + \sigma(t,x,u(t,x))\dot{W}(t,x),\;\; t>0, x \in \X,
	\end{align}
	where ${\mathcal L}$ is the infinitesimal generator  of an $\X$-valued Markov process 
	which possesses transition densities  $(G_t(x,y))_{t >0 ,x,y \in \X}$ with respect to the reference measure $m$, $\mathcal L^*$ is the adjoint of $\mathcal L$ with respect to $m$, 
	$\dot{W}$ is a space-time Gaussian noise defined on an augmented filtered probability space
	$(\Omega, \mathscr{F}, (\mathscr{F}_t)_{t\geq 0},\mathbb P)$, and the coefficients $b$ and $ \sigma$ are $\R$-valued (random) predictable functions on $\R_+ \times \X \times \R$.
	
Following Walsh's theory \cite{MR876085}, we  regard \eqref{eq:spde_intro} as a shorthand of the integral equation
	\begin{align}
		u_t(x)&=
		\int_\X G_t(y,x) u_0(y) m(\mathrm dy) + \int_0^t \mathrm ds \int_\X G_{t - s}(y,x)b(s,y,u_s(y)) m(\mathrm dy)
		\\&\quad + \iint_0^t G_{t-s}(y,x)\sigma(s,y,u_s(y)) W(\mathrm ds,\mathrm dy), \label{eq:mild}
	\end{align}
	where $u_0$ is the initial condition, and we write $u_t(x):=u(t,x)$ (not a partial derivative with respect to $t$). 
    The appearance of $\mathcal L^*$ in \eqref{eq:spde_intro} is consistent with this formulation: terms such as $\int_\X G_t(y,x)f(y)m(\mathrm dy)$ can be written formally as $P_t^*f(x)$, where $P_t^*$ is the adjoint of the semigroup associated with $\mathcal L$. 	Roughly, a stochastic process $u=(u_t)_{t\in \R_+}$ is called a \textbf{mild-solution} to \eqref{eq:spde_intro} if it is defined on the same probability space $(\Omega, \mathscr{F},\mathbb P)$, is adapted to the filtration $(\mathscr{F}_t)_{t\geq 0}$, and satisfies \eqref{eq:mild} almost surely for each $t\in\R_+$.

	A basic example is the one-dimensional Euclidean case $(\X,d)=(\R,|\cdot|)$ with Lebesgue measure and $\mathcal L=\Delta$, where
	\[
		G_t(x,y)=\frac{1}{\sqrt{4\pi t}}\exp\pr{\frac{-|x-y|^2}{4t}}
	\]
	is the Gaussian heat kernel. Much existing work treats the symmetric case, in which $\mathcal L^*=\mathcal L$ and $G_t(x,y)=G_t(y,x)$ for all $t>0$ and $x,y\in\X$. Our framework also allows nonsymmetric kernels.

    Much of the existing solution theory for parabolic SPDEs was developed in Euclidean settings. When the driving noise $\dot{W}$ is space-time white noise,  classical function-valued well-posedness results are largely restricted to one-dimensional state spaces, such as the circle, bounded intervals, or the real line. This restriction reflects the singularity of space-time white noise: in higher-dimensional spaces, the heat kernel typically does not provide enough smoothing to define the stochastic convolution as a function-valued process without additional spatial correlation, renormalization, or a weaker solution concept. 
    Working on suitable metric measure spaces (MMS) can mitigate this difficulty, enabling well-posedness while retaining space-time white noise as the driving noise. Moreover, there are compelling theoretical and practical reasons to study SPDEs beyond Euclidean spaces. A metric-measure-space formulation provides a common framework for manifolds, metric graphs, and fractals, and it also accommodates nonsymmetric transition kernels.

Indeed, SPDEs on metric measure spaces have attracted increasing attention in recent years. Besides the abstract functional-analytic formulations \cite{da2014stochastic},  earlier studies for SPDEs on metric measure spaces trace back to the Dirichlet form approaches in \cite{hinz2013vector}, and the semi-group approaches with heat kernel estimates in  \cites{hinz2012semigroups,issoglio2015regularity}. 
Although Walsh's original theory of mild solutions  focuses on the heat kernel of $\Delta$ on $\R$, his stochastic integration theory with respect to martingale measures was formulated on Lusin spaces  \cite{MR876085}. 
Further developments include well-posedness and regularity results for parabolic-type SPDEs on post-critically finite fractals \cites{hambly2018existence, hambly2017continuous}. There, the authors also investigated intermittency and invariant measures of solutions. The stochastic wave equations on fractals were studied in \cite{Zbl145435418}. Cerrai and Freidlin obtained SPDEs on metric graphs as asymptotic limits of SPDEs on narrowing 2-dimensional domains \cite{Zbl136760078, Zbl902444267} and of SPDEs in $\R^2$ with fast advection \cite{Zbl141860070, Zbl148360090}. Later,  Fan  \cite{MR4278798}  showed that solutions to stochastic Fisher-KPP equations on metric graphs can be obtained as scaling limits of interacting particle systems. Recently in \cite{Zbl08159034}, the authors obtained  estimates on the intermittency properties of the parabolic Anderson model on a class of bounded recurrent metric measure spaces, which includes compact metric graphs and the bounded Sierpi\'{n}ski gasket. Finally, singular SPDEs on fractals were studied in \cite{chen2026wick}, and those on other non-flat spaces in \cites{bailleul2016heat,Zbl900576144,Zbl902307062}.

One aim of this paper is to give a unified well-posedness framework for such equations. Despite the progress above, many existing results are tailored to particular spaces, noises, or heat kernels. We formulate checkable assumptions on the space, the noise, the heat kernel, and the coefficients that recover several known examples and also indicate new settings in which mild-solutions can be constructed.

Our second aim is to study various comparison principles for \eqref{eq:spde_intro} relative to initial conditions. We say that the {\bf comparison principle} holds for \eqref{eq:spde_intro} if the following is true: if $u_0(x)\ge u'_0(x)$ for all $x\in\X$, then
	\[
	\bbP\bigl(u_t(x)\ge u'_t(x)\ \text{for all }t\ge 0,\ x\in\X\bigr)=1
	\]
	whenever  $u=(u_t)_{t\in\R_+}$ and $u'=(u'_t)_{t\in\R_+}$ are strong solutions to \eqref{eq:spde_intro} with initial values $u_0$ and $u'_0$ respectively. 
	We say that the {\bf strong comparison principle} holds for \eqref{eq:spde_intro} if, in addition, the following is true: if $u_0(x)\ge u'_0(x)$ for all $x\in\X$ and $u_0(x_0)>u'_0(x_0)$ for some $x_0\in\X$, then
	\[
	\bbP\bigl(u_t(x)>u'_t(x)\ \text{for all }t>0,\ x\in\X\bigr)=1
	\]
	whenever  $u=(u_t)_{t\in\R_+}$ and $u'=(u'_t)_{t\in\R_+}$ are strong solutions to \eqref{eq:spde_intro} with initial values $u_0$ and $u'_0$ respectively. 

    It is known that comparison principles relative to initial conditions can be used to study the positivity of the solution \cites{Zbl074960057, kotelenez1992comparison}. We say that the {\bf strict positivity} holds for \eqref{eq:spde_intro} if 
    	\[
	\bbP\bigl(u_t(x)>0\ \text{for all }t>0,\ x\in\X\bigr)=1
	\]
    whenever the initial condition $u_0$ is non-negative and satisfies $u_0(x_0)>0$ for some $x_0\in\X$.

Most proofs of strong comparison principles for SPDEs trace back to Mueller’s original argument in \cite{Zbl074960057} for the one-dimensional stochastic heat equation driven by space-time white noise and with no drift (i.e. the coefficient $b\equiv0$). Mueller proved that, if $\sigma(u)=|u|^{\gamma}$ for some $\gamma\ge 1$, then the solution $u$ remains strictly positive up to the explosion time $\tau_{\infty}$, where $\tau_{\infty}=\lim_{n\to\infty}\tau_n$ and $\tau_n=\inf\{t\ge0:\,\sup_{x\in\R}u(t,x)\ge n\}$.
Shiga \cite{MR1271224} generalized this result, proving strong comparison and strict positivity when $b(u)$ and $\sigma(u)$ are Lipschitz continuous and $\dot{W}$ is space-time white noise. Later, Chen and Kim \cite{chen2017comparison} extended these results to fractional Laplacians and rough initial data. Gubinelli and Perkowski \cite{gubinelli2017kpz} gave a pathwise proof of strict positivity for the stochastic heat equation. Their argument is robust with respect to the precise structure of the noise, but it requires strictly positive initial data, whereas the results in \cites{Zbl074960057, MR1271224, chen2017comparison} allow initial data that are nonnegative and not identically zero.
See also Section 6.4 of \cite{dalang2024stochastic} for an alternative proof of comparison principles for the stochastic heat equation driven by space-time white noise. There, the comparison principle on the real line is obtained by approximation from the corresponding principles on intervals $[-L,L]$ with Dirichlet boundary conditions, followed by letting $L\to \infty$. 

The results discussed above concern one-dimensional Euclidean spaces. In arbitrary spatial dimension, comparison and positivity results are also available for the stochastic heat equation on $\R^d$ driven by Gaussian noise that is white in time and  homogeneous in space. Kotelenez \cite{kotelenez1992comparison} developed an abstract comparison theorem for mild, function-valued SPDEs with Lipschitz coefficients driven by cylindrical Hilbert-space-valued Brownian motions, under the assumption that the evolutionary operator is positivity-preserving. As an application, Kotelenez constructed unique mild solutions that remain in the cone of positive functions for certain reaction-diffusion equations with non-Lipschitz drift terms. 
Tessitore and Zabczyk \cite{tessitore1998strict} studied the linear multiplicative stochastic heat equation on $\R^d$ driven by spatially homogeneous Gaussian noise. Under integrability assumptions on the spectral density of the noise, they proved that if the initial function is nonnegative and uniformly positive on a set with nonempty interior, then, for each $t>0$, the solution is positive at Lebesgue-a.e. spatial point almost surely. 
More recently, Chen and Huang \cite{chen2019comparison} proved, for globally Lipschitz $\sigma$ and $b\equiv0$ and on $\R^d$, the weak comparison principle under Dalang's condition; under a stronger spectral integrability condition on the spatial covariance, they also established the strong comparison principle and strict positivity for measure-valued initial data. The proofs of these results typically go through finite-dimensional approximations of the noise, where a Pardoux-type comparison argument can be applied, and then prove strict positivity by combining Mueller's iterative propagation argument with exponential tail estimates for  stochastic convolutions.

To the best of our knowledge, comparison, strong comparison, and strict positivity have not previously been established for SPDEs on metric measure spaces at the level of generality considered here. Our results fill in this gap and provide a unified framework with both theoretical and potential practical applications. The work \cite{Zbl08159034}, for example, focuses on the linear parabolic Anderson model (PAM) on bounded spaces satisfying sub-Gaussian heat kernel estimates and a dimension restriction (see Example~\ref{Eg:PolySubGaussianHK}). By contrast, our framework accommodates bounded or unbounded spaces, non-linear drift and diffusion coefficients, and abstract heat kernel assumptions that cover both diffusion-type and stable-type processes. Figure~\ref{fig:pam-fractals} illustrates the solution profiles for the diffusion-type and the stable-type PAMs on the Sierpi\'{n}ski gasket. 

\begin{figure}[htbp]
    \centering
    \begin{subfigure}[t]{0.48\textwidth}
        \centering
        \includegraphics[width=\linewidth]
            {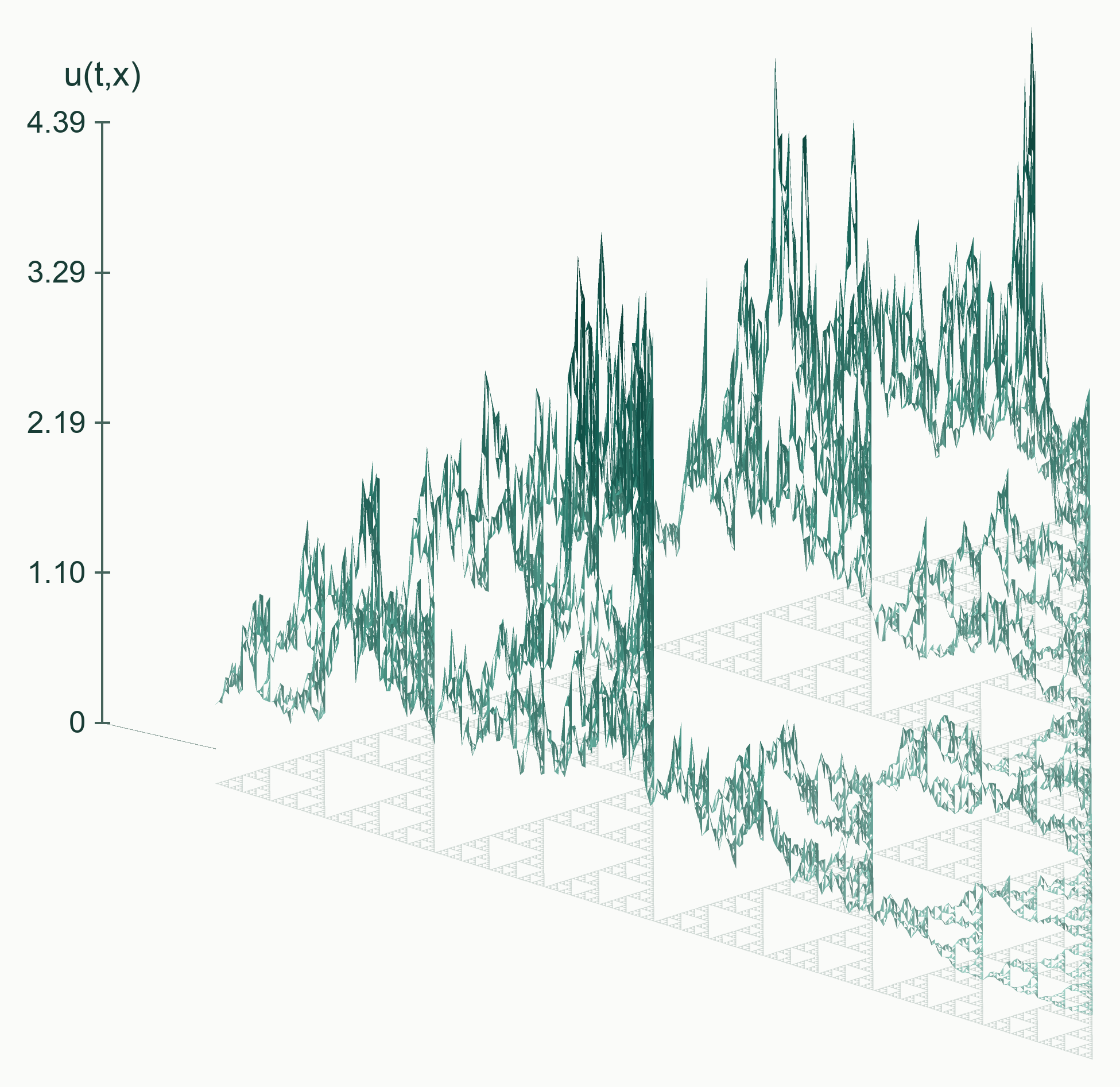}
        \caption{
            $\partial_t u = 0.15\mathcal L u + u\dot W$.
        }
        \label{fig:pam-usual}
    \end{subfigure}\hfill
    \begin{subfigure}[t]{0.48\textwidth}
        \centering
        \includegraphics[width=\linewidth]
            {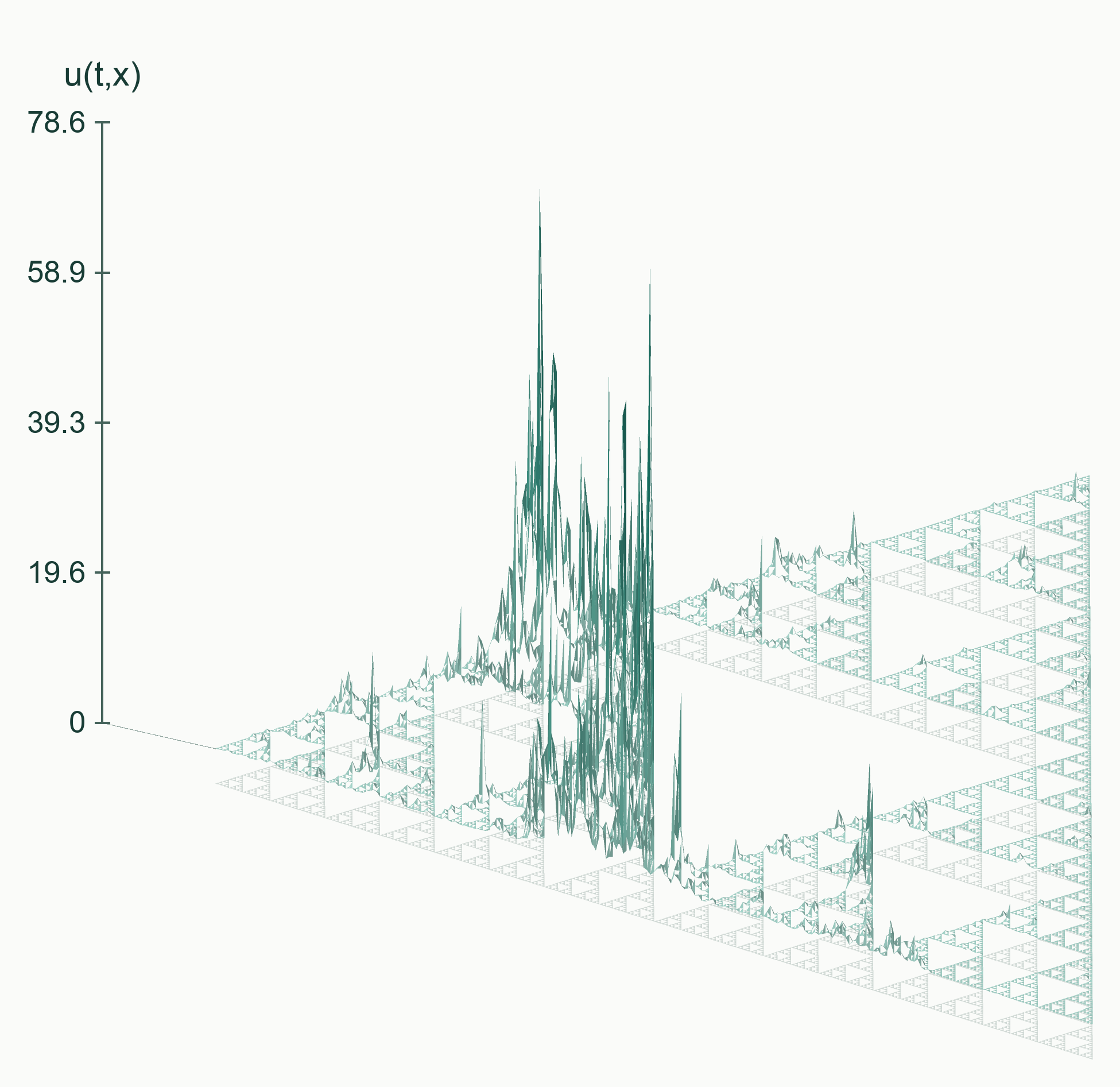}
        \caption{
            $\partial_t u = -0.15(-\mathcal L)^{0.8}u + u\dot W$.
        }
        \label{fig:pam-stable}
    \end{subfigure}
    \caption{
        Simulations of two fractional PAMs \eqref{eq:fractionalPAM} on the finite Sierpi\'{n}ski gasket
        at $t=0.1$, with the same $\delta$-like initial
        condition and realization of space-time white noise.
        Here, $\mathcal L$ is the  Laplacian on the gasket.
            }
    \label{fig:pam-fractals}
\end{figure}

In particular, our results provide tools for  studying stochastic reaction-diffusion equations driven by space-time white noise on metric graphs and fractals, which is useful for understanding how stochastic extinction, propagation phenomena and other long-time behaviors depend on the underlying geometry \cite{MR4278798,Zbl08159034}.
Our strict positivity result also enables the study of the KPZ equation on fractals and other spaces via the Cole-Hopf transform: Strict positivity of $u$ ensures that $\log u$ is  well-defined; formally,
if $u$ solves a multiplicative-noise equation (e.g., a parabolic Anderson type equation), then $h:=\log u$ is the corresponding KPZ-type height field \cite{hairer2013solving}.

	In this article, we study both well-posedness and comparison principles for SPDE \eqref{eq:spde_intro} in a general geometric setting (see Section \ref{sec:backgroup_setups}), and we provide easy-to-check sufficient conditions for both properties. 

	Our only explicit requirement is that the metric space $(\X,d)$ has a metric-entropy bound (see \eqref{ineq:entropy bound}) which supports a Kolmogorov-type continuity lemma.
	However, one observes that many implicit requirements of $\X$ are encoded in the heat kernel (see for example \cite{grigor2003heat}). Hence, instead of making explicit assumptions on the geometry of $(\X,d,m)$ and the heat kernel bounds, we  make implicit assumptions about the interactions between the space, the heat kernels and the Gaussian noise. The advantage of such an approach is that it helps identify rather weak sufficient conditions for our results. The downside is that each condition needs to be verified separately for different cases.
	However, we will demonstrate in Section \ref{sec:examples} how to verify our assumptions in terms of heat kernel bounds, with examples such as metric graphs and certain fractals like the Sierpi\'{n}ski gasket.
    In particular, we show that if the Markov process generated by ${\mathcal L}$ behaves either like a diffusion or a stable process, then the strong comparison principle holds for \eqref{eq:spde_intro} provided $b$ and $\sigma$ are globally Lipschitz in $u$ and continuous in $(x,u)$, under the noise conditions of Section~\ref{sec:examples}.
    
	A difficulty of studying SPDE on general space is that Mueller-type arguments rely on fine local estimates for stochastic convolutions and on the immediate spreading of mass by the heat kernel. In our setting these ingredients are replaced by metric-entropy estimates, compatibility conditions between the heat kernel and the Gaussian noise, and an abstract instant-propagation condition. This formulation is flexible enough to include nonsymmetric heat kernels and a range of diffusion-type and stable-type examples. More specifically,
	Mueller's proof of strict positivity \cite{Zbl074960057} iterates over many (say, $n$) short time steps: in each step the heat semigroup spreads a positive lower bound to a slightly larger region, while an exponential tail estimate shows that the stochastic convolution destroys this gain with probability $o(n^{-1})$, so that all steps succeed with probability tending to one as $n\to\infty$. On Euclidean spaces, this tail estimate rests on increments of the stochastic convolution on nested dyadic grids and on explicit Gaussian heat-kernel bounds, which also give the spreading of mass. In our proof for strict positivity, the dyadic decomposition is replaced by chaining over metric coverings, based on \eqref{ineq:entropy bound} and on increment bounds in the covariance norm of the noise (Proposition~\ref{prop:MLDP}); these local estimates are summed over the covering in \eqref{ineq:faster_than_ms} to obtain weighted bounds on the whole space (Lemma~\ref{lem:mldp_global}). The Gaussian computation for the spreading of mass is replaced by the instant-propagation condition \eqref{ineq:instant_propagation}, and the weak comparison principle in Theorem~\ref{thm:comparison}~\eqref{part:main2 weak comparison} allows us to restart each step from the random profile reached at the end of the previous one.

\smallskip
	
	This paper is organized as follows.
	In Section \ref{sec:backgroup_setups} we introduce the components of the equation \eqref{eq:spde_intro}: the metric measure space, the generator of a Markov process and a Gaussian noise.
	We also give our regularity assumptions at the end of Section \ref{sec:backgroup_setups}.
	We state our main results on well-posedness and the comparison principles in Section \ref{sec:main}.
	The sufficient conditions for the key assumptions in our main theorems, together with illustrative examples, can be found in Section \ref{sec:examples}. Section~2 collects preliminary results on stochastic integrals, heat kernels and random fields. Theorem~\ref{thm:main2} is proved in Section~\ref{sec:strong well posed}, Theorem~\ref{thm:comparison}~\eqref{part:main2 weak comparison} in Section~\ref{sec:comparison}, and Theorem~\ref{thm:comparison}~\eqref{part:main2 strong comparison}--\eqref{part:main2 strictly positivity} in Section~\ref{sec:proof of strong comparison}. Section \ref{sec:appendix} contains the proofs of auxiliary lemmas used throughout the paper.

    Unless otherwise stated, proofs of all results are provided either in the main text or in Appendix~\ref{sec:appendix}, which also collects the proofs of auxiliary lemmas used throughout the paper.
	
	\subsection{Model setups}\label{sec:backgroup_setups}
	
	\subsubsection*{Notations}
	
	We will be using notations of bounded-up-to-a-constant `$\lesssim$', comparable `$\asymp$', big-O `$\mathbf O$', small-o `$\mathbf o$' and big-Omega `$\Omega$'. To make these notations precise, let $f$ and $g$ be arbitrary real-valued functions on the Cartesian product $A\times B$ of two abstract topological spaces $A$ and $B$.
	For every $b\in B$, we say $|f(a,b)| \lesssim |g(a,b)|$ uniformly in $a\in A$, if 
	$\sup_{a\in A} \frac{\abs{f(a,b)}}{\abs{g(a,b)}} < \infty.$
	(Here, we are using the convention that $0/0 = 0.$)
	We write ``$f  \asymp  g$'' to represent ``$f \lesssim g $ and $g \lesssim f$''. 
	We say $f(a,b) = \mathbf O(g(a,b))$ uniformly in $a\in A$ as $b\to b_0$ in $B$ if $ \limsup_{b\to b_0} \sup_{a\in A} \frac{\abs{f(a,b)}}{\abs{g(a,b)}} < \infty$. 
	We say $f(a,b) = \mathbf o(g(a,b))$ uniformly in $a\in A$ as $b\to b_0$ in $B$, if $ \limsup_{b\to b_0} \sup_{a\in A} \frac{\abs{f(a,b)}}{\abs{g(a,b)}} = 0$. 
	We write ``$f=\boldsymbol{\Theta}(g)$'' to represent ``$g = \mathbf O(f)$''.

	\subsubsection*{The metric measure space}
	
	Throughout our presentations, let $(\X,d)$ be a locally compact separable complete metric space. The Borel $\sigma$-field of $\X$ is denoted as $\mathscr B_\X$. 
	Let $m$ be a Radon measure on $\X$ with full support. We call $(\X,d,m)$ the (underlying) metric measure space. Denote by $\mathcal B_\X$ the collection of real-valued Borel functions on $\X$.
	For every $p\geq 1$, denote by $L^p_\X$ the collection of $f\in \mathcal B_\X$ such that $\|f\|_{L^p_\X}:= \pr{\int |f(x)|^p m(\mathrm dx)}^{1/p} < \infty$.
	For $x \in \X$ and $r >0$, denote by $B_r(x) := \{ y \in \X: d(x,y) < r\}$ the open ball with radius $r$ centered at $x$. 
	Let $D_\varepsilon(A)$ be  the smallest number of  open balls with radius $\varepsilon>0$ required to cover a given subset $A\subset \X$.

	We assume throughout this paper the following local condition on the space
	\begin{itemize}
		\item[\eq\label{ineq:entropy bound}]
		for each $R\geq 1$ there exists $b>0$ such that $D_\varepsilon( B_R(x)) \lesssim \varepsilon^{-b}$ uniformly in $x\in \X$ and $0<\varepsilon \le 1$. 
	\end{itemize}
	This is a sufficient geometric condition for the Kolmogorov--Chentsov type theorem (see \cite[Theorem~1.1]{kratschmer2022kolmogorov}), which yields continuous modifications of random fields on metric spaces. Condition~\eqref{ineq:entropy bound} is also used in the chaining argument for the large deviation estimates of stochastic integrals in Section~\ref{sec:proof of strong comparison}.

	\subsubsection*{The state spaces}
	Let $\N = \{1, 2, \dots \}$.
	Let $\mathcal C_\X$ be the collection of real-valued continuous functions on $\X$. Denote by $\mathcal C^{\mathrm b}_\X$ the collection of bounded functions in $\mathcal C_\X$.
	Throughout this paper,  we fix a weight function $h \in \mathcal C_\X$ with $h \ge 1$ such that the following two statements hold.
	\begin{enumerate}
		\item[\eq \label{ineq:faster_than_ms}]There exist 
		$\lambda_0 > 0$, an index set $\mathcal I\subseteq\N$, and a family $(x_n)_{n \in \mathcal I}$ in $\X$ 
		such that $\X = \cup_{n \in \mathcal I} B_1(x_n)$ and
		\[   
		\sum_{n \in \mathcal I} h(x_n)^{-\lambda}< \infty,\quad \text{ for all } \lambda \ge \lambda_0.
		\]
		The index set $\mathcal I$ is allowed to be finite.
		\item[\eq\label{ineq:ref_regular}] 
		Uniformly for $x,y\in \X$ with $d(x,y)\leq 1$, $h(y) \lesssim h(x)$.
		
	\end{enumerate}
	For $\lambda \in \R$ and $f \in \mathcal B_\X$, define \[\|f\|_{(\lambda)} := \|f\|_{h^\lambda L_\X^\infty} :=\esssup_{\text{$m$-a.e. $x\in \X$}} h(x)^{-\lambda}|f(x)|.\]
	Denote by
	\[
	\mathcal C^{\mathrm{tem}}_{\X,h}  := \brc{ f \in \mathcal C_\X : \|f\|_{(\lambda)}< \infty
		\text{ for all } \lambda > 0 },
	\]
	the space of continuous functions on $\X$ with tempered growth. We also write $\mathcal C^{\mathrm{tem},+}_{\X,h}:=\{f\in\mathcal C^{\mathrm{tem}}_{\X,h}: f\ge 0\}$. In Lemma~\ref{lem:PS} we show that $\mathcal C^{\mathrm{tem}}_{\X,h}$ is a Polish space.
	The topology of $\mathcal C^{\mathrm{tem}}_{\X,h}$
	is generated by the family of norms $\{\|\cdot\|_{(\lambda)}: \lambda > 0\}$
	in the sense of \cite[Theorem 1.37]{MR1157815}. Denote by $\mathcal C_{\R_+}\mathcal C^{\mathrm{tem}}_{\X,h}:=\mathcal C(\R_+, \mathcal C^\mathrm{tem}_{\X,h})$ the collection of continuous paths in $\mathcal C^{\mathrm{tem}}_{\X,h}$ indexed by $\R_+$. Similarly, denote $\mathcal C_{[0,T]}\mathcal C^{\mathrm{tem}}_{\X,h} :=\mathcal C([0,T],\mathcal C^\mathrm{tem}_{\X,h})$.

	\subsubsection*{Generator}
	Let $\mathcal C^0_\X$ be the Banach space of real-valued continuous functions on $\X$ vanishing at infinity, equipped with the uniform norm. This space is dense in $\mathcal C^{\mathrm{tem}}_{\X,h}$, by Lemma \ref{lem:C_temp_density_0}. 
        
	Throughout the paper, let ${\mathcal L}$ be a linear operator from its domain $\mathcal D(\mathcal L)\subset \mathcal C^0_\X$ to $\mathcal C^0_\X$. 
	We assume that $\mathcal L$ is the generator of an $\X$-valued  Feller process $(Y_t)_{t \ge 0}$, whose transition semigroup and expectation operators are denoted by $(P_t)_{t\geq 0}$ and $(\Pi_x)_{x\in \X}$, in the sense of \cite[Theorem 17.4]{MR4226142}.
	We also assume the existence of a  measurable function $(G_t(x,y))_{t >0 ,x,y \in \X}$, referred to as the transition densities of $Y$ w.r.t.~ the reference measure $m$, so that
	\begin{align}\label{eq:G}
		G_{t+s}(x,z) = \int G_t(x,y) G_s(y,z) m(\mathrm dy), \quad s,t>0, x,z \in \X,
	\end{align}
	and 
	\begin{align}\label{eq:pi_hk}
		P_t f(x) =\Pi_x\left[f(Y_t) \right] = \int_\X  G_t(x,y) f(y) m(\mathrm dy), \quad t>0, x\in \X, f\in  \mathcal B^{\mathrm{b}+}_\X.
	\end{align}
	Here, $\mathcal B^{\mathrm{b}+}_\X$ represents the collection of bounded non-negative Borel functions on $\X$. 
	Throughout, $\mathcal L^*$ denotes the formal infinitesimal operator associated with the $m$-transposed kernel family $(P_t^*)_{t\geq 0}$ defined below.
	Define 
	$(P_t^*)_{t\geq 0}$ to be 
	the family of kernels from $\X$ to itself such that
	\begin{align}
		P_t^* f(x) =  \mathbf 1_{\{0\}}(t)f(x) + \mathbf 1_{(0,\infty)}(t) \int_\X G_t(y,x) f(y) m(\mathrm dy), \quad t\geq 0, 
		x\in \X, 
		f \in \mathcal B^{\mathrm{b}+}_\X.
	\end{align}
	For convention, define 
	\begin{align}\label{eq:GC}
		G_t(x,y) := 0, \quad t\leq 0, x,y \in \X.
	\end{align}
	
	Throughout this paper, we assume that the heat kernel $\{G_t(x, y)\}_{t>0,x,y\in\X}$ satisfies the following:
	\begin{itemize}
		\item[\eq\label{cond:ctem continuous}] for each $f\in \mathcal C^{\mathrm{tem}}_{\X,h}$, $(P_t^*f)_{t\geq 0}$ is a $\mathcal C^{\mathrm{tem}}_{\X,h}$-valued continuous process. 
	\end{itemize}
	Note that this is a natural requirement to ensure the existence of a solution in $\mathcal C^{\mathrm{tem}}_{\X,h}$ of the deterministic PDE $\partial_t u = \mathcal L^* u$ (when $b\equiv 0$ and $\sigma \equiv 0$ in \eqref{eq:spde_intro}).
	In addition, to obtain the desired regularity of the solution, we assume the following condition to hold throughout the paper:
	\begin{itemize}
		\item[\eq\label{cond:time_holder}] 
		There exist $\beta \in [0,1)$ and $\xi_1,\xi_2 >0$ so that, uniformly for every $x\in \X$ and $t,t'\in (0,1]$,
		\begin{equation}
			\| G_t(\cdot,x) -  G_{t'}(\cdot,x) \|_{L^1_\X} 
			\lesssim \frac{|t' - t|^{\xi_1} }{(t\wedge t')^\beta },
		\end{equation}
		and that, uniformly for every 
		$x,x' \in \X$ with $d(x,x') \leq 1$ and $t\in (0,1]$,
		\begin{equation}
			\|G_t(\cdot,x) -  G_t(\cdot,x') \|_{L^1_\X}
			\lesssim \frac{d(x,x')^{\xi_2}}{t^\beta}.
		\end{equation}
	\end{itemize}

	\subsubsection*{Stochastic basis}

	Throughout this paper, let $K$ be a Radon measure on $\X^2$ which is symmetric and non-negative definite, and 
    	denote by
	\[
	\mathcal H_{\X,K}:=\mathcal H_{K}:= \left\{f\in \mathcal B_\X:\;\int_{\X^2} |f(x)f(y)| K(\mathrm dx, \mathrm dy) <  \infty \right\}.
	\]
For any $f\in \mathcal H_{K}$, we define
		\[
		 \|f\|_{\mathcal H_{K}}^2:= \iint_{\X^2} f(x)f(y) K(\mathrm dx, \mathrm dy)
		\] 
and write  $[f]_K:=\{f+g: g\in \mathcal H_{K}, \|g\|_{\mathcal H_{K}} = 0\}$. We state the following standard result  for later use.
	\begin{lemma}\label{lem:H0}
		The space $(\mathcal H_{K}$, $\|\cdot\|_{\mathcal H_{K}})$ is a separable Hilbertian seminormed space in the sense of \cite[Section 1.1]{Zbl054760064}. 
		Furthermore, $[0]_K$ is a linear subspace of $\mathcal H_{K}$, and there exists a unique norm $\|\cdot\|_{\tilde{\mathcal H}_{K}}$ on the quotient space $\mathcal H_{K}/[0]_K:= \{[f]_K: f\in \mathcal H_{K}\}$ such that $\|[f]_K\|_{\tilde{\mathcal H}_{K}} = \|f\|_{\mathcal H_{K}}$ for each $f\in \mathcal H_{K}.$
		The completion of the normed space $(\mathcal H_{K}/[0]_K, \|\cdot\|_{\tilde{\mathcal H}_{K}})$ is a separable Hilbert space, denoted by $(\tilde{\mathcal H}_{K}, \|\cdot\|_{\tilde{\mathcal H}_{K}})$.

	\end{lemma}

	To control the spatial growth of the solutions to the SPDE under study, we assume the heat kernel and the Gaussian noise satisfy the following condition:
	\begin{itemize}
		\item[\eq\label{cond:exp_hk_lamda}]  There exists $\alpha \in (0,1)$ so that for every $\lambda  \in [0,1]$, uniformly  in $t \in (0,1]$ and $x\in \X$,
		$\|G_t(\cdot, x) h(\cdot)^\lambda\|_{L^1_\X} \lesssim h(x)^\lambda$ and
		$\| G_t(\cdot,x) h(\cdot)^\lambda \|_{\mathcal H_{K}}^2 \lesssim t^{-\alpha}h(x)^{2\lambda}$.
	\end{itemize}

\smallskip
    
	We say $(\Omega, \mathscr{F}, (\mathscr{F}_t)_{t\geq 0},\mathbb P)$ is an augmented filtered probability space if $(\Omega, \mathscr{F}, \mathbb P)$ is a complete probability space and $(\mathscr{F}_t)_{t\geq 0}$ is an augmented filtration  in the sense of \cite[Lemma 9.8]{MR4226142}. 
	We say  $(\Omega, \mathscr{F}, (\mathscr{F}_t)_{t\geq 0}, \mathbb P, W)$ is a {\bf stochastic basis} if $(\Omega, \mathscr{F}, (\mathscr{F}_t)_{t\geq 0}, \mathbb P)$ is an augmented filtered probability space and  $W =(W_t)_{t\geq 0}=(W_t(f): t\geq 0,f\in \tilde{\mathcal H}_{K})$ is a generalized $(\mathscr{F}_t)_{t\geq 0}$-Wiener process on $\tilde{\mathcal H}_{K}$, in the sense of \cite[Section 4.1.2. and p.~95]{MR1207136}, defined in the probability space $(\Omega, \mathscr{F}, \mathbb P)$.
	Let $\mathcal P$ denote the predictable $\sigma$-algebra on $\Omega\times\mathbb R_+$.
	Throughout this paper, a real-valued predictable random field on $\mathbb R_+\times\X$ means a $\mathcal P\otimes\mathscr B_\X$-measurable function on $\Omega\times\mathbb R_+\times\X$.
    
	Given such a stochastic basis, there exists a self-adjoint non-negative definite bounded linear operator $Q$ on $\tilde{\mathcal H}_{K}$ 
        such that 
	\begin{align}\label{eq:W_covariance}
		\mathbb E[W_t(f)W_s(g)] = (t\wedge s) \langle Qf,g \rangle_{\tilde{\mathcal H}_{K}}, \quad t,s\geq 0, f,g\in \tilde{\mathcal H}_{K}.
	\end{align}
	We refer to $Q$ as \textbf{the covariance of $W$}.
	Denote by $\|Q\|$ the operator norm of $Q$ \cite[p.~310]{MR1157815}
    For each $\phi \in Q^{1/2}\tilde{\mathcal H}_{K}$,
	denote by $Q^{-1/2}\phi$ the unique element $f$ in $\operatorname{Ker}(Q^{1/2})^{\bot}$ such that $Q^{1/2}f=\phi$; and define the induced norm $\|\phi\|_{Q^{1/2}\tilde{\mathcal H}_{K}}:=\|Q^{-1/2}\phi\|_{\tilde{\mathcal H}_{K}}$.
	It is known that  $(Q^{1/2}\tilde{\mathcal H}_{K}, \|\cdot\|_{Q^{1/2}\tilde{\mathcal H}_{K}})$  is a Hilbert space and is referred to as the reproducing kernel of $W$.
	For each $g\in \mathcal H_{K}$, noticing $[g]_K \in \tilde{\mathcal H}_{K}$, we define $\Gamma_g(\phi):=\langle [g]_K,\phi\rangle_{\tilde{\mathcal H}_{K}}$ for every $\phi \in Q^{1/2}\tilde{\mathcal H}_{K} \subset \tilde{\mathcal H}_{K}$.
		\begin{lemma}\label{lem:HS}
		For each  $g\in \mathcal H_{K}$,
		$\Gamma_g$ is a Hilbert-Schmidt operator from 
		$Q^{1/2}\tilde{\mathcal H}_{K}$
		to $\mathbb R$ satisfying that, for any orthonormal basis  $(\eta_i)_{i\in I}$ of 
		$Q^{1/2}\tilde{\mathcal H}_{K}$,
		
		\[
			\|\Gamma_g\|^2_{\mathcal L^{\mathrm{HS}}_{Q^{1/2}\tilde{\mathcal H}_{K}}}
			:=\sum_{i\in I} |\Gamma_g \eta_i|^2 
			= \norm{Q^{1/2}[g]_K}^2_{\tilde{\mathcal H}_{K}} 
			\leq \|Q^{1/2}\|^2 \|g\|^2_{\mathcal H_{K}}
		\]
	\end{lemma}

	Given a stochastic basis $(\Omega, \mathscr{F}, (\mathscr{F}_t)_{t\geq 0}, \mathbb P, W)$, the stochastic integral w.r.t.\ $W$ is defined as in \cite[Chapter 4]{MR1207136}.
    	In particular, for any predictable process $g:=(g_s)_{s\geq 0}$ in 
	$\mathcal H_{K}$
	such that \[ \norm{g}^2_{L^2_{[0,t]}\mathcal H_{K}}:=\int_0^t \norm{g_s}^2_{\mathcal H_{K}} \mathrm ds < \infty, \quad t\geq 0, \text{a.s.}\]
	the stochastic integral
	\begin{equation}\label{E:StocIntegral}
	\iint_0^\cdot g_s(y) \dot W_s(y)\mathrm ds m(\mathrm dy)
	:=\iint_0^\cdot g_s(y) W(\mathrm d s, \mathrm d y)
	:= \int_0^\cdot \Gamma_{g_s} \mathrm dW_s  
	\end{equation}
	is an adapted continuous local martingale with quadratic variation 
	\[
	\int_0^\cdot \|\Gamma_{g_s}\|_{\mathcal L^{\mathrm{HS}}_{Q^{1/2}\tilde{\mathcal H}_{K}}}^2\mathrm ds 
	\overset{\text{Lemma \ref{lem:HS}}}\leq \|Q^{1/2}\|^2 \int_0^\cdot\norm{g_s}^2_{\mathcal H_{K}}\mathrm ds.
	\]
	Here, we write $\dot W_t(x)$ as a formal expression of the improper derivative $W(\mathrm dt, \mathrm dx)/(\mathrm dt \otimes m(\mathrm dx))$. We also write $W(\mathrm dx,\mathrm dt)$ for $W(\mathrm dt,\mathrm dx)$. 
	We say $W$ is cylindrical if $Q=I$ is the identity operator on $\tilde{\mathcal H}_{K}$.
	We say $\dot W_t(x)$ is a space-time white noise if $W$ is cylindrical and $K(\mathrm dx, \mathrm dy) = \delta_x(\mathrm dy) m(\mathrm dx)$
	where $\delta_x$ is the Dirac delta measure at $x \in \X$.

	\subsection{Main results}\label{sec:main}
	
	The main object of study in this paper is the SPDE
	\begin{equation}\label{eq:spde}
		\begin{cases}
			\partial_t u_t(x) 
			= \mathcal L^* u_t(x) + b(t, x, u_t(x)) + \sigma(t, x ,u_t(x)) \dot W_t(x),
			\quad t \in (0,\infty),\, x \in \X,
			\\ u_0(x) = f(x).
		\end{cases}
	\end{equation}
	where 	
	\begin{itemize}
		\item[\eq\label{eq:SB}]  $(\Omega, \mathscr{F}, (\mathscr{F}_t)_{t\geq 0}, \mathbb P, W)$ is a stochastic basis as described earlier;
        		\item[\eq\label{eq:DC}] the drift coefficient \[b:=(b(t,x,u))_{t\geq 0, x\in \X, u\in \mathbb R} := (b(\omega, t, x, u))_{\omega \in \Omega, t\geq 0,x\in \X, u\in \mathbb R}\]
		which is a predictable function
		on the product space $\Omega \times \mathbb R_+ \times \X \times \mathbb R$; 
		\item[\eq\label{eq:NC}] the noise coefficient
		\[\sigma:=(\sigma(t,x,u))_{t\geq 0, x\in \X, u\in \mathbb R} := (\sigma(\omega, t, x, u))_{\omega \in \Omega, t\geq 0,x\in \X, u\in \mathbb R}\]
		which is a predictable function 
		on the product space $\Omega \times \mathbb R_+ \times \X \times \mathbb R$; and 
		\item[\eq\label{eq:IV}] the initial value  $f\in \mathcal C^{\mathrm{tem}}_{\X,h}$.
	\end{itemize}
	
	We will  focus on the  probabilistically strong continuous random field mild-solution theory.
		To be more precise, a stochastic process
    $(u_t)_{t\geq 0}$ defined in the probability space $(\Omega, \mathscr{F},\mathbb P)$
    is a \textbf{mild-solution} to the SPDE \eqref{eq:spde}, if it is  a continuous $\mathcal C^{\mathrm{tem}}_{\X,h}$-valued process that is  $(\mathscr{F}_t)_{t\geq 0}$-adapted and such that, for every $t\geq 0$ and $x\in \X$, 
	\begin{align}
		u_t(x)&=
		P_t^*f(x) + \int_0^t \mathrm ds \int_\X G_{t - s}(y,x)b(s,y,u_s(y)) m(\mathrm dy)
		\\&\quad + \iint_0^t G_{t-s}(y,x)\sigma(s,y,u_s(y)) W(\mathrm ds,\mathrm dy), \label{eq:mild2}
	\end{align}
    $\mathbb{P}$-almost surely, where the stochastic integral in \eqref{eq:mild2} is as in \eqref{E:StocIntegral}. Note that, by definition, the process $(u_t)_{t\geq 0}$ has sample paths in $\mathcal C_{\R_+}\mathcal C^{\mathrm{tem}}_{\X,h}$ almost surely.
	
	\begin{thm}
		\label{thm:main2}
		Let $(\X, d,m)$ be a metric measure space such that \eqref{ineq:entropy bound} is satisfied. 
		Let the stochastic basis $(\Omega, \mathscr{F}, (\mathscr{F}_t)_{t\geq 0}, \mathbb P, W)$, drift coefficient $b$, noise coefficient
		$\sigma$, and the initial value $f$ be given as in \eqref{eq:SB}--\eqref{eq:IV}.
        Suppose
		$\mathcal L$ admits a heat kernel $\{G_t(x, y)\}_{t>0,x,y\in\X}$ (with respect to the Radon measure $m$) that satisfies \eqref{cond:ctem continuous},
		\eqref{cond:exp_hk_lamda} and
		\eqref{cond:time_holder}.
		Suppose
		the existence of
		a $\mathbb P$-null subset $\mathrm N_{\counter\label{c:Lip}}$ of $\Omega$ s.t.
		\begin{itemize}
			\item[\eq\label{eq:LiG}]
			(linear growth condition)
			for any $T>0$, uniformly for
			$\omega \in \Omega \setminus \mathrm N_{\ref{c:Lip}}$,
			$(t,x) \in [0,T] \times \X$ and $u \in \R$,
			\begin{equation}
				|b(\omega, t,x,u) |+|\sigma(\omega, t,x,u)| \lesssim 1+|u|;
				\quad \text{and}
			\end{equation}
		\end{itemize}
		\begin{itemize}
			\item[\eq\label{ineq:g_lip}]
			(Lipschitz condition)
			for any $T>0$, uniformly for
			$\omega \in \Omega \setminus \mathrm N_{\ref{c:Lip}}$,
			$(t,x) \in [0,T] \times \X$ and $u,v \in \R$,
			\begin{equation}
				|b(\omega, t,x,u) - b(\omega, t,x,v)|+|\sigma(\omega, t,x,u)- \sigma(\omega, t,x,v)| \lesssim |u - v|.
			\end{equation}
		\end{itemize}
		Then, there exists a 
		mild-solution $(u_t)_{t\geq 0}$ to the SPDE \eqref{eq:spde}, and the solution is unique up to indistinguishability.
	\end{thm}
	
	Next, we shall state the comparison principles. For this we will use the following conditions.
	\begin{itemize}
		\item[\eq\label{cond:hk_uni_bound}] 
		For each $\delta \in (0,1)$, 
		$\| G_\varepsilon(\cdot, x) \mathbf 1_{B_\delta(x)^\mathrm c}(\cdot) \|^2_{\mathcal H_{K}} = \mathbf o(1)$
		uniformly in $x\in \X$ while $\varepsilon \downarrow 0$.
		\item[\eq\label{ineq:instant_propagation}]
		For every $0<r_0<v$, and $x_0\in \X$, as $\varepsilon \downarrow 0$, uniformly in $t\in (\varepsilon, 2\varepsilon)$, $r \in (r_0, v)$ and $x\in B_{r+v\varepsilon}(x_0)$, 
		\begin{equation}\label{E:hk_uni_bound}
			\int_{B_r(x_0)} G_t(y,x) m(\mathrm dy) = \boldsymbol{\Theta}(1).
		\end{equation}
                        \item[\eq\label{eq:con}] for any $(t,\omega) \in \R_+\times \Omega \setminus \mathrm N_{\ref{c:Lip}}$, the maps $(x,u) \mapsto b(\omega,t,x,u)$ and $(x,u) \mapsto \sigma(\omega,t,x,u)$ are continuous.
	\end{itemize}

	\begin{thm}
		\label{thm:comparison}
		Suppose the same assumptions as in Theorem \ref{thm:main2} hold. Let $(u_t)_{t\ge 0}$ and $(u'_t)_{t \ge 0}$ be the mild-solutions with initial conditions $u_0=f\in \mathcal C^{\mathrm{tem}}_{\X,h}$ and  
        $u_0' = f' \in \mathcal C^{\mathrm{tem}}_{\X,h}$ respectively.
		\begin{enumerate}
			\item \label{part:main2 weak comparison} (Comparison principle)
			Suppose  \eqref{cond:hk_uni_bound} and \eqref{eq:con}  hold. If
			$f(x) \ge f'(x)$ for  $x \in \X$, then $\bbP\bigl(u_t(x)\ge u'_t(x)\ \text{for all }t\ge 0 \text{ and } x\in\X\bigr)=1$.
			\item \label{part:main2 strong comparison}
			(Strong comparison principle) Suppose   \eqref{cond:hk_uni_bound}, \eqref{ineq:instant_propagation} and \eqref{eq:con}    hold. If $f(x) \ge f'(x)$ for  $x \in \X$ and $f(x_0) >f'(x_0)$ for some $x_0 \in \X$,
			then $\bbP\bigl(u_t(x)> u'_t(x)\ \text{for all }t> 0 \text{ and } x\in\X\bigr)=1$.
			\item  \label{part:main2 strictly positivity}
			(Strict positivity)
			Suppose   \eqref{cond:hk_uni_bound}, \eqref{ineq:instant_propagation} and \eqref{eq:con}    hold, and that almost surely,
			\begin{align}\label{cond:pos_zero}
				b(t,x,0) \ge 0\quad \text{ and }\quad \sigma(t,x,0) = 0 \quad\text{for all }(t,x) \in \R_+ \times \X.
			\end{align}
            If $f$ is non-negative and $f(x)>0$ for some $x\in \X$, then 
                        $\bbP\bigl(u_t(x)> 0\ \text{for all }t> 0 \text{ and } x\in\X\bigr)=1$.
		\end{enumerate}
	\end{thm}

	\subsection{Examples}\label{sec:examples}
	
	\subsubsection{Metric measure Dirichlet spaces with sub-Gaussian heat kernel estimates}
	A metric measure Dirichlet space (MMD) is a collection $(\X,d,m,\mathcal E,\mathcal F)$ such that
	\begin{enumerate}
		\item $(\X,d)$ is a locally compact separable complete metric space;
		\item $m$ is a Radon measure on $\X$ with full support;
		\item $(\mathcal E,\mathcal F)$ is a strongly local regular symmetric Dirichlet form on $L^2(\X,m)$ in the sense of \cite{fukushima2011dirichlet}.
	\end{enumerate}
	Associated with an MMD space is a strongly continuous symmetric Markov semigroup $(P_t)_{t\ge0}$ on $L^2(\X,m)$ with its nonpositive self-adjoint generator $\mathcal L$ in the sense of \cite[Theorem 1.4.1]{fukushima2011dirichlet}.
	We say such an MMD admits heat kernels if there exists a measurable function $(G_t(x,y))_{t>0, x,y \in \X}$ such that \eqref{eq:G} and \eqref{eq:pi_hk} hold for every $t>0$, $x\in \X$ and $f\in L^2(\X,m)$.
	
	To describe the sub-Gaussian estimates of the heat kernel, we recall the notion of a scale function $\Psi$, i.e., an increasing homeomorphism $\Psi:[0,\infty)\to[0,\infty)$ satisfying
	\[
	\Bigl(\frac{R}{r}\Bigr)^{\beta_1} \lesssim \frac{\Psi(R)}{\Psi(r)} \lesssim \Bigl(\frac{R}{r}\Bigr)^{\beta_2},
	\qquad\text{uniformly in } 0<r\le R,
	\]
	for some constants $1<\beta_1\le \beta_2$, and its Legendre transform
	\[
	\Psi^*(s):=\sup_{r>0}\left(\frac{s}{r}-\frac{1}{\Psi(r)}\right), \quad s\geq 0.
	\]
	We say a function $V:[0,\infty) \to [0,\infty)$ is \textit{doubling} if it is a homeomorphism such that
	\begin{align}\label{ineq:doubling_function}
		V(2r) \lesssim V(r),\qquad\text{uniformly in } r>0.
	\end{align}
	Note that if $V,\Psi:[0,\infty)\to[0,\infty)$ are doubling functions satisfying
	\begin{align} \label{eq:2doublingfunctions}
		\frac{R^2}{r^2} \lesssim \frac{\Psi(R)}{\Psi(r)} \lesssim \frac{R\,V(R)}{r\,V(r)},\qquad\text{uniformly in } 0<r\le R,
	\end{align}
	then $\Psi$ is a scale function.
	
	We say an MMD $(\X,d,m, \mathcal E, \mathcal F)$ with heat kernels $(G_t)_{t>0}$ satisfies the sub-Gaussian heat kernel estimate with respect to a scale function $\Psi$, shortened as $\mathrm{HKE}\pr{\Psi}$, if there exist constants $C_1,c_U,b_U,C_2,\delta \in (0,\infty)$ such that, for all $t>0$ and $m\otimes m$-a.e.\ $(x,y)\in\X^2$,
	\[
	G_t(x,y)
	\le
	\frac{C_1}{m(B_{\Psi^{-1}(t)}(x))}
	\exp\!\left[-c_U\, t\, \Psi^*\!\left(b_U\frac{d(x,y)}{t}\right)\right],
	\]
	and, whenever $d(x,y)\le\delta\Psi^{-1}(t)$,
	\[
	G_t(x,y)\ge
	\frac{C_2}{m(B_{\Psi^{-1}(t)}(x))}.
	\]
	We say an MMD $(\X,d,m,\mathcal E,\mathcal F)$ with heat kernels $(G_t)_{t>0}$ satisfies the \textit{full} sub-Gaussian heat kernel estimate with respect to the scale function $\Psi$, shortened as $\mathrm{HKE}_{\rm full}\pr{\Psi}$, if in addition to $\mathrm{HKE}\pr{\Psi}$, there exist $c_L,b_L>0$ such that, for all $t>0$ and $m\otimes m$-a.e.\ $(x,y)\in\X^2$,
	\begin{align*}
		G_t(x,y) \ge \frac{C_2}{m(B_{\Psi^{-1}(t)}(x))}\exp \pr{-c_L t \Psi^*\pr{b_L \frac{d(x,y)}{t}}}.
	\end{align*}
	We say the metric measure space $(\X,d,m)$ has \textit{volume growth profile} $V$ if 
	\begin{align*}
		V(r) \asymp m\pr{B_r(x)}\qquad \text{uniformly in }x \in \X,\,0 <r< \text{diam}(\X).
	\end{align*}
	Under $\mathrm{HKE}(\Psi)$, the heat kernel admits a jointly continuous version (see \cite{barlow2012equivalence}), which we use throughout, so that the above estimates hold for all $(x,y)\in\X^2$, and the associated semigroup is Feller \cite[Proposition 3.2]{lierl2015scale}.

	Inspired by \cites{Zbl08159034,chen2025global}, we consider the following family of symmetric and non-negative definite Radon measures $\{K_\gamma\}_{\gamma \ge 0}$ on $\X\times \X$ given by:
	\begin{align}\label{eq:K0}
		K_0(\mathrm dy,\mathrm dx):= \delta_{x}(\mathrm dy)m(\mathrm dx),
	\end{align} 
	and for $\gamma \neq 0$,  
	\begin{align}\label{eq:Kgamma}
		K_\gamma\pr{\mathrm dx,\mathrm dy}:= \pr{\int_0^\infty e^{-s}s^{\gamma - 1}G_s(x,y) \mathrm ds} m(\mathrm dx)\otimes m(\mathrm dy).
	\end{align}
	
	The following result shows that the MMD spaces described in this subsection provide a concrete class of examples to which our main theorems apply.
	\begin{prop}\label{prop:sub_g_mmd_framework}
		Let $V,\Psi:[0,\infty)\to[0,\infty)$ be doubling functions satisfying \eqref{eq:2doublingfunctions}.
		Let $(\X,d,m,\mathcal E,\mathcal F)$ be an MMD that admits heat kernels, satisfies $\mathrm{HKE}_{\rm full}(\Psi)$ and has volume growth profile $V$.
		Let $\mathcal L$ be the operator associated with $(\mathcal E,\mathcal F)$.
		Fix $x_0\in\X$ and set $h(x):=\exp(d(x_0,x))$ for $x\in\X$.
		Let $\gamma\geq 0$ and let $K_\gamma$ be as in \eqref{eq:K0} and \eqref{eq:Kgamma}.
		Suppose there exists $\alpha_0\in(0,1)$ such that
		\begin{itemize}
			\item if $\gamma = 0$, then 
			\[
			\pr{V\pr{\Psi^{-1}(t)}}^{-1}\lesssim t^{-\alpha_0},\qquad\text{uniformly in }t\in(0,1];
			\]
			\item else if $\gamma \neq 0$, then 
			\[
			\int_0^\infty \frac{s^\gamma}{V\pr{\Psi^{-1}(t+s)}}\,\frac{\mathrm ds}{s}\lesssim t^{-\alpha_0},\qquad\text{uniformly in }t\in(0,1].
			\]
		\end{itemize}
		Let stochastic basis $(\Omega, \mathscr{F}, (\mathscr{F}_t)_{t\geq 0}, \mathbb P, W)$, drift coefficient $b$, noise coefficient
		$\sigma$, and initial value $f$ be given as in \eqref{eq:SB}--\eqref{eq:IV} with $K$ being replaced by $K_\gamma$.
		Suppose there exists a $\mathbb P$-null subset $\mathrm N_{\ref{c:Lip}}$ of $\Omega$ such that \eqref{eq:LiG} and \eqref{ineq:g_lip} hold.
		Then, there exists a unique (up to indistinguishability) strong solution $(u_t)_{t\geq 0}$ to the SPDE \eqref{eq:spde}.
				Furthermore, if in addition \eqref{eq:con} and \eqref{cond:pos_zero} hold, then the unique strong solution satisfies the strict positivity stated in Theorem~\ref{thm:comparison}.
		In particular, for the Parabolic Anderson model
		\[
		\partial_t u = \mathcal L u + u\dot W_t,\qquad u(0,\cdot)=u_0,
		\]
		with either white noise $K_0$ or colored noise $K_\gamma$ under the respective conditions above,
		the unique strong solution satisfies
		\[
		\mathbb P\pr{u(t,x)>0\text{ for all }t>0,\;x\in\X}=1
		\]
		whenever $u_0\in\mathcal C^{\mathrm{tem},+}_{\X,h}$ is not identically zero.
	\end{prop}
	The proof of Proposition \ref{prop:sub_g_mmd_framework} follows from verifying the assumptions of Theorems~\ref{thm:main2} and~\ref{thm:comparison}, which is carried out in Appendix~\ref{sec:appendix}.
	
	\begin{remark}\label{rem:sub-g-general}
		Let \(V,\Psi : [0,\infty) \to [0,\infty)\) be doubling functions such that
		\begin{align}\label{ineq:sub_g_escape_vol_constrain}
			\frac{R^2}{r^2} \lesssim \frac{\Psi(R)}{\Psi(r)} \lesssim \frac{R\,V(R)}{r\,V(r)},\qquad\text{uniformly in } 0<r\le R.
		\end{align}
		Then there exists an unbounded (that is, infinite diameter) MMD space
		\[
		(\X,d,m,\mathcal E,\mathcal F)
		\]
		that satisfies the full sub-Gaussian kernel estimate \(\mathrm{HKE}_{\rm full}(\Psi)\), and the metric measure space $(\X,d,m)$ has volume growth profile $V$.
		By tracing the proofs in \cite{murugan2024diffusions}, one sees that the metric space $(\X,d)$ here is a complete locally compact metric space (see \cite[Proofs of Theorem 2.7  and Proposition 3.18]{murugan2024diffusions}).
		Note that the doubling property of $\Psi$ alone is weaker than the scale function definition (with two-sided exponents $\beta_1,\beta_2$) given earlier in this subsection; the two-sided polynomial bounds for $\Psi$ follow from \eqref{ineq:sub_g_escape_vol_constrain} together with the doubling property and are not assumed a priori.
	\end{remark}
	
	\begin{example}[Polynomial sub-Gaussian heat kernels]\label{Eg:PolySubGaussianHK}
		In many standard examples we have $\Psi(r)\asymp r^{d_w}$ for some $d_w \ge 2$ and $V(r) \asymp r^{d_h}$ for some $d_h \ge 1$; in that case this setting corresponds to the usual sub-Gaussian heat kernel estimate. Examples include
		\begin{itemize}
			\item  $\R^n$ with the usual Gaussian heat kernel, in which case $d_h = n$ and $d_w = 2$. 
			\item  metric graphs with volume doubling property, with $d_h = 1$ and $d_w = 2$;  see \cite{MR4278798}. 
			\item (Unbounded) Sierpi\'{n}ski gasket in $\R^n$ with $d_h = \log_2(n+1)$ and $d_w = \log_2(n+3)$; see \cites{barlow1988brownian,barlow1998diffusions}.
		\end{itemize}
		
		Under these choices of $\Psi$ and $V$, the standing assumptions in Section~\ref{sec:backgroup_setups} are satisfied when the parameters meet appropriate conditions (see Appendix~\ref{sec:appendix}). In particular, for white noise ($\gamma=0$) one needs $d_h<d_w$; for colored noise $K_\gamma$, the condition reads
		\[
		\frac{d_h}{d_w} > \gamma > \max\brc{0,\frac{d_h}{d_w}-1}{,}
		\]
		which is consistent with the conditions for the resolvability of Parabolic Anderson models in \cite[Theorem 1.1]{Zbl08159034}  and the discussion thereafter.
	\end{example}
	\subsubsection{Metric measure spaces with stable-type heat kernel bound}
	
	Another class of examples to which our main theorems apply consists of stable-type heat kernels.
	For simplicity, we shall assume $(\X,m,d,\mathcal E,\mathcal F)$ satisfies the conclusion of Remark \ref{rem:sub-g-general} with $V(r) = r^{d_h}$ and $\Psi(r) = r^{d_w}$ for $1\le d_h\le d_w$.
	Let $\mathcal L$ be the associated operator, then for $\beta \in (0,1)$, the operator $-(-\mathcal L)^\beta$ admits heat kernel $(G_t^\beta(x,y))_{t>0,\,x,y \in \X}$ so that
	\begin{align*}
		G_t^\beta(x,y) \asymp t^{-\frac{d_h}{d_w \beta}} \pr{1+ \frac{d(x,y)}{t^{1/(\beta d_w)}}}^{-d_h - \beta d_w},\qquad \text{ for }t>0,\, x,y \in \X. 
	\end{align*}
	The proof of the following result is similar to that of Propositions~\ref{prop:sub_g_cond_no_noise}--\ref{prop:sub_g_strict_pos_condition} in Appendix~\ref{sec:appendix}, hence we omit its proof. 
	\begin{prop}\label{prop:fractionalPAM}
		Let $(\X,d,m,\mathcal E,\mathcal F)$, $\mathcal L$ and $1 \le d_h \le d_w$ be as above and let $\beta \in (d_h/d_w, 1)$.
		Fix an $x_0 \in \X$ and let $h(x):= 1+d(x_0,x)$ and let $\dot{W}$ be the space-time white noise on $\R_+\times \X$.
		For any $u_0 \in \mathcal C^{\mathrm{tem},+}_{\X,h}$ that is not identically zero, there exists a unique solution to
		\begin{equation}\label{eq:fractionalPAM}
			\partial_t u = -(-\mathcal L)^\beta u + u \dot{W},\qquad u(0,\cdot) = u_0.
		\end{equation}
		Furthermore, $\mathbb P \left(u(t,x)>0 \,\text{ for }\, t>0,\, x \in \X \right) =1$.
	\end{prop}
	
	    Take
	$\X=\R$ with Lebesgue measure, $\mathcal L=\Delta$,
	$d_h=1$, $d_w=2$, $\beta=\alpha/2$ with $\alpha\in(1,2)$, and the space-time white noise. Then \eqref{eq:fractionalPAM} is the
    fractional PAM studied in \cite{chen2017comparison}.    
	With $h(x)=1+|x|$, Proposition~\ref{prop:fractionalPAM} yields existence, uniqueness,  strong comparison principle
	and strict positivity for this equation for initial conditions in the space $\mathcal C^{\mathrm{tem}}_{\R,h}$, thereby extending some results of \cite{chen2017comparison}. 

    Figure~\ref{fig:pam-fractals} illustrates the fractional PAM  on the Sierpi\'{n}ski gasket.

\subsubsection{Riemannian manifolds with short-time Gaussian heat kernel estimates}\label{sec:local_manifolds}

Let $(M,g)$ be a complete connected smooth Riemannian manifold of dimension $n\geq1$, without boundary. Write $d=d_g$ for the distance induced by the Riemannian metric $g$, and $m$ for its Riemannian volume measure. By the Hopf--Rinow theorem \cite[Chapter~7, Theorem~2.8]{docarmo1992riemannian}, $(M,d)$ is geodesic: every two points are joined by a distance-minimizing geodesic.

Let $\Delta_M=\operatorname{div}_g\nabla_g$ be the Laplace--Beltrami operator, realized as a nonpositive self-adjoint operator on $L^2(M,m)$. Set $\mathcal L=\tfrac12\Delta_M$ and let $G_t(x,y)$ denote the smooth integral kernel of the heat semigroup $P_t=e^{t\mathcal L}$. Self-adjointness gives
\[
 G_t(x,y)=G_t(y,x),\qquad t>0,\quad x,y\in M.
\]
We assume the following uniform short-time two-sided Gaussian estimates: there are constants $c_0,C_0>0$ such that
\begin{align}
 c_0 t^{-n/2}e^{-C_0 d(x,y)^2/t}
 \leq G_t(x,y)\leq C_0 t^{-n/2}e^{-c_0 d(x,y)^2/t},
 \quad x,y\in M,\; t\in(0,1].
 \label{eq:manifold_gaussian}
\end{align}
This assumption allows exponential volume growth at large distances and does not require curvature bounds; examples are given below.

Fix $o\in M$, set $h(x)=\exp(d(o,x))$, and, for
\begin{align}
 \gamma>\max\{0,n/2-1\},
 \label{eq:manifold_gamma}
\end{align}
define
\begin{align}
 K(\mathrm dx,\mathrm dy)
 :=\frac1{\Gamma(\gamma)}
 \left(\int_0^\infty e^{-s}s^{\gamma-1}G_s(x,y)\,\mathrm ds\right)
 m(\mathrm dx)m(\mathrm dy).
 \label{eq:manifold_covariance}
\end{align}
The factor $e^{-s}$ makes this definition applicable on both compact and noncompact manifolds. The parabolic Anderson model on Riemannian manifolds, driven by Gaussian noises with covariances of this type, is studied in \cites{chen2025global, chen2026sharp} for compact manifolds and in \cite{baudoin2025moment} for Cartan--Hadamard manifolds.

It turns out that the geometric, semigroup and noise  assumptions of Theorems~\ref{thm:main2} and~\ref{thm:comparison} are satisfied under the above setting.
\begin{prop}\label{prop:local_manifolds}
Suppose $(\X,d,m)=(M,d_g,m)$ and $\mathcal L=\tfrac12\Delta_M$ are the Riemannian manifold and operator  described above  such that \eqref{eq:manifold_gaussian} holds.
Suppose the function $h$, the constant $\gamma$ and the measure $K$ are as in the paragraph containing \eqref{eq:manifold_gamma}-\eqref{eq:manifold_covariance}.
Then, $(M,d,m)$ satisfies \eqref{ineq:entropy bound}, $h$ satisfies \eqref{ineq:faster_than_ms} and \eqref{ineq:ref_regular}, the heat semigroup $(P_t)_{t\geq0}$ is Feller, $K$ is a symmetric, non-negative definite Radon measure, and \eqref{cond:ctem continuous}, \eqref{cond:time_holder}{,}
\eqref{cond:exp_hk_lamda}{, \eqref{cond:hk_uni_bound} and \eqref{ineq:instant_propagation}} are all satisfied.
\end{prop}

Consequently, for a stochastic basis and coefficients as in \eqref{eq:SB}--\eqref{eq:NC}, satisfying \eqref{eq:LiG} and \eqref{ineq:g_lip}, and every $f\in\mathcal C^{\mathrm{tem}}_{M,h}$, it holds that (by Theorem~\ref{thm:main2}) equation~\eqref{eq:spde} has a unique mild-solution. Furthermore, for the Parabolic Anderson model
\[
\partial_t u=\mathcal L u+u\dot W_t,\qquad u(0,\cdot)=u_0,
\]
the unique mild-solution satisfies $\mathbb P\pr{u(t,x)>0\text{ for all }t>0,\;x\in M}=1$ whenever $u_0\in\mathcal C^{\mathrm{tem},+}_{M,h}$ is not identically zero, by Theorem~\ref{thm:comparison}{\eqref{part:main2 strictly positivity}}.

The proof of Proposition \ref{prop:local_manifolds} will be given in the Appendix.

\begin{example}\label{ex:local_manifolds}
 By \cite[Theorem~2.7]{hebisch2001relation}, a sufficient criterion for \eqref{eq:manifold_gaussian} consists of \eqref{eq:manifold_local_volume} and a constant $C_P>0$ such that
 \[
  \int_{B_r(x)}|v-v_{B_r(x)}|^2\,dm
  \leq C_P r^2\int_{B_r(x)}|\nabla v|^2\,dm,
  \qquad v_B:=\frac1{m(B)}\int_B v\,dm,
 \]
 for every $x\in M$, $0<r\leq1$ and $v\in C_c^\infty(M)$. This includes the following manifolds.
 \begin{enumerate}
  \item Every compact connected smooth Riemannian manifold without boundary satisfies \eqref{eq:manifold_gaussian}; see \cite[Theorem~6.1, p.~428]{saloff1992uniformly}. Here $h$ is bounded and $\mathcal C^{\mathrm{tem}}_{M,h}=C(M)$.
  \item Real hyperbolic space $\mathbb H^n$, $n\geq2$, satisfies \eqref{eq:manifold_gaussian} by \cite[Theorem~6.1, p.~428]{saloff1992uniformly}.
         \item Let $M=\mathbb R^n$, let $g_{\mathrm E}$ be the Euclidean metric, and let $g=(g_{ij})$ be a smooth Riemannian metric on $M$ such that, for some $a\geq1$,
  \begin{align}
   a^{-1}g_{\mathrm E}\leq g\leq a g_{\mathrm E}.
   \label{eq:manifold_metric_comparison}
  \end{align}
  Set
  $d=d_g,\,m(\mathrm dx)=\sqrt{\det g(x)}\mathrm dx,
   \, \mathcal L=\tfrac12\operatorname{div}_g\nabla_g$. 
    By \eqref{eq:manifold_metric_comparison}, the global parabolic Harnack inequality for $g_{\mathrm E}$ transfers to $g$ by \cite[Theorem~5.5.9, pp.~168--169]{saloff2002aspects}. Its equivalence with two-sided Gaussian heat-kernel estimates follows from \cite[Theorems~5.5.1 and~5.5.3, pp.~157--161]{saloff2002aspects}. Then \eqref{eq:manifold_gaussian} follows from \eqref{eq:manifold_metric_comparison} and the fact that $m(B_r(x))\asymp r^n$.

   \end{enumerate}
\end{example}

\par

	\section{Preliminaries}

	This section establishes the preliminary results about stochastic integrals, heat kernels and random fields for our general setting. These results are necessary for the later analysis.
	Unless otherwise stated, we assume that $(\X,\mathcal{L}, \dot{W})$ is as in Theorem~\ref{thm:main2}. 

	We first state the following random-field version of the stochastic Fubini theorem (see \cite[Theorem~4.33]{da2014stochastic} and its proof), while recalling the notation in Lemma~\ref{lem:H0}.

	\begin{lemma}[Stochastic Fubini]\label{thm:stochastic_fubini}
		Suppose that $(\Omega, \mathscr{F}, (\mathscr{F}_t)_{t\geq 0}, \mathbb P, W)$ is a stochastic basis given as in \eqref{eq:SB}. 
		Let  $(E, \mathscr E, \nu)$
		be a $\sigma$-finite measure space and let $T\geq 0$. 
		Let $\mathcal P_T$ denote the restriction of $\mathcal P$ to $\Omega\times[0,T]$.
		Let $\phi:=(\phi(t,\cdot, \theta): t \ge 0, \theta \in E )$ be a predictable random field in $\mathcal H_{K}$
		such that the map
		\[
			(\omega,t,\theta)\longmapsto [\phi(\omega,t,\cdot,\theta)]_K
		\]
		is $\mathcal P_T\otimes\mathscr E$-measurable with values in $\widetilde{\mathcal H}_K$, and
		\begin{align} 
			\norm{\phi}_{L^1_{E} L^2_{\Omega\times[0,T]}\mathcal H_{K}}:=
			\int_{E} \E\left[\int_0^T \norm{\phi(t,\cdot, \theta)}^2_{\mathcal H_{K}} \mathrm dt\right]^{1/2} \nu(\mathrm d\theta) <\infty.
		\end{align} 
		Then, there exists a real-valued measurable function  $(\Phi_T(\omega,\theta))_{\omega\in \Omega, \theta \in E}$ on 
		$(\Omega \times E, \mathscr{F}_T \otimes \mathscr E)$ such that
		\[
		\Phi_T(\cdot, \theta)=\iint_0^T \phi(t,x,\theta)W(\mathrm dt,\mathrm dx), \quad \text{a.s.,}\quad \text{for $\nu$-almost every $\theta\in E$},
		\]
		and that
		\[
		\int_{E} \Phi_T(\cdot, \theta) \nu(\mathrm d\theta) = \iint_0^T \left( \int_{E} [\phi(t,x,\theta)]_K\,\nu(\mathrm d\theta)\right) W(\mathrm dt,\mathrm dx),\quad \text{a.s.}
		\]
		Here the inner integral on the right-hand side is a Bochner integral in $\widetilde{\mathcal H}_K$.
	\end{lemma}

	The following BDG-type estimate controls moments of the stochastic convolutions appearing throughout the paper.
	\begin{lemma}\label{cor:mtg+mink_cs}
		Suppose that $(\Omega, \mathscr{F}, (\mathscr{F}_t)_{t\geq 0}, \mathbb P, W)$ is a stochastic basis given as in \eqref{eq:SB}. 
		Then for every $p\geq 2$, uniformly for every $t\geq 0$ and real-valued predictable random field $(\Phi(s,y))_{s\geq 0, y\in \X}$ 
		such that 
		\[
		\|\Phi\|_{L_{[0,t]}^2 \mathcal H_{K} L^p_\Omega} :=\left(\int_0^t \left\|\left\|\Phi(s,\cdot)\right\|_{L^p_{\Omega}}\right\|_{\mathcal H_{K}}^2 \mathrm ds \right)^\frac{1}{2} < \infty, 
		\]
		we have
		\begin{align}\label{ineq:mtg+mink+cs}
			\left\|\iint_0^t \Phi(s,y)W(\mathrm ds,\mathrm dy)\right\|_{L^p_\Omega} 
			\lesssim \|\Phi\|_{L_{[0,t]}^2 \mathcal H_{K} L^p_\Omega}.
		\end{align}
	\end{lemma}

	The next lemma generalizes  \eqref{cond:exp_hk_lamda} and  \eqref{cond:time_holder}  up to an arbitrary given time.
	
	\begin{lemma}\label{lem:hk longtime}
		Suppose that $\lambda \in \R$ and 
		$\tilde \alpha \in (0,1)$ satisfy \eqref{cond:exp_hk_lamda} with $\alpha$ being replaced by $\tilde \alpha$,
		then there exist $\kappa, \tilde \kappa >0$ so that 
		uniformly for every $t>0$ and $x\in \X$,
		\begin{align}\label{ineq:hk_m long time}
			\| G_t(\cdot,x)h(\cdot)^\lambda\|_{L^1_\X} \lesssim e^{\kappa t}h(x)^\lambda,
		\end{align}
		and
		\begin{align}\label{ineq:hk_k long time}
			\| G_t(\cdot,x)h(\cdot)^\lambda\|_{\mathcal H_{K}}^2 
			\lesssim t^{-\tilde \alpha}
			e^{\tilde \kappa t}h(x)^{2\lambda}.
		\end{align}
		Let $\beta\in [0,1)$ and $\xi_1,\xi_2>0$ be given as in \eqref{cond:time_holder}. For each $T>0$, it holds uniformly in $t,t' \in (0,T]$ and $x \in \X$ that
		\begin{align}\label{ineq:time_holder_long_time}
			\norm{G_t(\cdot,x) - G_{t'}(\cdot,x)}_{L^1_\X} \lesssim  \frac{\abs{t-t'}^{\xi_1 \wedge 1}}{\pr{t\wedge t'}^\beta}.
		\end{align}
		Moreover, it holds uniformly in $t\in(0,T]$ and $x,x'\in\X$ with $d(x,x')\le1$ that
		\begin{align}\label{ineq:spatial_holder_long_time}
			\norm{G_t(\cdot,x)-G_t(\cdot,x')}_{L^1_\X}
			\lesssim \frac{d(x,x')^{\xi_2}}{t^\beta}.
		\end{align}
	\end{lemma}
	
	In the rest of the paper, we write $G(s,y; t, x):= G_{t-s}(y,x)$ and $G(t,y,x):=G_t(y,x)$ for every $(s,y)$ and $(t,x)$ in $\mathbb R_+\times \X$. 
	Pay attention that $G(s,\cdot; t, \cdot) \equiv 0$ if $s\geq t$.
	For any real-valued measurable function $(f(s,y))_{(s,y)\in \mathbb R_+\times \X}$, we define (semi)norms
	\[ 
	\| f \|_{L^1_{[0,t)} L^1_\X}:= \int_0^{t} \mathrm ds \int_\X |f(s,y) | m(\mathrm d y), \quad t \in [0,\infty) \cup \{\infty\},
	\]
	and
	\[
	\|f\|_{L^2_{[0,t)}\mathcal H_{K}}:= \left(\int_0^t \| f(s, \cdot) \|_{\mathcal H_{K}}^2 \mathrm ds \right)^{1/2}, \quad t\in [0,\infty)\cup\{\infty\}.
	\]
	The next lemma gives the increment estimates used to establish the continuity of the solutions.
	\begin{lemma} \label{lem:hk_abs}
		Let $\xi_1,\, \xi_2 \in (0,1]$ be as in \eqref{cond:time_holder}, then
		for any $T>0$ and $\lambda \in [0,1/2]$, it holds uniformly for 
		$t,t'$ in $[0,T]$
		and  $x,x' \in \X$ with $d(x,x') \le 1$,
		that 
		\begin{align}
			h(x)^{-\lambda}\| h(\cdot)^{\lambda}\pr{G(\cdot,\cdot;t,x) - G(\cdot, \cdot; t', x')} \|_{L^1_{[0,T)} L^1_\X}
			\lesssim  |t-t'|^{{{\xi_1 \wedge 1}\over 2}} + d(x,x')^{{\xi_2 \over 2}\wedge 1}. \label{ineq:hk abs holder m}
		\end{align}
		Let $\alpha\in (0,1)$ and $\beta \in [0,1)$ be as in \eqref{cond:exp_hk_lamda} and \eqref{cond:time_holder} respectively.  
		There exist $\widetilde \xi_1,\widetilde \xi_2 \in (0,1)$, so that for every $\lambda_1\in[0,1/2]$ and $T>0$,
		it holds uniformly for $t,t'$ in $[0,T]$ and  $x,x' \in \X$ with $d(x,x') \le 1$ that
		\begin{align}
			h(x)^{-2\lambda_1}\| h(\cdot)^{\lambda_1}\abs{G(\cdot,\cdot;t,x) - G(\cdot, \cdot; t', x') }\|^2_{L^2_{[0,T)}\mathcal H_{K}}
			\lesssim |t-t'|^{2\widetilde \xi_1} + d(x,x')^{2\widetilde \xi_2}. \label{ineq:hk abs holder k}
		\end{align}
		Moreover, there exist $\xi\in(0,1]$ and $\alpha_0\in(0,1)$ such that, for every $\lambda\in[0,1/2]$ and $T>0$, uniformly for $x,x'\in\X$ with $d(x,x')\le1$ and $0\le t\le t'\le T$,
		\begin{align}
			\int_0^{t'}\left\|\left|G_{t-s}(\cdot,x)-G_{t'-s}(\cdot,x')\right|h(\cdot)^\lambda\right\|_{\mathcal H_K}^2\,\mathrm ds
			\lesssim t'^{1-\alpha_0}\left(|t-t'|^\xi+d(x,x')^\xi\right)h(x)^{2\lambda}.
			\label{ineq:hk abs holder k weighted}
		\end{align}
	\end{lemma}

	We equip $\mathcal C_{\R_+}\mathcal C^{\mathrm{tem}}_{\X,h}$ with its compact-open topology. 
	The following lemma gives criteria for the existence of a $\mathcal C_{\R_+}\mathcal C^{\mathrm{tem}}_{\X,h}$-valued modification of a given random field.

	\begin{lemma}\label{lem:tight_tem_moment}
		Let $X$ be a real-valued random field on  $\R_+ \times \X$.
		Suppose
		that for each  $\lambda\in (0,1]$, there exists $p_0 = p_0(\lambda) \ge 2$ so that for all $p \ge p_0$, it holds uniformly in $x\in \X$ that
		\begin{align}
			\E\left[|X(0,x)|^p\right] \lesssim h(x)^{\lambda}. 
		\end{align}
		Suppose further that there exist
		$\tilde \xi_1,\tilde \xi_2 \in (0,1]$
		so that for each $T>0$, $\lambda \in (0,1]$ and $p \ge 2$,  uniformly for 
		$t,t'$ in $[0,T]$ with $|t-t'| \le 1$,
		$x \in \X$ and $y_1,y_2 \in B_1(x)$,
		\begin{align}\label{cond:tight_moment}
			\E\left[\left| X(t,y_1) - X(t',y_2) \right| ^p\right] \lesssim  \left( |t'-t|^{\tilde \xi_1} + d(y_1,y_2)^{\tilde \xi_2}\right)^p h(x)^{\lambda}.
		\end{align}
		Then,
		the random field $X$ admits a continuous modification $\tilde X$ taking values in $\mathcal C_{\mathbb R_+}\mathcal C^{\mathrm{tem}}_{\X,h}$.
	\end{lemma}

	\section{Proof of strong well-posedness}\label{sec:strong well posed}

	This section is dedicated to the proof of Theorem \ref{thm:main2}.
	To do this, let the stochastic basis $(\Omega, \mathscr{F}, (\mathscr{F}_t)_{t\geq 0}, \mathbb P, W)$, drift coefficient $b$, noise coefficient 
	$\sigma$, and the initial value $f$ be given as in \eqref{eq:SB}--\eqref{eq:IV}. 
	Suppose that the linear growth condition \eqref{eq:LiG} and the Lipschitz condition \eqref{ineq:g_lip} hold.
	Let $T>0$ be arbitrary.
	For any two predictable random fields $(X(\omega, t, x))_{\omega\in \Omega, t\in [0,T], x\in \X}$ and $(X'(\omega, t, x))_{\omega \in \Omega, t\in [0,T], x\in \X}$, we say $X'$ is a predictable modification of $X$ if, for every $x\in \X$, $X(\omega, t, x) = X'(\omega, t, x)$ holds for $(\mathbb P\otimes \Lambda)$-a.e.~ $(\omega,t)\in \Omega \times [0,T]$; 
	we say $X'$ is a modification of $X$ if, for every $(t,x)\in [0,T] \times \X$, $X(\omega, t, x) = X'(\omega, t, x)$ holds for $\mathbb P$-a.s.~ $\omega\in \Omega$.
	Here, $\Lambda$ represents the Lebesgue measure on the  real line, and $\mathbb P\otimes \Lambda$ represents the product measure of $\mathbb P$ and $\Lambda$.
	Clearly, 
	\begin{itemize}
		\item[\eq\label{eq:Modification}] if $X$ and $X'$ are predictable random fields on $[0,T]\times \X$ such that $X'$ is a modification of $X$, then $X'$ is a predictable modification of $X$. 
	\end{itemize}
	
	We say $(B, \Sigma)$ is a Duhamel-integrable pair if 
	\begin{itemize}
		\item $B$ and $\Sigma$ are real-valued predictable random fields on $[0,T] \times \X$; and
		\item for any $(t,x)\in [0,T]\times \X$, almost surely
		\[
		\|G(\cdot, \cdot; t,x)B\|_{L^1_{[0,T]}L^1_\X} + \|G(\cdot, \cdot; t,x)\Sigma\|^2_{L^2_{[0,T]}\mathcal H_K} < \infty.
		\]
	\end{itemize}
	In particular, noticing \eqref{cond:ctem continuous}, for any $g\in \mathcal C^{\mathrm{tem}}_{\X,h}$, Duhamel-integrable pair $(B, \Sigma)$, and $(t,x)\in [0,T] \times \X$,
	the following random variable 
	\begin{equation}
		\Psi(g,B,\Sigma)(t,x) := 
		\Psi_0(g)(t,x)
		+ \Psi_1(B)(t,x) + \Psi_2(\Sigma)(t,x) \label{eq:Psi}
	\end{equation}
	is  (almost surely) well-defined where
	\begin{align} 
		& \Psi_0(g)(t,x):= (P_t^* g)(x),  \label{eq:Psi_0}
		\\&\Psi_1(B)(t,x):=\int_0^t \mathrm ds \int_\X G(s,y;t,x) B(s,y) m(\mathrm dy),
		\\&\Psi_2(\Sigma)(t,x):= \iint_0^t G(s,y;t,x) \Sigma(s,y)W(\mathrm ds,\mathrm dy).
	\end{align}

    The next lemma ensures that the Duhamel terms depend only on the predictable modifications of their integrands.
	\begin{lemma}\label{lem:PsiM}
		Suppose that both $(B, \Sigma)$ and $(B', \Sigma')$ are Duhamel-integrable pairs. Then 
		\begin{itemize}
			\item $\Psi_1(B')$ is a modification of $\Psi_1(B)$ provided $B'$ is a predictable modification of $B$; 
			\item $\Psi_2(\Sigma')$ is a modification of $\Psi_2(\Sigma)$ provided $\Sigma'$ is a predictable modification of $\Sigma$.
		\end{itemize}
	\end{lemma}

	Fix arbitrary $\lambda \in \mathbb R$ and $p\geq 1$ in this paragraph.
	Denote by $\mathbb M_{\lambda, p}$ the collection of real-valued predictable random fields $X$ on $[0,T]\times \X$ such that 
	\[
	\|X\|_{\mathbb M_{\lambda, p}(T)}:= 
	\|X\|_{\mathbb M_{\lambda,p}}
	:= \sup_{x \in \X} \esssup_{\Lambda\text{-a.e.~}t \in [0,T]}  
	h(x)^{-\lambda}
	\norm{X(t,x)}_{L^p_\Omega} <\infty.
	\]
	For every $X\in \mathbb M_{\lambda, p}$, denote by 
	$[X]_{\mathrm{p}}$
	the collection of predictable modifications of $X$.
	It is clear that $\mathbb M_{\lambda, p}$ is a vector space and that $[0]_{\mathrm{p}}$ is a  
	subspace of $\mathbb M_{\lambda, p}$. Denote by 
	\[
	\mathbb M_{\lambda,p}^{\sharp} := \mathbb M_{\lambda,p}/[0]_{\mathrm{p}}=\brc{[X]_{\mathrm{p}}: X \in  \mathbb M_{\lambda,p}}
	\] 
	the quotient of $\mathbb M_{\lambda,p}$ by the equivalence relation of  predictable modification.
	Note also that, for every $A\geq 0$ and $X\in \mathbb M_{\lambda, p}$,
	\begin{align*}
		\|X\|_{\mathbb M_{\lambda,p,A}}:= \sup_{x \in \X} \esssup_{\Lambda\text{-a.e.~} t \in [0,T]} e^{-A t}h(x)^{-\lambda} \|X(t,x)\|_{L^p_\Omega} < \infty.
	\end{align*}
	For every $A\geq 0$, noticing that $\|X\|_{\mathbb{M}_{\lambda,p,A}} = \|X'\|_{\mathbb{M}_{\lambda,p,A}}$ for every $X,X' \in \mathbb M_{\lambda,p}$ provided $[X']_{\mathrm{p}}=[X]_{\mathrm{p}}$, there exists a unique map $\|\cdot\|_{\mathbb M^\sharp_{\lambda,p,A}}$ from $\mathbb M^\sharp_{\lambda,p}$ to $[0,\infty)$ such that
	\begin{equation} \label{eq:NormSharp}
		\| [X]_{\mathrm{p}} \|_{\mathbb M_{\lambda,p,A}^{\sharp}} = \|X\|_{\mathbb{M}_{\lambda,p,A}}, \quad X\in \mathbb{M}_{\lambda,p}.
	\end{equation}
	Denote by $\mathbb M^f_{\lambda, p}$ the collection of $X\in \mathbb M_{\lambda, p}$ satisfying that $X(0,x)=f(x)$ almost surely for every $x\in \X$; and define $\mathbb M^{\sharp,f}_{\lambda, p}:= \{[X]_{\mathrm{p}}: X\in \mathbb M^f_{\lambda, p}\}$. 
	For each $A\geq 0$, define a map $d_{\lambda, p, A}$ from $\mathbb M^{\sharp,f}_{\lambda, p} \times \mathbb M^{\sharp,f}_{\lambda, p}$ to $[0,\infty)$ such that 
	\[
	d_{\lambda, p, A}([X]_{\mathrm{p}}, [X']_{\mathrm{p}}) = \norm{[X]_{\mathrm{p}} - [X']_{\mathrm{p}}}_{\mathbb M^\sharp_{\lambda, p, A}}.
	\]
	
    The following completeness result provides the metric space used in the fixed-point construction of the solution.
	\begin{lemma} \label{lem:CMS} 
		For every
		$\lambda \in \mathbb R$, $p \ge 1$ and $A\geq 0$, $(\mathbb M_{\lambda,p}^{\sharp}, \|\cdot\|_{\mathbb M_{\lambda,p,A}^{\sharp}})$ is a Banach space. 
		In particular, $(\mathbb M_{\lambda,p}^{\sharp,f}, d_{\lambda,p,A})$ is a complete metric space.
	\end{lemma}
	
	    The following estimates show that the mild-solution map is well-defined and becomes a contraction after introducing an exponential time weight.
	\begin{lemma}\label{lem:b_contract_combined}
		Let $\lambda \ge0$ and $\tilde \alpha \in (0,1)$ satisfy \eqref{cond:exp_hk_lamda} with $\alpha$ being replaced by $\tilde \alpha$.
		Let $p \ge 2$ and $g \in \mathcal C^{\mathrm{tem}}_{\X,h}$. When $\lambda=0$, assume in addition that $g\in\mathcal C_\X^{\mathrm b}$. For each $U \in \mathbb M_{\lambda,p}$ and $(t,x) \in [0,T]\times \X$, define
		\begin{equation}\label{eq:BuSu}
			B_U(t,x):= b(t,x,U(t,x)),\qquad \Sigma_U(t,x):= \sigma(t,x,U(t,x)).
		\end{equation}
		Then  for each $U\in \mathbb M_{\lambda,p}$, $(B_U,\Sigma_U)$ is a Duhamel-integrable pair, and  $\Psi(g,B_U,\Sigma_U) \in \mathbb M_{\lambda,p}$.
		Moreover,
		uniformly in $ U,\, V \in \mathbb M_{\lambda, p }$, $A\ge 1$ and $g \in \mathcal C^{\mathrm{tem}}_{\X,h}$, we have 
		\begin{align}
			\norm{\Psi(g,B_U,\Sigma_U) - \Psi(g,B_V, \Sigma_V)}_{\mathbb M_{\lambda, p, A}} 
			\lesssim  A^\frac{\tilde \alpha - 1}{2} 
			\norm{U - V}_{\mathbb M_{\lambda, p, A}}.
		\end{align}
	\end{lemma}
	
	Thanks to \eqref{cond:exp_hk_lamda} and Lemma \ref{lem:b_contract_combined}, we can define
	\begin{itemize}
		\item[\eq\label{eq:pp}] for any $\lambda \in (0,1], p\geq 2$ (where $\lambda=0$ is also allowed if $f\in\mathcal C^{\mathrm b}_\X$) and $U\in \mathbb M_{\lambda,p}$,
		$\Phi(U) := \Psi(f,B_U,\Sigma_U) \in \mathbb M_{\lambda,p}$
		where $B_U,\, \Sigma_U$  are given by \eqref{eq:BuSu}.
	\end{itemize}

	\begin{prop}\label{prop:wellposted}
		Suppose that  $\lambda \in (0,1]$ (or $\lambda=0$ if $f\in\mathcal C^{\mathrm b}_\X$) and $p \ge 2$. Then there exists a $U \in \mathbb M_{\lambda,p}^f$, unique up to predictable modification,
		such that $\Phi(U)$ is a  predictable modification of $U$.
	\end{prop}
	\begin{proof}
		From \eqref{eq:pp}, we know that $\Phi(U) \in \mathbb M_{\lambda,p}$.
		Note that for any $U \in \mathbb M_{\lambda,p}^{ f}$,
		\begin{equation}
			\Phi(U)(0,x)
			= \Psi(f, B_U, \Sigma_U)(0,x) 
			= f(x), \quad x\in \X.
		\end{equation}
		Therefore, 
		$\Phi: U \mapsto \Phi(U)$ is a map from $\mathbb M_{\lambda,p}^{f}$ to $\mathbb M_{\lambda,p}^{f}$.
		From Lemma \ref{lem:PsiM} and \eqref{eq:Modification} we can verify that, for any $U \in \mathbb M_{\lambda,p}^{f}$
		and any 
		predictable
		modification $U'$ of $U$, $\Phi(U')$ is a 
		predictable
		modification of 
		$\Phi(U)$.
		Therefore, we can define a map $\Phi^\sharp$ from $\mathbb M_{\lambda,p}^{\sharp,f}$ to $\mathbb M_{\lambda,p}^{\sharp,f}$ such that 
		\begin{equation}\label{eq:Phisharp}
			\Phi^\sharp([U]_{\mathrm{p}}) = [\Phi(U)]_{\mathrm{p}}, \quad \forall U\in \mathbb M^f_{\lambda, p}.
		\end{equation}
		Recall that $\alpha\in (0,1)$ is given by \eqref{cond:exp_hk_lamda}.
		Note that uniformly for $A\ge 1$ and $U,V \in \mathbb M_{\lambda,p}^f$,
		\begin{align}\label{eq:PhiD}
			&\|\Phi(U) - \Phi(V)\|_{\mathbb{M}_{\lambda,p,A}} 
			=\norm{\Psi(f,B_U,\Sigma_U) - \Psi(f, B_V, \Sigma_V)}_{\mathbb M_{\lambda, p, A}} 
			\\&\overset{\text{Lemma \ref{lem:b_contract_combined}}}\lesssim A^{\frac{\alpha - 1}{2}} \|U - V\|_{\mathbb{M}_{\lambda,p,A}}. 
		\end{align}
		From the fact that $\alpha \in (0,1)$, there exists a large $A_0>0$ such that for any $A\geq A_0$ and $U,V \in \mathbb M_{\lambda,p}^{f}$,
		\begin{equation} \label{eq:PC}
			\|\Phi(U) - \Phi(V)\|_{\mathbb{M}_{\lambda,p,A}} \leq \frac{1}{2}\|U-V\|_{\mathbb{M}_{\lambda,p,A}}.
		\end{equation}
		
		Now, for any $A\geq A_0$ and $\mathbf U, \mathbf V$ in $\mathbb M_{\lambda,p}^{\sharp,f}$, taking representatives $U\in \mathbf U\cap\mathbb M_{\lambda,p}^f$ and $V\in \mathbf V\cap\mathbb M_{\lambda,p}^f$ respectively, it holds that 
		\begin{align}
			&d_{\lambda,p, A}\pr{\Phi^\sharp(\mathbf U), \Phi^\sharp(\mathbf V)} 
			= \norm{\Phi^\sharp([U]_{\mathrm{p}}) - \Phi^\sharp([V]_{\mathrm{p}})}_{\mathbb{M}_{\lambda,p,A}^{\sharp}} 
			\\&\overset{\eqref{eq:Phisharp}}{=} \norm{ [\Phi(U)]_{\mathrm{p}} - [\Phi(V)]_{\mathrm{p}} }_{\mathbb{M}_{\lambda,p,A}^{\sharp}} 
			=\norm{ [\Phi(U) - \Phi(V)]_{\mathrm{p}} }_{\mathbb{M}_{\lambda,p,A}^{\sharp}} 
			\\&\overset{\eqref{eq:NormSharp}}{=} \left\| \Phi(U) - \Phi(V)\right\|_{\mathbb{M}_{\lambda,p,A}}  
			\overset{\eqref{eq:PC}}\leq \frac{1}{2}\left\|U-V\right\|_{\mathbb{M}_{\lambda,p,A}}
			\overset{\eqref{eq:NormSharp}}{=} \frac{1}{2}\left\|[U-V]_{\mathrm{p}}\right\|_{\mathbb{M}_{\lambda,p,A}^\sharp}
			\\&= \frac{1}{2}\left\|[U]_{\mathrm{p}}-[V]_{\mathrm{p}}\right\|_{\mathbb{M}_{\lambda,p,A}^\sharp}
			=  \frac{1}{2}d_{\lambda,p, A}\pr{\mathbf U, \mathbf V} .
		\end{align}
		
		Now noticing Lemma \ref{lem:CMS}, $\Phi^\sharp$ is a contraction map from the complete metric space $(\mathbb{M}_{\lambda,p}^{\sharp,f}, d_{\lambda,p,A})$ to $( \mathbb{M}_{\lambda,p}^{\sharp,f}, d_{\lambda,p,A})$ for every $A\geq A_0$. 
		By the Banach fixed-point theorem, 
		\begin{itemize}
			\item[\eq\label{eq:UMf}] there exists a unique element $\mathbf U$ in $\mathbb{M}_{\lambda,p}^{\sharp,f}$ such that $\Phi^\sharp(\mathbf U) = \mathbf U$.
		\end{itemize}
		Let $U\in \mathbf U\cap \mathbb M_{\lambda,p}^f$ be a representative of the $\mathbf U$ in \eqref{eq:UMf}, then by \eqref{eq:Phisharp}, $[\Phi(U)]_{\mathrm{p}} = [U]_{\mathrm{p}}$. This proves the existence part of the desired result.
		Suppose that there is any other $U'\in \mathbb M_{\lambda,p}^f$ such that $[\Phi(U')]_{\mathrm{p}} = [U']_{\mathrm{p}}$, then by \eqref{eq:Phisharp}, $\Phi^\sharp([U']_{\mathrm{p}}) = [U']_{\mathrm{p}}$, and by \eqref{eq:UMf}, $[U']_{\mathrm{p}} = [U]_{\mathrm{p}}$.
		This proves the uniqueness part of the desired result. We are done.
	\end{proof}
	
	Pay attention that the space $\mathbb M^{f}_{\lambda, p}$ increases in $\lambda$, and decreases in $p$, i.e.,
	\begin{itemize}
		\item[\eq\label{eq:MM}] for any $\lambda_1 \leq \lambda_0$ in $(0,\infty)$ and $p_1 \geq p_0$ in $[1,\infty)$, we have  $\mathbb M^{f}_{\lambda_1, p_1} \subset \mathbb M^{f}_{\lambda_0, p_0}$.
	\end{itemize}
	In light of this, we define
	\begin{equation}
		\mathbb M^f:= \lim_{\lambda \downarrow 0, p \uparrow \infty} \mathbb M^{f}_{\lambda, p} := \bigcap_{\lambda \in (0,1], p \geq 2} \mathbb M^{f}_{\lambda, p}.
	\end{equation}
	
	\begin{corollary}\label{cor:MfT}
		There exists  $U \in \mathbb M^{f}$, unique up to  predictable modification, such that $\Phi(U)$ is a predictable modification of  $U$. If, in addition, $f\in\mathcal C^{\mathrm b}_\X$, then $U\in\mathbb M^f_{0,p}$ for every $p\geq 2$.
	\end{corollary}
	\begin{proof}
		Let $(\lambda_n)_{n\in \mathbb N}$ be a sequence in $(0,1]$ such that $\lambda_n \downarrow 0$, and let $(p_n)_{n\in \mathbb N}$ be a sequence in $[2,\infty)$ such that $p_n \uparrow \infty$. 
		From \eqref{eq:MM}, we know that $\mathbb M^{f}_{\lambda_{n+1}, p_{n+1}}\subset \mathbb M^{f}_{\lambda_{n}, p_{n}}$ for each $n\in \mathbb N$.
		From Proposition \ref{prop:wellposted} for each $n\in \mathbb N$, there exists a 
		$U^{(n)}\in \mathbb M^{f}_{\lambda_n, p_n}$, unique up to  predictable modification, such that $\Phi(U^{(n)})$ is a  predictable modification of $U^{(n)}$. 
		For each $n\in \mathbb N$, since $U^{(n+1)}\in \mathbb M^{f}_{\lambda_{n+1}, p_{n+1}} \subset \mathbb M^{f}_{\lambda_{n}, p_{n}}$ and $\Phi(U^{(n+1)})$ is a  predictable modification of $U^{(n+1)}$, it follows that $U^{(n+1)}$ is a  predictable modification of $U^{(n)}$. 
		From this and induction, $U^{(n)}$ is a  predictable modification of $U:= U^{(1)}$ for each $n\in \mathbb N$.

		For any $\lambda \in (0,1]$ and $p\geq 2$, by taking $n\in \mathbb N$ large enough such that $\lambda_n \leq \lambda$ and $p_n \geq p$, we have $U \in \mathbb M^{f}_{\lambda_{n}, p_{n}} \subset \mathbb M^{f}_{\lambda, p}$. 
		Therefore, $U \in \mathbb M^{f}$ and $[U]_{\mathrm p} = [\Phi(U)]_{\mathrm p}$.
		This proves the existence part of the desired result.
		Suppose that there is any other $U' \in \mathbb M^f$ satisfying that $[U']_{\mathrm p} = [\Phi(U')]_{\mathrm p}$, then from the fact that $\mathbb M^f \subset \mathbb M^f_{\lambda_1, p_1}$, it must hold that $[U']_{\mathrm p} = [U^{(1)}]_{\mathrm p} = [U]_{\mathrm p}$.
		This proves the uniqueness part of the desired result.

		Finally, suppose that $f\in\mathcal C^{\mathrm b}_\X$ and let $p\geq 2$. By Proposition~\ref{prop:wellposted} with $\lambda=0$, there exists $U_0\in\mathbb M^f_{0,p}$ such that $\Phi(U_0)$ is a predictable modification of $U_0$. Since $h\geq1$, both $U_0$ and $U$ belong to $\mathbb M^f_{1,p}$, so the uniqueness part of Proposition~\ref{prop:wellposted} with $\lambda=1$ (and the same $p$) shows that $U$ is a predictable modification of $U_0$. Hence, by Fubini's theorem, for every $x\in\X$, $\norm{U(t,x)}_{L^p_\Omega}=\norm{U_0(t,x)}_{L^p_\Omega}$ for $\Lambda$-a.e.\ $t\in[0,T]$, and thus $\norm{U}_{\mathbb M_{0,p}}=\norm{U_0}_{\mathbb M_{0,p}}<\infty$.
	\end{proof}
	
	\begin{prop}\label{prop:existence of C_tem modification}
		Suppose that $U \in \mathbb M^{f}$ is  the random field given as in Corollary  \ref{cor:MfT},
		then $U$ admits a  predictable modification $\tilde U$ taking values in $\mathcal C_{[0,T]}\mathcal C^{\mathrm{tem}}_{\X,h}$.  In addition, $\Phi(U)$ admits a modification in $\mathcal C_{[0,T]}\mathcal C^{\mathrm{tem}}_{\X,h}$ and  there exists $\xi \in (0,1)$ so that for any $T>0$, $p \ge 2$ and $\lambda \in (0,1]$, it holds uniformly in $t,t'\in [0,T]$, $x,x' \in \X$ with $d(x,x') \le 1$, 
		\begin{align}
			&\norm{\Psi_1(B_U)(t,x) - \Psi_1(B_U)(t',x')}_{L^p_\Omega}+\norm{\Psi_2(\Sigma_U)(t,x) - \Psi_2(\Sigma_U)(t',x')}_{L^p_\Omega} \\
			&\lesssim \pr{|t-t'|^\xi +d(x,x')^\xi }h(x)^\lambda.
		\end{align}
	\end{prop}
	
	\begin{proof}
		\firststep
		Notice that $U$ is the predictable modification of 
		\begin{equation}\label{eq:PhiU}
			\Phi(U) := \Psi_0(f) + \Psi_1(B_U) + \Psi_2(\Sigma_U).
		\end{equation}
		To prove the desired result for this proposition, we only need to show that each term on the right hand side of \eqref{eq:PhiU} admits a modification taking values in $\mathcal C_{[0,T]}\mathcal C^{\mathrm{tem}}_{\X,h}$.
		In fact, if this has been shown, then $\Phi(U)$ admits a modification $\tilde \Phi(U)$ taking values in $\mathcal C_{[0,T]}\mathcal C^{\mathrm{tem}}_{\X,h}$. 
		And since the continuous random field $\tilde \Phi(U)$ is predictable, by \eqref{eq:Modification}, it is also a predictable modification of both $\Phi(U)$ and $U$.

		\nextstep
		From \eqref{cond:ctem continuous}, we know that $\Psi_0(f)$ takes values in $\mathcal C_{[0,T]}\mathcal C^{\mathrm{tem}}_{\X,h}$ already.
		
		\nextstep
		We will show that $\Psi_1(B_U)$ admits a 
		modification taking values in $\mathcal C_{[0,T]} \mathcal C^{\mathrm{tem}}_{\X,h}$.
		Let $\xi_1, \xi_2$
		be the constants given as in 
		\eqref{cond:time_holder}.
		Let $\lambda_0 \in [0,1]$ (where $\lambda_0=0$ is allowed only if $f\in\mathcal C^{\mathrm b}_\X$, in which case $U\in\mathbb M_{0,p}$ by Corollary~\ref{cor:MfT}) and $p \geq 2$ be arbitrary. 
		Define $\lambda := \lambda_0 / p \in [0,1/2]$. 
		Note that uniformly  for $(s,y)\in [0,T]\times \X$ and $\omega \in \Omega\setminus \mathrm N_{\ref{c:Lip}}$, 
		\begin{align}
			\abs{B_U(s,y)} + \abs{\Sigma_U(s,y)}&= \abs{b(s,y, U(s,y))} 
			+ \abs{\sigma(s,y,U(s,y))}
			\\&{\overset{\eqref{eq:LiG}}\lesssim} 1 + |U(s,y)|. \label{eq:BUles}
		\end{align}
		Therefore, we have the existence of a $\Lambda\otimes m$-null subset $\mathrm N_{\counter\label{c:BUh}}$ of $[0,T]\times \X$ such that uniformly for every $(s,y)\in [0,T]\times \X\setminus \mathrm N_{\ref{c:BUh}}$, 
		\begin{align}
			&\norm{B_U(s,y)}_{L^p_\Omega} 
			\overset{\eqref{eq:BUles}}{\lesssim} \norm{1 + |U(s,y)|}_{L^p_\Omega} 
			\overset{\text{Minkowski}}\leq 1 + \norm{U(s,y)}_{L^p_\Omega}
			\\& \leq h(y)^{\lambda} + \norm{U}_{\mathbb M_{\lambda, p}} h(y)^{\lambda}
			\lesssim h(y)^{\lambda}. \label{eq:BUp}
		\end{align}
		
		Uniformly for every $t,t'\in [0,T]$ and $x,x'\in \X$ with $d(x,x') \le 1$,
		\begin{align}
			& \mathbb E\left[\left|\Psi_1(B_U)(t,x) - \Psi_1 (B_U)(t',x')\right|^p\right]^{\frac{1}{p}} = \left\|\Psi_1(B_U)(t,x) - \Psi_1 (B_U)(t',x')\right\|_{L^p_\Omega} 
			\\&= \left\| \int_0^T \mathrm ds \int_\X \left(G(s,y;t,x) - G(s,y;t',x')\right) B_U(s,y) m(\mathrm dy)\right\|_{L^p_\Omega} 
			\\&\overset{\text{Minkowski}}\le  \left\| \left| G(\cdot,\cdot;t,x) - G(\cdot,\cdot;t',x')\right| \left\|B_U(\cdot,\cdot)\right\|_{L^p_\Omega} \right\|_{L^1_{[0,T]\times\X}}
			\\&\overset{\eqref{eq:BUp}} \lesssim \norm{\abs{G(\cdot,\cdot;t,x) - G(\cdot,\cdot;t',x')}h(\cdot )^\lambda }_{L^1_{[0,T]\times \X}} 
			\\&\overset{\text{Lemma \ref{lem:hk_abs} }}\lesssim  \pr{\abs{t-t'}^{\pr{\xi_1 \wedge 1}\over 2} +  d(x,x')^{\frac{\xi_2}{2}\wedge 1} }h(x)^\lambda 
			\\&= \pr{\abs{t-t'}^{\pr{\xi_1 \wedge 1}\over 2} +  d(x,x')^{\frac{\xi_2}{2}\wedge 1} }h(x)^{\lambda_0/p}     
		\end{align}
		Since $\Psi_1(B_U)(0,\cdot) \equiv 0$ is deterministic, and since $p\geq 2$ and $\lambda_0 \in (0,1]$ are chosen arbitrarily, from Lemma \ref{lem:tight_tem_moment}, $\Psi_1(B_U)$ admits a modification taking values in  $\mathcal C_{[0,T]} \mathcal C^{\mathrm{tem}}_{\X,h}$ as desired for this step.
		
		\nextstep We will show that $\Psi_2(\Sigma_U)$ admits a predictable modification taking values in $\mathcal C_{[0,T]} \mathcal C^{\mathrm{tem}}_{\X,h}$.
										Let $\lambda_0 \in [0,1]$ (where $\lambda_0=0$ is allowed only if $f\in\mathcal C^{\mathrm b}_\X$, in which case $U\in\mathbb M_{0,p}$ by Corollary~\ref{cor:MfT}) and $p\geq 2$ be arbitrary.
		Take $\lambda_1:=\lambda_0/p\in[0,1/2]$ with $p\lambda_1 \leq \lambda_0$.
		Pay attention that \eqref{cond:exp_hk_lamda} holds with $\lambda$ being replaced by $\lambda_1$.
		
		Uniformly for every $(t,x)$ and $(t',x')$ in $[0, T]\times \X$, we have
		\begin{align}
			&\mathbb E\left[\left|\Psi_2(\Sigma_U)(t,x) - \Psi_2(\Sigma_U)(t',x')\right|^p \right]^{1/p}
			\\&= \norm{\iint_0^T \pr{G(s,y;t,x)-G(s,y;t',x')} \Sigma_U(s,y)W(\mathrm ds, \mathrm dy)}_{L^p_\Omega}
			\\&\overset{\text{Lemma \ref{cor:mtg+mink_cs}}} \lesssim \left\|\left(G(\cdot, \cdot;t,x) - G(\cdot,\cdot;t',x')\right)  \Sigma_U \right\|_{L^2_{[0,T]} \mathcal H_K L_\Omega^p}
			\\&=\pr{\int_0^T \norm{
					\abs{G_{t-s}(\cdot,x) - G_{t'-s}(\cdot,x')} 
					\norm{\Sigma_U(s,\cdot)}_{L^p_\Omega}
				}_{\mathcal H_K}^2\mathrm ds}^{1/2}. \label{ineq:continuity stochastic part }
		\end{align}
		Note that we have the existence of a $\Lambda \otimes K$-null subset $\mathrm N_{\counter\label{c:SigmaNull}}$ of $[0,T]\times \X^2$ such that uniformly for every $(s,x,y)\in [0,T]\times \X^2 \setminus {\mathrm N_{\ref{c:SigmaNull}}}$, 
		\begin{align}
			\norm{\Sigma_U(s,x)}_{L^p_\Omega} \cdot \norm{\Sigma_U(s,y)}_{L^p_\Omega} 
			&{\overset{\eqref{eq:BUles}}{\lesssim}} \norm{1 + |U(s,x)|}_{L^p_\Omega} \cdot \norm{1 + |U(s,y)|}_{L^p_\Omega}
			\\& {\overset{\text{Minkowski}}\leq} \pr{ 1 + \norm{U(s,x)}_{L^p_\Omega}} \cdot \pr{ 1 + \norm{U(s,y)}_{L^p_\Omega}}
			\\& \leq \pr{h(x)^{\lambda_1} + \norm{U}_{\mathbb M_{\lambda_1, p}} h(x)^{\lambda_1}}\cdot \pr{h(y)^{\lambda_1} + \norm{U}_{\mathbb M_{\lambda_1, p}} h(y)^{\lambda_1}}\\
			&\lesssim h(x)^{\lambda_1} h(y)^{\lambda_1}. \label{eq:SUp}
		\end{align}
		Therefore, uniformly for every $(t,x)$ and $(t',x')$ in $[0, T]\times \X$ with $d(x,x') \le 1$, we have
		\begin{align}
			&\mathbb E\left[\left|\Psi_2(\Sigma_U)(t,x) - \Psi_2(\Sigma_U)(t',x')\right|^p \right]^{1/p}
			\\&\overset{\eqref{ineq:continuity stochastic part }}\lesssim \pr{\int_0^T \norm{
					\abs{G_{t-s}(\cdot,x) - G_{t'-s}(\cdot,x')} 
					\norm{\Sigma_U(s,\cdot)}_{L^p_\Omega}
				}_{\mathcal H_K}^2\mathrm ds}^{1/2}
			\\&\overset{\eqref{eq:SUp}}\lesssim \pr{\int_0^T \norm{
					\abs{G_{t-s}(\cdot,x) - G_{t'-s}(\cdot,x')} 
					h(\cdot)^{\lambda_1}
				}_{\mathcal H_K}^2\mathrm ds}^{1/2}
			\\& \overset{\text{Lemma \ref{lem:hk_abs} }}\lesssim  h(x)^{\lambda_1} \pr{\abs{t-t'}^{\widetilde \xi_1} + d(x,x')^{\widetilde \xi_2}} 
			\lesssim 
			h(x)^{\lambda_0/p}
			\pr{\abs{t-t'}^{\widetilde \xi_1} + d(x,x')^{\widetilde \xi_2}}
		\end{align}
		for some $\widetilde\xi_1,\, \widetilde\xi_2 \in(0,1]$, where we used the fact that 
		$h(x) \ge 1$ for all 
		$x \in \X$ and $p\lambda_1 \leq \lambda_0$ in the last line.
		Since $\Psi_2(\Sigma_U)(0,\cdot) \equiv 0$ is deterministic, and since $p \ge 2$ and $\lambda_0\in (0,1]$ 
		were
		chosen arbitrarily  independent of $\tilde \xi_1$ and $\tilde \xi_2$, from Lemma \ref{lem:tight_tem_moment}, $\Psi_2(\Sigma_U)$ admits a modification 
		taking values in 
		$\mathcal C_{[0,T]}\mathcal C^{\mathrm{tem}}_{\X,h}$ as desired for this step.
		
		\nextstep
		Combining results from
		Steps 1, 2, 3 and 4
		with the argument at the beginning of this proof, we know that the desired result for this proposition holds.
	\end{proof}

	    The next lemma upgrades equality of predictable modification classes to indistinguishability for continuous random fields, and will be needed to complete the existence argument.
	\begin{lemma}\label{eq:TX}
		If predictable random fields $U$ and $V$ in $\mathcal C_{[0,T]}\mathcal C^{\mathrm{tem}}_{\X,h}$ satisfy that $[U]_{\mathrm p} = [V]_{\mathrm p}$, then they are indistinguishable, that is, almost surely $U(t,x) = V(t,x)$ for every $(t,x)\in [0,T] \times \X$.
	\end{lemma}
	
	\begin{proof}[Proof of Lemma \ref{eq:TX}]
		The condition $[U]_{\mathrm p} = [V]_{\mathrm p}$ says that
		\begin{equation}
			U(\omega,t,x) = V(\omega, t, x), \quad \text{$\mathbb P\otimes \Lambda$-a.e. $(\omega, t)\in \Omega \times [0,T]$,}\quad x\in \X.
		\end{equation}
		By Fubini's theorem, this implies that 
		\begin{equation}
			U(\omega,t,x) = V(\omega, t, x), \quad \text{$\mathbb P$-a.s. $\omega\in \Omega$,}\quad \text{$\Lambda$-a.e.~$t\in [0,T]$}, \quad x\in \X.
		\end{equation}
		Since $\X$ is separable, we know the existence of a countable dense subset $\X_0$ of $\X$. 
		Now the above display implies that for $\Lambda$-almost every $t\in [0,T]$,
		\begin{equation} \label{eq:CD}
			U(\omega,t,x) = V(\omega, t, x), \quad x\in \X_0, \quad \text{$\mathbb P$-a.s. $\omega\in \Omega$}.
		\end{equation}
		Denote by $A$ the collection of $t\in [0,T]$ such that \eqref{eq:CD} holds. Then $\Lambda(A^{\mathrm c}) = 0$, where $A^{\mathrm c} := [0,T]\setminus A$.
		Denote by $(\mathcal O_k)_{k=1}^\infty$ a countable topological base of $[0,T]$.
		For every $k\in \mathbb N$, since $\Lambda(\mathcal O_k)>0$, we must have $\mathcal O_k \not\subset A^{\mathrm c}$, and therefore, $\mathcal O_k\cap A \neq \emptyset$. 
		Taking $t_k \in \mathcal O_k\cap A \neq \emptyset$ for each $k\in \mathbb N$, we obtain a countable dense subset $A_0:=\{t_k:k\in \mathbb N\}$ of $[0,T]$ such that 
		\begin{equation}
			U(\omega,t,x) = V(\omega, t, x), \quad x\in \X_0, \quad \text{$\mathbb P$-a.s. $\omega\in \Omega$,}\quad t\in A_0.
		\end{equation}
		This implies that 
		\begin{equation}
			U(\omega,t,x) = V(\omega, t, x), \quad x\in \X_0, \quad t\in A_0, \quad \text{$\mathbb P$-a.s. $\omega\in \Omega$}.
		\end{equation}
		Finally, since $U$ and $V$ are continuous random fields on $[0,T]\times \X$, we obtain the desired result for this lemma.
	\end{proof}

	\begin{proof}[Proof of Theorem \ref{thm:main2}]
		\firststep\label{step:LT}
		Since $T>0$ was chosen arbitrarily, we only need to show the well-posedness of SPDE \eqref{eq:spde} up to time $T$.
		That is, we only need to show that there exists an up-to-indistinguishability unique $(\mathscr F_t)_{t\geq 0}$-adapted $\mathcal C_{[0,T]}\mathcal C^{\mathrm{tem}}_{\X,h}$-valued random field $(u_t(x))_{t\in [0,T], x\in \X}$ such that, for every $t\in [0,T]$ and $x\in \X$, \eqref{eq:mild2} holds almost surely.
		It is then standard to extend this unique solution to infinite time. 
		
		\nextstep
		We show in this step the existence of $(u_t(x))_{t\in [0,T], x\in \X}$ as desired in Step \ref{step:LT}. 
		Let the predictable random fields $U$ and $\widetilde U$ be given as in Proposition \ref{prop:existence of C_tem modification}.
		Note that 
		\begin{equation}
			[\tilde U]_{\mathrm p} = [U]_{\mathrm p} = [\Phi(U)]_{\mathrm p} 
			\overset{\eqref{eq:Phisharp}}{=} \Phi^\sharp([U]_{\mathrm p}) =  
			\Phi^\sharp([\tilde U]_{\mathrm p}) = [\Phi(\tilde U)]_{\mathrm p}.
			\label{eq:Upm}
		\end{equation}
		Therefore, $\tilde U$ is also a random field as desired in Corollary \ref{cor:MfT}.
		As a consequence of this and Proposition \ref{prop:existence of C_tem modification}, 
		\begin{itemize}
			\item[\eq\label{eq:PU}] $\Phi(\tilde U)$ admits a modification in $\mathcal C_{[0,T]}\mathcal C^{\mathrm{tem}}_{\X,h}$, denoted by $\tilde \Phi(\tilde U)$.
		\end{itemize}
		Now, since both $\tilde U$ and $\tilde \Phi(\tilde U)$ are $\mathcal C_{[0,T]}\mathcal C^{\mathrm{tem}}_{\X,h}$-valued predictable random fields and 
		\begin{equation}
			[\tilde U]_{\mathrm p} 
			\overset{\eqref{eq:Upm}}{=} [\Phi(\tilde U)]_{\mathrm p} 
			\overset{\eqref{eq:PU}}{=} [\tilde \Phi(\tilde U)]_{\mathrm p},
		\end{equation}
		by Lemma \ref{eq:TX}, $\tilde U$ and $\tilde \Phi(\tilde U)$ are indistinguishable.
		From this and \eqref{eq:PU}, $\tilde U$ is a modification of $\Phi(\tilde U)$.
		The desired existence result for this step now follows by taking $(u_t(x))_{t\in [0,T], x\in \X} = \tilde U$.
		
		\nextstep 
		We prove the uniqueness of the random field $(u_t(x))_{t\in [0,T], x\in \X}$ as desired in Step \ref{step:LT} following the argument outlined in \cite[Theorem 2.2]{MR1271224}. 
		To do this, let $U$ and $V$ be two such random fields, and let $\lambda \in (0,1]$ be arbitrary. We only need to show that $U$ and $V$ are indistinguishable.
		For the rest of this proof, $n$ always denotes an integer satisfying $n>\|f\|_{(\lambda)}$. We set
		\begin{align*}
			\tau_n:= { T\wedge }\inf\left\{
			t\in [0,T]:\,
			\norm{U(t,\cdot)}_{ (\lambda)
			} \ge n \text{ or } \norm{V(t,\cdot)}_{(\lambda)} \ge n \right\}.
		\end{align*}
		Noticing that $t\mapsto \|U(t,\cdot)\|_{(\lambda)}$ and $t\mapsto \|V(t,\cdot)\|_{(\lambda)}$ are almost surely continuous on $[0,T]$, it is easy to see that
		\begin{equation}\label{eq:tnT}
			\lim_{n\to \infty}\tau_n = T, \quad \text{a.s.}
		\end{equation}
		For every $t\in[0,T]$, $x\in\X$ and $n\in\mathbb N$, define
		\[
			U_n(t,x):=\mathbf 1_{\{t\leq\tau_n\}}U(t,x),
			\qquad
			V_n(t,x):=\mathbf 1_{\{t\leq\tau_n\}}V(t,x).
		\]
		Since $\tau_n$ is a stopping time, $t\mapsto\mathbf 1_{\{t\leq\tau_n\}}$ is adapted and left-continuous, and hence predictable; thus $U_n$ and $V_n$ are predictable.
		Clearly, for every $n\in \mathbb N$, \[\sup_{t \in [0,T]}\pr{\norm{U_n(t,\cdot)}_{(\lambda)}+ \norm{V_n(t,\cdot)}_{(\lambda)}} \le 2n, \quad \text{a.s.} \] and in particular $U_n, V_n \in \mathbb M_{\lambda, 2}$.
		For every $n\in \mathbb N$, define predictable functions $b^{(n)}$ and $\sigma^{(n)}$ on $\Omega \times [0,T] \times \X \times \mathbb R$ such that for every $t\in [0,T], x\in \X$ and $u\in \mathbb R$, 
		\begin{align}
			&(b^{(n)}(\omega, t, x, u))_{\omega \in \Omega} := b^{(n)}(t,x,u):= b(t, x, u)\mathbf 1_{\{t\leq \tau_n\}},
			\\&(\sigma^{(n)}(\omega, t, x, u))_{\omega \in \Omega} :=\sigma^{(n)}(t, x, u) := \sigma(t, x, u) \mathbf 1_{\{t\leq \tau_n\}}.
		\end{align}
		Pay attention that, almost surely for every $n\in \mathbb N$, $t\in [0,T]$, $x\in \X$ and $u,v\in \mathbb R$, we have
		\begin{equation}
			|b^{(n)}(t,x,u)| + |\sigma^{(n)}(t,x,u)|  \leq |b(t,x,u)| + |\sigma(t,x,u)|
		\end{equation}
		and
		\begin{align}
			&|b^{(n)}(t,x,u) - b^{(n)}(t,x,v)| + |\sigma^{(n)}(t,x,u)- \sigma^{(n)}(t,x,v)|  
			\\&= |b(t,x,u) - b(t,x,v)|\mathbf 1_{\{t\leq \tau_n\}} + |\sigma(t,x,u)- \sigma(t,x,v)|  \mathbf 1_{\{t\leq \tau_n\}} 
			\\&\leq  |b(t,x,u) - b(t,x,v)| + |\sigma(t,x,u)- \sigma(t,x,v)|.
		\end{align}
		Therefore, for every $n\in \mathbb N$, 
		\begin{itemize}
			\item[\eq] the linear growth condition \eqref{eq:LiG} and the Lipschitz condition \eqref{ineq:g_lip} still hold with $b$, and $\sigma$, being replaced by $b^{(n)}$, and $\sigma^{(n)}$, respectively.
		\end{itemize} 
		For any $n\in \mathbb N$, $\tilde U \in \mathbb M_{\lambda, 2}$, $t\in [0,T]$ and $x\in \X$, define
		\begin{equation}
			B^{(n)}_{\tilde U}(t,x):= b^{(n)}(t,x, \tilde U(t,x)) \quad \text{and} \quad \Sigma^{(n)}_{\tilde U}(t,x):= \sigma^{(n)}(t,x, \tilde U(t,x)).
		\end{equation}
		Now, for every $n\in \mathbb N$, since all the conditions of Lemma \ref{lem:b_contract_combined} are met with $(\tilde \alpha, p ,g, b, \sigma)$ being replaced by $(\alpha, 2, f, b^{(n)}, \sigma^{(n)})$, we conclude that 
		\begin{itemize}
			\item[\eq] for each $\tilde U\in \mathbb M_{\lambda,2}$, $(B^{(n)}_{\tilde U},\Sigma^{(n)}_{\tilde U})$ is a Duhamel-integrable pair, and  $\Psi(f,B^{(n)}_{\tilde U},\Sigma^{(n)}_{\tilde U}) \in \mathbb M_{\lambda,2}$; moreover, 
			\item[\eq\label{eq:uA}] uniformly in $ \tilde U, \tilde V \in \mathbb M_{\lambda, 2}$ and $A\ge 1$, we have 
			\begin{align}
				\norm{\Psi(f,B^{(n)}_{\tilde U},\Sigma^{(n)}_{\tilde U}) - \Psi(f,B^{(n)}_{\tilde V},\Sigma^{(n)}_{\tilde V})}_{\mathbb M_{\lambda, 2, A}} 
				\lesssim  A^\frac{\alpha - 1}{2} 
				\norm{\tilde U - \tilde V}_{\mathbb M_{\lambda, 2, A}}.
			\end{align}
		\end{itemize}
		In fact, we first note that for every $t\in [0,T]$ and $x\in \X$, almost surely,
		\begin{align} 
			&U(t,x) - V(t,x) 
			\\&= \iint_0^t G(s,y;t,x) b(s,y, U(s,y)) \mathrm ds m(\mathrm dy) 
			\\&\quad + \iint_0^t G(s,y;t,x) \sigma(s,y, U(s,y))W(\mathrm ds, \mathrm dy) 
			\\& \quad - \iint_0^t G(s,y;t,x) b(s,y, V(s,y)) \mathrm ds m(\mathrm dy) 
			\\& \quad - \iint_0^t G(s,y;t,x) \sigma(s,y, V(s,y))W(\mathrm ds, \mathrm dy). \label{eq:CF}
		\end{align}
		Since for every $x\in \X$, almost surely, both sides of \eqref{eq:CF} are continuous in $t\in [0,T]$,
		we have actually
		\begin{itemize}
			\item[\eq\label{eq:Cv}] for every $x\in \X$, almost surely, \eqref{eq:CF} holds for every $t\in [0,T]$.
		\end{itemize}
		Define $Z_n:=U_n-V_n\in\mathbb M_{\lambda,2}$ for every $n\in\mathbb N$.
		Therefore, for every sufficiently large $n\in\mathbb N$, $t\in[0,T]$ and $x\in\X$, almost surely,
		\begin{align}
			&Z_n(t,x)=U_n(t,x)-V_n(t,x)
			\\&=\mathbf 1_{\{t\leq\tau_n\}}\big(U(t,x)-V(t,x)\big)
			\\&\overset{\eqref{eq:Cv}}{=} \mathbf 1_{\{t\leq\tau_n\}}
			\iint_0^{t} G(s,y;t,x) b^{(n)}(s,y, U_n(s,y)) \mathrm ds m(\mathrm dy)
			\\&\qquad + \mathbf 1_{\{t\leq\tau_n\}}
			\iint_0^{t} G(s,y;t,x) \sigma^{(n)}(s,y, U_n(s,y))W(\mathrm ds, \mathrm dy)
			\\& \qquad - \mathbf 1_{\{t\leq\tau_n\}}
			\iint_0^{t} G(s,y;t,x) b^{(n)}(s,y, V_n(s,y)) \mathrm ds m(\mathrm dy)
			\\& \qquad - \mathbf 1_{\{t\leq\tau_n\}}
			\iint_0^{t} G(s,y;t,x) \sigma^{(n)}(s,y, V_n(s,y))W(\mathrm ds, \mathrm dy)
			\\&=\mathbf 1_{\{t\leq\tau_n\}}
			\left[\Psi(f, B^{(n)}_{U_n}, \Sigma^{(n)}_{U_n})
			-\Psi(f, B^{(n)}_{V_n}, \Sigma^{(n)}_{V_n})\right](t,x).
			\label{eq:UnVn}
		\end{align}
		We obtain that, uniformly for every sufficiently large $n\in \mathbb N$ and $A\geq 1$,
		\begin{align}
			\norm{Z_n}_{\mathbb M_{\lambda,2,A}}
			&\overset{\eqref{eq:UnVn}}{\leq}
			\norm{\Psi(f,B^{(n)}_{U_n},\Sigma^{(n)}_{U_n})
			-\Psi(f,B^{(n)}_{V_n},\Sigma^{(n)}_{V_n})}_{\mathbb M_{\lambda,2,A}}\\
			&\overset{\eqref{eq:uA}}\lesssim
			A^\frac{\alpha - 1}{2}
			\norm{U_n-V_n}_{\mathbb M_{\lambda,2,A}}
			=A^\frac{\alpha - 1}{2}\norm{Z_n}_{\mathbb M_{\lambda,2,A}}.
		\end{align}
		Therefore, there exists an $A\geq 1$ sufficiently large such that
		$\norm{Z_n}_{\mathbb M_{\lambda,2,A}}
		\leq \frac{1}{2}\norm{Z_n}_{\mathbb M_{\lambda,2,A}}$
		for every sufficiently large $n\in\mathbb N$, which implies that
		\begin{align}
			\sup_{x \in \X}\esssup_{\Lambda\text{-a.e. }t\in[0,T]}
			e^{-At}h(x)^{-\lambda}\norm{Z_n(t,x)}_{L^2_\Omega}
			=\norm{Z_n}_{\mathbb M_{\lambda,2,A}}=0.
		\end{align}
		Thus, for every $x\in\X$,
		\[
			\mathbf 1_{\{t\leq\tau_n\}}\big(U(t,x)-V(t,x)\big)=0
			\quad\text{for $\mathbb P\otimes\Lambda$-almost every $(\omega,t)$.}
		\]
		Let $\X_0$ be a countable dense subset of $\X$.
		By Fubini's theorem, on a common event of probability one, for every sufficiently large $n\in\mathbb N$ and $x\in\X_0$,
		\[
			U(t,x)=V(t,x)
			\quad\text{for $\Lambda$-almost every $t\in[0,T]$ satisfying $t\leq\tau_n(\omega)$.}
		\]
		Since $n>\|f\|_{(\lambda)}$, the continuity used in the definition of $\tau_n$ gives $\tau_n>0$ almost surely.
		The joint continuity of $U-V$ therefore first gives equality for every $t\in[0,\tau_n]$ and $x\in\X_0$, and then, by density, for every $t\in[0,\tau_n]$ and $x\in\X$.

		Let $\Omega_0$ be a probability-one event on which $\tau_n\uparrow T$, as in \eqref{eq:tnT}, and $U-V$ is jointly continuous. Since $(\tau_n)_{n\in\mathbb N}$ is nondecreasing, for some sufficiently large $n_0\geq1$,
		\begin{align}
			&\Omega_0\cap\left\{U(t,x)=V(t,x),\ \forall\,(t,x)\in[0,T]\times\X\right\}\nonumber\\
			&\qquad=\Omega_0\cap\bigcap_{n\geq n_0}\left\{U(t,x)=V(t,x),\ \forall\,(t,x)\in[0,\tau_n]\times\X\right\}.
			\label{eq:continuity of sets_temp}
		\end{align}
		Since $\mathbb P(\Omega_0)=1$ and the events in the intersection above decrease with $n$, continuity from above of $\mathbb P$ gives
		\begin{equation}\label{eq:Zn1}
			\begin{aligned}
			1&=\lim_{n \to \infty}\mathbb P\left(
			U(t,x)=V(t,x),\
			\forall\,(t,x) \in [0,\tau_n]\times\X
			\right)\\
			&=\mathbb P\left( U(t,x)=V(t,x),\ \forall\,(t,x) \in [0,T]\times\X \right).
			\end{aligned}
		\end{equation}
		Therefore, $U$ and $V$ are indistinguishable random fields as desired.
	\end{proof}

	The following moment bound for the solution will be used in Section~\ref{sec:comparison} and in the Appendix.
	\begin{corollary}\label{cor:moment}
		Let $u$ be the  mild-solution given by Theorem~\ref{thm:main2}, and let $p\ge2$. Then, for every $T>0$ and every $\lambda\in(0,1]$, and also for $\lambda=0$ if $f\in\mathcal C^{\mathrm b}_\X$,
		\begin{equation}\label{ineq:moment_every_t}
			\sup_{t\in[0,T]}\sup_{x\in\X}h(x)^{-\lambda}\norm{u(t,x)}_{L^p_\Omega}<\infty.
		\end{equation}
	\end{corollary}
	\begin{proof}
		Fix $T>0$, and let $U$ and $\tilde U$ be as in the proof of Theorem~\ref{thm:main2}, constructed on $[0,T]$. By the uniqueness part of Theorem~\ref{thm:main2}, $u=\tilde U$ on $[0,T]\times\X$ up to indistinguishability, and by \eqref{eq:Upm}, $\tilde U$ is a predictable modification of $U$. Hence, by Corollary~\ref{cor:MfT}, $\norm{u}_{\mathbb M_{\lambda,p}}=\norm{U}_{\mathbb M_{\lambda,p}}<\infty$ for every $\lambda\in(0,1]$, and also for $\lambda=0$ if $f\in\mathcal C^{\mathrm b}_\X$.
		
		Now fix $x\in\X$ and $t\in[0,T]$. We see that the bound $h(x)^{-\lambda}\norm{u(s,x)}_{L^p_\Omega}\le\norm{u}_{\mathbb M_{\lambda,p}}$ holds for $\Lambda$-a.e.\ $s\in[0,T]$. Hence, we may choose such $s=t_j$ with $t_j\to t$. Since $u$ has continuous paths in $\mathcal C^{\mathrm{tem}}_{\X,h}$ and $m$ has full support, $u(t_j,x)\to u(t,x)$ almost surely, and Fatou's lemma gives
		\begin{align*}
			h(x)^{-\lambda}\norm{u(t,x)}_{L^p_\Omega}\le\liminf_{j\to\infty}h(x)^{-\lambda}\norm{u(t_j,x)}_{L^p_\Omega}\le\norm{u}_{\mathbb M_{\lambda,p}}.
		\end{align*}
		Taking the supremum over $(t,x)\in[0,T]\times\X$ proves \eqref{ineq:moment_every_t}.
	\end{proof}
	
    \section{Proof of the comparison principle}\label{sec:comparison}
	This section is dedicated to the proof of Theorem \ref{thm:comparison} \eqref{part:main2 weak comparison}. 
	To do this, suppose all assumptions in Theorem \ref{thm:main2} hold, and assume that conditions \eqref{cond:hk_uni_bound} and \eqref{eq:con}  hold.
Let $(u_t)_{t\ge 0}$ and $(u'_t)_{t \ge 0}$ be the mild-solutions with initial conditions $u_0=f\in \mathcal C^{\mathrm{tem}}_{\X,h}$ and  
        $u_0' = f' \in \mathcal C^{\mathrm{tem}}_{\X,h}$ respectively.

	We will need several results, the first of which is the following result on continuity with respect to the initial condition. 
	\begin{prop}\label{prop:continuity_initial}
		Let $(f_n)_{n \in \N}$  be a sequence converging to $f$ in $\mathcal C^{\mathrm{tem}}_{\X,h}$.
		For each $n\in \mathbb N$, let $u^{(n)}$ be the unique strong solution to \eqref{eq:spde} with $f$ being replaced by $f_n$. Then for any $p \ge 2$, $\lambda \in (0,1]$ and $T>0$, we have
		$\|u^{(n)} - u\|_{\mathbb M_{\lambda, p}(T)} \to 0$ when $n\uparrow \infty$.
	\end{prop}
	\begin{proof}
		Fix $p \ge 2$, $\lambda \in (0,1]$ and $T>0$. 
		Define $v^{(n)}_t(x):= u^{(n)}_t(x) - u_t(x)$ for every $(t,x) \in [0,T] \times \X$ and $n\in \mathbb N$.
		Uniformly for every $(t,x)\in [0,T]\times \X$ and $n\in \mathbb N$, 
		\begin{align}
			&\norm{v^{(n)}_t(x)}_{L^p_\Omega} 
			\\&= \Bigg\| P_t^*(f_n-f)(x) +   \int_0^t \mathrm ds \int_\X G_{t-s}(y,x) \pr{b(s,y, u^{(n)}_s(y)) - b(s,y, u_s(y)) } m(\mathrm dy) 
			\\& \qquad + \iint_0^t G_{t-s}(y,x)\pr{\sigma(s,y, u^{(n)}_s(y)) - \sigma(s,y, u_s(y)) }W(\mathrm ds\mathrm dy)\Bigg\|_{L^p_\Omega}
			\\&\overset{\text{Minkowski}}{\leq}  P_t^*(|f_n-f|)(x) +  \int_0^t \mathrm ds \int_\X G_{t-s}(y,x) \norm{b(s,y, u^{(n)}_s(y)) - b(s,y, u_s(y)) }_{L^p_\Omega} m(\mathrm dy)
			\\& \qquad + \norm{\iint_0^t G_{t-s}(y,x)\pr{\sigma(s,y, u^{(n)}_s(y)) - \sigma(s,y, u_s(y)) }W(\mathrm ds\mathrm dy)}_{L^p_\Omega}
			\\&\overset{\text{Lemma \ref{cor:mtg+mink_cs}}}{\lesssim} P_t^*(|f_n-f|)(x) +  \int_0^t \mathrm ds \int_\X G_{t-s}(y,x) \norm{b(s,y, u^{(n)}_s(y)) - b(s,y, u_s(y)) }_{L^p_\Omega} m(\mathrm dy)
			\\& \qquad + \pr{\int_0^t \norm{ G_{t-s}(\cdot,x)\norm{\sigma(s,\cdot, u^{(n)}_s(\cdot)) - \sigma(s,\cdot, u_s(\cdot)) }_{L^p_\Omega} }_{\mathcal H_K}^2 \mathrm ds}^{1/2}
			\\&\overset{\eqref{ineq:g_lip}}{\lesssim}P_t^*(|f_n-f|)(x) +  \int_0^t \mathrm ds \int_\X G_{t-s}(y,x) \norm{v^{(n)}_s(y)}_{L^p_\Omega} m(\mathrm dy)
			\\& \qquad + \pr{\int_0^t \norm{ G_{t-s}(\cdot,x)\norm{v^{(n)}_s(\cdot)}_{L^p_\Omega} }_{\mathcal H_K}^2 \mathrm ds}^{1/2}. \label{eq:vntx}
		\end{align}
		We denote the three terms on the right-hand side of \eqref{eq:vntx} by $I_0^{(n)}(t,x)$, $I_1^{(n)}(t,x)$ and $I_2^{(n)}(t,x)$, respectively.
		
		It holds uniformly in $(t,x) \in [0,T] \times \X$ and $n \in \mathbb N$  that 
		\begin{align}
			&
			I_0^{(n)}(t,x) 
			= \int_\X G_t(y,x)|f_n(y)-f(y)| m(\mathrm dy)
			\le \int_\X G_t(y,x) h(y)^{\lambda} 
			\|f_n - f\|_{(\lambda)}m(\mathrm dy)
			\\&=  \norm{G_t(\cdot, x)h(\cdot)^\lambda}_{L^1_\X} \|f_n - f\|_{(\lambda)} 
			\overset{
				\text{Lemma \ref{lem:hk longtime}}
			}{\lesssim} h(x)^\lambda \norm{f_n - f}_{(\lambda)}. \label{eq:vntx0}
		\end{align}
		Similarly, it holds uniformly in $(t,x) \in [0,T]\times \X$  and $n\in \mathbb N$ that 
		\begin{align}
			&
			I_1^{(n)}(t,x) 
			\le \iint_0^t  G_{t-s}(y,x) h(y)^\lambda 
			\norm{\norm{
					v^{(n)}_s(\cdot)
				}_{L^p_\Omega}}_{(\lambda)} \mathrm ds m(\mathrm d y)
			\\& \overset{\text{Fubini}}{=} \int_0^t \pr{\int G_{t-s}(y,x) h(y)^\lambda m(\mathrm dy)} \norm{\norm{v^{(n)}_s(\cdot)}_{L^p_\Omega}}_{(\lambda)} \mathrm ds 
			\overset{
				\text{Lemma \ref{lem:hk longtime}}
			}{\lesssim}  h(x)^\lambda\int_0^t  
			\norm{v^{(n)}}_{\mathbb M_{\lambda,p}(s)}
			\mathrm ds. \label{eq:vntx1}
		\end{align}
		Finally, it holds uniformly in $(t,x) \in [0,T]\times \X$  and $n\in \mathbb N$ that
		\begin{align}
			&
			I_2^{(n)}(t,x) 
			\leq \pr{\int_0^t \norm{ G_{t-s}(\cdot,x)h(\cdot)^\lambda \norm{\norm{v^{(n)}_s(\cdot)}_{L^p_\Omega}}_{(\lambda)} }_{\mathcal H_K}^2 \mathrm ds}^{1/2}
			\\
			&\leq \pr{\int_0^t \|G_{t-s}(\cdot,x) h(\cdot)^\lambda\|_{\mathcal H_K}^2 \left\|\|v^{(n)}_s(\cdot)\|_{L^p_\Omega}\right\|_{(\lambda)}^2 \mathrm ds }^\frac{1}{2}
			\\&\overset{
				\text{Lemma \ref{lem:hk longtime}}
			}{\lesssim} h(x)^\lambda \left(\int_0^t (t-s)^{-\alpha} 
			\norm{v^{(n)}}_{\mathbb M_{\lambda,p}(s)}^2
			\mathrm ds \right)^\frac{1}{2}. \label{eq:vntx2}
		\end{align}
		
		Combining \eqref{eq:vntx}, \eqref{eq:vntx0}, \eqref{eq:vntx1} and \eqref{eq:vntx2}, we see that uniformly in $\tau \leq t$ in $[0,T]$, $x\in \X$ and $n\in \mathbb N$, 
		\begin{align}
			&h(x)^{-\lambda} \norm{v^{(n)}_{\tau}(x)}_{L^p_\Omega} 
			\overset{\eqref{eq:vntx}}{\lesssim} h(x)^{-\lambda} I_0^{(n)}(\tau,x) + h(x)^{-\lambda} I_1^{(n)}(\tau,x) + h(x)^{-\lambda} I_2^{(n)}(\tau,x)
			\\&\overset{ \eqref{eq:vntx0}, \eqref{eq:vntx1}, \eqref{eq:vntx2}}{\lesssim} \norm{f_n - f}_{(\lambda)} + \int_0^\tau \norm{v^{(n)}}_{\mathbb M_{\lambda,p}(s)}
			\mathrm ds + \pr{\int_0^\tau (\tau-s)^{-\alpha} 
				\norm{v^{(n)}}_{\mathbb M_{\lambda,p}(s)}^2
				\mathrm ds }^\frac{1}{2}
			\\&= \norm{f_n - f}_{(\lambda)} + \int_0^\tau \norm{v^{(n)}}_{\mathbb M_{\lambda,p}(s)}
			\mathrm ds + \pr{\int_0^\tau s^{-\alpha} 
				\norm{v^{(n)}}_{\mathbb M_{\lambda,p}(\tau-s)}^2
				\mathrm ds }^\frac{1}{2}
			\\&\leq  \norm{f_n - f}_{(\lambda)} + \int_0^t \norm{v^{(n)}}_{\mathbb M_{\lambda,p}(s)}
			\mathrm ds + \pr{\int_0^t s^{-\alpha} 
				\norm{v^{(n)}}_{\mathbb M_{\lambda,p}(t-s)}^2
				\mathrm ds }^\frac{1}{2}
			\\&=  \norm{f_n - f}_{(\lambda)} + \int_0^t \norm{v^{(n)}}_{\mathbb M_{\lambda,p}(s)}
			\mathrm ds + \pr{\int_0^t {(t-s)}^{-\alpha} 
				\norm{v^{(n)}}_{\mathbb M_{\lambda,p}(s)}^2
				\mathrm ds }^\frac{1}{2}. \label{eq:hxl}
		\end{align}
		Then uniformly in $t\in [0,T]$ and $n\in \mathbb N$,
		\begin{align}
			&\|v^{(n)}\|_{\mathbb M_{\lambda,p}(t)}^2 = \sup_{x\in \X} \esssup_{\text{$\Lambda$-a.e. $\tau \in [0,t]$}} h(x)^{-2\lambda} \norm{v^{(n)}_\tau(x)}^2_{L^p_\Omega} 
			\\&\overset{\eqref{eq:hxl}}{\lesssim} \pr{\norm{f_n - f}_{(\lambda)} + \int_0^t \norm{v^{(n)}}_{\mathbb M_{\lambda,p}(s)}
				\mathrm ds + \pr{\int_0^t {(t-s)}^{-\alpha} 
					\norm{v^{(n)}}_{\mathbb M_{\lambda,p}(s)}^2
					\mathrm ds }^\frac{1}{2}}^2
			\\&\overset{\text{Jensen}}\lesssim  \norm{f_n - f}^2_{(\lambda)} + \int_0^t \norm{v^{(n)}}^2_{\mathbb M_{\lambda,p}(s)}
			\mathrm ds + \int_0^t {(t-s)}^{-\alpha} 
			\norm{v^{(n)}}_{\mathbb M_{\lambda,p}(s)}^2
			\mathrm ds.
		\end{align}
		For $t \in [0,T]$, set $U_n(t):=\|v^{(n)}\|_{\mathbb M_{\lambda,p}(t)}$, which is finite by Corollary~\ref{cor:moment}. By Gr\"{o}nwall's inequality (see for example, \cite{chu1967gronwall}), we obtain uniformly in $t \in [0,T]$ that 
		\begin{align}
			U_n(t)^2 \lesssim \|f_n - f\|_{(\lambda)}^2\left(1+ \int_0^t \left( 1+ (t-s)^{-\alpha} \right) \exp\left( \int_0^s(1+(s-\tau)^{-\alpha})\mathrm d\tau  \right)\mathrm ds\right)
		\end{align}
		which converges to zero uniformly in $t \in [0,T]$ as $n \to \infty$ since $\alpha \in (0,1)$.
	\end{proof}

	We define for each $\varepsilon > 0$ the integral operator $\mathcal L^{\varepsilon,*} := \frac{P_\varepsilon ^*- I}{\varepsilon}$ on $L^\infty_\X$ where $I$ denotes the identity operator, and let $\pr{P^{\varepsilon,*}_t}_{t \ge 0}$ be the semi-group  generated by $\mathcal L^{\varepsilon,*}$.
	Note that, for each $\varepsilon > 0$ and $t\geq 0$,
	\begin{align}\label{approx_kernel}
		&P^{\varepsilon,*}_t = \exp\left(t \mathcal L^{\varepsilon,*}\right) 
		= \exp\left(-\frac{t}{\varepsilon}I \right) \exp\left(\frac{t}{\varepsilon} P^*_\varepsilon\right) 
		= e^{-t/\varepsilon} \sum_{n = 0}^\infty \frac{(t/\varepsilon)^n}{n!}P^*_{n\varepsilon}
		\\&=: e^{-t/\varepsilon}I +R_t^{\varepsilon,*}
	\end{align}
	where
	\begin{align}
		&(R_t^{\varepsilon,*} f)(x)=e^{-t/\varepsilon} \sum_{n = 1}^\infty \frac{(t/\varepsilon)^n}{n!}(P^*_{n\varepsilon}f)(x)
		= e^{-t/\varepsilon} \sum_{n = 1}^\infty \frac{(t/\varepsilon)^n}{n!}\int_\X G_{n\varepsilon}(y,x)f(y)m(\mathrm dy)
		\\& \overset{\text{Fubini}}{=} \int_\X r_t^\varepsilon(y,x) f(y)m(\mathrm dy), \quad f\in L_\X^\infty, x\in \X,
	\end{align}
	and
	\begin{equation}
		r_t^\varepsilon(y,x) := e^{-t/\varepsilon}\sum_{n = 1}^\infty \frac{(t/\varepsilon)^n}{n!} G_{n\varepsilon}(y,x), \quad x,y\in \X.
	\end{equation}

	\begin{lemma}\label{lem:approx_estimates}
		\begin{enumerate}
			\item \label{lem:approx_estimates1} 
			Let $\alpha \in (0,1)$ be given as in \eqref{cond:exp_hk_lamda}. Then there exists $\kappa \in \R$, independent of $\varepsilon$, such that
			uniformly in $t>0$, $x\in\X$ and $\varepsilon\in(0,1]$,
			\begin{align}\label{ineq:R_bound}
				\left\|r_t^\varepsilon(\cdot,x)\right\|_{\mathcal H_K}^2
				\lesssim t^{-\alpha}e^{\kappa t},\qquad
				\left\|r_t^\varepsilon(\cdot,x)\right\|_{L^1_\X}
				\lesssim e^{\kappa t}.
			\end{align}
			\item  For each $t >0$,
			\begin{align}\label{lim:p_approx}
				\lim_{\varepsilon \downarrow 0}\sup_{x\in \X} \int_0^t \left\| \left|\int_\X  r_s^\varepsilon(y,x) G_\varepsilon(\cdot,y) m(\mathrm dy) -  G_s(\cdot, x)
				\right|\right\|^2_{\mathcal H_K}\,\mathrm ds = 0.
			\end{align}
			\item For each $T>0$, there exists a $c_T \in (0,1]$, and a function $\varepsilon \to \gamma_\varepsilon$ with $\lim_{\varepsilon \downarrow 0} \gamma_\varepsilon = 0$, so that,  uniformly in $(t,x) \in (0,T]\times \X$ and sufficiently small $\varepsilon >0$, it holds that
			\begin{align}\label{ineq:p_approx_m}
				\int_\X \left|r_t^\varepsilon(y,x)  - G_t(y,x)\right| m(\mathrm dy)\lesssim e^{-c_T\frac{t}{\varepsilon}} +  \frac{\gamma_\varepsilon}{t^\beta},
			\end{align}
			where $\beta$ is as in \eqref{cond:time_holder}.
		\end{enumerate}
	\end{lemma}
	
	We say a continuous random field $(a(t,x))_{t\geq 0,x\in \X}$ is non-negative if $a(t,x)\geq 0$ almost surely for every $t\geq 0$ and $x\in \X$; equivalently, if $a(t,x)\geq 0$ for every $t\geq 0$ and $x\in\X$ almost surely.
	\begin{prop}
		\label{thm:non_negtive}
		Suppose that
		\begin{itemize}
			\item[\eq\label{ineq:non_negativity_coef_bound_at_zero}] for any $T>0$, uniformly for $\omega \in \Omega \setminus \mathrm N_{\ref{c:Lip}}$
			and $(t,x) \in [0,T]\times \X$,
			\begin{align}
				b(\omega,t,x,0) \geq 0, \qquad \sigma(\omega,t,x,0) = 0.
			\end{align}
		\end{itemize}
		\begin{itemize}
			\item[\eq\label{ineq:non_negative_approx_coef_ineq}]
			there exists an $\eta >0$ 
			so that for any $T>0$, 
			uniformly in $\omega \in \Omega \setminus \mathrm N_{\ref{c:Lip}}$, $u\in \mathbb R$,
			and $(t,x,y) \in [0,T]\times \X^2$,
			\begin{align}
				\abs{b(\omega,t,x,u) - b(\omega,t,y,u)} \lesssim d(x,y)^\eta (1+|u|),\\
				\abs{\sigma(\omega,t,x,u) - \sigma(\omega,t,y,u)} \lesssim d(x,y)^\eta |u|;
			\end{align}
		\end{itemize}
		and that $f$ is non-negative.
		Then  $u$ is non-negative.
	\end{prop}
	
	\begin{proof}
		In light of Proposition \ref{prop:continuity_initial} and Lemma \ref{lem:C_temp_density_0}, it is enough to assume that $f$ is bounded, continuous, and non-negative.
		
		For $\varepsilon>0$ and $(t,x)\in\R_+\times\X$, define $W^\varepsilon(t,x)=:W^\varepsilon_x(t)$ via
		\begin{align}\label{def:white_noise_cov}
			W^\varepsilon(t,x) := \iint_0^t  G_\varepsilon(y,x)W(\mathrm ds,\mathrm dy).
		\end{align}
		Observe that for each $(t,x),\, (t',x') \in \R_+\times \X$, 
		\begin{align}
			\E\left[W^\varepsilon(t,x)W^\varepsilon(t',x')\right] &= \pr{t\wedge t'}\left \langle Q\left[G_\varepsilon(\cdot, x)\right]_K, \left[G_\varepsilon(\cdot,x') \right]_K \right \rangle_{\tilde{\mathcal H}_K}\\
			&:=\pr{t'\wedge t} k_\varepsilon(x,x').
		\end{align}
		By \eqref{cond:exp_hk_lamda}, we see there is some $\alpha \in (0,1)$ so that $\sup_{x}k_\varepsilon(x,x)\lesssim\varepsilon^{-\alpha}$ uniformly in $\varepsilon \in (0,1]$.
		
		For each fixed $x\in\X$, the process $(W^\varepsilon(t,x),\,t\geq0)$ is a Brownian motion with 
		\begin{align}\label{eq:q_var_ep}
			\left[\int_0^\cdot W^\varepsilon(\mathrm ds,x)\right]_t =\left[ \iint_0^\cdot  G_\varepsilon(y , x)W(\mathrm dy,\mathrm ds)\right]_t = t k_\varepsilon(x,x)
		\end{align}
		where $([M]_t)_{t \ge 0}$ is the quadratic variation process of a martingale $(M_t)_{t\ge0}$. 

		Let $\varepsilon\in(0,1]$ and recall $\mathcal L^{\varepsilon,*}$, $P^{\varepsilon,*}$, and the density $r_t^\varepsilon$ defined after \eqref{approx_kernel}. Formally, we consider the following equation with initial condition $u_\varepsilon(0,x) = f(x)$, 
		\begin{align*}
			\mathrm d u_\varepsilon(t,x)
			={}&\left(\mathcal L^{\varepsilon,*}u_\varepsilon(t,x)+b(t,x,u_\varepsilon(t,x))\right)\mathrm dt+\sigma(t,x,u_\varepsilon(t,x))W^\varepsilon(\mathrm dt,x)
		\end{align*}
		whose mild form is given by
		\begin{align}\label{eq:cont_1d_mild}
			u_\varepsilon(t,x) = P^{\varepsilon,*}_t f(x) +&  \int_0^t e^{-(t-s)/\varepsilon}b\pr{s,x,u_\varepsilon(s,x)}\mathrm ds\nonumber\\
			&+\int_0^t\mathrm ds \int_\X r_{t-s}^\varepsilon(y,x)b(s,y,u_\varepsilon(s,y))m(\mathrm dy)\nonumber\\
			&+\int_0^t e^{- (t-s)/\varepsilon}\sigma(s,x,u_\varepsilon(s,x)) W^\varepsilon(\mathrm ds,x)\nonumber\\
			&+\int_\X\int_0^t r_{t-s}^\varepsilon(y,x)\sigma(s,y,u_\varepsilon(s,y))W^\varepsilon(\mathrm ds,y)m(\mathrm dy).
		\end{align}
		One can show that \eqref{eq:cont_1d_mild} has a $\mathcal C_{[0,T]} \mathcal C^{\mathrm{tem}}_{\X,h}$-valued solution satisfying
		\begin{align}\label{boud: lp_loc_both}
			\sup_{0 < \varepsilon\le 1}\sup_{0 \le t \le T} \sup_{x \in \X} \E \left[|u_\varepsilon(t,x)|^2\right] < \infty,\qquad \forall T>0.
		\end{align}
		 Since $\mathcal L^{\varepsilon,*}$ is bounded and
		\[
			P_t^{\varepsilon,*}-I=\int_0^t\mathcal L^{\varepsilon,*}P_s^{\varepsilon,*}\mathrm ds,
		\]
		ordinary and stochastic Fubini (Lemma \ref{thm:stochastic_fubini}), conditions \eqref{eq:LiG}, \eqref{cond:exp_hk_lamda}, \eqref{ineq:hk_m long time}, \eqref{ineq:R_bound}, and the bound \eqref{boud: lp_loc_both} ensure that  \eqref{eq:cont_1d_mild} satisfies
		\begin{align}\label{eq:cont_1d_sde}
			u_\varepsilon(t,x)=f(x)&+\int_0^t\left(\mathcal L^{\varepsilon,*}u_\varepsilon(s,x)+b(s,x,u_\varepsilon(s,x))\right)\mathrm ds\nonumber\\
			&+\int_0^t\sigma(s,x,u_\varepsilon(s,x))W^\varepsilon(\mathrm ds,x).
		\end{align}
		The last term is a one-dimensional It\^{o} integral. Hence, by \eqref{def:white_noise_cov} and \eqref{eq:cont_1d_sde}, $\left(u_\varepsilon(t,x)\right)_{t\geq0}$ is a continuous semimartingale for each fixed $\varepsilon\in(0,1]$ and $x\in\X$.
		
		Let $\varphi(u) := - \left(u \wedge 0\right) = u^-,\, u \in \R $, which is convex. By It\^o--Meyer--Tanaka's formula (\cite[Proposition 9.2]{le2016brownian}), we see for each $t \ge0$ and $x \in \X$,
		\begin{equation}\label{eq:u_epsilon negative part}
			\begin{aligned}
				\E[\varphi(u_\varepsilon(t,x))]= \varphi(f(x)) &- \E\left[\int_0^t \mathbf 1_{u_\varepsilon(s,x) \le 0} \left(\mathcal L^{\varepsilon,*} u_\varepsilon(s,x) +b(s,x,u_\varepsilon(s,x)) \right)\mathrm ds\right]\\
				&- \E\left[\int_0^t \mathbf 1_{u_\varepsilon(s,x) \le 0} \,\sigma(s,x,u_\varepsilon(s,x))\mathrm dW_x^\varepsilon(s) -\frac12L^0_x(t)\right]
			\end{aligned}
		\end{equation}
		where for each $x \in \X$, $\left(L_x^0(t)\right)_{t \ge 0 }$ is the local time of $(u_\varepsilon(t,x))_{t \ge 0}$ at zero, which is non-negative and non-decreasing. Now,  
		conditions \eqref{ineq:g_lip} and \eqref{ineq:non_negativity_coef_bound_at_zero} imply that there is some $L\ge0$ which is the common Lipschitz constant for $b$ and $\sigma$ on $[0,T]$. Hence, for $r\le0$ and $(t,x) \in [0,T]\times \X$, $b(t,x,r)\ge -L|r|$. 
		By the definition of  $\mathcal L^{\varepsilon,*}$, we get the second term on the right hand side  of \eqref{eq:u_epsilon negative part} (including the minus sign) is smaller than
		\begin{equation}\label{eq:u_epsilon second}
			\begin{aligned}
				&\E\left[\int_0^t \frac{1}{\varepsilon}\boldsymbol{1}_{u_\varepsilon(s,x) \le 0} u_\varepsilon(s,x)-\frac{1}{\varepsilon}\boldsymbol{1}_{u_\varepsilon(s,x) \le 0}P_\varepsilon^* u_{\varepsilon}(s)(x) +L\,\boldsymbol{1}_{u_\varepsilon(s,x) \le 0} |u_\varepsilon(s,x)| \mathrm ds\right]\\
				\le& \left(L +1/\varepsilon\right) \int_0^t \E[\varphi(u_\varepsilon(s,x))]\,\mathrm ds-\frac{1}{\varepsilon}\int_0^t \int_\X G(\varepsilon,y,x)\E\left[ \boldsymbol{1}_{u_\varepsilon(s,x) \le 0}u_\varepsilon(s,y)\right]\,m(\mathrm dy)\,\mathrm ds
			\end{aligned}
		\end{equation}
		where $\boldsymbol{1}_A(\cdot)$ denotes the indicator function on the set $A$.  By \cite[Proposition 9.9]{le2016brownian}, we see for each $x \in \X$, it holds almost surely that
		\begin{align}\label{ineq:localtime epsilon}
			L^0_x(t) &\,\,= \lim_{\delta\downarrow 0}
			\frac{1}{\delta}\int_0^t 1_{\left\{0\le u_\varepsilon(s,x)\le \delta \right\}}\mathrm d\left[ u_\varepsilon(\cdot,x)\right]_s \nonumber\\
				&\,\,=\lim_{\delta\downarrow 0}
				\frac{1}{\delta} \int_0^t 1_{\left\{0\le u_\varepsilon(s,x)\le \delta \right\}}\sigma\big(s,x,u_\varepsilon(s,x)\big)^2 \mathrm d\left[ W_x^\varepsilon(\cdot)\right]_s\nonumber\\
				&\overset{
				\eqref{ineq:g_lip}, \eqref{ineq:non_negativity_coef_bound_at_zero}}{\le} L^2\lim_{\delta\downarrow 0}
				\frac{1}{\delta}\int_0^t 1_{\left\{0\le u_\varepsilon(s,x)\le \delta \right\}} \delta^2\mathrm d\left[ W_x^\varepsilon(\cdot)\right]_s\nonumber\\
				&\overset{\eqref{eq:q_var_ep}}{\le} L^2 k_\varepsilon(x,x) \lim_{\delta \downarrow 0}
				\frac{1}{\delta} \int_0^t  1_{\left\{0\le u_\varepsilon(s,x)\le \delta \right\}} \delta^2 \mathrm ds\nonumber\\
			&\,\,=0.
		\end{align}
		In addition, we see by \eqref{boud: lp_loc_both}, \eqref{eq:q_var_ep} and the Lipschitz bound on $\sigma$ that the stochastic integral in \eqref{eq:u_epsilon negative part} is square-integrable and hence has mean zero.
		Inserting \eqref{eq:u_epsilon second} and \eqref{ineq:localtime epsilon} into equation \eqref{eq:u_epsilon negative part} and recalling $\varphi(f) \equiv 0$ on $\X$ since $f \ge 0$, we see that 
		\begin{align}
			\E[\varphi(u_\varepsilon(t,x))] &\le ( 
			L +1/\varepsilon) \int_0^t \E\left[ \varphi(u_\varepsilon(s,x))\right] \mathrm ds \\
			&\quad \quad\quad+  (1/\varepsilon) \int_0^t\mathrm ds \int_\X G(\varepsilon,y,x)\E\left[\varphi(u_\varepsilon(s,y)
			)\right]m(\mathrm dy).
		\end{align}
		By condition \eqref{cond:exp_hk_lamda} when $\lambda = 0$, we may take the supremum over $x \in \X$ on both sides and apply Gronwall's lemma to see that $\sup_{x \in \X} \E\left[\varphi(u_\varepsilon(t,x))\right] = 0$ for $t\in[0,T]$. Since $T>0$ is arbitrary, we see for each $(t,x)\in\R_+\times\X$, $u_\varepsilon(t,x)\ge0$ almost surely.
		
		Let $(u_t)_{t\geq0}$ be the unique solution to \eqref{eq:spde}. We will show that $\lim_{\varepsilon \downarrow 0}u_\varepsilon = u$ in an appropriate sense. By Lemma \ref{lem:hk longtime}, Lemma \ref{lem:approx_estimates} together with the bound \eqref{boud: lp_loc_both}, we may apply stochastic Fubini (Lemma \ref{thm:stochastic_fubini}) to see
		\begin{align}
			&\int_\X\int_0^t  r_{t - s}^\varepsilon(y,x)\sigma\big(s,y,u_\varepsilon(s,y)\big) W^{\varepsilon}(\mathrm ds,y) \,m(\mathrm dy)\\
			&\quad \quad \quad \quad \quad \quad \quad\quad \quad \quad \quad \quad \quad-\iint_0^t G(t-s,z,x)\sigma\big(s,z,u(s,z)\big) W( \mathrm dz, \mathrm ds)\\
			=&  \iint_0^t \left(\int_\X  r_{t-s}^\varepsilon(y,x)\sigma\big(s,y,u_\varepsilon(s,y) \big) G_\varepsilon(z,y)\,m(\mathrm dy)\right)\,W(\mathrm dz,\mathrm ds)\\
			&\quad \quad \quad \quad \quad \quad \quad\quad \quad \quad \quad \quad \quad -\iint_0^t  G(t - s,z,x)\sigma\big(s,z,u(s,z)\big) W(\mathrm dz, \mathrm ds)\\
			=&\iint_0^t \left(\int_\X  r_{t-s}^\varepsilon(y,x) \big(\sigma\big(s,y,u_\varepsilon(s,y) \big)-\sigma\big(s,y,u(s,y) \big) \big) G_\varepsilon(z,y)\,m(\mathrm dy)\right)\,W(\mathrm dz,\mathrm ds)\\
			&\quad+\iint_0^t\left(\int_\X  r_{t-s}^\varepsilon(y,x) \big(\sigma\big(s,y,u(s,y) \big)-\sigma\big(s,z,u(s,z) \big)\big) G_\varepsilon(z,y)\,m(\mathrm dy)\right)\,W(\mathrm dz,\mathrm ds)\\
			&\quad+\iint_0^t\bigg(\int_\X  r_{t-s}^\varepsilon(y,x)\sigma\big(s,z,u(s,z) \big) G_\varepsilon(z,y)\,m(\mathrm dy) \\
			&\qquad \qquad \qquad \qquad \qquad \qquad \qquad \qquad- G(t - s,z,x)\sigma\big(s,z,u(s,z)\big)\bigg)\,W(\mathrm dz,\mathrm ds).
		\end{align}
		We obtain uniformly in  $(t,x) \in [0,T]\times \X$,
		\begin{align}
			&\E\left[|u_\varepsilon(t,x) - u(t,x)|^2\right] \\
			\lesssim &\left|P^{\varepsilon,*}_tf(x) - P_t^*f(x)\right|^2 + \E\left[\left|\int_0^t e^{-(t-s)/\varepsilon} b(s,x,u_\varepsilon(s,x) )\,\mathrm ds \right|^2\right]\\
			&+\E\left[\left|\int_0^t \mathrm ds\int_\X r_{t-s}^\varepsilon(y,x)\Big(b(s,y,u_\varepsilon(s,y)) - b(s,y,u(s,y))\Big) m(\mathrm dy)\, \right|^2\right]\\
			&+ \E\left[ \left|\int_0^t \mathrm ds\int_\X \Big(r_{t-s}^\varepsilon(y,x) - G(t-s,y,x) \Big) b(s,y,u(s,y)) \,m(\mathrm dy)\right|^2\right]\\
			&+ \int_0^t e^{-2(t-s)/\varepsilon} \sup_{x \in \X}\E\left[\sigma(s,x,u_\varepsilon(s,x))^2\right] \varepsilon^{-\alpha}\mathrm ds\\
			&+\int_0^t \norm{\norm{\int_\X r_{t-s}^\varepsilon(y,x)\abs{\sigma(s,y,u_\varepsilon(s,y))-\sigma(s,y,u(s,y))} G_\varepsilon(\cdot,y)m(\mathrm dy)}_{L^2_\Omega}}_{\mathcal H_K}^2 \mathrm ds\\
			&+\int_0^t \norm{ \norm{\int_\X r_{t-s}^\varepsilon(y,x) \abs{\sigma(s,y,u(s,y)) - \sigma(s,\cdot,u(s,\cdot))} G_\varepsilon(\cdot,y) m(\mathrm dy)}_{L^2_\Omega}}_{\mathcal H_K}^2 \mathrm ds\\
			&+\int_0^t \E\left[\left\|\left(\int_\X r_{t - s}^\varepsilon(y,x) G_\varepsilon(\cdot,y) m(\mathrm dy) - G(t-s,\cdot,x) \right) \sigma(s,\cdot,u(s,\cdot))\right\|^2_{\mathcal H_K} \right] \mathrm ds\\
			:=& \sum_{i = 1}^8 J_i(\varepsilon,t,x).
		\end{align}
		Using \eqref{ineq:p_approx_m} of Lemma \ref{lem:approx_estimates} and the boundedness of $f$, we see that, for each $T>0$, there exists $c_T>0$ such that, uniformly in sufficiently small $\varepsilon>0$ and $t\in(0,T]$,
		\begin{align}
			\sup_{x \in \X}J_1(\varepsilon,t,x)
			\lesssim e^{-c_Tt/\varepsilon}+\frac{\gamma_\varepsilon}{t^\beta},
		\end{align}
		where $\beta$ is as in \eqref{cond:time_holder}.
		At $t=0$, $P_0^{\varepsilon,*}f=P_0^*f=f$, and hence $J_1(\varepsilon,0,x)=0$.

		Since $f$ is bounded, we see by Corollary~\ref{cor:moment} with $\lambda = 0$ and $p=2$ that
		\begin{align}\label{ineq:u_uniform_second_moment}
			\sup_{t\in[0,T]}\sup_{x\in\X}\E\bigl[|u(t,x)|^2\bigr]<\infty.
		\end{align}
		Together with \eqref{ineq:p_approx_m}  and Minkowski's inequality,
		\begin{align}
\sup_{t\in[0,T]}\sup_{x\in\X}J_4(\varepsilon,t,x)\lesssim(\varepsilon+\gamma_\varepsilon)^2\longrightarrow0,\qquad \text{ as }\varepsilon\downarrow0.
		\end{align}
		Also, by \eqref{cond:exp_hk_lamda} with $\lambda=0$, \eqref{eq:LiG}, and \eqref{boud: lp_loc_both}, we obtain, uniformly in $\varepsilon\in(0,1]$,
		\begin{align*}
\sup_{t\in[0,T]}\sup_{x\in\X}J_2(\varepsilon,t,x)\lesssim\varepsilon^2,\qquad
\sup_{t\in[0,T]}\sup_{x\in\X}J_5(\varepsilon,t,x)\lesssim\varepsilon^{1-\alpha},
		\end{align*}
		where $\alpha\in(0,1)$ is the exponent in \eqref{cond:exp_hk_lamda}.
		Hence, 
		\begin{align}
\lim_{\varepsilon\downarrow0}\sup_{t\in[0,T]}\sup_{x\in\X}\big(J_2(\varepsilon,t,x)+J_5(\varepsilon,t,x)\big)=0.
		\end{align}
		By \eqref{eq:LiG}, \eqref{ineq:u_uniform_second_moment}, and Lemma \ref{lem:approx_estimates},
		\begin{align}
			\lim_{\varepsilon \downarrow 0} \sup_{t \in [0,T]} \sup_{x \in \X} J_8(\varepsilon,t,x) &\lesssim \lim_{\varepsilon \downarrow 0}\sup_{x \in \X}\int_0^T\left\|\left|\int_\X r_s^\varepsilon(y,x) G_\varepsilon(\cdot,y) m(\mathrm dy) - G(s,\cdot,x)\right|\right\|^2_{\mathcal H_K} \mathrm ds \\ &\overset{\eqref{lim:p_approx}}{=}0.
		\end{align}
		Let $U_\varepsilon(t) = \sup_{x\in \X} \norm{u_\varepsilon(t,x) - u(t,x)}_{L^2_\Omega}^2$. By Minkowski's inequality, uniformly in $(t,x) \in [0,T]\times \X$ and sufficiently small $\varepsilon >0$,
		\begin{align}
			J_3(\varepsilon,t,x) &\overset{\eqref{ineq:g_lip}}{\lesssim} \int_0^t \left( \int_\X r_{t-s}^\varepsilon(y,x)m(\mathrm dy)\right)^2U_\varepsilon(s) \mathrm ds\\
			&\,\overset{\eqref{ineq:R_bound}}{\lesssim} \int_0^tU_\varepsilon(s)\mathrm ds\\
			&\,\,\,\lesssim \int_0^t (t-s)^{-\alpha}U_\varepsilon(s) \mathrm ds.
		\end{align}
		By Lemma \ref{cor:mtg+mink_cs} and Minkowski's inequality, it holds uniformly in  $(t,x) \in [0,T]\times \X$ and sufficiently small $\varepsilon>0$ that
		\begin{equation}\label{ineq:J6 step 1}
			\begin{aligned}
				J_6(\varepsilon,t,x) &\,\,\,\lesssim \int_0^t \norm{\norm{\int_\X r_{t-s}^\varepsilon(y,x)\abs{\sigma(s,y,u_\varepsilon(s,y))-\sigma(s,y,u(s,y))} G_\varepsilon(\cdot,y)m(\mathrm dy)}_{L^2_\Omega}}_{\mathcal H_K}^2 \mathrm ds\\
				&\overset{\eqref{ineq:g_lip}}{\lesssim} \int_0^t\norm{\int_\X r_{t-s}^\varepsilon(y,x)\abs{U_\varepsilon(s) }^\frac{1}{2}G_\varepsilon(\cdot,y)m(\mathrm dy)}_{\mathcal H_K}^2 \mathrm ds\\
				&\,\,\,= \int_0^t \norm{\int_\X r_{t-s}^\varepsilon(y,x) G_\varepsilon(\cdot,y) m(\mathrm dy)}^2_{\mathcal H_K} U_\varepsilon(s)\, \mathrm ds.
			\end{aligned}
		\end{equation}
		Let $\widetilde \kappa \in \R$ be given by Lemma \ref{lem:hk longtime}. By the semigroup property, uniformly in $\varepsilon \in (0,1]$ and $(t,x) \in (0,T]\times \X$,
		\begin{equation}\label{ineq:mollified_conv_H_norm}
			\begin{aligned}
				\norm{\int_\X r_t^\varepsilon(y,x) G_\varepsilon(\cdot,y) m(\mathrm dy)}^2_{\mathcal H_K} &\,\,= \norm{e^{-t/\varepsilon}\sum_{n = 1}^\infty \frac{\pr{t/\varepsilon}^n}{n!} G(\varepsilon (n+1),\cdot,x)}^2_{\mathcal H_K}\\
				&\overset{\text{Minkowski\&Jensen}}{\le} \pr{e^{-t/\varepsilon}\sum_{n = 1}^\infty \frac{(t/\varepsilon)^n}{n!} \norm{G\pr{\varepsilon(n+1),\cdot,x}}_{\mathcal H_K}^2}\\
				&\overset{\eqref{ineq:hk_k long time} }{\lesssim} e^{-t/\varepsilon}\sum_{n = 0}^\infty \frac{(te^{\varepsilon\widetilde\kappa}/\varepsilon)^n}{n!} (\varepsilon(n+1))^{-\alpha}\\
				&\,\,\le e^{-t/\varepsilon + te^{\varepsilon \widetilde\kappa}/\varepsilon} \pr{e^{-te^{\varepsilon \widetilde\kappa}/\varepsilon}\sum_{n = 0}^\infty \frac{\pr{te^{\varepsilon \widetilde\kappa}/\varepsilon}^{n+1}}{(n+1)!} t^{-1}e^{-\varepsilon \widetilde\kappa}}^{\alpha}\\
				&\,\,\lesssim t^{-\alpha} \exp \left(\frac{te^{\varepsilon \widetilde\kappa} - t}{\varepsilon} \right)\\
				&\lesssim t^{-\alpha},
			\end{aligned}
		\end{equation}
		where Jensen's inequality was used in the second and fourth lines, $\alpha\in(0,1)$ is from \eqref{cond:exp_hk_lamda}, and there exists $\widetilde \kappa_0 \in \R$ such that, for all $t \ge 0$,
		\begin{align*}
			\sup_{\varepsilon \in (0,1]}\exp\pr{\frac{te^{\varepsilon \widetilde\kappa} - t}{\varepsilon}} \le e^{\widetilde \kappa_0 t}.
		\end{align*}
		Inserting this back into the inequality \eqref{ineq:J6 step 1}, we see uniformly in $\varepsilon \in (0,1]$ and $(t,x) \in [0,T]\times \X$ that
		\begin{align*}
			J_6(\varepsilon,t,x) \lesssim \int_0^t (t-s)^{-\alpha} U_\varepsilon(s) \mathrm ds.
		\end{align*}

		Now we consider $J_7$. By Lemma \ref{lem:hk longtime} and Minkowski's inequality, for each $\varepsilon \in (0,1)$ and $(t,x) \in [0,T]\times \X$,
		\begin{align}
			J_7(\varepsilon,t,x) &\le \int_0^t \left\|\int_\X r_{t-s}^\varepsilon(y,x) \|\sigma(s,y,u(s,y)) - \sigma(s,\cdot,u(s,\cdot))\|_{L^2_\Omega} G_\varepsilon(\cdot,y)m(\mathrm dy) \right\|_{\mathcal H_K}^2 \mathrm ds.
		\end{align}
		By the boundedness of $f$, \eqref{ineq:u_uniform_second_moment}, \eqref{ineq:spatial_holder_long_time}, and the calculations in the proof of Proposition \ref{prop:existence of C_tem modification} with $p=2$ and $\lambda=0$, there exists $\xi'\in(0,1)$ such that, uniformly in $s\in(0,T]$ and $y,z\in\X$,
		\begin{align}\label{ineq:u_bounded_spatial_increment}
			\norm{u(s,y)-u(s,z)}_{L^2_\Omega}
			\lesssim \left(\frac{d(y,z)^{\xi'}}{s^\beta}\right)\wedge1.
		\end{align}
		Set $\xi:=\eta\wedge\xi'$. For $d(y,z)\le1$, uniformly in $s\in(0,T]$,
		\begin{align}
			&{\quad} \norm{\sigma(s,y,u(s,y))-\sigma(s,z,u(s,z))}_{L^2_\Omega}\\
			&{\le} \norm{\sigma(s,y,u(s,y))-\sigma(s,z,u(s,y))}_{L^2_\Omega}
			+\norm{\sigma(s,z,u(s,y))-\sigma(s,z,u(s,z))}_{L^2_\Omega}\\
			&{\overset{\eqref{ineq:non_negative_approx_coef_ineq},\,\eqref{ineq:g_lip}}{\lesssim}}d(y,z)^\eta\norm{u(s,y)}_{L^2_\Omega}+\norm{u(s,y)-u(s,z)}_{L^2_\Omega}\\
			&{\overset{\eqref{ineq:u_bounded_spatial_increment}}{\lesssim}}
			\left(\frac{d(y,z)^\xi}{s^\beta}\right)\wedge1,
		\end{align}
		where $\beta\in[0,1)$ is as in \eqref{cond:time_holder}.
		For $d(y,z)>1$, \eqref{ineq:g_lip}, \eqref{ineq:non_negativity_coef_bound_at_zero}, and \eqref{ineq:u_uniform_second_moment} give the cap by $1$. Thus the last estimate holds for all $y,z\in\X$.
		
		Therefore,  it holds uniformly in $x,z \in \X$ and $\delta \in (0,1]$ that
		\begin{align}
			&\int_\X r_{t-s}^\varepsilon(y,x) \|\sigma(s,y,u(s,y)) - \sigma(s,z,u(s,z))\|_{L^2_\Omega} G_\varepsilon(z,y)m(\mathrm dy)\\
			\lesssim &\left( \left(\frac{\delta^\xi}{s^\beta}\wedge 1\right)\int_{d(z,y)\le \delta} + \int_{d(z,y) > \delta}\right) r_{t-s}^\varepsilon(y,x) G_\varepsilon(z,y)m(\mathrm{d}y)
		\end{align}
		Hence, for each $T>0$, it holds uniformly in $\varepsilon\in(0,1]$ and $(t,x)\in(0,T]\times\X$ that
		\begin{align}
			J_7(\varepsilon,t,x) \lesssim \int_0^{t} J_{7,1}(\varepsilon,t,s,x) + J_{7,2}(\varepsilon,t-s,x) \mathrm d s 
		\end{align}
		where 
		\begin{equation}\label{ineq:J71}
			\begin{aligned}
				J_{7,1}(\varepsilon,t,s,x)&\,\,: = \left(\frac{\delta^\xi}{s^\beta}\wedge 1 \right) \norm{\int_\X r_{t-s}^\varepsilon(y,x)G_\varepsilon(\cdot,y)m(\mathrm dy)}_{\mathcal H_K}^2 \\
				&\overset{\eqref{ineq:mollified_conv_H_norm}}{\lesssim} \left(\frac{\delta^\xi}{s^\beta}\wedge 1 \right)(t-s)^{-\alpha},
			\end{aligned}
		\end{equation}
		uniformly in $\delta, \varepsilon \in (0,1]$, $0<s<t\le T$, $x \in \X$,  and 
		\begin{align}
			J_{7,2}(\varepsilon,t-s,x):&\,\,= \left\|\int_\X \mathbf 1_{B_\delta(y)^c}(\cdot) r_{t-s}^\varepsilon(y,x) G_\varepsilon(\cdot,y) m(\mathrm dy)\right\|_{\mathcal H_K}^2\\
			&\,\,\le  \pr{\int_\X  r_{t-s}^\varepsilon(y,x) \norm{\mathbf 1_{B_\delta(y)^c}(\cdot) G_\varepsilon(\cdot,y)}_{\mathcal H_K} m(\mathrm dy)}^2\\
			&\overset{\eqref{ineq:R_bound}}{\lesssim} \sup_{y \in \X} \norm{\mathbf 1_{B_\delta(y)^c}(\cdot) G_\varepsilon(\cdot,y)}_{\mathcal H_K}^2,
		\end{align}
		uniformly in $\delta,\varepsilon\in(0,1]$, $0<s<t\le T$, and $x\in\X$. Here we used Minkowski's inequality.

		Fix $\theta\in(0,1]$ so that $\alpha+\theta\beta<1$. Since $a\wedge1\le a^\theta$ for every $a\ge0$, \eqref{ineq:J71} gives, uniformly in $(t,x)\in (0,T]\times \X$ and $\delta, \varepsilon \in (0,1)$,
		\begin{align}
			\int_0^t J_{7,1}(\varepsilon,t,s,x) \mathrm ds
			&\lesssim \int_0^t\left(\frac{\delta^\xi}{s^\beta}\wedge 1 \right)(t-s)^{-\alpha}\mathrm ds \\
			&\le \delta^{\xi\theta}\int_0^t s^{-\theta\beta}(t-s)^{-\alpha}\mathrm ds\\
			&= \delta^{\xi\theta}t^{1-\alpha-\theta\beta}
			\int_0^1 r^{-\theta\beta}(1-r)^{-\alpha}\mathrm dr\\
			&\lesssim \delta^{\xi\theta}.
		\end{align}
		Hence, it holds uniformly in  $\delta \in (0,1]$ that 
		\begin{align}
			\limsup_{\varepsilon\downarrow0}\sup_{t\in[0,T]}\sup_{x\in\X}J_7(\varepsilon,t,x)
			&\lesssim \delta^{\xi\theta}+T\limsup_{\varepsilon\downarrow0}\sup_{y\in\X}
			\left\|\mathbf1_{B_\delta(y)^c}(\cdot)G_\varepsilon(\cdot,y)\right\|_{\mathcal H_K}^2
			\overset{\eqref{cond:hk_uni_bound}}{\le}\delta^{\xi\theta}.
		\end{align}
		Since $\delta>0$ is arbitrary, we see $\lim_{\varepsilon \downarrow 0 }\sup_{t\in[0,T]}\sup_{x \in \X} J_7(\varepsilon,t,x) = 0.$

		Since $u_\varepsilon(0,x)=u(0,x)=f(x)$ for all $x \in \X$, we have $U_\varepsilon(0)=0$. Collecting all terms, for every $t\in(0,T]$ and sufficiently small $\varepsilon>0$, we have
		\begin{align}
			&U_\varepsilon(t) \lesssim \int_0^t (t-s)^{-\alpha}U_\varepsilon(s) \mathrm ds +H_\varepsilon(t)+e^{-c_T \frac{t}{\varepsilon}} + \frac{\gamma_\varepsilon}{t^\beta}
		\end{align}
		for some $H_\varepsilon(t)$ with $\lim_{\varepsilon\downarrow 0} \sup_{t \in [0,T]} H_\varepsilon(t) = 0.$
		By the Volterra--Gr\"{o}nwall inequality (see for example, \cite{chu1967gronwall}), we obtain, uniformly in $t\in(0,T]$ and all sufficiently small $\varepsilon>0$,
		\begin{align}
			U_\varepsilon(t) &\lesssim H_\varepsilon(t) + e^{-c_T \frac{t}{\varepsilon}} + \frac{\gamma_\varepsilon}{t^\beta} \\
			&\quad\quad\quad+ \int_0^t \left( H_\varepsilon(s) + e^{-c_T \frac{s}{\varepsilon}} + \frac{\gamma_\varepsilon}{s^\beta}\right)(t-s)^{-\alpha} \exp \left( \int_s^t (t-\tau)^{-\alpha} \mathrm d\tau\right) \mathrm ds
		\end{align}
		which converges to zero for each $t>0$ by the dominated convergence theorem. 
		Since $u_\varepsilon(t,x)\ge 0$ almost surely for each $(t,x) \in \R_+\times \X$ and $\varepsilon \in (0,1]$, we see that $u(t,x)\ge 0$ almost surely for each $(t,x) \in \R_+\times \X$.
		Thus, by separability of $\R_+\times \X$ and the joint continuity of $u$, we see $u$ is non-negative.
	\end{proof}
	
	We also need the following approximation result which considers SPDEs whose coefficients are regularized.
	\begin{prop}[Approximation theorem] \label{prop:approximation of coefficients}
		Under the standing assumptions of this section, suppose that $f$ is bounded and \eqref{ineq:non_negativity_coef_bound_at_zero} holds.
		Then there exist sequences $(b_n)_{n \in \N}$ and  $(\sigma_n)_{n \in \N}$ of predictable functions
		on the product space $\Omega \times \mathbb R_+ \times \X \times \mathbb R$ so that 
		\begin{itemize}
			\item[\eq\label{eq:FineNConditions}] for each $n\in \mathbb N$, \eqref{eq:LiG}, \eqref{ineq:g_lip}, \eqref{eq:con}, \eqref{ineq:non_negativity_coef_bound_at_zero} and \eqref{ineq:non_negative_approx_coef_ineq} hold with $b$ and $\sigma$ being replaced by $b_n$ and $\sigma_n$; and that
			\item[\eq\label{eq:Approximation}] for every $T>0$ and $\lambda>0$,
			\begin{align}
				\lim_{n\to\infty}\sup_{t\in[0,T]}\sup_{x\in\X}
				h(x)^{-\lambda}\norm{u^{(n)}(t,x)-u(t,x)}_{L^2_\Omega}=0,
			\end{align}
			where, for each $n\in \mathbb N$, $(u^{(n)}_t)_{t\geq 0}$ is the unique strong solution to the equation
			\begin{align}\label{eq:n}
				\begin{cases}
					\partial_t u^{(n)}_t(x) 
					= \mathcal L^* u^{(n)}_t(x)+b_n\pr{t,x,u^{(n)}_t(x)} + \sigma_n\pr{t,x,u^{(n)}_t(x)} \dot{W}_t(x), \quad (t,x)\in \mathbb R_+\times \X,
					\\u^{(n)}_0(x) = f(x), \quad x\in \X.
				\end{cases}
			\end{align}
		\end{itemize}
	\end{prop}
	
	\begin{prop}\label{prop:non-negative2}
		Suppose that $f$ is non-negative and bounded, and that \eqref{ineq:non_negativity_coef_bound_at_zero} holds.
		Then, the unique strong solution $u$ to the SPDE  \eqref{eq:spde} is non-negative.
	\end{prop}
	\begin{proof}
		Let sequences $(b_n)_{n \in \N}$ and  $(\sigma_n)_{n \in \N}$ of predictable functions on the product space $\Omega \times \mathbb R_+ \times \X \times \mathbb R$ be given as in Proposition \ref{prop:approximation of coefficients}. 
		For every $n\in \mathbb N$, by Proposition \ref{thm:non_negtive}, \eqref{eq:FineNConditions}, and the condition that $f$ is non-negative, we obtain that $u^{(n)}$, the unique strong solution to the equation \eqref{eq:n}, is non-negative. 
		By \eqref{eq:Approximation}, for each fixed $(t,x)\in\R_+\times\X$,
		$u^{(n)}(t,x)\to u(t,x)$ in $L^2_\Omega$. Since every $u^{(n)}(t,x)$ is non-negative,
		the closedness of the non-negative cone in $L^2_\Omega$ implies that $u(t,x)\geq0$ almost surely.
		Thus, by joint continuity and separability of $\R_+\times\X$, $u$ is non-negative.
	\end{proof}
	
	\begin{proof}[Proof of Theorem \ref{thm:comparison} \eqref{part:main2 weak comparison}]
		Define an adapted $\mathcal C^{\mathrm{tem}}_{\X,h}$-valued continuous process $(v_t)_{t \ge 0}$ via
		\begin{align}
			v_t(x) = u_t(x) - u'_t(x),\qquad (t,x) \in \R_+\times \X.
		\end{align}
		Note that for every $T>0$ and  $(t,x)\in [0,T]\times \X$, almost surely,
		\begin{align}
			v_t(x) = P^*_t f_0(x)
			&+ \int_0^t \mathrm ds \int_\X  G(s,y;t,x) \widetilde b(s,y, v_s(y)) 
			m(\mathrm d y) \\
			&+ \iint_0^t  G(s,y;t,x) \widetilde \sigma(s,y,v_s(y)) 
			W(\mathrm ds,\mathrm dy), 
			\label{eq:vtx}
		\end{align}
		where $f_0 := f - f' \in \mathcal C^{\mathrm{tem}}_{\X,h}$ and for all $(s,y) \in \R_+\times \X$ and $r\in \mathbb R$ 
		\begin{align}
			(\tilde b(\omega,s,y,r))_{\omega \in \Omega} := 
			\widetilde b(s,y,r)&:= b\pr{s,y,
				r
				+u_s'(y)} - b(s,y,u'_s(y))\\
			(\tilde \sigma(\omega,s,y,r))_{\omega \in \Omega} := 
			\widetilde \sigma(s,y,
			r)&:= \sigma\pr{s,y, r+u_s'(y)} - \sigma(s,y,u'_s(y)).
		\end{align}
		Since $u'$ is adapted with continuous
		$\mathcal C^{\mathrm{tem}}_{\X,h}$-valued paths, it is predictable; hence
		$\widetilde b$ and $\widetilde\sigma$ are predictable. Moreover,
		\eqref{eq:con} and the continuity of $u'_t$ show that they also satisfy
		\eqref{eq:con}.
		We can verify that 
		\begin{itemize}
			\item[\eq\label{eq:tildeL}] $\tilde b$ and $\tilde \sigma$ satisfy the linear growth condition \eqref{eq:LiG} and the Lipschitz condition \eqref{ineq:g_lip} with $b$ and $\sigma$ being replaced by $\tilde b$ and $\tilde \sigma$.  
		\end{itemize}
		In fact, for every $T>0$, uniformly for every $\omega\in \Omega\setminus\mathrm{N}_{\ref{c:Lip}}$, $t\in [0,T]$, $r,r'\in \mathbb R$ and $x\in \X$, it holds that
		\begin{align}
			|\tilde b(\omega, t,x,r) |+|\tilde\sigma(\omega, t,x,r)| 
			&= |b\pr{\omega,t,x,r+u'(\omega, t, x)} - b(\omega,t,x,u'(\omega, t, x))| 
			\\& \qquad +|\sigma\pr{\omega,t,x,r+u'(\omega, t, x)} - \sigma(\omega,t,x,u'(\omega, t, x))| 
			\\& \overset{\eqref{ineq:g_lip}}{\lesssim} |r|, \label{eq:tbtsr}
		\end{align}
		and that
		\begin{align}
			&\quad |\tilde b(\omega, t,x,r) - \tilde b(\omega, t,x,r')|+|\tilde \sigma(\omega, t,x,r)-\tilde \sigma(\omega, t,x,r')| 
			\\&=| b\pr{\omega,t,x,r+u'(\omega, t, x)}  - b\pr{\omega,t,x,r'+u'(\omega, t, x)}  | 
			\\&\qquad +| \sigma\pr{\omega,t,x,r+u'(\omega, t, x)}  - \sigma\pr{\omega,t,x,r'+u'(\omega, t, x)}  |
			\\&{\overset{\eqref{ineq:g_lip}}{\lesssim} }|r - r'|.
		\end{align}
		In particular,  Theorem \ref{thm:main2}, \eqref{eq:tildeL} and \eqref{eq:vtx} now say that $(v_t)_{t\geq 0}$ is the unique strong solution to the SPDE
		\begin{equation}
			\begin{cases}
				\partial_t v_t(x) 
				= \mathcal L^* v_t(x) + \tilde b\left(t, x, v_t(x)\right) + \tilde\sigma\left(t, x ,v_t(x)\right) \dot W_t(x),
				\quad t \in \mathbb R_+,\, x \in \X,
				\\ v_0(x) = f_0(x), \quad x\in \X.
			\end{cases}
		\end{equation}
		
		Since $f_0\in \mathcal C^{\mathrm{tem}}_{\X,h}$ is non-negative, we can take a sequence $(f^{(k)}_0)_{k \in \mathbb N}$ of non-negative bounded elements in $\mathcal C^{\mathrm{tem}}_{\X,h}$ which converges to $f_0$ in $\mathcal C^{\mathrm{tem}}_{\X,h}$.
		From Theorem \ref{thm:main2} and \eqref{eq:tildeL}, for each $k\in \mathbb N$, there exists a unique strong solution $(v^{(k)}_t)_{t\geq 0}$ to the SPDE 
		\begin{equation}
			\begin{cases}
				\partial_t v^{(k)}_t(x) 
				= \mathcal L^* v^{(k)}_t(x) + \tilde b\left(t, x, v^{(k)}_t(x)\right) + \tilde\sigma\left(t, x ,v^{(k)}_t(x)\right) \dot W_t(x),
				\quad t \in \mathbb R_+,\, x \in \X,
				\\ v^{(k)}_0(x) = f_0^{(k)}(x), \quad x\in \X.
			\end{cases}
		\end{equation}
		By \eqref{eq:tildeL} and Proposition \ref{prop:continuity_initial}, we know that 
		\begin{equation}\label{eq:v}
			\lim_{k\to \infty}\|v^{(k)}  -v\|_{\mathbb M_{\lambda, p}(T)} = 0
		\end{equation}
		for every $p \ge 2$, $\lambda >0$ and $T>0$.
		
		Taking $r=0$ in \eqref{eq:tbtsr}, we obtain
		$\widetilde b(t,x,0)=\widetilde\sigma(t,x,0)=0$, and hence
		\begin{itemize}
			\item[\eq\label{eq:tildeLM}] $\tilde b$ and $\tilde \sigma$ satisfy  \eqref{ineq:non_negativity_coef_bound_at_zero} with $b$ and $\sigma$ being replaced by $\tilde b$ and $\tilde \sigma$.  
		\end{itemize}
		Now, for each $k\in \mathbb N$, by Proposition \ref{prop:non-negative2}, \eqref{eq:tildeLM}, and the fact that $f^{(k)}_0\in \mathcal C^{\mathrm{tem}}_{\X,h}$ is non-negative and bounded, we have that $v^{(k)}$ is non-negative.  
		Combining this with the separability of $\X$, continuity of $v$  and \eqref{eq:v}, we see that $v=u-u'$ is also non-negative as desired.
	\end{proof}

\section{Proof of strong comparison principle and strict positivity}\label{sec:proof of strong comparison}
	In this section, we prove Theorem \ref{thm:comparison} \eqref{part:main2 strong comparison} and \eqref{part:main2 strictly positivity}. We first establish Mueller-type large-deviation estimates for stochastic integrals in Proposition \ref{prop:MLDP} and Corollary \ref{lem:m_ldp_g}, and obtain a global weighted estimate for the stochastic convolution in Lemma \ref{lem:unit_scale_stochastic_part_bound}. 
	We then apply the iteration method in \cite{Zbl074960057}  to prove strict positivity for the drift-free equation \eqref{eq:spde_driftless}. Combining this result with Theorem \ref{thm:comparison} \eqref{part:main2 weak comparison} gives the strong comparison principle Theorem \ref{thm:comparison} \eqref{part:main2 strong comparison}, from which Theorem \ref{thm:comparison} \eqref{part:main2 strictly positivity} follows under \eqref{cond:pos_zero}.

	\subsection{Large deviation principle on metric measure space}
	In this section, for any metric space $(\mathcal T,d_\mathcal T)$, we denote $D(\mathcal T,d_\mathcal T; \varepsilon)$ as the smallest number of open balls with radius $\varepsilon>0$ in $\mathcal T$ required to cover $\mathcal T$. 

        The following product-space covering estimate is used in the chaining argument in the proof of Proposition~\ref{prop:MLDP}.
	\begin{lemma}\label{lem:product metric entropy}
		Suppose $(\mathcal X,d_\mathcal X)$ and $\pr{\mathcal Y,d_\mathcal Y}$ are metric spaces with finite radius such that for some $\beta_\mathcal X,\beta_\mathcal Y >0$, it holds uniformly in $\varepsilon\in (0,1]$ that 
		\begin{align}
			D(\mathcal X,d_\mathcal X; \varepsilon)\lesssim \varepsilon^{-\beta_\mathcal X},\qquad D(\mathcal Y,d_\mathcal Y; \varepsilon)\lesssim \varepsilon^{-\beta_\mathcal Y}.
		\end{align}
		Suppose $\xi_\mathcal X,\xi_\mathcal Y \in (0,1]$ and let $\mathcal S:= \mathcal X \times \mathcal Y$ equipped with the metric $d_\mathcal S\pr{(x,y),(x',y')}:= d_\mathcal X(x,x')^{\xi_\mathcal X}+d_\mathcal Y(y,y')^{\xi_\mathcal Y}$. Then 
		\begin{align}
			D\pr{\mathcal S,d_\mathcal S; \varepsilon} \lesssim \varepsilon^{-\pr{\frac{\beta_\mathcal X}{\xi_\mathcal X}+\frac{\beta_\mathcal Y}{\xi_\mathcal Y}}},\quad \text{ uniformly in }\varepsilon \in (0,1].
		\end{align}
	\end{lemma}

	We will need the following local version of large deviation type estimates of stochastic integrals with respect to Gaussian noises that  is similar to that of \cite[Theorem 4.2]{dalang2009some} (also see \cite[Lemma 9.1]{barnes2021effect}). The proof relies on the ball-chaining technique on  metric spaces that can be found, for example, in \cite{ledoux2013probability} and \cite[Chapter 2.4]{khoshnevisan2002multiparameter}. Recall that $Q$ in \eqref{eq:W_covariance} is the covariance operator of $W$, with operator norm denoted by $\|Q\|$.
	\begin{prop}\label{prop:MLDP}
		Let $T,R >0$, $x_0 \in \X$ and assume the following assumptions hold:
		\begin{enumerate}
			\item There is some $\beta_0 > 0$ so that $D(B_R(x_0),d; \varepsilon) \lesssim \varepsilon^{-\beta_0}$ uniformly for $\varepsilon \in (0,1]$, where $D$ is defined at the beginning of this subsection.
			\item $g_{t,x}(s,y):=g(t,s,x,y)$ on $ \R_+\times \R_+\times \X \times \X$ and $\eta(s,y)$ on $\R_+ \times \X$ are non-negative deterministic real-valued functions. There are $\xi_1,\xi_2 \in (0,1]$ such that the supremum below is finite, and we set
			\begin{align}\label{ineq:ldp local holder type condition}
				B:= \|Q\|\sup_{t,t' \in [0,T],\, x,x' \in B_R(x_0)} \frac{\int_0^\infty \left\| \left|g_{t,x}(s,\cdot) - g_{t',x'}(s,\cdot)\right| \eta(s,\cdot)\right\|_{\mathcal H_K}^2 \mathrm ds}{\left(\frac{d(x,x')}{2R}\right)^{\xi_1} + \left|\frac{t-t'}{T}\right|^{\xi_2}}.
			\end{align}
			Moreover, suppose there exists $(t_*,x_*)\in[0,T]\times B_R(x_0)$ such that
			\begin{align}
				\int_0^\infty
				\norm{g_{t_*,x_*}(s,\cdot)\eta(s,\cdot)}_{\mathcal H_K}^2
				\mathrm ds<\infty.
			\end{align}
			\item $\sigma$ is a predictable random field on $\R_+ \times \X$ so that 
			\begin{align}
				\left|\sigma(t,x) \right| \le \eta(t,x),\quad \text{almost surely for all } (t,x) \in \R_+ \times \X.
			\end{align}
		\end{enumerate}
		For $(t,x) \in [0,T]\times B_R(x_0)$, let
		\begin{align}\label{eq:s_conv}
			N(t,x):= \int_0^\infty \int_\X g(t,s,x,y)\sigma(s,y)  W(\mathrm ds,\mathrm dy).
		\end{align}
		Then there is a jointly continuous modification $\widetilde{N}(t,x)$ of $N(t,x)$ on $[0,T] \times B_R(x_0)$, so that for  $c:=  4\sum_{n= 1}^\infty  \sqrt{n}2^{-n/2} $, the following inequality holds uniformly in $\lambda > 0$, with an implicit constant independent of $T$, $B$ and $\sigma$,
		\begin{align}\label{ineq:MLDP_local}
			\bbP \left( \sup_{(t,x),(t',x') \in [0,T]\times B_R(x_0)} \left| \widetilde{N}(t,x) - \widetilde{N}(t',x') \right| > \lambda\right) \lesssim \exp\left(- \frac{\lambda^2}{2c^2B}\right).
		\end{align}
		When $B=0$, the right-hand side of \eqref{ineq:MLDP_local} is understood as $e^{-\infty}=0$.
	\end{prop}

	We will need the following lemma.
	\begin{lemma}\label{lem:mod_cont_sconv}
		Under the assumptions and notations of Proposition \ref{prop:MLDP}, for all \(t,t'\in[0,T]\), \(x,x'\in B_R(x_0)\), and \(\lambda>0\),
		\begin{align}\label{ineq:s_cont1}
			\mathbb{P} \left(\left|N(t,x) - N(t',x') \right| > \sqrt{B}\lambda \right) &\le 4 \exp\left(\frac{ -\,\lambda^2}{2\left( \left(\frac{d(x,x')}{2R} \right)^{\xi_1}+  \left| \frac{t-t'}{T}\right|^{\xi_2}\right)}\right).
		\end{align}
	\end{lemma}
	\begin{proof}
		If $B=0$, then the stochastic-integral isometry, Lemma \ref{lem:HS}, \eqref{ineq:ldp local holder type condition}, and $|\sigma|\le\eta$ imply that $N(t,x)-N(t',x')=0$ in $L^2(\Omega)$ for every $(t,x),(t',x')\in[0,T]\times B_R(x_0)$, so \eqref{ineq:s_cont1} is immediate. Hence, we may assume $B>0$.
		
		For $t,t'\in[0,T]$, $x,x'\in B_R(x_0)$, and $\tau\ge0$, define the difference martingale
		\begin{align*}
			M_{(t,t'),(x,x')}(\tau)
			:=\int_0^\infty\int_\X \mathbf 1_{[0,\tau]}(s)
			\big(g_{t,x}(s,y)-g_{t',x'}(s,y)\big)\sigma(s,y)
			W(\mathrm ds,\mathrm dy).
		\end{align*}
		Then $\big(M_{(t,t'),(x,x')}(\tau)\big)_{\tau\ge0}$ is a continuous martingale. By  Lemma \ref{lem:HS}, it holds almost surely for every $\tau\ge0$ that
		\begin{align}
			\left[ M_{(t,t'),(x,x')} \right]_\tau &= \int_0^\tau \norm{Q^{1/2}\left[\left(g_{t,x}(s,\cdot) - g_{t',x'}(s,\cdot)\right) \sigma(s,\cdot)\right]_K}_{\tilde{\mathcal H}_{K}}^2 \mathrm ds\\
			&\,\,\le \|Q\|\int_0^\tau \left\| \left(g_{t,x}(s,\cdot) - g_{t',x'}(s,\cdot)\right)\sigma(s,\cdot)\right\|_{\mathcal H_K}^2 \mathrm ds\\
			&\,\,\le \|Q\|\int_0^\tau \left\| \left|g_{t,x}(s,\cdot) - g_{t',x'}(s,\cdot)\right| \eta(s,\cdot)\right\|_{\mathcal H_K}^2 \mathrm ds\\
			&\overset{\eqref{ineq:ldp local holder type condition}}{\le} \left( \left(\frac{d(x,x')}{2R} \right)^{\xi_1}+  \left| \frac{t-t'}{T}\right|^{\xi_2}\right) B.
		\end{align}
		The last bound shows that this martingale has finite energy and is bounded in $L^2(\Omega)$. Hence, by the martingale convergence theorem and the stochastic-integral isometry,
		\begin{align*}
			M_{(t,t'),(x,x')}(\tau)
			\longrightarrow M_{(t,t'),(x,x')}(\infty)
			=N(t,x)-N(t',x')
		\end{align*}
		almost surely and in $L^2(\Omega)$ as $\tau\to\infty$. By the Dambis--Dubins--Schwarz theorem \cite[Chapter~V, Theorem~(1.7)]{revuz1999continuous}, after enlarging the probability space if necessary, there is a Brownian motion $(\mathcal B_r)_{r\ge0}$ such that
		\(M_{(t,t'),(x,x')}(\tau)=\mathcal B_{[M_{(t,t'),(x,x')}]_\tau}\) for every $\tau\ge0$. Therefore, martingale convergence and the reflection principle give
		\begin{align}
			\mathbb{P}\left(\left|N(t,x)-N(t',x')\right| > \sqrt{B}\lambda\right)
			&=\mathbb{P}\left(\left|M_{(t,t'),(x,x')}(\infty)\right| > \sqrt{B}\lambda\right)\\
			&\le \mathbb{P}\left(\sup_{0 \le r \le  \left( \left(\frac{d(x,x')}{2R} \right)^{\xi_1}+  \left| \frac{t-t'}{T}\right|^{\xi_2}\right) B} |\mathcal B_r| > \sqrt{B}\lambda\right)\\
			&\le 4 \exp\left(\frac{ -\,\lambda^2}{2\left( \left(\frac{d(x,x')}{2R} \right)^{\xi_1}+  \left| \frac{t-t'}{T}\right|^{\xi_2}\right)}\right).
		\end{align}
		Where we used the standard Gaussian tail bound: if $X$ is a centered Gaussian random variable with variance $\nu>0$, then
		\begin{align}
			\bbP \left(|X|> \lambda \right) \le 2 e^{- \frac{\lambda^2}{2\nu}}.
		\end{align}
	\end{proof}
	
	\begin{proof}[Proof of Proposition \ref{prop:MLDP}]
		If $B=0$, then the same zero-increment argument shows that $\widetilde N(t,x):=N(t_*,x_*)$ is a jointly continuous modification of $N$ on $[0,T]\times B_R(x_0)$. Consequently, the probability in \eqref{ineq:MLDP_local} is zero for every $\lambda>0$, and the conclusion follows from the convention above. Hence, suppose $B>0$ for the remainder of the proof.
		
		Since $\xi_1,\xi_2 \in (0,1]$, we may define the set ${\mathcal T}:= [0,T] \times B_R(x_0)$ and the new metric $d_{\mathcal T}$ on $\mathcal T$ by 
		\begin{align}
			d_{\mathcal T}((s,x),(s',x')):= \frac{1}{2} \left(\left(\frac{d(x,x')}{2R} \right)^{\xi_1}+ \left|\frac{s-s'}{T} \right|^{\xi_2}\right),\quad (s,x),\, (s',x') \in \mathcal T.
		\end{align}
		Then $({\mathcal T}, d_\mathcal T)$ is a metric space with diameter  at most $1$.   Recall $D(\mathcal T,d_\mathcal T; \varepsilon)$ defined right before Lemma \ref{lem:product metric entropy}, by Lemma \ref{lem:product metric entropy}, we see for  $\beta = \frac{1}{\xi_2} + \frac{\beta_0}{\xi_1}$, there is a positive constant $\C\label{c:lld_tem1} = \Cr{c:lld_tem1}>0$, depending on the covering constant in assumption~(1) and on $R,\beta_0,\xi_1,\xi_2$, so that 
		\begin{align}\label{ineq:comb_entro}
			D({\mathcal T},d_{\mathcal T}; \varepsilon) \le \Cr{c:lld_tem1} \varepsilon^{-\beta},\quad  \text{ uniformly in }0<\varepsilon  <1.
		\end{align}
		Next, we construct a countable dense sequence of points in $\mathcal T$.
		Denote $B_{\mathcal T}(t,r)$ as the open ball in ${\mathcal T}$ centered at $t \in {\mathcal T}$ with radius $r>0$ under the metric $d_{\mathcal T}$.
		Pick any $t_0 \in {\mathcal T}$ and let ${\mathcal T}_0 = \{t_0\}$.
		For $n \ge 1$, we can find  ${\mathcal T}_n := \{x_i\}_{i \in I_n}$ with $|I_n| \le \Cr{c:lld_tem1} 2^{\beta n}$, where $|I_n|$ denotes the cardinality of $I_n$,  and
		\begin{align}
			{\mathcal T} \subset \bigcup_{i \in I_n} B_{\mathcal T}\left(x_i,\frac{1}{2^{n}}\right).
		\end{align}
		Finally, we let ${\mathcal T}_\infty := \bigcup_{n \ge 0} {\mathcal T}_n$ which is dense in ${\mathcal T}$. 
		
		Now we define the chaining map $\pi_n : {\mathcal T}_{n} \to {\mathcal T}_{n-1}$ for $n \ge 1$ that ‘chains' the points that are close together. By construction, for any $n \ge 1$ and $t \in {\mathcal T}_n$, there is $t_{n-1} \in {\mathcal T}_{n-1}$ so that $d_{\mathcal T}(t,t_{n-1}) \le 2^{-n+1}.$ For each $t \in {\mathcal T}_n$, we choose exactly one such $t_{n-1} \in {\mathcal T}_{n-1}$ and define $\pi_n(t):= t_{n-1}$ and  set $\pi_0(t_0) = t_0$.
		
		For $\rho > 2^{2\beta }$, we define the  events
		$$
		A_\rho := \bigcap_{n \ge 1} \left\{\left|N(t) - N(\pi_n(t)) \right| \le c_n(\rho),\quad \text{ for all } t \in {\mathcal T}_n\right\}
		$$
        where $c_n(\rho) = 2\sqrt{Bn\ln(\rho)}\,2^{-n/2}$.
		Observe that on the event $A_\rho$, for any $n \ge 1$ and $t \in {\mathcal T}_n$, where $\pi_{n+1}\circ \dots \circ \pi_n$ denotes the identity map,
			\begin{align}
				\left|N(t) - N(t_0)\right| &\le \sum_{i = 1}^n \left|N\left(\pi_{i+1}\circ \dots \circ \pi_n(t)\right) - N\left(\pi_i \circ \dots \circ \pi_n(t) \right) \right|\\
				&\le 2\sqrt{B \ln (\rho)} \sum_{i = 1}^\infty  \sqrt{i2^{-i}},
			\end{align}
		and hence  
		\begin{align}\label{ineq: sup distance a rho}
			\sup_{t,t' \in {\mathcal T}_\infty}\left|N(t) - N(t') \right| \le 4\sqrt{B\ln (\rho)} \sum_{n \ge 1}\sqrt{n}2^{-n/2}.
		\end{align}
		Then by Lemma \ref{lem:mod_cont_sconv},
		\begin{align}
			\bbP \left( A^c_\rho\right) &\,\,\,\le \sum_{n = 1}^\infty \sum_{t \in {\mathcal T}_n} \bbP \left(\left|N(t) - N(\pi_n(t)) \right| > c_n(\rho)\right)\\
			&\,\,\,\le 4 \sum_{n = 1}^\infty \sum_{t \in {\mathcal T}_n} \exp\left(-  \frac{c_n(\rho)^2}{4d_{\mathcal T}(t,\pi_n(t)) B} \right)\\
			&\overset{\eqref{ineq:comb_entro}}{\le} 4\Cr{c:lld_tem1}\sum_{n = 1}^\infty 2^{\beta n} \exp \left(-  n \ln(\rho)/2 \right)\\
			&\,\,\,= 4 \Cr{c:lld_tem1} \sum_{n = 1}^\infty\left( 2^\beta  \rho^{-\frac{1}{2}}\right)^n\\
			&\,\,\,= \frac{4 \Cr{c:lld_tem1} 2^\beta }{ \rho^{\frac{1}{2}}- 2^\beta }.
		\end{align}
		
		Fix $\lambda>0$ and set $\rho:=\exp(\lambda^2)$. Let $c:=  4\sum_{n = 1}^\infty  \sqrt{n}2^{-n/2} $. If $\rho > 2^{2\beta +2}$, then by \eqref{ineq: sup distance a rho},
		\begin{align}\label{ineq:mldp_last}
			\bbP \left(\sup_{t,t' \in {\mathcal T}_\infty}\left|N(t) - N(t') \right| > c \lambda \sqrt{B} \right) \le \bbP \left( A_\rho^c\right) \le  \frac{4 \Cr{c:lld_tem1} 2^\beta }{ \rho^{\frac{1}{2}}- 2^\beta } \le \frac{8 \Cr{c:lld_tem1} 2^\beta}{\rho^{\frac{1}{2}}} = \frac{8 \Cr{c:lld_tem1} 2^{\beta}}{\exp\left(\lambda^2/2\right)}.
		\end{align}
		If $\rho  \in \left[1, 2^{2\beta+2}\right]$, we see $\exp \left(-\frac{\lambda^2}{2} \right)\ge 2^{-\beta -1}$.
		Hence $8 \max\left\{\Cr{c:lld_tem1},1 \right\}2^\beta \exp\left(- \lambda^2/2 \right) \ge 1$.
		So we get for all $\lambda >0$
		\begin{align}
			\bbP \left(\sup_{t,t' \in {\mathcal T}_\infty}\left|N(t) - N(t') \right| > c \lambda \sqrt{B} \right) \le \frac{8\max\left\{ \Cr{c:lld_tem1} ,1\right\} 2^{\beta}}{\exp\left(\lambda^2 / 2 \right)}.
		\end{align}
			Finally, integrating \eqref{ineq:s_cont1} yields, for every $p\ge1$ and $z,z'\in\mathcal T$,
			\[
				\mathbb E\left[|N(z)-N(z')|^p\right]\le C_p B^{p/2}d_{\mathcal T}(z,z')^{p/2}.
			\]
			Choosing $p>2\beta$, \cite[Theorem 2.3.1]{khoshnevisan2002multiparameter} gives a continuous modification. Since ${\mathcal T}_\infty$ is countable, on a common full-measure event the modification agrees with $N$ on ${\mathcal T}_\infty$, and by density and continuity the preceding dense-set supremum bound transfers to the supremum over all of ${\mathcal T}$, yielding the desired inequality.
	\end{proof}
	Next, we will extend Proposition \ref{prop:MLDP} globally to the whole space. The following elementary inequality makes the resulting weighted tail estimates summable over the spatial covering in \eqref{ineq:faster_than_ms}.
	\begin{lemma}\label{lem:exp_ineq_elementary}
		Suppose $y >0$ and $x \ge 1$, then
		$\exp\left( - yx\right) \le  \frac{\exp\left( - y/2\right)}{xy}$.
	\end{lemma}
	The key estimate we will be using is the following,  similar to that of \cite{Zbl074960057}, which will allow us to have tail probability estimates of stochastic integrals globally.
	
	\begin{lemma}\label{lem:mldp_global}
		Suppose the conditions of Proposition \ref{prop:MLDP} hold, $(\X,d)$ satisfies \eqref{ineq:entropy bound} and  $h:\X \to [1,\infty)$ is a weight function.
		Let  $\lambda_0>0$ be as in \eqref{ineq:faster_than_ms} and $\lambda > 0$.
			Assume there is $\alpha > 0$ so that for each $T>0$,  the following inequality holds uniformly in $  (t,x) \in (0,T]\times \X$,
			\begin{align}\label{ineq:var_bound_h}
				\sup_{s,s' \in [0,t],\, y,y' \in B_1(x)} \frac{\int_0^\infty \left\| \left|g_{s,y}(\tau,\cdot) - g_{s',y'}(\tau,\cdot)\right| \eta(\tau,\cdot)\right\|_{\mathcal H_K}^2 \mathrm d\tau}{\left(\frac{d(y,y')}{2}\right)^{\xi_1} + \left|\frac{s-s'}{t}\right|^{\xi_2}}
				\lesssim h(x)^{2\lambda - \lambda_0}t^\alpha.
			\end{align}
			Let $N(t,x)$ be as in \eqref{eq:s_conv}, and assume that the continuous modification provided by Proposition \ref{prop:MLDP} satisfies $N(0,\cdot)\equiv0$ almost surely. Then, for each $T>0$, there exists $ \C\label{c:gld_4} = \Cr{c:gld_4}(T,\lambda) > 0$ so that uniformly for $t \in (0,T]$ and $\varepsilon \in (0,1)$,
			\begin{align}
				\bbP \left(\sup_{(s,x) \in [0,t]\times \X} |N(s,x)|h(x)^{-\lambda} > \varepsilon \right) \lesssim  \varepsilon^{-2} t^\alpha \exp\left( - \frac{\Cr{c:gld_4}}{2t^\alpha}\varepsilon^2\right).
			\end{align}
			For $t=0$, the left hand side is zero, so the estimate is trivial.
	\end{lemma}
	\begin{proof}
		We see from condition \eqref{ineq:ref_regular} that for each $\lambda  \ge 0$, there is $C_\lambda>0$ so that for all $x \in \X$,
		\begin{align}\label{ineq:w_ball_reg}
			C_\lambda h(x)^{-\lambda} \ge \sup_{z \in B_1(x)} h(z)^{-\lambda} \ge h(y)^{-\lambda},\quad \text{ for all } y \in B_1(x).
		\end{align}
			Let $\lambda, \lambda_0$ be as in \eqref{ineq:var_bound_h}. By Proposition \ref{prop:MLDP}, applied on the time interval $[0,t]$, with $R=1$ and $x_0=x$, there is $\Cr{c:gld_4} = \Cr{c:gld_4}(\lambda,T)>0$ so that, uniformly in $\varepsilon \in (0,1)$ and $(t,x) \in (0,T] \times \X$,
			\begin{align}
				\bbP \left(\sup_{(s,y) \in [0,t] \times B_1(x)} |N(s,y)|h(y)^{-\lambda } > \varepsilon \right)&\overset{\eqref{ineq:w_ball_reg}}{\le} \bbP \left( \sup_{(s,y) \in [0,t] \times B_1(x)} |N(s,y)| \cdot C_\lambda h(x)^{-\lambda}> \varepsilon \right) \\
				&\,\,\,= \bbP \left(\sup_{(s,y) \in [0,t] \times B_1(x)} |N(s,y)| > C_\lambda^{-1}h(x)^\lambda\varepsilon\right)\\
				&\,\,\,\le \bbP \left(\sup_{\substack{s,s'\in[0,t]\\y,y'\in B_1(x)}} |N(s,y)-N(s',y')| > C_\lambda^{-1}h(x)^\lambda\varepsilon\right)\\
				&\overset{\eqref{ineq:MLDP_local}}{\lesssim} \exp \left(- \Cr{c:gld_4} \frac{h(x)^{\lambda_0}}{t^\alpha}\varepsilon^2\right)\\
				& \,\,\,\lesssim \varepsilon^{-2} t^\alpha \exp\left( - \frac{\Cr{c:gld_4}}{2t^\alpha}\varepsilon^2\right)h(x)^{-\lambda_0},
			\end{align}
			where the assumption $N(0,y) = 0$ for all $y \in \X$ was used in the third line by taking $(s',y')=(0,y)$; in the last inequality we used Lemma \ref{lem:exp_ineq_elementary}, absorbed constants into the implicit constant, and used the fact $h(x) \ge 1$ for all $x \in \X$.
		Now let $\{x_n\}_{n \in \mathcal I}\subset \X$ so that  \eqref{ineq:faster_than_ms} holds.
		Then we see uniformly in $t \in [0,T]$ and $\varepsilon \in (0,1)$
		\begin{align}
			\bbP \left(\sup_{(s,x) \in   [0,t]\times \X}|N(s,x)|h(x)^{-\lambda} > \varepsilon \right) &= \bbP \left(\bigcup_{n \in \mathcal I} \left\{\sup_{(s,x) \in [0,t] \times B_1(x_n)}|N(s,x)|h(x)^{-\lambda} >\varepsilon  \right\}\right)\\
			&\le \sum_{n \in \mathcal I} \bbP \left(\sup_{(s,x) \in [0,t] \times B_1(x_n)}|N(s,x)|h(x)^{-\lambda} >\varepsilon  \right)\\
			&\lesssim \varepsilon^{-2} t^\alpha \sum_{n \in \mathcal I} \exp\left( - \frac{\Cr{c:gld_4}}{2t^\alpha}\varepsilon^2\right)h(x_n)^{-\lambda_0}\\
			&\lesssim \varepsilon^{-2}t^\alpha \exp \left( - \frac{\Cr{c:gld_4}}{2t^\alpha} \varepsilon^2 \right).
		\end{align}
	\end{proof}

	\begin{corollary}\label{lem:m_ldp_g}
		Under the assumptions of Theorem  \ref{thm:main2} with $\lambda_0>0$ as in \eqref{ineq:faster_than_ms}{, let} $\alpha \in (0,1)$, $\lambda>0$ and $\eta:\R_+ \times \X \to \R_+$ be a deterministic measurable function. Suppose there exist $\xi_1,\xi_2\in(0,1]$ so that, for each $T>0$, the following inequality holds uniformly in $(t,x) \in (0,T]\times \X$:
		\begin{align}\label{ineq:var_bound_h_k}
			\sup_{s,s' \in [0,t],\, y,y' \in B_1(x)} \frac{\int_0^{s\vee s'} \left\|\left|G_{s - \tau}(\cdot,y) - G_{s' - \tau}(\cdot,y')\right|\eta(\tau,\cdot)\right\|_{\mathcal H_K}^2 \mathrm d\tau}{d(y,y')^{\xi_1} + \left|s-s'\right|^{\xi_2}}
			\lesssim h(x)^{2\lambda - \lambda_0}t^{1-\alpha}.
		\end{align}
		 Let $T>0$ and suppose $(f(t,x),\, (t,x) \in \R_+ \times \X)$ is a predictable random field so that   $|f(t,x)| \le \eta(t,x)$ almost surely for all $(t,x) \in [0,T] \times \X$. If we let 
		\begin{align}
			N(t,x):= \int_0^t \int_\X G_{t-s}(y,x)f(s,y)W(\mathrm dy\,\mathrm ds),\quad (t,x) \in \R_+ \times \X,
		\end{align}
		then  there exists $ \C\label{c:gld_3} = \Cr{c:gld_3}(T,\lambda) > 0$ so that 
		\begin{align}
			\bbP \left(\sup_{(s,x) \in [0,t]\times \X} |N(s,x)|h(x) ^{-\lambda}> \varepsilon \right) \lesssim \varepsilon^{-2} t^{1-\alpha} \exp\left( - \frac{\Cr{c:gld_3}}{2t^{1-\alpha}}\varepsilon^2\right),
		\end{align}
		uniformly in $t\in (0,T],\, \varepsilon\in (0,1)$.
	\end{corollary}
	
	\subsection{Proof of strict positivity}
	In the sequel, we work with the following drift-free integral equation and assume that all the hypotheses of  Theorem~\ref{thm:comparison} \eqref{part:main2 strictly positivity}, except the ones on the initial condition, hold, 
	\begin{equation}\label{eq:spde_driftless}
		\begin{cases}
			u_t(x) &= P_t^*f(x) + \int_0^t \int_\X G_{t-s}(y,x) \sigma(s,y,u_s(y)) W(\mathrm dy\,\mathrm ds), \quad (t,x) \in \R_+ \times \X,\\
			u_0 &=f \in \mathcal C^{\mathrm{tem}}_{\X,h}.
		\end{cases}
	\end{equation}
	Let $u$ be the solution to the integral equation \eqref{eq:spde_driftless} and set for $(t,x) \in \R_+ \times \X$
	\begin{align}\label{eq:nt}
		Z(t,x) = \int_0^t \int_\X G_{t-s}(y,x) \sigma(s,y,u_s(y)) W(\mathrm dy,\mathrm ds).
	\end{align}
	
	Recall $b>0$ as in \eqref{ineq:entropy bound}, chosen for balls of radius $2$, and denote $\mathfrak{b}:= (b+1)/\xi$, where 
	$\xi \in (0,1]$ is as in Proposition \ref{prop:existence of C_tem modification}. The following Lemma is a direct consequence of \cite[Theorem 1.1]{kratschmer2022kolmogorov}.
	\begin{lemma}\label{lem:metric_entropy_Lp_chain}
	Let $(\mathcal T, d_{\mathcal T})$ be a complete metric space of diameter at most $2$, suppose  there exists $c_0>0$ so that $D(\mathcal T, d_{\mathcal T}; \varepsilon) \le c_0 \varepsilon^{-\mathfrak{b}}$ for all $\varepsilon \in (0,2)$. Suppose $B>0$, $p >\mathfrak{b}$ and  $(N(t))_{t \in \mathcal T}$ is an almost surely continuous random field on $\mathcal T$ such that
	\begin{align}\label{ineq:metric_entropy_Lp_chain_moment_bound}
		\norm{N(t) - N(t')}_{L^p_\Omega} \le B d_{\mathcal T}(t,t'),\quad \text{ for all } t,t' \in \mathcal T.
	\end{align}
	Then there exists $\C\label{c:kol_constant} = \Cr{c:kol_constant}(p,\mathfrak{b},c_0)>0$ so that for every $t_0 \in \mathcal T$, 
	\begin{align*}
		\norm{\sup_{t \in \mathcal T} |N(t) - N(t_0)|}_{L^p_\Omega} \le \Cr{c:kol_constant} B.
	\end{align*}
	\end{lemma}

	Here we shall make the constants that appeared previously explicit. Denote $S_0:= \sum_{n \in \mathcal I} h(x_n)^{-\lambda_0}$, where $\{x_n\}\subset \X$ is as in \eqref{ineq:faster_than_ms}. By \eqref{ineq:ref_regular}, there exists $\mathfrak{c}>0$ so that for all $x,y \in \X$ with $d(x,y) \le 1$, 
	\begin{align}
		\mathfrak{c}^{-1} h(x) \le h(y) \le \mathfrak{c} h(x).
	\end{align}
	Let $\C\label{c:semi_group_norm}>0$ be such that $\sup_{t \in (0,1]}\norm{P_t^* \mathbf 1}_{L^\infty_\X} \le \Cr{c:semi_group_norm}$.
	
	\begin{lemma}\label{lem:unit_scale_stochastic_part_bound}
		Let $u$ be the $\mathcal C^{\mathrm{tem}}_{\X,h}$-valued solution to \eqref{eq:spde_driftless} with $f \in \mathcal C^{\mathrm{tem}}_{\X,h}$ so that $|f(x)| \le 2$ for all $x \in \X$. Recall $Z$ as in \eqref{eq:nt}. Then for each $\lambda >0$, there exists $p_0 = p_0(\lambda) \ge 2$ such that for all $p \ge p_0$, there exists $\kappa := \kappa(p,\lambda)>0$ so that for all $N >0$, 
		\begin{align*}
			\mathbb P \left( \sup_{(t,x) \in [0,1]\times \X} |Z(t,x)|h(x)^{-\lambda} > N \right) \le S_0 \pr{\frac{\kappa}{N}}^p.
		\end{align*}
		In particular, the constants depend on the noise coefficient $\sigma$ only through the implicit constant $L>0$ in \eqref{ineq:g_lip} with $T=1$.
	\end{lemma}
	\begin{proof}
		Note that since $f$ is continuous and bounded by $2$, we have $|u(0,x)|  = |f(x)| \le 2$ for all $x \in \X$. Let $\{x_n\}_n \subset \X$ be as in \eqref{ineq:faster_than_ms} and set $\mathcal T_n := [0,1] \times \overline{B_1(x_n)}$ for each $n \in \mathcal I$ and $d_\mathcal T((t,x),(t',x')) := \frac12\left(|t-t'|^{\xi} + d(x,x')^{\xi}\right)$ for all $(t,x),(t',x') \in \mathcal T_n$, where $\xi \in (0,1]$ is as in Proposition \ref{prop:existence of C_tem modification}. By examining the proof of Proposition \ref{prop:existence of C_tem modification}, we see for each $\lambda'\in(0,1/2)$, there exists $\C\label{c:moment_for_stochastic_integral} = \Cr{c:moment_for_stochastic_integral}(p,\lambda')>0$ so that for all $n \in \mathcal I$ and $(t,x),(t',x') \in \mathcal T_n$, it holds that
		\begin{align*}
			\norm{Z(t,x) - Z(t',x')}_{L^p_\Omega} \le \Cr{c:moment_for_stochastic_integral}L h(x)^{\lambda'} \left(|t-t'|^{\xi} + d(x,x')^{\xi}\right) = 2\Cr{c:moment_for_stochastic_integral} L h(x)^{\lambda'}d_{\mathcal T}((t,x),(t',x')),
		\end{align*}
		where $L$ is as in \eqref{ineq:g_lip} with $T=1$. First let $\lambda\in(0,1/2]$, and take $p$ sufficiently large so that $\lambda':=\lambda-\lambda_0/p\in(0,1/2)$.
		
		Since the implicit constant in \eqref{ineq:entropy bound} is independent of $x \in \X$, we see by Lemma \ref{lem:metric_entropy_Lp_chain} and Proposition \ref{prop:existence of C_tem modification}, by choosing $p > \mathfrak{b}$,  there exists $\Cr{c:kol_constant} = \Cr{c:kol_constant}(p,\mathfrak{b})>0$ so that for each $n \in \mathcal I$, upon taking $t_0=(0,x_n)$ and using $Z(0,x_n)=0$,
		\begin{align}
			\norm{\sup_{(t,x) \in \mathcal T_n} |Z(t,x)|}_{L^p_\Omega} \le 2\Cr{c:kol_constant} \Cr{c:moment_for_stochastic_integral}L h(x_n)^{\lambda'}.
		\end{align}
		This implies 
		\begin{align*}
			\norm{\sup_{(t,x) \in \mathcal T_n} |Z(t,x)|h(x)^{-\lambda}}_{L^p_\Omega}  \le 2\mathfrak{c}\Cr{c:kol_constant} \Cr{c:moment_for_stochastic_integral}L h(x_n)^{\lambda' - \lambda}= 2\mathfrak{c}\Cr{c:kol_constant} \Cr{c:moment_for_stochastic_integral}L h(x_n)^{-\lambda_0/p}.
		\end{align*}
		Therefore, by Markov's inequality, we see for each $N>0$, it holds that 
		\begin{align*}
			\mathbb P \left(\sup_{(t,x) \in [0,1]\times \X} |Z(t,x)|h(x)^{-\lambda} > N \right) &\le \mathbb P \left(\bigcup_{n \in \mathcal I} \left\{\sup_{(t,x) \in \mathcal T_n} |Z(t,x)|h(x)^{-\lambda} > N \right\}\right)\\
			&\le \sum_{n \in \mathcal I} \mathbb P \left(\sup_{(t,x) \in \mathcal T_n} |Z(t,x)|h(x)^{-\lambda} > N \right)\\
			&\le \sum_{n \in \mathcal I} N^{-p} \norm{\sup_{(t,x) \in \mathcal T_n} |Z(t,x)|h(x)^{-\lambda}}_{L^p_\Omega}^p\\
			&\le \sum_{n \in \mathcal I} N^{-p} \left(2\mathfrak{c}\Cr{c:kol_constant} \Cr{c:moment_for_stochastic_integral}L\right)^p h(x_n)^{-\lambda_0}\\
			&:= S_0 \pr{\frac{\kappa}{N}}^p,
		\end{align*}
		which is the desired result for $\lambda\in(0,1/2]$. If $\lambda>1/2$, then $h\ge1$ implies $h^{-\lambda}\le h^{-1/2}$, so the estimate follows from the case $\lambda=1/2$.
	\end{proof}
	Note that \eqref{ineq:instant_propagation} implies, for each $t,M>0$, $0<r_0<R$ and $x_0\in\X$, there exists $N_0 = N_0(x_0,t,M,r_0,R) \in \N$ so that, for all sufficiently large $Q$ and all $r \in [r_0,R]$,
	\begin{align}\label{ineq:ip_explicite}
		\inf_{s \in [t/Q,2t/Q]} \inf_{x \in B_{r+M/Q}(x_0)}\int_{B_r(x_0)} G_s(y,x)m(\mathrm dy) \ge \frac{1}{N_0}.
	\end{align}
	\begin{prop}\label{prop:one_step_estimates}
		Let $\lambda \in (0,1/2]$, $x_0 \in \X$, $M,t>0$ and $0<r_0<R$ be arbitrary but fixed. Let $N_0=N_0(x_0,t,M,r_0,R)$ be as in \eqref{ineq:ip_explicite}, and denote $H:= \sup_{x \in B_R(x_0)}h(x)^{\lambda+\lambda_0/2}$.
		
		Suppose $\sigma(t,x,0)=0$ almost surely for all $(t,x)\in\R_+\times\X$, and let $L$ be a Lipschitz constant for $\sigma$ on $[0,1]$.
		Then, for all sufficiently large $p \ge2$, all $N \ge4 \Cr{c:semi_group_norm}$, and all sufficiently large $Q\ge 1$ (independent of $N$ and $\sigma$), there exists $\kappa := \kappa(p,\lambda) >0$ such that the following holds uniformly for $r\in[r_0,R-M/Q]$: if $u$ is the solution to \eqref{eq:spde_driftless} with initial condition $g \in \mathcal C^{\mathrm{tem}}_{\X,h}$ satisfying $\mathbf 1_{B_r(x_0)}(x) \le g(x) \le 2$ for all $x \in \X$, then
		\begin{align*}
			\mathbb{P}\left(u(s,x) \ge  \frac{1}{2N_0} \mathbf 1 _{B_{r+M/Q}(x_0)}(x),\,\, \forall(s,x) \in \left[\frac{t}{Q},\frac{2t}{Q}\right]\times \X\right) \ge 1- \varepsilon_{N,Q,p}
		\end{align*}
		where 
		\begin{align*}
			\varepsilon_{N,Q,p} = \Cr{c:one_step_pre}S_0 \pr{LNHN_0}^2 \pr{\frac{t}{Q}}^{1-\alpha_0} \exp \left(-\Cr{c:one_step_rate}\pr{LNHN_0}^{-2} \pr{\frac{Q}{t}}^{1-\alpha_0} \right) + S_0 \pr{\frac{\kappa}{N}}^p,
		\end{align*}
		for some $\C\label{c:one_step_pre},\,\C\label{c:one_step_rate}>0$, independent of $L,N,Q,N_0$, where $\alpha_0\in(0,1)$ is as in Lemma~\ref{lem:hk_abs}, \eqref{ineq:hk abs holder k weighted}.
	\end{prop}
	\begin{proof}
		For each $N >0$, let us define 
		\begin{align}
			\tau_N:=\inf\left\{ t>0: \sup_{x \in \X} \abs{u(t,x)}h(x)^{-\lambda} > N \right\},
		\end{align}
		which is a stopping time by the continuity of sample paths. Define for each $N >0$ that 
		\begin{align}
			Z_N(t,x):= \int_0^{t } \int_\X \mathbf 1_{\{s<\tau_N\}}G_{t-s}(y,x) \sigma(s,y,u(s,y)) W(\mathrm dy,\mathrm ds),\quad (t,x) \in \R_+ \times \X.
		\end{align}
		Note that for $s < \tau_N\wedge 1$, by \eqref{ineq:g_lip} and the fact that $\sigma(t,x,0) = 0$ almost surely for all $(t,x) \in \R_+ \times \X$, we see
		\begin{align}
			|\sigma(s,y,u(s,y))| \le L |u(s,y)| \le L N h(y)^{\lambda},\quad \text{ for all } (s,y) \in [0,\tau_N\wedge 1)\times \X.
		\end{align}
		Therefore, by Corollary \ref{lem:m_ldp_g} applied to $(LN)^{-1}Z_N$, with $\eta=h^\lambda$, $\lambda$ replaced by $\lambda+\lambda_0/2$, $\alpha=\alpha_0$, $T=1$, and $\xi_1=\xi_2$ equal to the H\"older exponent in Lemma \ref{lem:hk_abs} \eqref{ineq:hk abs holder k weighted} (which gives its hypothesis), with $\varepsilon= (2LNN_0H)^{-1}$ (we may take $L\ge1$, so that $\varepsilon<1$), there exist $\Cr{c:one_step_pre},\, \Cr{c:one_step_rate}>0$ so that for all $N \ge 1$ and sufficiently large $Q \ge 1$,
		\begin{align}\label{ineq:one_step_noise_LDP}
			&\bbP \left(\sup_{(s,x) \in [0,2t/Q]\times \X} |Z_N(s,x)|h(x)^{-\lambda- \lambda_0/2} > (2N_0H)^{-1} \right)\\
			 \le &\Cr{c:one_step_pre}S_0 \pr{LNHN_0}^2 \pr{\frac{t}{Q}}^{1-\alpha_0}\exp \left(-\Cr{c:one_step_rate}\pr{LNHN_0}^{-2} \pr{\frac{Q}{t}}^{1-\alpha_0} \right).
		\end{align}

		Since $g(x) \ge \mathbf{1}_{B_r(x_0)}(x)$ for all $x \in \X$, we see by \eqref{ineq:ip_explicite} that for all $s \in [t/Q,2t/Q]$ and $x \in B_{r+M/Q}(x_0)$,
		\begin{align}\label{ineq:in_proof_inst_prop}
			P_s^*g(x) \ge \int_{B_r(x_0)} G_s(y,x) m(\mathrm dy) \ge \frac{1}{N_0}.
		\end{align}
		By weak comparison with the zero solution (Theorem \ref{thm:comparison} \eqref{part:main2 weak comparison}), there is a full-measure event
		\begin{align*}
			\Omega_+:=\{u(s,x)\ge0\text{ for all }(s,x)\in\R_+\times\X\}.
		\end{align*}
		On the event $\{\tau_N>2t/Q\}$, one has $u(s,x) = P^*_sg(x) +Z_N(s,x)$ for all $(s,x) \in [0,2t/Q]\times \X$. Hence, on the event
		$$
		\mathcal A_Q:=\left\{\sup_{(s,x) \in [0,2t/Q]\times \X} |Z_N(s,x)|h(x)^{-\lambda- \lambda_0/2} \le  (2N_0H)^{-1}\right\}\bigcap \left\{\tau_N>\frac{2t}{Q}\right\}\bigcap\Omega_+,
		$$
		we see for all  $x \in B_{r+M/Q}(x_0)$, it holds for all $s \in [t/Q,2t/Q]$ that
		\begin{align}\label{ineq:one_step_inclusion_relation}
			 u(s,x) \overset{\eqref{ineq:in_proof_inst_prop}}\ge \frac{1}{N_0} - \abs{Z_N(s,x)} \ge \frac{1}{N_0} - \sup_{x \in B_{r+M/Q}(x_0)}\frac{h(x)^{\lambda+\lambda_0/2}}{2HN_0} \ge \frac{1}{2N_0}.
		\end{align}
		Since $u\ge0$ on $\Omega_+$ we see 
		\begin{align*}
			\mathcal A_Q \subset \left\{u(s,x) \ge  \frac{1}{2N_0} \mathbf 1 _{B_{r+M/Q}(x_0)}(x),\,\, \forall(s,x) \in \left[\frac{t}{Q},\frac{2t}{Q}\right]\times \X \right\}.
		\end{align*}		
		By the continuity of sample paths, we see 
		$$
    		\left\{\tau_N \le \frac{2t}{Q}\right\}\subset \left\{\sup_{(s,x) \in [0,2t/Q]\times \X} \abs{u(s,x)}h(x)^{-\lambda } \ge N\right\}.
		$$
		Since $g(x) \le 2$ for all $x\in \X$, we see $P^*_sg(x) \le 2\Cr{c:semi_group_norm}$ for all $s \in [0,1]$ and $x \in \X$. So if $N \ge 4\Cr{c:semi_group_norm}$, then for large $Q>0$ so that $2t/Q\le 1$, 
		$$
			\left\{\sup_{(s,x) \in[0,2t/Q]\times \X }\abs{u(s,x)}h(x)^{-\lambda } \ge N \right\} \subset \left\{\sup_{(s,x) \in [0,2t/Q]\times \X} \abs{Z(s,x)}h(x)^{-\lambda } \ge N/2 \right\}.
		$$
		We see by Lemma \ref{lem:unit_scale_stochastic_part_bound} that, for $p\ge p_0(\lambda)$, with $p_0(\lambda)$ as in Lemma~\ref{lem:unit_scale_stochastic_part_bound}, there exists $\kappa = \kappa(p,\lambda) >0$, enlarged by a factor of $2$ if necessary, so that  
		\begin{align*}
			\mathbb P\left(\tau_N \le\frac{2t}{Q} \right) \le \mathbb P \left(\sup_{(s,x) \in [0,2t/Q]\times \X} \abs{Z(s,x)}h(x)^{-\lambda} \ge \frac{N}{2}\right)\le S_0 \pr{\frac{\kappa}{N}}^p.
		\end{align*}
		Therefore, by \eqref{ineq:one_step_noise_LDP}, the preceding bound, and the fact that $\mathbb P(\Omega_+)=1$, we see 
		\begin{align*}
			\mathbb{P}\left(u(s,x) \ge  \frac{1}{2N_0} \mathbf 1 _{B_{r+M/Q}(x_0)}(x)\, \forall(s,x) \in \left[\frac{t}{Q},\frac{2t}{Q}\right]\times \X\right)
			&\overset{\eqref{ineq:one_step_inclusion_relation}}{\ge}\mathbb P \left( \mathcal A_Q\right)\\
			&=1- \mathbb P \left(\mathcal A_Q ^c\right)\ge 1 - \varepsilon_{N,Q,p}
		\end{align*}
		which is the desired inequality.
	\end{proof}

	\begin{proof}[Proof of Theorem \ref{thm:comparison} \eqref{part:main2 strong comparison} and \eqref{part:main2 strictly positivity}]
		To prove Theorem \ref{thm:comparison} \eqref{part:main2 strong comparison}, let $u$ and $u'$ be as in Theorem~\ref{thm:comparison}~\eqref{part:main2 strong comparison}, and denote $v_t(x) : =u_t(x) - u'_t(x)$ and $f_0(x):= f(x) - f'(x)$ for $t > 0$ and $x \in \X$. Then $v$ solves the integral equation
		\begin{align}
			v_t(x) = P_t^*f_0(x) + &\int_0^t \int_\X G_{t-s}(y,x)\widetilde{b}(s,y,v_s(y)) m(\mathrm{d}y)\mathrm{d}s\\
			&+\int_0^t \int_\X G_{t-s}(y,x)\widetilde{\sigma}(s,y,v_s(y))W(\mathrm{d}y,\mathrm{d}s),
		\end{align}
		for $t>0$ and $x \in \X$, where for $r\in \R$,
		\begin{align}
			\widetilde{b}(t,x,r) &:= b(t,x,r +u'_t(x)) - b(t,x,u'_t(x))\\ \widetilde{\sigma}(t,x,r) &:= \sigma(t,x,r +u'_t(x)) - \sigma(t,x,u'_t(x)). \label{eq:tilde_coef_diff}
		\end{align}
		By \eqref{ineq:g_lip}, $\widetilde b$ and $\widetilde\sigma$ satisfy the same local Lipschitz condition, and $\widetilde b(t,x,0)=\widetilde\sigma(t,x,0)=0$ almost surely for all $(t,x)\in\R_+\times\X$.

		By Theorem~\ref{thm:comparison}~\eqref{part:main2 weak comparison} and the argument of \cite[pp.~427--428, equations~(4.17)--(4.22)]{MR1271224}, it is enough to consider \eqref{eq:spde_driftless} and assume that, for some $x_0\in\X$ and $\beta,a>0$, the initial value $f$ satisfies
		\begin{align}
			2\beta\ge f(x)\ge\beta\mathbf1_{B_a(x_0)}(x),\qquad x\in\X.
		\end{align}
		
		Let $t,M>0$ be arbitrary but fixed, and let $L>0$ be a Lipschitz constant for $\sigma$ in \eqref{ineq:g_lip} on $[0,t+1]$. Fix $\lambda\in(0,1/2]$ and any $\gamma\in(0,(1-\alpha_0)/2)$, where $\alpha_0\in(0,1)$ is as in Lemma~\ref{lem:hk_abs}, \eqref{ineq:hk abs holder k weighted}.
		By \eqref{ineq:instant_propagation}, we see there exists $N_0 = N_0(x_0,t,M,a/2,a+M) \ge 1$ so that for all sufficiently large $Q\ge 1$, 
		\begin{align*}
			\inf_{s \in [t/Q,2t/Q]} \inf_{r \in [a/2,a+M]} \inf_{x \in B_{r+M/Q}(x_0)}\int_{B_r(x_0)} G_s(y,x)m(\mathrm dy) \ge \frac{1}{N_0}.
		\end{align*}
		Let $H:=\sup_{B_{a+M}(x_0)}h(x)^{\lambda+\lambda_0/2}$ and let $p \ge 2$ be sufficiently large as in Proposition \ref{prop:one_step_estimates} and  assume in addition that $p > 1/\gamma$. Define
		\begin{align*}
			\varepsilon_{Q}: =\Cr{c:one_step_pre}S_0 \pr{LQ^\gamma HN_0}^2 \pr{\frac{t}{Q}}^{1-\alpha_0}\exp \left(-\Cr{c:one_step_rate}\pr{LQ^\gamma HN_0}^{-2} \pr{\frac{Q}{t}}^{1-\alpha_0} \right) + S_0 \pr{\frac{\kappa}{Q^\gamma}}^p.
		\end{align*}
		We note that 
		\begin{align}\label{ineq:strict_pos_Qepsilon_to_0}
			\lim_{Q \to \infty}Q \varepsilon_Q = 0. 
		\end{align}
		Indeed, since $\gamma < \frac{1-\alpha_0}{2}$ and $p > 1/\gamma$, we see for some $C>0$, it holds uniformly in $Q\ge 1$ that 
		\begin{align*}
			Q \varepsilon_Q \lesssim Q^{\alpha_0+2\gamma} \exp \left(-C Q^{1-\alpha_0 - 2\gamma}\right) +Q^{-p\gamma +1}
		\end{align*}
		which converges to $0$ as $Q \to \infty$, which verifies \eqref{ineq:strict_pos_Qepsilon_to_0}.

		For $Q\in \N$ sufficiently large, set $s_j=jt/Q$ for $j=0,\dots,Q+1$, and, for $j=0,\dots,Q$, set
		\begin{align*}
			\zeta_j:=\frac{jM}{2Q^2},\qquad B_j:=B_{a+\frac{jM}{Q}-\zeta_j}(x_0),\qquad \theta_j:=\beta(2N_0)^{-j}.
		\end{align*}
		For $k=1,\dots,Q$, set $B_k':=B_{a+\frac{(k-1)M}{Q}-\zeta_k}(x_0)$ and define
		\begin{gather}\label{def:strict_pos_set_A_J_def}
			\mathcal A_k:=\brc{u(s,x)\ge\theta_k\mathbf 1_{B_k}(x)\ \forall(s,x)\in\brk{s_k,s_{k+1}}\times\X},\\
			J_j:=\brc{u(s_j,\cdot)\ge\theta_j\mathbf 1_{B_j}(\cdot)},\qquad j=0,\dots,Q.
		\end{gather}
		Since $\theta_0=\beta$, $B_0=B_a(x_0)$, and $f\ge\beta\mathbf 1_{B_a(x_0)}$, we have $J_0=\Omega$.

		For each sufficiently large $Q\in \N$ and $k = 1,\dots, Q$, let $\varphi_{k-1} :\X \to [0,1]$ be a continuous function so that $\varphi_{k-1}\equiv 1$ on the closure of $B_k'$ and $\varphi_{k-1}\equiv 0$ outside of $B_{k-1}$. 

		Define for $k = 1,\dots, Q$ that $\sigma_{k}(s,x,r):=\theta_{k-1}^{-1}\sigma(s_{k-1}+s,x,\theta_{k-1}r)$ for $r \in \R$, which is predictable with respect to the filtration $(\mathcal F_{s_{k-1}+s})_{s \ge 0}$ and satisfies \eqref{ineq:g_lip} on $[0,1]$ with the same Lipschitz constant $L>0$ (since $s_{k-1}+1\le t+1$) and $\sigma_k(t,x,0) = 0$ almost surely for all $(t,x) \in \R_+\times \X$.
		
		Therefore, by Theorem \ref{thm:main2}, there exists a unique continuous $\mathcal C^{\mathrm{tem}}_{\X,h}$-valued solution to the integral equation with the shifted noise,
		\begin{align*}
			w_k(t,x) = P^*_t\varphi_{k-1}(x)+ \int_0^t \int_\X G_{t-s}(y,x)\sigma_{k}(s,y,w_k(s,y))W\pr{s_{k-1}+\mathrm ds, \mathrm dy}.
		\end{align*}
		Denote 
		\begin{align*}
			\mathcal B_k:=\left\{w_k(s,x) \ge \frac{1}{2N_0}\mathbf 1_{B_k}(x),\, \forall (s,x) \in \left[\frac{t}{Q},\frac{2t}{Q} \right]\times \X\right\}.
		\end{align*}
		For all large $Q$, the radius of $B_k'$ lies in $[a/2,a+M-M/Q]$ and the radius of $B_k'$ plus $M/Q$ is exactly the radius of $B_k$. Hence Proposition \ref{prop:one_step_estimates}, applied with $r$ equal to the radius of $B_k'$ and $N=Q^\gamma$, gives
		\begin{align}\label{ineq:strict_pos_one_step_estimates}
			\mathbb P \left(\mathcal B_k\right) \ge 1 - \varepsilon_{Q}.
		\end{align}
		We claim that 
		\begin{align}\label{ineq:strict_pos_main_step_claim}
			\mathcal A_{k}^c \cap J_{k-1} \subset \mathcal B_k^c\quad \text{and hence }\mathbb{P}\left( \mathcal A_{k}^c \cap J_{k-1} \right) \le \varepsilon_Q.
		\end{align}
		Now define $v_k(t,x):= \theta_{k-1}w_k(t,x)$ for $(t,x) \in \mathbb R_+\times \X$. Then $v_k$ is the solution to the following shifted integral equation:
		\begin{align*}
			v_k(t,x)= P^*_t\pr{\theta_{k-1}\varphi_{k-1}}(x) + \int_0^t \int_\X G_{t-s}(y,x)\sigma(s_{k-1}+s,y,v_k(s,y))W\pr{s_{k-1}+\mathrm ds, \mathrm dy}.
		\end{align*}
		Therefore, on the event $\mathcal B_k$, we have 
		\begin{align*}
			v_k(s,x) \ge \theta_{k-1} (2N_0)^{-1}\mathbf{1}_{B_k}(x) = \theta_{k}\mathbf{1}_{B_k}(x),\qquad \text{ for all }(s,x)\in \left[\frac{t}{Q},\frac{2t}{Q} \right]\times \X.
		\end{align*}

		Now we set $\widetilde{\varphi}_{k}:= \mathbf{1}_{J_{k-1}} \theta_{k-1}\varphi_{k-1}$, then $\widetilde{\varphi}_{k}$ is $\mathcal F_{s_{k-1}}$ measurable since $J_{k-1}\in \mathcal F_{s_{k-1}}$ and it belongs to $\mathcal C^{\mathrm{tem}}_{\X,h}$. We see 
		\begin{align*}
			\mathbf{1}_{J_{k-1}}v_k(t,x) &= P_t^*\widetilde{\varphi}_k(x)+ \int_0^t \int_\X G_{t-s}(y,x)\mathbf{1}_{J_{k -1}}\sigma(s_{k-1}+s, y, v_k(s,y))W\pr{s_{k-1}+\mathrm ds, \mathrm dy}\\
			&=P_t^*\widetilde{\varphi}_{k}(x) + \int_0^t\int_\X G_{t-s}(y,x)\sigma(s_{k-1}+s,y,\mathbf{1}_{J_{k-1}}v_k(s,y))W\pr{s_{k-1}+\mathrm ds, \mathrm dy},
		\end{align*}
		where we used the fact that $\sigma(t,x,0) = 0$ in the last line. In other words, $\mathbf{1}_{J_{k-1}}v_k$ is a solution to the shifted equation with initial condition $\widetilde{\varphi}_k$.  

		By Corollary \ref{cor:MfT}, \eqref{ineq:bcc D pair explicit}, and \eqref{ineq:hk_m long time}, Lemma \ref{thm:stochastic_fubini} applies. Thus, splitting the mild equation at $s_{k-1}$ and using the semigroup property, we obtain, for $s\ge0$,
		\begin{align*}
			u(s_{k-1}+s,x)
			&=P_{s_{k-1}+s}^*f(x)+\int_0^{s_{k-1}}\int_\X G_{s_{k-1}+s-r}(y,x)\sigma(r,y,u_r(y))W(\mathrm dr,\mathrm dy)\\
			&\quad+\int_{s_{k-1}}^{s_{k-1}+s}\int_\X G_{s_{k-1}+s-r}(y,x)\sigma(r,y,u_r(y))W(\mathrm dr,\mathrm dy)\\
			&=P_s^*\big(u(s_{k-1},\cdot)\big)(x)+\int_0^s\int_\X G_{s-r}(y,x)\sigma(s_{k-1}+r,y,u_{s_{k-1}+r}(y))W\pr{s_{k-1}+\mathrm dr,\mathrm dy}.
		\end{align*}
		Thus $(u(s_{k-1}+s,\cdot))_{s\ge0}$ solves the same shifted equation with initial condition $u(s_{k-1},\cdot)\in\mathcal C^{\mathrm{tem},+}_{\X,h}$ almost surely. It holds for all $x \in \X$ that 
		\begin{align*}
			u(s_{k-1},x) \ge \mathbf{1}_{J_{k-1}} u(s_{k-1},x) \ge  \mathbf{1}_{J_{k-1}}\theta_{k-1} \mathbf{1}_{B_{k-1}}(x) \ge \widetilde{\varphi}_k(x),
		\end{align*}
		where we used the fact that $\widetilde \varphi_{k}$ is supported on $B_{k-1}$. 
		        Set $\mathcal G_s:=\mathcal F_{s_{k-1}+s}$ for $s\ge0$. The shifted noise is a (generalized) Wiener process relative to $(\mathcal G_s)_{s\ge0}$, and the shifted coefficient satisfies the same assumptions. The same contraction, nonnegativity, and truncation arguments extend Theorem~\ref{thm:main2} and Theorem~\ref{thm:comparison}~\eqref{part:main2 weak comparison} to $\mathcal G_0$-measurable $\mathcal C^{\mathrm{tem}}_{\X,h}$-valued initial conditions $g$ satisfying
		\begin{equation}\label{ineq:random_initial_moment}
			\sup_{x\in\X}h(x)^{-\lambda}\norm{g(x)}_{L^p_\Omega}<\infty,
			\qquad \lambda\in(0,1],\quad p\ge2,
		\end{equation}
		with the initial ordering understood in the almost sure sense. The bound \eqref{ineq:random_initial_moment} holds for $u(s_{k-1},\cdot)$ by Corollary~\ref{cor:moment}, and for $\widetilde\varphi_k$ since $0\le\widetilde\varphi_k\le\theta_{k-1}$.
                
        		Therefore, by the comparison principle (Theorem \ref{thm:comparison} \eqref{part:main2 weak comparison}), we see
		\begin{align*}
			u(s_{k-1}+s,x) \ge \mathbf{1}_{J_{k-1}}v_k(s,x),\qquad \forall (s,x) \in\R_+\times \X. 
		\end{align*}
		This implies, on $J_{k-1}$, we have $\mathcal A^c_k \subset \mathcal B_k^c$, combined with \eqref{ineq:strict_pos_one_step_estimates}, we see \eqref{ineq:strict_pos_main_step_claim} holds. 

		Observe that for sufficiently large $Q\in \N$, 
		\begin{align}\label{ineq:strict_pos_last_inclusion}
			\cap_{k = 1}^Q\mathcal A_k \subset \left\{ u_s(x)>0\ \text{on }\brk{\tfrac t2,t}\times B_{a+\tfrac M4}(x_0)\right\}.
		\end{align}
		Indeed, on $\cap_{k = 1}^Q\mathcal A_k$, for each $s \in [t/2,t]$, there is some $Q/2 -1\le k < Q$ so that $s \in [s_{k},s_{k+1}]$, and hence for sufficiently large $Q \in \N$,
		\begin{align*}
			B_{a+ \frac{M}{4}}(x_0) \subset B_{a+ \frac{kM}{Q} - \frac{kM}{2Q^2}}(x_0) \subset B_k,\qquad \frac{Q}{2}-1 \le k <Q. 
		\end{align*}
		Then $\mathcal A_k$ gives $u(s,x) \ge\theta_k>0$ on $B_k$, which implies \eqref{ineq:strict_pos_last_inclusion}.

		Since $J_0=\Omega$ and, for $k\ge2$, $\mathcal A_{k-1} \subset J_{k-1}$ and $\bigcap_{j<k}\mathcal A_j \subset \mathcal A_{k-1}$, we see for all sufficiently large $Q \in \N$,
		\begin{align*}
			1-\bbP\pr{\bigcap_{k=1}^Q\mathcal A_k}
			\ =\ \bbP\pr{\bigcup_{k=1}^{Q}\mathcal A_k^c}
			\ &=\ \sum_{k=1}^{Q}\bbP\pr{\mathcal A_k^c\cap\bigcap_{j<k}\mathcal A_j}\\
			\ &\le\ \sum_{k=1}^{Q}\bbP\pr{\mathcal A_k^c\cap J_{k-1}}
			\ \overset{\eqref{ineq:strict_pos_main_step_claim}}{\le }\ \sum_{k=1}^{Q}\bbP\pr{\mathcal B_k^c}
			\ \le\ Q\,\varepsilon_Q.
		\end{align*}
		whence
		\[
			\bbP\pr{u_s(x)>0\ \text{on }\brk{\tfrac t2,t}\times B_{a+\tfrac M4}(x_0)}\ \overset{\eqref{ineq:strict_pos_last_inclusion}}{\ge}\ \limsup_{Q \uparrow \infty} \bbP\pr{\bigcap_{k=1}^Q\mathcal A_k}\ \ge\ 1-\lim_{Q \uparrow \infty}Q\,\varepsilon_Q \overset{\eqref{ineq:strict_pos_Qepsilon_to_0}}{=} 1.
		\]
		Apply this with $t_n=2^n$, $n\in\mathbb Z$, and $M_m=m$, $m\in\mathbb N$, and intersect the resulting countable family of probability-one events. Since the union of the intervals $[2^{n-1},2^n]$ covers $(0,\infty)$ for $n \in \Z$ and the balls $B_{a+m/4}(x_0)$ exhaust $\X$ for $m \in \N$, we obtain
		\begin{align*}
			\bbP\pr{u(t,x)>0\ \text{for all }(t,x) \in (0,\infty) \times \X } = 1.
		\end{align*}
		This proves Theorem \ref{thm:comparison} \eqref{part:main2 strong comparison}.

		Finally, Theorem \ref{thm:comparison} \eqref{part:main2 strictly positivity} follows by applying \eqref{part:main2 strong comparison} to the solutions with initial conditions $f$ and $0$, since the latter is non-negative by Proposition \ref{prop:non-negative2} under \eqref{cond:pos_zero}.
	\end{proof}

\section*{Acknowledgements}
Much of this work was carried out as part of the Ph.D. thesis of Yifan (Johnny) Yang, submitted to Indiana University in 2025 and supported by the National Science Foundation (DMS-2152103). 
Zhenyao Sun is partially supported by the National Key R\&D Program of China (Grant No. 2023YFA1010102) and the National Natural Science Foundation of China (Grant No. 12301173).
Louis W.-T. Fan is partially supported by National Science Foundation  (DMS-2152103, DMS-2532574 and DMS-2534011).

\section*{AI use disclosure}

The authors used AI tools during the revision of this paper, after the main results and their proofs had been completed. 
Namely,  Claude-5 and Claude-5.5 family were used to
write simulation codes for Figure \ref{fig:pam-stable}. 
GPT-5.6, GPT-6 and DeepSeek V4 were used to
improve English grammar and sentence flow and to help identify typographical errors, inconsistencies in notation, and potential errors in the manuscript. 
The authors reviewed all suggestions and revised the manuscript as appropriate. They take full responsibility for the final content, including the correctness of the mathematical arguments and results and the completeness and accuracy of references to prior work.
	
\newpage
	
	\appendix

	\section{Appendix}\label{sec:appendix}

	\subsection{Proof of Proposition \ref{prop:sub_g_mmd_framework}}
    In this section, we verify the sub-Gaussian heat kernel estimates in Proposition \ref{prop:sub_g_mmd_framework}. Due to the lengthiness of the proof, we break it down into several propositions (Propositions \ref{prop:sub_g_cond_no_noise}, \ref{prop:sub_g_condition_noise} and \ref{prop:sub_g_strict_pos_condition} below), whose proofs will be provided at the end of the section.

	\begin{prop}\label{prop:sub_g_cond_no_noise}
		Let $V,\Psi : [0,\infty) \to [0,\infty)$ be doubling functions satisfying 
		\[
		\frac{R^2}{r^2} \lesssim \frac{\Psi(R)}{\Psi(r)} \lesssim \frac{R\,V(R)}{r\,V(r)},\qquad\text{uniformly in } 0<r\le R.
		\]
		Then $\Psi$ is a scale function.
		Let $(\X,d,m,\mathcal E,\mathcal F)$ be an MMD with heat kernel satisfying $\mathrm{HKE}_{\rm full}(\Psi)$ and having volume growth profile $V$. Fix an arbitrary $x_0 \in \X$ and define weight function
		\[
		h(x) := \exp \pr{d(x_0,x)},\qquad x \in \X.
		\]
		Then the following hold:
		\begin{enumerate}
			\item \label{it_prop_subg_no_noise:1} The MMS $(\X,d,m)$ satisfies \eqref{ineq:entropy bound}, and the weight function $h$ satisfies \eqref{ineq:faster_than_ms} and \eqref{ineq:ref_regular}. 
			\item \label{it_prop_subg_no_noise:2} Condition \eqref{cond:ctem continuous} and the first inequality in condition \eqref{cond:exp_hk_lamda} hold. 
			\item \label{it_prop_subg_no_noise:3} Condition \eqref{cond:time_holder} holds. 
		\end{enumerate}
	\end{prop}
	
	\begin{prop}\label{prop:sub_g_condition_noise}
		Let $V,\Psi$, and $(\X,d,m, \mathcal E, \mathcal F)$ and the function $h$ on $\X$ be as in Proposition \ref{prop:sub_g_cond_no_noise}.  Suppose $\alpha_0 \in (0,1)$. 
		Then the following holds:  
		\begin{enumerate}
			\item \label{it_prop_subg_noise:1} If $\pr{V\pr{\Psi^{-1}(t)}}^{-1} \lesssim t^{-\alpha_0}$ uniformly in $t \in (0,1]$, then the second inequality in \eqref{cond:exp_hk_lamda} and \eqref{cond:hk_uni_bound} hold with $\alpha$ being replaced by $\alpha_0$ and the underlying measure $K$ in the definition of $\mathcal H_K$ being replaced by $K_0$.
			\item \label{it_prop_subg_noise:2} If $\gamma\neq 0$ and it holds uniformly in $t \in (0,1]$ that 
			\begin{align*}
				\int_0^\infty \frac{s^\gamma}{V\pr{\Psi^{-1}(t+s)}} \frac{\mathrm ds}{s} \lesssim t^{-\alpha_0},
			\end{align*}
			then the second inequality in \eqref{cond:exp_hk_lamda} and \eqref{cond:hk_uni_bound} hold with $\alpha$ being replaced by $\alpha_0$, and the underlying measure $K$ in the definition of $\mathcal H_K$ being replaced by $K_\gamma$. 
		\end{enumerate}
	\end{prop}

	\begin{prop}\label{prop:sub_g_strict_pos_condition}
		Let $V,\Psi$ and $(\X,d,m,\mathcal E,\mathcal F)$ be as in Proposition \ref{prop:sub_g_cond_no_noise}. There exists a metric $\rho$ that is bi-Lipschitz equivalent to $d$ on $\X$ so that condition \eqref{ineq:instant_propagation} holds under $\rho$.
	\end{prop}
	\begin{remark}\label{rmk:bi_lip_metric_remark}
		Since $\rho$ and $d$ are bi-Lipschitz equivalent, the heat kernel has similar $\mathrm{HKE}_{\rm full}(\Psi)$ under the metric $\rho$ and the conclusions of Propositions \ref{prop:sub_g_cond_no_noise} and \ref{prop:sub_g_condition_noise} hold under the metric $\rho$ as well. 
	\end{remark}
	\begin{proof}[Proof of Proposition \ref{prop:sub_g_mmd_framework}]
		From Propositions \ref{prop:sub_g_cond_no_noise}--\ref{prop:sub_g_condition_noise} and the assumptions, we see all conditions of Theorem \ref{thm:main2} hold. Hence the well-posedness part of the result holds. In addition, by Proposition \ref{prop:sub_g_strict_pos_condition} and Remark \ref{rmk:bi_lip_metric_remark}, we see all conditions of Theorem \ref{thm:comparison} hold, which implies the strict positivity property. 
	\end{proof}
    
	We shall prove Propositions \ref{prop:sub_g_cond_no_noise}--\ref{prop:sub_g_strict_pos_condition} in the rest of this subsection. 
	We will need the following elementary results.
	\begin{lemma}\label{lem_CM:doubling_function_polynomial_bound}
		Suppose $V:\R_+\to \R_+$ is a doubling function as in \eqref{ineq:doubling_function}. Then there exists a constant $D>1$ such that
		\begin{align*}
			\frac{V(R)}{V(r)} \le D \pr{\frac{R}{r}}^{\log_2(D)},\qquad \text{ for }0<r\le R.
		\end{align*}
			\end{lemma}
	\begin{proof}
		Let $0<r\le R$, choose $n \in \N\cup\{0\}$ so that $2^{n-1} < \frac{R}{r}\le 2^n$, then by \eqref{ineq:doubling_function}, there exists $D>1$ so that 
		\begin{align*}
			V(R) \le V(2^n r) \le D^n V(r).
		\end{align*}
		Since $D^n = 2^{n \log_2(D)}$, we get for $0<r\le R$, 
		\begin{align*}
			\frac{V(R)}{V(r)} \le D^n \le D\pr{\frac{R}{r}}^{\log_2(D)}.
		\end{align*}
	\end{proof}
	
	\begin{lemma}\label{lem_CM:escape_polynomial_bound}
		Let $V,\Psi$ be as in Remark \ref{rem:sub-g-general}. Then there exists $\beta\ge 1$ so that 
		\begin{align*}
			\frac{V(R)}{V(r)} \lesssim \pr{\frac{R}{r}}^\beta,\qquad \text{ uniformly in }0<r\le R.
		\end{align*}
		In addition, for any $T>0$, it holds uniformly in $t \in (0,T]$ that 
		\begin{align*}
			\sqrt{t} \lesssim \Psi^{-1}(t) \lesssim t^{\frac{1}{\beta+1}}.
		\end{align*}
	\end{lemma}
	\begin{proof}
		By \eqref{ineq:sub_g_escape_vol_constrain} and Lemma \ref{lem_CM:doubling_function_polynomial_bound}, we see there exist $C,C'>1$ and $\beta>0$ so that 
		\begin{align*}
			C^{-1} \pr{\frac{R}{r}}^2 \le \frac{\Psi(R)}{\Psi(r)} \le C'\frac{RV(R)}{rV(r)} \le C\pr{\frac{R}{r}}^{\beta+1},\qquad \text{for }0<r\le R. 
		\end{align*}
		This implies $\beta \ge 1$. It also implies that for $0<t\le T$, 
		\begin{align*}
			C^{-1} \pr{\frac{T}{t}}^{\frac{1}{\beta+1}} \le \frac{\Psi^{-1}(T)}{\Psi^{-1}(t)} \le C \pr{\frac{T}{t}}^{1/2}.
		\end{align*}
		We may now fix any $T>0$ and multiply this inequality by $\Psi^{-1}(t)$ to see the desired result.
	\end{proof}

	Note that, in both cases $\text{diam}(\X)<\infty$ and $\text{diam}(\X)=\infty$, the volume growth profile together with $\mathrm{HKE}_{\rm full}(\Psi)$ implies
	\begin{align}\label{ineq_CM:volume_all_scales}
		m\pr{B_r(x)} \asymp V\pr{r\wedge \text{diam}(\X)},\qquad \text{uniformly in }x \in \X,\, r>0.
	\end{align}

	\begin{proof}[Proof of Proposition \ref{prop:sub_g_cond_no_noise}]
		For \eqref{it_prop_subg_no_noise:1}, it is known that when a metric measure space $(\X,d,m)$ has volume growth profile $V$ for a doubling function $V$, then by \eqref{ineq_CM:volume_all_scales} a stronger version of \eqref{ineq:entropy bound} holds (cf. \cite[pp.~167--168]{kaenmaki2017measures}): there exists $b >0$ so that
		\begin{align}\label{ineq_CM:entropy_grwoth_sub_g_case}
			D_r\pr{B_R(x)} \lesssim \pr{\frac{r}{R}}^{-b},\qquad \text{uniformly in }x \in \X, \, 0<r<R.
		\end{align}
		It can be easily checked that condition \eqref{ineq:ref_regular} holds for exponential functions.
		For \eqref{ineq:faster_than_ms}, we see from \eqref{ineq_CM:entropy_grwoth_sub_g_case} by setting $r = 1$ that for each $N \in \N$, the ball $B_N(x_0)$, hence $A_N:= B_{N}(x_0) \backslash B_{N-1}(x_0)$, can be covered by $\lfloor CN^b\rfloor$ unit balls, where $C\ge1$ does not depend on $N$, denote their centers as $\{x^N_i\}_{i = 1}^{\lfloor CN^b\rfloor}$.
		Now define $I_N:=\{x_i^N: A_N \cap B_1(x_i^N) \neq \emptyset\}$.
		Clearly,
		\begin{align*}
			A_N \subset\bigcup_{x_i \in I_N}B_1(x_i),\qquad \abs{I_N} \lesssim N^b \text{ uniformly in }N.
		\end{align*}
		Let $I = \cup_{N =1}^\infty I_N$, then 
		\begin{align*}
			\X = \bigcup_{N \in \N}A_N \subset \bigcup_{N \in \N}\bigcup_{x_i \in I_N}B_1(x_i) = \bigcup_{x_i \in I} B_1(x_i).
		\end{align*}
		Now consider 
		\begin{align*}
			\sum_{x_i \in I} \exp \pr{-d(x_0,x_i)} = \sum_{N = 1}^\infty \sum_{x_i \in I_N} \exp \pr{-d(x_0,x_i)}\lesssim\sum_{N = 1}^\infty e^{-(N-2)}N^b < \infty,
		\end{align*}
		which gives \eqref{ineq:faster_than_ms}. 
		
		For \eqref{it_prop_subg_no_noise:2}, we see from \cite[Lemma 4.8 \& Theorem 3.1]{barlow2012equivalence}, there exist $c_1,c_2>0$ so that for each $k \in \N$, 
		\begin{align}\label{ineq_CM:sub_g_exponential_moment_bound}
			\int_\X \exp \left(c_1 \frac{d(x,y)}{\Psi^{-1}(t/k)} \right)G_t(x,y) m(\mathrm dy) \le c_2^k,\qquad\text{ for all }x \in \X,\, t>0. 
		\end{align}
		Note that for each $\lambda >0$, we may select a sufficiently large $k_\lambda  \ge 1$ so that $\frac{c_1}{\Psi^{-1}(t/k_\lambda)} \ge \lambda$ for all $t \in (0,1]$. Hence, consider for $t \in (0,1]$ that 
		\begin{align*}
			e^{-\lambda d(x_0,x)}\int_\X e^{\lambda d(x_0,y)}G_t(x,y) m(\mathrm dy)&\le \int_\X e^{\lambda d(x,y)}G_t(x,y)m(\mathrm dy)\\
			&\le \int_\X \exp \left(c_1 \frac{d(x,y)}{\Psi^{-1}(t/k_\lambda)} \right) G_t(x,y) m(\mathrm dy)\\
			&\le c_2^{k_\lambda},
		\end{align*}
		which implies the first inequality of \eqref{cond:exp_hk_lamda}. 
		
		It remains to show \eqref{cond:ctem continuous}, that is, for any $f \in \mathcal C^{\mathrm{tem}}_{\X,h}$, 
		\begin{align}\label{ineq_CM:ctemp_continuous_subg}
			\lim_{t \downarrow 0}\norm{P_tf - f}_{(\lambda)}	 = 0,\qquad \text{ for all } \lambda > 0. 
		\end{align}
		Let us denote $r(x):=d(x_0,x)$ for $x \in \X$, then $e^{-\lambda r(x)} f(x) \in \mathcal C^0_\X$ for any $\lambda >0$, hence uniformly continuous on $\X$. Next, we note that by \eqref{ineq_CM:volume_all_scales} and \cite[Theorem 3.2]{lierl2015scale}, we have $\int_\X G_t(x,y)m(\mathrm dy) = 1$ for all $x \in \X $ and $t>0$.
		
		Let $\lambda >0$ be arbitrary but fixed,  denote $g(x):= e^{-\lambda r(x)/2}f(x)$ and consider
		\begin{align*}
			P_tf(x) - f(x)&=\int_\X G_t(x,y) \pr{g(y)e^{\lambda r(y)/2} - g(x)e^{\lambda r(x)/2}}m(\mathrm dy)\\
			&=\int_\X G_t(x,y) \pr{\bigl(g(y)-g(x)\bigr)e^{\lambda r(y)/2}
				+
				g(x)\pr{e^{\lambda r(y)/2}-e^{\lambda r(x)/2}}}m(\mathrm dy).
		\end{align*}
		Hence 
		\begin{align*}
			e^{-\lambda r(x)}\abs{P_tf(x) - f(x)} \le A_t(x)+B_t(x)
		\end{align*}
		where 
		\begin{align*}
			A_t(x) &:= \int_\X G_t(x,y)e^{-\lambda r(x)}\abs{g(y) - g(x)}e^{\lambda r(y)/2}m(\mathrm dy)\\
			B_t(x) &:= \int_\X G_t(x,y) e^{-\lambda r(x)}\abs{g(x)}\, \abs{e^{\lambda r(y) / 2} - e^{\lambda r(x) /2}} m(\mathrm dy).
		\end{align*}
		We will consider $B_t(x)$ first by noting 
		\begin{align*}
			e^{-\lambda r(x)}\abs{e^{\lambda r(y)/2} - e^{\lambda r(x)/2}} = e^{-\lambda r(x)/2}\abs{e^{\lambda (r(y) - r(x))/2} -1} \le e^{-\lambda r(x)/2}\abs{e^{\lambda d(x,y)/2}-1}
		\end{align*}
		Therefore, 
		\begin{align*}
			B_t(x) &\le \abs{g(x)}e^{-\lambda r(x)/2}\int_\X G_t(x,y) \pr{e^{\lambda d(x,y)/2}-1}m(\mathrm dy)\\
			&\mathrel{\le} \norm{g}_{L^\infty_\X}\int_\X G_t(x,y) \pr{e^{\lambda d(x,y)/2}-1}m(\mathrm dy)
		\end{align*}
		Let $c_1 > 0$ be as in \eqref{ineq_CM:sub_g_exponential_moment_bound} and denote $\eta_t:= c_1/ \Psi^{-1}(t)$.
		Let $\varepsilon >0$ be arbitrary but fixed, then choose $\delta >0$ so that $e^{\lambda \delta/2} - 1 < \varepsilon $.
		Consider
		\begin{align*}
			\int_\X G_t(x,y) \pr{e^{\lambda d(x,y) /2} -1}m(\mathrm dy)  &= \pr{\int_{d(x,y) \mathrel{<} \delta} + \int_{d(x,y) \ge \delta}} G_t(x,y) \pr{e^{\lambda d(x,y) /2} -1}m(\mathrm dy)\\
			&:=I_1(t,x) +I_2(t,x). 
		\end{align*}
		Note that $I_1(t,x) < \varepsilon $ for all $t>0$ and $x \in \X$.
		Now, let $t_0>0$ be sufficiently small so that $\eta_t > \lambda$ for $0< t< t_0$.
		Hence for $d(x,y) \mathrel{\ge} \delta$ and $0<t<t_0$, we have
		\begin{align*}
			e^{\lambda d(x,y)/2} = e^{-(\eta_t -\lambda/2)\delta}e^{(\eta_t - \lambda/2) \delta}e^{\lambda d(x,y)/2} \mathrel{\le} e^{-(\eta_t - \lambda/2)\delta}e^{\eta_t d(x,y)}.
		\end{align*}
		Therefore, for all $0<t<t_0$ 
		\begin{align*}
			\sup_{x \in \X}I_2(t,x) &\le e^{-(\eta_t - \lambda/2)\delta} \sup_{x\in \X} \int_\X G_t(x,y) e^{\eta_t d(x,y)} m(\mathrm dy)\\
			&\overset{\eqref{ineq_CM:sub_g_exponential_moment_bound}}{\le} e^{-(\eta_t-\lambda/2)\delta}c_2.
		\end{align*}
		Hence 
		\begin{align*}
			\mathop{\limsup}\limits_{t \downarrow 0} \sup_{x \in \X}B_t(x) \le \norm{g}_{L^\infty_\X}\mathop{\limsup}\limits_{t \downarrow 0} \sup_{x\in \X} \pr{I_1(t,x) +I_2(t,x)} \le \norm{g}_{L^\infty_\X}\varepsilon,
		\end{align*}
		since $\varepsilon >0$ can be arbitrarily small, we see the limit is zero.
		
		For ${A}_t(x)$, for any $\varepsilon > 0$,  we can choose sufficiently small $\delta \mathrel{\in} (0,1]$ so that $\abs{g(x) - g(y)} < \varepsilon$ if $d(x,y) < \delta$. Hence, 
		\begin{align*}
			A_t(x) &\le \pr{\int_{d(x,y) < \delta} +\int_{d(x,y) \ge \delta}}G_t(x,y) \abs{g(y)-g(x)}e^{-\lambda r(x) +\lambda r(y)}{m(\mathrm dy)}\\
			&\le \pr{\int_{d(x,y) < \delta} +\int_{d(x,y) \ge \delta}}G_t(x,y) \abs{g(y)-g(x)}e^{\lambda d(x,y)}m(\mathrm dy)\\
			&:= I_3(t,x)+I_4(t,x).
		\end{align*}
		Note that $I_3(t,x) \le e^{\lambda}\varepsilon$ for all $t>0,\, x \in \X$. In addition, since $\sup_{x \in \X} \abs{g(x)} < \infty$, it follows from the same argument as in $I_2$ that 
		\begin{align*}
			\lim_{t \downarrow 0} \sup_{x \in \X} I_4(t,x) = 0. 
		\end{align*}
		Hence, $\mathop{\limsup}_{t \downarrow 0} \sup_{x\in \X} A_t(x) \mathrel{\le} e^{\lambda}\varepsilon$, and since $\varepsilon>0$ can be arbitrarily small, we see the limit is zero. Finally, collecting all terms, we see \eqref{ineq_CM:ctemp_continuous_subg} holds.  
		
		For \eqref{it_prop_subg_no_noise:3}, we let $L^{\infty,c}_\X$ denote the collection of $L^\infty_\X$ functions with compact essential support. Note by \cite[Theorem 3.1 \& Corollary  4.2 \& Corollary 4.5]{barlow2012equivalence} (also see \cite[p. 1122]{barlow2012equivalence}), there exist $\alpha, \delta \in (0,1)$ and $C>0$ so that for any $(t,x) \in (0,\infty)\times \X$, $0< R< \Psi^{-1}(t)$,  $f \in L^{\infty,c}_\X$, the function $u(s,y):= P_{s}f(y)$ for $s>0$ and $y \in \X$ satisfies
		\begin{align}\label{ineq_CM:osc_inequality}
			\abs{u(s_1,y_1) - u(s_2,y_2)} \le C \pr{\frac{\Psi^{-1}\pr{\abs{s_1-s_2}} +d(y_1,y_2)}{R}}^\alpha \sup_{(s,y) \in D((t,x), R)} \mathopen{\lvert}u(s,y)\mathclose{\rvert}
		\end{align}
		for all $(s_1,y_1),(s_2,y_2) \in D((t,x),\delta R)$, where $D((t,x),r):= (t-\Psi(r), t]\times B_r(x)$ for $t\in \R$, $x \in \X$ and $r>0$.
		
		Recall $\delta, \alpha \in (0,1)$ in \eqref{ineq_CM:osc_inequality}, we note by \cite[p.1126 Proof of Lemma 4.6]{barlow2012equivalence}, it holds uniformly in $t>0$, $x',x'' \in \X$ with $d(x',x'') \le \delta \Psi^{-1}(t/2)$ and $f \in L^{\infty,c}_\X$,
		\begin{align*}
			\abs{\int_\X \pr{G_t(x',y) - G_t(x'',y)} f(y)m(\mathrm dy)} \lesssim \pr{\frac{d(x',x'')}{\Psi^{-1}(t/2)}}^\alpha \norm{f}_{L^\infty_\X}.
		\end{align*}
		On the other hand, if $\norm{f}_{L^\infty_\X} \le 1$, then by conservativeness, the left hand side is bounded above by $2$. Hence, 
		\begin{align*}
			\norm{G_t(x',\cdot) - G_t(x'',\cdot)}_{L^1_\X}&= \sup_{f \in L^{\infty,c}_\X\atop \norm{f}_{L^\infty_\X} \le 1 }\abs{\int_\X \pr{G_t(x',y) - G_t(x'',y)} f(y)m(\mathrm dy)}\\
			&\lesssim \min\left\{1,\pr{\frac{d(x',x'')}{\Psi^{-1}(t/2)}}^\alpha  \right\}.
		\end{align*}
		Now let $x \in \X$ be arbitrary and let $0<t< t'< \infty$. Suppose first that 
		$$
		t'-t < \Psi(\delta \Psi^{-1}(t/2)),
		$$
		then $(t,x),(t',x) \in D\pr{(t^{\prime},x),\delta \Psi^{-1}(t/2)}$. By \eqref{ineq_CM:osc_inequality}, we see it holds uniformly in   $f \in L^{\infty,c}_\X$ with $\norm{f}_{L^\infty_\X} \le 1$, $0<t'-t<\Psi(\delta\Psi^{-1}(t/2))$ and $x \in \X$ that
		\begin{align*}
			\abs{P_tf(x) - P_{t'}f(x)}\lesssim\pr{\frac{\Psi^{-1}\pr{\abs{t-t'}}}{\Psi^{-1}(t/2)}}^\alpha.
		\end{align*}
		If instead $t'-t\ge  \Psi\pr{\delta \Psi^{-1}(t/2)}$, by monotonicity of $\Psi^{-1}$, 
		\begin{align*}
			\Psi^{-1}\pr{t'-t} \ge \delta \Psi^{-1}(t/2),
		\end{align*}
		which implies 
		\begin{align*}
			\delta^\alpha\le \pr{\frac{\Psi^{-1}(t'-t)}{\Psi^{-1}(t/2)} }^\alpha.
		\end{align*}
		Hence, for any $f \in L^{\infty,c}_\X$ with $\norm{f}_{L^\infty_\X} \le 1$, $x \in \X$ and $\Psi\pr{\delta \Psi^{-1}(t/2)}\le t'-t$,
		\begin{align*}
			\abs{P_tf(x) - P_{t'}f(x)} \le 2\mathrel{\le}2\delta^{-\alpha} \pr{\frac{\Psi^{-1}(t'-t)}{\Psi^{-1}(t/2)} }^\alpha.
		\end{align*}
		Therefore, it holds uniformly in $0<t< t'< \infty$ and $x \in \X$ that 
		\begin{align*}
			\norm{G_t(x,\cdot) - G_{t'}(x,\cdot)}_{L^1_\X} &= \sup_{f \in L^{\infty,c}_\X\atop \norm{f}_{L^\infty_\X} \le 1}\abs{P_tf(x) -P_{t'}f(x)}\\
			&\lesssim \min\left\{ 1, \pr{\frac{\Psi^{-1}(t'-t)}{\Psi^{-1}(t/2)} }^\alpha\right\}.
		\end{align*}
		Finally, we see \eqref{it_prop_subg_no_noise:3} holds by Lemma \ref{lem_CM:escape_polynomial_bound} and the fact that $\alpha \in (0,1]$. 
	\end{proof}

	We will need the following elementary result.
	\begin{lemma}\label{lem_CM:subg_space_time_scaling}
		There exists $a>0$ so that for each $\lambda>0$, 
		\begin{align*}
			\exp \pr{\lambda \frac{d(x,y)}{\Psi^{-1}(t)}}G_t(x,y) \lesssim G_{at}(x,y),
		\end{align*}
		uniformly in $t>0$, $x,y \in \X$. Furthermore, it holds uniformly in $t>0,x,y \in \X$ that
		\begin{align*}
			d(x,y)^\lambda G_{t}(x,y) \lesssim \Psi^{-1}(t)^\lambda G_{at}\left( x,y\right).
		\end{align*}
	\end{lemma}
	
	\begin{proof}
		Let $b_U,b_L,c_U,c_L>0$ be as in $\mathrm{HKE}_{\rm full}(\Psi)$, and we may assume without loss of generality that $b_U\le b_L$, $c_U \le c_L$ and $b_U,c_U \le 1$. Observe by Lemma \ref{lem_CM:escape_polynomial_bound}, for each $\theta >0$, 
		\begin{align*}
			\Psi(\theta r) \asymp \Psi(r),\qquad \text{ uniformly in }r >0. 
		\end{align*}
		
		Let $r_\lambda(t):= \frac{b_Uc_U}{2\lambda}\Psi^{-1}(t)$ and $K_\lambda(t) := \frac{c_U}{2 \Psi\pr{r_\lambda(t)}}$, then for $r>0$, 
		\begin{align*}
			\frac{c_U}{2}t\Psi^*\pr{b_U \frac{r}{t}} \ge \frac{c_Ub_U r}{2r_\lambda(t)} - \frac{c_U t}{2\Psi(r_\lambda(t))}= \frac{\lambda r}{\Psi^{-1}(t)} - K_\lambda(t) t,
		\end{align*}
		In other words, 
		\begin{align*}
			\frac{\lambda r}{\Psi^{-1}(t)} - c_U t\Psi^*\pr{b_U \frac{r}{t}} \le K_\lambda(t) t - \frac{c_U}{2} t\Psi^*\pr{b_U \frac{r}{t}}.
		\end{align*}
		However, $K_\lambda(t) \asymp \frac{1}{t}$, hence by \eqref{ineq_CM:volume_all_scales} it holds uniformly in $t>0,x,y \in \X$ that
		\begin{align*}
			\exp \left(\lambda \frac{d(x,y)}{\Psi^{-1}(t)} \right)G_t(x,y) \lesssim \frac{1}{V(\Psi^{-1}(t)\wedge \text{diam}(\X))} \exp \pr{-\frac{c_U}{2} t\Psi^*\pr{b_U \frac{d(x,y)}{t}} }.
		\end{align*}
		By Lemma \ref{lem_CM:escape_polynomial_bound}, for any $\theta\ge 1$, there exist $\beta \ge 1$ and $C>0$ so that $\Psi(\theta r) \le C \theta^{\beta + 1}\Psi(r)$ for all $r>0$.  Hence, we may choose an appropriate $a>0$ so that for all $t,r>0$, 
		\begin{align*}
			c_L (at)\Psi^*\pr{b_L  \frac{r}{at}} &= \sup_{s>0} \frac{c_Lb_L r}{s}- \frac{c_Lat}{\Psi(s)}\\
			&= \sup_{s>0} \frac{c_Ub_U r}{s} - \frac{c_L at}{\Psi\pr{s \frac{c_Lb_L}{c_Ub_U}}}\\
			&\le \sup_{s>0} \frac{c_Ub_Ur}{s} - \frac{c_La}{C \pr{\frac{c_L b_L}{c_Ub_U}}^{\beta+1}}\frac{t}{\Psi(s)}\\
			&\le \sup_{s>0} \frac{c_Ub_U r}{s} - \frac{c_U t}{\Psi(s)}\\
			&=c_Ut\Psi^*\pr{b_U \frac{r}{t}}
		\end{align*}
		Hence, by replacing $c_U$ by $c_U/2$ if needed and \eqref{ineq_CM:volume_all_scales},  it holds uniformly in $t>0, x,y \in \X$ that
		\begin{align*}
			\exp \left(\lambda \frac{d(x,y)}{\Psi^{-1}(t)} \right)G_t(x,y) \lesssim \frac{1}{V(\Psi^{-1}(t)\wedge \text{diam}(\X))} \exp \pr{-c_L at\Psi^*\pr{b_L \frac{d(x,y)}{at}} } \lesssim G_{at}(x,y).
		\end{align*}
		The second inequality is obtained by noting  
		\begin{align*}
			d(x,y)^\lambda = \Psi^{-1}(t)^\lambda \pr{\frac{d(x,y)}{\Psi^{-1}(t)}}^{\lambda} \lesssim \Psi^{-1}(t)^\lambda \exp \pr{\lambda \frac{d(x,y)}{\Psi^{-1}(t)}}.
		\end{align*}
	\end{proof}

	\begin{proof}[Proof of Proposition \ref{prop:sub_g_condition_noise}]
		We note \eqref{cond:exp_hk_lamda} of \eqref{it_prop_subg_noise:1} is a direct consequence of the sub-Gaussian heat kernel upper bound and Proposition \ref{prop:sub_g_cond_no_noise}. For \eqref{cond:hk_uni_bound}, fix $\delta\in(0,1)$ and let $k \ge 1$ to be determined later and consider by Lemma \ref{lem_CM:subg_space_time_scaling} that 
		\begin{align*}
			\int_\X \mathbf 1_{d(x,y)\ge \delta} G_t(x,y)^2 m(\mathrm dy) &\lesssim  t^{-\alpha_{0}} \delta^{-k}\int_\X d(x,y)^k G_t(x,y) m(\mathrm dy)\\
			&\lesssim t^{-\alpha_{0}} \delta^{-k} \Psi^{-1}(t)^k.
		\end{align*}
		By Lemma \ref{lem_CM:escape_polynomial_bound}, we see there exists $k \ge 1$ so that $\Psi^{-1}(t)^k \lesssim t$ uniformly in $t \in (0,1]$, which implies the desired inequality. 
		
		To see \eqref{it_prop_subg_noise:2}, let us denote $r(x):= d(x_0,x)$ and let $\lambda \ge 0$ and $T\ge 1$ be fixed. It holds uniformly in $t \in (0, T]$ and $x \in \X$ that  
		\begin{align*}
			&e^{-2\lambda r(x)}\iint_{\X^2}G_t(x,y)e^{\lambda \pr{r(y)+r(z)}}G_s(y,z) G_t(x,z) m(\mathrm dz)m(\mathrm dy)\\
			\le & \quad\iint_{\X^2}  e^{\lambda \pr{d(x,y)+d(x,z)}} G_t(x,y)G_{s}(z,y)G_t(x,z)m(\mathrm dy)m(\mathrm dz)\\
			\overset{\text{Lem }\ref{lem_CM:subg_space_time_scaling} }{\lesssim} &\quad \iint_{\X^2} G_{at}(x,y)G_{at}(x,z) G_s(z,y) m(\mathrm dy)m(\mathrm dz)\\
			=& \quad G_{2at+s}(x,x).
		\end{align*}
		Since $V$ is increasing, $\frac{1}{V(r\wedge \text{diam}(\X))}\le \frac{1}{V(r)}+\frac{1}{V(1\wedge \text{diam}(\X))}$ for all $r>0$. Therefore, by \eqref{ineq_CM:volume_all_scales}, uniformly in $t \in (0,1]$ and $x \in \X$
		\begin{align*}
			e^{-2\lambda r(x)}\norm{G_t(x,\cdot)h(\cdot)^\lambda}_{\mathcal H_{K_{\gamma}}}^2
			&\lesssim \int_0^\infty e^{-s}s^{\gamma} G_{2at +s}(x,x) \frac{\mathrm ds}{s}\\
			&\mathrel{\lesssim} \int_0^\infty e^{-s}s^\gamma \frac{1}{V\pr{\Psi^{-1}(2at+s)\wedge \text{diam}(\X)}}\frac{\mathrm ds}{s}\\
			&\mathrel{\le} \int_0^\infty s^\gamma \frac{1}{V\pr{\Psi^{-1}(2at+s)}}\frac{\mathrm ds}{s} + \frac{\Gamma(\gamma)}{V(1\wedge \text{diam}(\X))}\\
			&\lesssim \int_0^\infty s^{\gamma} \frac{1}{V\pr{\Psi^{-1}(t+s)}}\frac{\mathrm ds}{s} \mathbin{+} 1\\
			&\lesssim t^{-\alpha_{0}},
		\end{align*}
		which implies \eqref{cond:exp_hk_lamda} holds. Finally, we may take $k \ge 1$ sufficiently large so that $\Psi^{-1}(t/a)^{2k} \lesssim t$ uniformly in $t \in (0,1]$ and consider for each $\delta>0$ that
		\begin{align*}
			\norm{\mathbf 1_{B_\delta(x)^{\mathrm c}} G_{t/a}(x,\cdot)}_{\mathcal H_{K_{\gamma}}}^2&\le \delta^{-2k}\norm{d(x,\cdot)^kG_{t/a}(x,\cdot)}_{\mathcal H_{K_{\gamma}}}^2\\
			&\lesssim \delta^{-2k} \Psi^{-1}(t/a)^{2k} \norm{G_{t}(x,\cdot)}_{\mathcal H_{K_{\gamma}}}^2\\
			&\lesssim \delta^{-2k}\, t\, t^{-\alpha_0}\\
			&\lesssim \delta^{-2k}\, t^{1-\alpha_{0}}
		\end{align*}
		uniformly in $t \in (0,1]$ and $x \in \X$. Since $0< \alpha_{0}<1$, we see the desired result holds.
	\end{proof}
		The next Lemma shows that the on-diagonal lower heat kernel estimates give \eqref{ineq:instant_propagation} if the metric space is geodesic and local volume doubling property holds. 
	\begin{lemma}\label{prop:on_diagonal_to_instant_propagation}
		Let $\rho:\X\times \X \to \R_+$ be a geodesic metric on the measure space $(\X,m)$, denote $B_r^\rho(x)$ be the open ball of radius $r>0$ centered at $x \in \X$ under the metric $\rho$. Let $V:\R_+\to \R_+$ be a doubling function so that for some $s_0>0$, 
		\begin{align}\label{ineq:local_volume_doubling}
			V(s) \asymp m(B_s^\rho(x)),\qquad \text{uniformly in } x \in \X, \, s \in (0,s_0]. 
		\end{align}
		Let $\Psi:\R_+\to \R_+$ be an increasing homeomorphism for which there exists $\delta_0>0$ so that 
		\begin{align}\label{ineq:on_diagonal_lower_bound}
			G_t(y,x) \gtrsim \frac{1}{V(\Psi^{-1}(t))}, \qquad \text{ uniformly for } t \in (0,1]\text{ and } \rho(x,y) \le \delta_0 \Psi^{-1}(t),
		\end{align}
		and 
		\begin{align}\label{ineq:on_diagonal_space_time_scaling}
			\lim_{t\downarrow 0}\frac{t}{\Psi^{-1}(t)} =0.
		\end{align}
		Then condition \eqref{ineq:instant_propagation} holds under the metric $\rho$. 
	\end{lemma}
	\begin{proof}[Proof of Lemma \ref{prop:on_diagonal_to_instant_propagation}]
		Let $x_0 \in \X$ and $0<r_0<v$ be arbitrary but fixed. By \eqref{ineq:on_diagonal_space_time_scaling}, there exists $\varepsilon_0 \in (0,1/2]$ so that for all $\varepsilon \in (0, \varepsilon_0)$, 
		\begin{align}\label{ineq:on_diagonal_to_ip_step_1}
			v \varepsilon \le \frac{\delta_0}{8} \Psi^{-1}(t),\qquad \frac{\delta_0}{8}\Psi^{-1}(t) < r_0 \wedge s_0,\qquad \text{ for all }t \in (\varepsilon, 2\varepsilon).
		\end{align}
		Let $\varepsilon \in (0,\varepsilon_0),\, t \in (\varepsilon, 2\varepsilon),\, r \in (r_0, v)$ and $x \in B_{r+\varepsilon v}^\rho(x_0)$. Set 
		\begin{align*}
			r_x:= \rho(x, B^\rho_r(x_0)):= \inf_{y \in B^\rho_r(x_0)} \rho(x,y).
		\end{align*}
		Since $\rho$ is a geodesic metric, $r_x=\max\{\rho(x,x_0)-r,0\}\leq v\varepsilon$. Let $\zeta:[0, \rho(x,x_0)]\to \X$ be a geodesic from $x$ to $x_0$ and set
		\begin{align*}
			\ell:= \min\{ r_x +\frac{\delta_0}{4} \Psi^{-1}(t), \rho(x,x_0)\},\qquad w := \zeta(\ell).
		\end{align*}
		We claim that
		\begin{align}\label{ineq:on_diagonal_set_inclusion}
			B^\rho_{\frac{\delta_0}{8}\Psi^{-1}(t)}(w) \subset B_r^{\rho}(x_0).
		\end{align}
		Indeed, if $\ell = r_x +\frac{\delta_0}{4}\Psi^{-1}(t)$, then $\rho(x_0,w) = \rho(x_0,x) - \rho(x,w)= \rho(x,x_0) - r_x - \frac{\delta_0}{4}\Psi^{-1}(t)$.
		Hence for all $y \in B^\rho_{\frac{\delta_0}{8}\Psi^{-1}(t)}(w)$, we see 
		\begin{align*}
			\rho(y,x_0)& \le \rho(w,y)+\rho(w,x_0) \\
			&\le \frac{\delta_0}{8}\Psi^{-1}(t) +\rho(x,x_0) - r_x - \frac{\delta_0}{4}\Psi^{-1}(t)\\
			&= \rho(x,x_0)-r_x-\frac{\delta_0}{8}\Psi^{-1}(t)\\
			&\le r - \frac{\delta_0}{8}\Psi^{-1}(t)\\
			&< r.
		\end{align*}
		If $\ell = \rho(x,x_0)$, then $w = x_0$ and \eqref{ineq:on_diagonal_set_inclusion} holds by \eqref{ineq:on_diagonal_to_ip_step_1}. This proves \eqref{ineq:on_diagonal_set_inclusion}.

		In addition, we see it holds for all $y \in B^\rho_{\frac{\delta_0}{8}\Psi^{-1}(t)}(w)$ that
		\begin{align*}
			\rho(y,x) &\le \rho(y,w) + \rho(w,x) \\
			&\le \frac{\delta_0}{8}\Psi^{-1}(t) + r_x +\frac{\delta_0}{4}\Psi^{-1}(t)\\
			& \le  \frac{3\delta_0}{8}\Psi^{-1}(t) + r_x \\
			&\le  \frac{3\delta_0}{8}\Psi^{-1}(t) +v\varepsilon\\
			&\overset{\eqref{ineq:on_diagonal_to_ip_step_1}}{\le} \frac{\delta_0}{2}\Psi^{-1}(t)\\
			&< \delta_0 \Psi^{-1}(t).
		\end{align*}
		Therefore,
		\begin{align*}
			\int_{B^\rho_r(x_0)} G_t(y,x) m(\mathrm dy) &\overset{\eqref{ineq:on_diagonal_set_inclusion}}{\ge} \int_{B^\rho_{\frac{\delta_0}{8}\Psi^{-1}(t)}(w)} G_t(y,x) m(\mathrm dy)\\
			&\overset{\eqref{ineq:on_diagonal_lower_bound}{,\,\eqref{ineq:local_volume_doubling}}}{\gtrsim} \frac{V(\frac{\delta_0}{8}\Psi^{-1}(t))}{V(\Psi^{-1}(t))}\\
			&\overset{\eqref{ineq:doubling_function}}{\gtrsim} 1,
		\end{align*}
		note that the implicit constant depends only on the doubling property of $V$ and the constant in \eqref{ineq:on_diagonal_lower_bound}, which gives the desired inequality.
	\end{proof}
	\begin{proof}[Proof of Proposition \ref{prop:sub_g_strict_pos_condition}]
		By \cite[Theorem 2.11]{murugan2020length} and \cite[Proposition A.1]{kajino2020singularity}, there exists a geodesic metric $\rho$ on $\X$ so that 
		\begin{align*}
			\rho(x,y) \asymp d(x,y),\qquad \text{uniformly in }x,y \in \X. 
		\end{align*}
		Note that $\mathrm{HKE}_{\rm full}(\Psi)$ holds with $d(x,y)$ replaced by $\rho(x,y)$ for $x,y \in \X$. 

		By Lemma \ref{lem_CM:escape_polynomial_bound}, we see \eqref{ineq:on_diagonal_space_time_scaling} holds and the desired inequality is a direct consequence of Lemma \ref{prop:on_diagonal_to_instant_propagation}.
	\end{proof}

\subsection{Proof of Proposition~\ref{prop:local_manifolds}}

We shall prove Proposition \ref{prop:local_manifolds} through three lemmas: consequences of the Gaussian estimates, verification of the geometric and semigroup assumptions, and verification of the noise assumptions.

\begin{lemma}[Consequences of the Gaussian estimates]\label{lem:manifold_gaussian_consequences}
 Assume \eqref{eq:manifold_gaussian}. The heat kernel is conservative:
 \begin{align}
  \int_M G_t(x,y)\,m(\mathrm dy)=1,
  \qquad t>0,\quad x\in M.
  \label{eq:manifold_conservation}
 \end{align}
  Moreover, the following estimates hold.
 \begin{enumerate}
  \item There are constants $c,C,\kappa>0$, independent of the centers and radii, such that
 \begin{align}
  c r^n\leq m(B_r(x))\leq C r^n,
  \qquad x\in M,\quad 0<r\leq1.
  \label{eq:manifold_local_volume}
 \end{align}
  At all radii,
 \begin{align}
  m(B_R(x))\leq C R^n e^{\kappa R},\qquad x\in M,\quad R>0.
  \label{eq:manifold_volume_growth}
 \end{align}
  \item For every $q\geq0$,
 \begin{align}
  \sup_{x\in M,\,0<t\leq1}
  \int_M G_t(x,y)e^{q d(x,y)}\,m(\mathrm dy)<\infty.
  \label{eq:manifold_exp_moment}
 \end{align}
  \item There are constants $\eta\in(0,1)$ and $C>0$ such that, uniformly for $x,x'\in M$ and $t,t'\in(0,1]$,
 \begin{align}
  \|G_t(x,\cdot)-G_t(x',\cdot)\|_{L^1_M}
  &\leq C\frac{d(x,x')^\eta}{t^{\eta/2}},\nonumber\\
  \|G_t(x,\cdot)-G_{t'}(x,\cdot)\|_{L^1_M}
  &\leq C\frac{|t-t'|^{\eta/2}}{(t\wedge t')^{\eta/2}}.
  \label{eq:manifold_holder}
 \end{align}
 \end{enumerate}
\end{lemma}

\begin{proof}
 For $0<r\leq1$, the sub-Markov property  gives
 \[
  1\geq\int_{B_r(x)}G_{r^2}(x,y)\,m(\mathrm dy)
  \overset{\eqref{eq:manifold_gaussian}}{\geq} c_0e^{-C_0}r^{-n}m(B_r(x)).
 \]
 For the reverse volume bound, set $a=c_0 2^{-n/2}$ and choose $A\geq\sqrt2$ such that $C_0e^{-c_0A^2}\leq a/2$. Splitting the semigroup integral over $B_{A\sqrt t}(x)$ and its complement gives, for $0<t\leq1/2$,
 \begin{align*}
	at^{-n/2}&\overset{\eqref{eq:manifold_gaussian}}{\le} G_{2t}(x,x)
	=\pr{\int_{B_{A\sqrt t}(x)}+\int_{B_{A\sqrt t}(x)^c}} G_t(x,y)^2\,m(\mathrm dy)\\
	&\leq C_0^2t^{-n}m(B_{A\sqrt t}(x))
       +\frac a2t^{-n/2}.
 \end{align*}
 Taking $t=(r/A)^2$ yields
 \[
  \frac{a}{2C_0^2A^n}r^n\leq m(B_r(x))
  \leq\frac{e^{C_0}}{c_0}r^n,
  \qquad x\in M,\quad 0<r\leq1,
 \]
 proving \eqref{eq:manifold_local_volume}.

 Conservation \eqref{eq:manifold_conservation} follows from \eqref{eq:manifold_gaussian}, \eqref{eq:manifold_local_volume} and \cite[p.~1001]{coulhon2012heat}. The volume bound \eqref{eq:manifold_volume_growth} follows from \eqref{eq:manifold_local_volume} and \cite[Section~2.3, (2.6)]{hebisch2001relation}.

 By \cite[Lemma~3.6]{hebisch2001relation} and \eqref{eq:manifold_conservation}, there are $c,C>0$ such that
 \begin{align}
  \int_{B_r(x)^c}G_t(x,y)\,m(\mathrm dy)
  \leq Ce^{-cr^2/t},\qquad x\in M,\quad 0<t\leq1,\quad r>0.
  \label{eq:manifold_gaussian_tail}
 \end{align}
 For $q=0$, \eqref{eq:manifold_exp_moment} is conservation. For $q>0$, Tonelli's theorem gives
 \begin{align*}
  \int_M G_t(x,y)e^{q d(x,y)}\,m(\mathrm dy)&=\int_M G_t(x,y)\pr{1+q\int_0^{d(x,y)}e^{qr}\,\mathrm dr}\,m(\mathrm dy)\\
  &=1+q\int_0^\infty e^{qr}
       \int_{\{y:d(x,y)>r\}}G_t(x,y)\,m(\mathrm dy)\,\mathrm dr\\
  &{\overset{\eqref{eq:manifold_gaussian_tail}}{\leq}} 1+Cq\int_0^\infty e^{qr-cr^2/t}\,\mathrm dr\\
  &\le 1+C\int_0^\infty e^{qr - cr^2}\,\mathrm dr
 \end{align*}
 which is finite and independent of $x$ and $0<t\leq1$. This proves \eqref{eq:manifold_exp_moment}.

 By \cite[pp.~1000--1001]{coulhon2012heat}, \cite[Theorem~2.3(2)]{hebisch2001relation} and the $L^\infty$-contractivity of $(P_t)$, there exist $C>0$ and $\eta\in(0,1)$ such that, for every bounded measurable $f$, $x,x'\in M$ and $t,t'\in(0,1]$,
 \[
  |P_tf(x)-P_{t'}f(x')|
  \leq C\left(\frac{d(x,x')+|t-t'|^{1/2}}{\sqrt{t\wedge t'}}\right)^\eta
       \|f\|_\infty.
 \]
 Taking the supremum over $\|f\|_\infty\leq1$ gives the same bound for $\|G_t(x,\cdot)-G_{t'}(x',\cdot)\|_{L^1_M}$ by duality. Setting $t'=t$ or $x'=x$ proves the two estimates in \eqref{eq:manifold_holder}.
\end{proof}

\begin{lemma}[Geometric and weighted semigroup assumptions]\label{lem:manifold_semigroup}
 Assume \eqref{eq:manifold_gaussian} and let $h(x)=e^{d(o,x)}$. Then $(M,d,m)$ satisfies \eqref{ineq:entropy bound}, and $h$ satisfies \eqref{ineq:faster_than_ms} and \eqref{ineq:ref_regular}. The heat semigroup is Feller and satisfies \eqref{cond:ctem continuous}, the $L^1$ inequality in \eqref{cond:exp_hk_lamda}, and \eqref{cond:time_holder} with $\beta=\xi_1=\eta/2$ and $\xi_2=\eta$, where $\eta$ is as in Lemma~\ref{lem:manifold_gaussian_consequences}. Moreover, \eqref{ineq:instant_propagation} holds with the geodesic distance $d$.
\end{lemma}

\begin{proof}
 The volume-to-covering argument in \cite[pp.~167--168]{kaenmaki2017measures}, applied at bounded scales using \eqref{eq:manifold_local_volume}--\eqref{eq:manifold_volume_growth}, gives $D_\varepsilon(B_R(x))\leq C_R\varepsilon^{-n}$ uniformly in $x\in M$ and $0<\varepsilon\leq1$, for each $R\geq1$. This proves \eqref{ineq:entropy bound}.

 By the $5r$-covering lemma \cite[Theorem~1.2, p.~2]{heinonen2001lectures} and separability, there is an at most countable family $(x_j)_{j\in\mathcal I}$ such that $M=\bigcup_{j\in\mathcal I}B_1(x_j)$ and the balls $B_{1/10}(x_j)$ are pairwise disjoint. Recall $\kappa>0$ in \eqref{eq:manifold_volume_growth} and fix $\lambda_0>\kappa$. By \eqref{eq:manifold_local_volume}, the triangle inequality and disjointness, followed by Tonelli's theorem, there exists $C_{\lambda_0}>0$ such that
 \begin{align*}
  \sum_{j\in\mathcal I}h(x_j)^{-\lambda_0}
  \leq C_{\lambda_0}\int_M e^{-\lambda_0d(o,y)}\,m(\mathrm dy)
  =C_{\lambda_0}\lambda_0\int_0^\infty e^{-\lambda_0r}m(B_r(o))\,\mathrm dr\overset{\eqref{eq:manifold_volume_growth}}{<}\infty,
 \end{align*}
this proves \eqref{ineq:faster_than_ms} for every $\lambda\geq\lambda_0$. Condition \eqref{ineq:ref_regular} follows from $h(y)\leq e^{d(x,y)}h(x)$.

 The Gaussian upper bound in \eqref{eq:manifold_gaussian} and the semigroup property yield $P_t\mathcal C_M^0\subset\mathcal C_M^0$ for all $t>0$. The Feller property then follows from \cite[Example~4.2, p.~10]{benamor2023essential}. Moreover, \eqref{eq:manifold_exp_moment} gives  the $L^1$ inequality in \eqref{cond:exp_hk_lamda}.
 
 Fix $\lambda>0$ and $f\in\mathcal C^{\mathrm{tem}}_{M,h}$. By Lemma~\ref{lem:C_temp_density_0} (whose proof does not depend on this  subsection), choose $f_j\in\mathcal C_M^0$ with $\|f-f_j\|_{(\lambda)}\to0$. Then for each $\lambda>0$, there exists $C_\lambda >0$ so that 
 \[
  \limsup_{t\downarrow0}\|P_t f-f\|_{(\lambda)}
  \leq C_\lambda\|f-f_j\|_{(\lambda)}
       +\limsup_{t\downarrow0}\|P_t f_j-f_j\|_\infty =C_\lambda\|f-f_j\|_{(\lambda)}
 \]
 which converges to $0$ as $j\to\infty$. The same approximation shows that $P_t f\in\mathcal C^{\mathrm{tem}}_{M,h}$, and the semigroup property gives continuity at every time. Since $P_t^*=P_t$, this proves \eqref{cond:ctem continuous}. Equation~\eqref{eq:manifold_holder} gives \eqref{cond:time_holder}.

 Finally, $(M,d)$ is geodesic by the Hopf--Rinow theorem \cite[Chapter~7, Theorem~2.8]{docarmo1992riemannian}. With $\rho=d$, $V(r)=r^n$ and $\Psi(r)=r^2$, \eqref{eq:manifold_local_volume} and the Gaussian lower bound verify the volume and near-diagonal hypotheses of Lemma~\ref{prop:on_diagonal_to_instant_propagation}. Since $t/\Psi^{-1}(t)=\sqrt t\to0$, that lemma gives \eqref{ineq:instant_propagation}.
\end{proof}

\begin{lemma}[Covariance estimates]\label{lem:manifold_covariance}
 Assume \eqref{eq:manifold_gaussian} and \eqref{eq:manifold_gamma}. The measure $K$ in \eqref{eq:manifold_covariance} is symmetric, non-negative definite and Radon. Moreover, $L^2_M\subset\mathcal H_K$ and
 \begin{align}
  0\leq\|v\|_{\mathcal H_K}^2\leq\|v\|_{L^2_M}^2,
  \qquad v\in L^2_M.
  \label{eq:manifold_noise_contraction}
 \end{align}
 For every $\alpha$ with $\max\{0,n/2-\gamma\}<\alpha<1$, the $\mathcal H_K$ inequality in \eqref{cond:exp_hk_lamda} and condition~\eqref{cond:hk_uni_bound} hold.
\end{lemma}

\begin{proof}
 Clearly, $K$ is a symmetric Radon measure. Applying Tonelli's theorem and $L^2$-contractivity to $|v|$ gives $v\in\mathcal H_K$ for every $v\in L^2_M$. Symmetry and the semigroup property then yield
 \[
  0\leq\|v\|_{\mathcal H_K}^2
  =\frac1{\Gamma(\gamma)}\int_0^\infty e^{-s}s^{\gamma-1}
       \|P_{s/2}v\|_{L^2_M}^2\,\mathrm ds
  \leq\|v\|_{L^2_M}^2.
 \]
 This proves non-negative definiteness and \eqref{eq:manifold_noise_contraction}. For $v(y)=G_t(x,y)h(y)^\lambda$, the Gaussian bound and \eqref{eq:manifold_exp_moment} give, uniformly for $x\in M$, $0<t\leq1$ and $0\leq\lambda\leq1$,
 \[
  \|v\|_{L^1_M}\leq C h(x)^\lambda,
  \qquad \|v\|_{L^2_M}^2\leq C t^{-n/2}h(x)^{2\lambda}.
 \]
 For $s\leq1$,
 \[
  \langle v,P_s v\rangle_{L^2_M}
  \leq\min\{\|v\|_{L^2_M}^2,\,
             \sup_{y,z}G_s(y,z)\|v\|_{L^1_M}^2\}
  \leq C h(x)^{2\lambda}(t+s)^{-n/2};
 \]
 for $s\geq1$, it is at most $C h(x)^{2\lambda}$ by the semigroup property. Since \eqref{eq:manifold_gamma} gives $\gamma>0$ and $n/2-\gamma<1$, we may fix
 \[
  \max\{0,n/2-\gamma\}<\alpha<1.
 \]
 In particular, $\gamma+\alpha-n/2>0$. For $0<t\leq1$, splitting at $s=t$ gives
 \begin{align*}
  \int_0^1s^{\gamma-1}(t+s)^{-n/2}\,\mathrm ds
  &\leq t^{-n/2}\int_0^t s^{\gamma-1}\,\mathrm ds
       +\int_t^1s^{\gamma-n/2-1}\,\mathrm ds\\
  &\leq\frac{t^{\gamma-n/2}}{\gamma}
       +t^{-\alpha}\int_t^1s^{\gamma+\alpha-n/2-1}\,\mathrm ds\\
  &\leq\left(\frac1\gamma+\frac1{\gamma+\alpha-n/2}\right)t^{-\alpha}.
 \end{align*}
 Therefore, using $1\leq t^{-\alpha}$,
 \begin{align}
  \|G_t(x,\cdot)h^\lambda\|_{\mathcal H_K}^2
  \leq C h(x)^{2\lambda}
    \left(1+\int_0^1s^{\gamma-1}(t+s)^{-n/2}\,\mathrm ds\right)
  \leq C t^{-\alpha}h(x)^{2\lambda},
  \label{eq:manifold_noise_bound}
 \end{align}
 This proves the $\mathcal H_K$ inequality in \eqref{cond:exp_hk_lamda} with $\alpha\in(0,1)$, as required. Equation~\eqref{eq:manifold_noise_contraction} and the Gaussian bound also give
 \[
  \|G_t(x,\cdot)\mathbf1_{B_\delta(x)^c}\|_{\mathcal H_K}^2
  \leq\|G_t(x,\cdot)\mathbf1_{B_\delta(x)^c}\|_{L^2_M}^2
  \leq C t^{-n/2}e^{-c\delta^2/t}\longrightarrow0,\quad t\downarrow0,
 \]
 uniformly in $x$ for each $\delta>0$, proving \eqref{cond:hk_uni_bound}.
\end{proof}

\begin{proof}[Proof of Proposition~\ref{prop:local_manifolds}]
 Lemmas~\ref{lem:manifold_semigroup} and~\ref{lem:manifold_covariance} verify all the geometric, semigroup and noise assumptions.
\end{proof}

	\subsection{Proof of \texorpdfstring{Lemmas}{Lemmas} \ref{lem:H0} and \ref{lem:HS}}

    The next three lemmas ensure that the state space in which the solution paths are constructed is nice.

	\begin{lemma}\label{lem:ref_grow}
		Let $(\X,d)$ be a metric space with infinite diameter and let $h$ be a  weight function on it satisfying \eqref{ineq:faster_than_ms} and \eqref{ineq:ref_regular}. Then for each $x_0 \in \X$,
		\begin{align}\label{cond:ref_grow}
			\lim_{R \to \infty} \inf_{x \in B_R(x_0)^c} h(x) = \infty.
		\end{align}
	\end{lemma}
	\begin{proof}
		Suppose, to the contrary, that \eqref{cond:ref_grow} fails. Then for some $x_0\in\X$ and $M\geq1$, there are points $y_k\in\X$ such that $d(x_0,y_k)\to\infty$ and $h(y_k)\leq M$. Let $\lambda_0>0$ and $(x_n)_{n\in\mathcal I}$ be as in \eqref{ineq:faster_than_ms}. For each $k$, choose $n_k$ so that $y_k\in B_1(x_{n_k})$. Since $d(x_0,x_{n_k})\geq d(x_0,y_k)-1\to\infty$ as $k \to \infty$, after passing to a subsequence we may assume that the indices $n_k$ are distinct. By \eqref{ineq:ref_regular}, there is some $C\geq1$ such that $h(x_{n_k})\leq Ch(y_k)\leq CM$. Consequently,
		\begin{align*}
			\sum_{n\in\mathcal I}h(x_n)^{-\lambda_0}
			\geq\sum_{k\in\N}h(x_{n_k})^{-\lambda_0}
			\geq\sum_{k\in\N}(CM)^{-\lambda_0}=\infty,
		\end{align*}
		contradicting \eqref{ineq:faster_than_ms}.
	\end{proof}

	\begin{lemma} \label{lem:PS}
		$\mathcal C^{\mathrm{tem}}_{\X,h}$ is a Polish space.
	\end{lemma}

	\begin{proof}[Proof of Lemma \ref{lem:PS}]
		It is well known that if a metric space $(\mathbb X,d)$ is a locally compact separable complete metric space, then $\mathcal C_\X^0$  is separable.
		Let us fix an arbitrary $x_0 \in \X$.
		Take $\{g_n\}_{n \in \mathbb N} \subset \mathcal C_\X^0$ to be a dense subset and let $\mathcal X_n$  be continuous functions with values in $[0,1]$ whose support is contained in $B_{n+1}(x_0)$ and $\mathcal X_n(x) = 1$ for all $x \in B_n(x_0)$.
		Denote $g_{n,k}:= g_n \cdot \mathcal X_k$ for each $n,k \in \mathbb N$.
		Then $\mathcal A_d:=\{g_{n,k}\}_{n,k \in \mathbb N}$ is countably dense in $\mathcal C_\X^0$.
		
		We first show $\mathcal C^{\mathrm{tem}}_{\X,h}$ is Polish. Define the metric $d_\mathrm{tem}$ on $\mathcal C^{\mathrm{tem}}_{\X,h}$ by 
		\begin{align*}
			d_\mathrm{tem}(f,g):= \sum_{n = 1}^\infty 2^{-n}\pr{1 \wedge \norm{f - g}_{(n^{-1})}},
		\end{align*}
		which makes it a complete metric space.
		Suppose $f \in \mathcal C^{\mathrm{tem}}_{\X,h}$, define $f_n:= f \cdot \mathcal X_n$ for each $n \in \mathbb N$ and recall that the support of $\mathcal X_n$ is contained in $B_{n+1} := B_{n+1}(x_0)$ for some $x_0$.
		Then consider for any fixed $\ell>0$,
		\begin{align}\label{ineq:first compact approx temp}
			\lim_{n\to \infty}\| f - f_n\|_{(\ell)} = \lim_{n\to \infty}\sup_{x \in B_n^c} h(x)^{-\ell}\left| f(x)\right| \le \|f \|_{(\ell/2)}\lim_{n\to \infty} \sup_{x \in B_n^c}h(x)^{-\ell/2},
		\end{align}
		which is equal to zero by Lemma \ref{lem:ref_grow} if $\X$ has infinite diameter; if $\X$ has finite diameter, then $f_n=f$ for all large $n$.
		Now for each $n$, take $\tilde g_{n,k}\mathrel{\in} \mathcal A_d$, $k\in\mathbb N$, so that $\lim_{k \to \infty} \|f_n - \tilde g_{n,k}\|_\infty = 0$; this is possible since $f_n\in\mathcal C^0_\X$ and $\mathcal A_d$ is dense in $\mathcal C^0_\X$.
		Hence, given $\varepsilon>0$, pick $n \in \mathbb N$ large enough so that $\| f - f_n\|_{(\ell)} < \frac{\varepsilon}{2}$, and pick $k = k(n) \in \mathbb N$ large enough so that $\| f_n - \tilde g_{n,k}\|_\infty < \frac{\varepsilon}{2}$.
		Hence we obtain
		\begin{align}
			\|f - \tilde g_{n,k}\|_{(\ell)}  \le \| f- f_n\|_{(\ell)} + \|f_n - \tilde g_{n,k}\|_{(\ell)} < \frac{\varepsilon}{2} + \| f_n - \tilde g_{n,k}\|_\infty < \varepsilon.
		\end{align}
		This shows the set of functions $\mathcal A_d$ is a dense subset of $\mathcal C^{\mathrm{tem}}_{\X,h}$.	
	\end{proof}

\begin{lemma}\label{lem:C_temp_density_0}
		$\mathcal C^0_\X$ is dense in $\mathcal C^{\mathrm{tem}}_{\X,h}$.
	\end{lemma}

	\begin{proof}[Proof of Lemma \ref{lem:C_temp_density_0}]
		If $\X$ has finite diameter, then \eqref{ineq:entropy bound} and completeness imply that $\X$ is compact, so the assertion is immediate. Suppose that $\X$ has infinite diameter. Fix $x_0\in\X$ and let $(\mathcal X_n)_{n\in\N}$ be the compactly supported cutoff functions used in the proof of Lemma \ref{lem:PS}, chosen so that $0\leq\mathcal X_n\leq1$ with $\mathcal X_n=1$ on $B_n(x_0)$ and $\mathcal X_n =0$ on $B_{n+1}(x_0)^c$. For $f\in\mathcal C^{\mathrm{tem}}_{\X,h}$, set $f_n:=f\mathcal X_n\in\mathcal C^0_\X$. Then for every $\lambda>0$, 
		\begin{align*}
			\lim_{n \to \infty}\norm{f-f_n}_{(\lambda)}
			&\leq \lim_{n \to \infty}\norm{f}_{(\lambda/2)}
			\sup_{x\in B_n(x_0)^\mathrm c}h(x)^{-\lambda/2}=0,
		\end{align*}
		by Lemma \ref{lem:ref_grow}. Hence $f_n\to f$ in $\mathcal C^{\mathrm{tem}}_{\X,h}$.
	\end{proof}

We can now finish the proofs of Lemma \ref{lem:H0} and Lemma \ref{lem:HS}.
	
	\begin{proof}[Proof of Lemma \ref{lem:H0}]
		Since $\X$ is locally compact and separable, fix an increasing sequence of compact sets $(C_n)_{n\in\N}$ with $\X=\bigcup_{n\in\N}C_n$.
		Let $\mathcal S$ be the space of compactly supported real-valued Borel simple functions on $\X$, and define
		\begin{align*}
			B(f,g):=\iint_{\X^2}f(x)g(y)K(\mathrm dx,\mathrm dy),\qquad f,g\in\mathcal S.
		\end{align*}
		The assumptions on $K$ imply that $B$ is a symmetric non-negative definite bilinear form on $\mathcal S$, and hence satisfies the Cauchy--Schwarz inequality.

		For each $f\in\mathcal H_K$, choose $f_n\in\mathcal S$ supported in $C_n$ such that $f_n\to f$ pointwise and $|f_n|\leq|f|$.
		For $f,g\in\mathcal H_K$, Fatou's lemma and the Cauchy--Schwarz inequality applied to $|f_n|,|g_n|$ give
		\begin{align}\label{ineq:H0 mixed absolute}
			\iint_{\X^2}|f(x)g(y)|K(\mathrm dx,\mathrm dy)
			&\leq\liminf_{n\to\infty}B(|f_n|,|g_n|)\nonumber\\
			&\leq\left(\iint_{\X^2}|f(x)f(y)|K(\mathrm dx,\mathrm dy)\right)^{1/2}\nonumber\\
			&\qquad\times\left(\iint_{\X^2}|g(x)g(y)|K(\mathrm dx,\mathrm dy)\right)^{1/2}<\infty.
		\end{align}
		Thus $\mathcal H_K$ is a vector space, and the domain and the integral formula of $B$ extend to $\mathcal H_K$.
		By dominated convergence, $B(f,f)=\lim_{n\to\infty}B(f_n,f_n)\geq0$ for every $f\in\mathcal H_K$.
		Consequently,
		\begin{align}\label{ineq:H0 Cauchy Schwarz}
			|B(f,g)|\leq B(f,f)^{1/2}B(g,g)^{1/2},\qquad f,g\in\mathcal H_K,
		\end{align}
		and $\norm{f}_{\mathcal H_K}=B(f,f)^{1/2}$ is a Hilbertian seminorm in the sense of \cite[Section 1.1]{Zbl054760064}.
		Moreover, since $|f-f_n|\leq2|f|$, dominated convergence also yields
		\begin{align}\label{ineq:H0 compact truncation}
			\lim_{n \to \infty}\norm{f-f_n}_{\mathcal H_K}^2
			\leq\lim_{n\to \infty}\iint_{\X^2}|f(x)-f_n(x)||f(y)-f_n(y)|K(\mathrm dx,\mathrm dy)=0.
		\end{align}

		To prove separability, for each $n\in\N$ let $\mu_n$ be the finite Radon measure on $C_n$ defined by $\mu_n(A):=K(A\times C_n)$ for Borel sets $A\subset C_n$.
		By the separability of $\mathcal C(C_n;\R)$ under the uniform norm and \cite[Proposition~7.9, p.~217]{folland1999real}, choose a countable set $\mathcal D_n\subset\mathcal C(C_n;\R)$ dense in $L^2(C_n,\mu_n)$, and extend its elements by zero outside $C_n$.
		For every bounded Borel function $\psi$ supported in $C_n$, we see by the symmetry of $K$ that
		\begin{align*}
			\norm{\psi}_{\mathcal H_K}^2
			&=\iint_{C_n^2}{\psi}(x)\psi(y)K(\mathrm dx,\mathrm dy)\\
			&\leq\frac12\iint_{C_n^2}\bigl(|\psi(x)|^2+|\psi(y)|^2\bigr)K(\mathrm dx,\mathrm dy)
			=\norm{\psi}_{L^2(C_n,\mu_n)}^2.
		\end{align*}
		Given $f\in\mathcal H_K$ and $\varepsilon>0$, first use \eqref{ineq:H0 compact truncation} to choose $n$ with $\norm{f-f_n}_{\mathcal H_K}<\varepsilon/2$.
		Since $f_n$ is bounded and supported in $C_n$, choose $g\in\mathcal D_n$ with $\norm{f_n-g}_{L^2(C_n,\mu_n)}<\varepsilon/2$.
		Applying the preceding estimate to $\psi=f_n-g$ gives
		\begin{align*}
			\norm{f-g}_{\mathcal H_K}
			&\leq\norm{f-f_n}_{\mathcal H_K}+\norm{f_n-g}_{\mathcal H_K}\\
			&\leq\norm{f-f_n}_{\mathcal H_K}+\norm{f_n-g}_{L^2(C_n,\mu_n)}<\varepsilon.
		\end{align*}
		Thus $\bigcup_{n\in\N}\mathcal D_n$ is a countable dense subset of $\mathcal H_K$.

		Finally, $[0]_K$ is a linear subspace by the seminorm properties, and \eqref{ineq:H0 Cauchy Schwarz} shows that $B$ vanishes whenever one of its arguments belongs to $[0]_K$.
		Thus
		\begin{align*}
			\langle[f]_K,[g]_K\rangle:=B(f,g)
		\end{align*}
		is a well-defined inner product on $\mathcal H_K/[0]_K$, whose induced norm is the unique norm satisfying $\norm{[f]_K}_{\widetilde{\mathcal H}_K}=\norm{f}_{\mathcal H_K}$.
		Its completion is therefore a separable Hilbert space, as asserted.
	\end{proof}
	
	\begin{proof}[Proof of Lemma \ref{lem:HS}]
		Fix $g\in\mathcal H_K$.
		By the self-adjointness of $Q^{1/2}$, for every $f\in\widetilde{\mathcal H}_K$,
		\[
			(\Gamma_g\circ Q^{1/2})(f)
			=\left\langle[g]_K,Q^{1/2}f\right\rangle_{\widetilde{\mathcal H}_K}
			=\left\langle Q^{1/2}[g]_K,f\right\rangle_{\widetilde{\mathcal H}_K}.
		\]
		Thus $\Gamma_g\circ Q^{1/2}$ is a Hilbert--Schmidt  operator on $\widetilde{\mathcal H}_K$, whose norm is $\|Q^{1/2}[g]_K\|_{\widetilde{\mathcal H}_K}$ by Parseval's identity.
		For every $f\in\operatorname{Ker}(Q^{1/2})^\bot$, the definitions of $Q^{-1/2}$ and the induced norm give
		\[
			\|Q^{1/2}f\|_{Q^{1/2}\widetilde{\mathcal H}_K}
			=\|Q^{-1/2}Q^{1/2}f\|_{\widetilde{\mathcal H}_K}
			=\|f\|_{\widetilde{\mathcal H}_K};
		\]
		see \cite[Appendix~B.2, p.~429]{da2014stochastic}.
		Thus $Q^{1/2}$ restricts to an isometric isomorphism from $\operatorname{Ker}(Q^{1/2})^\bot$ onto $Q^{1/2}\widetilde{\mathcal H}_K$.
		Since $\Gamma_g\circ Q^{1/2}$ vanishes on $\operatorname{Ker}(Q^{1/2})$, the Hilbert--Schmidt norm identity in \cite[p.~96 and Section~4.2.1]{da2014stochastic} applies.
		Consequently, $\Gamma_g$ is Hilbert--Schmidt and
		\begin{align*}
			\|\Gamma_g\|^2_{\mathcal L^{\mathrm{HS}}_{Q^{1/2}\widetilde{\mathcal H}_K}}
			=\|\Gamma_g\circ Q^{1/2}\|^2_{\mathcal L^{\mathrm{HS}}_{\widetilde{\mathcal H}_K}}
			=\left\|Q^{1/2}[g]_K\right\|^2_{\widetilde{\mathcal H}_K}
			\leq\|Q^{1/2}\|^2\|g\|^2_{\mathcal H_K}.
		\end{align*}
	\end{proof}

		\subsection{Proof of Lemmas  \ref{cor:mtg+mink_cs}, \ref{lem:hk longtime}, \ref{lem:hk_abs} and \ref{lem:tight_tem_moment} 
    }

	We  need the following analytical result in the proof of Lemma \ref{lem:hk_abs}.
		
	\begin{lemma}\label{lem:Min_H}
		Suppose that $(\Psi(z,x): z\in E, x\in \X)$ is a non-negative measurable function on the product space $(E\times \X, \mathscr E\otimes \mathscr B_\X, \nu\otimes m)$ where $(E, \mathscr E, \nu)$ is a $\sigma$-finite measure space.
		Set
		\begin{align*}
			F(x):=\int_{E}\Psi(z,x)\nu(\mathrm dz),\qquad x\in\X,
		\end{align*}
		and suppose that $F\in\mathcal B_\X$, $\Psi(z,\cdot)\in\mathcal H_K$ for $\nu$-almost every $z\in E$, and
		\begin{align*}
			\int_{E}\norm{\Psi(z,\cdot)}_{\mathcal H_K}\nu(\mathrm dz)<\infty.
		\end{align*}
		Then $F\in\mathcal H_K$ and
		\begin{equation}
			\norm{\int_{E} \Psi(z, \cdot)\nu(\mathrm dz)}_{\mathcal H_K}
			\leq \int_{E} \norm{\Psi(z, \cdot)}_{\mathcal H_K} \nu(\mathrm dz).
		\end{equation}
	\end{lemma}
	
	\begin{proof}[Proof of Lemma \ref{lem:Min_H}]
		By changing $\Psi$ on a $\nu$-null subset of $E$ if necessary, we may assume that $\Psi(z,\cdot)\in\mathcal H_K$ for every $z\in E$.
		By Tonelli's theorem and \eqref{ineq:H0 Cauchy Schwarz}, we see
		\begin{align*}
			\iint_{\X^2}F(x)F(y)K(\mathrm dx,\mathrm dy)
			&=\int_{E}\int_{E}\iint_{\X^2}\Psi(z,x)\Psi(z',y)K(\mathrm dx,\mathrm dy)\nu(\mathrm dz)\nu(\mathrm dz')\\
			&\leq\int_{E}\int_{E}\norm{\Psi(z,\cdot)}_{\mathcal H_K}\norm{\Psi(z',\cdot)}_{\mathcal H_K}\nu(\mathrm dz)\nu(\mathrm dz')\\
			&=\left(\int_{E}\norm{\Psi(z,\cdot)}_{\mathcal H_K}\nu(\mathrm dz)\right)^2<\infty.
		\end{align*}
		Therefore, $F\in\mathcal H_K$, and taking square roots proves the desired inequality.
	\end{proof}

	\begin{proof}[Proof of Lemma \ref{cor:mtg+mink_cs}]
		By Lemma \ref{lem:HS} and the boundedness of $Q$ on $\tilde{\mathcal H}_{K}$, we see by the  Burkholder-Davis-Gundy inequality (cf. \cite[Theorem 4.37]{da2014stochastic}) that, uniformly in $T >0$, 
		\begin{equation}\label{ineq:bdg_gen_mtg_bdg}
			\begin{aligned}
				\norm{\sup_{t \in [0,T]}\left|\iint_0^t \Phi(s,y)W(\mathrm ds,\mathrm dy)\right|}_{L^p_\Omega} &\lesssim \pr{\int_0^T \mathbb E\left[\langle Q[\Phi(s,\cdot)]_K,[\Phi(s,\cdot)]_K\rangle_{\tilde{\mathcal H}_{K}}^\frac{p}{2}\right]^\frac{2}{p}\mathrm ds}^\frac{1}{2}\\
				&\lesssim \pr{\int_0^T \mathbb E\left[\langle \Phi(s,\cdot),\Phi(s,\cdot)\rangle_{\mathcal H_K}^\frac{p}{2}\right]^\frac{2}{p}\mathrm ds}^\frac{1}{2}.
			\end{aligned}
		\end{equation}
		For each $s \ge 0$, we see by Minkowski's and the Cauchy-Schwarz inequalities  that
		\begin{align}
			\mathbb E\left[\langle \Phi(s,\cdot),\Phi(s,\cdot)\rangle_{\mathcal H_K}^\frac{p}{2}\right]^\frac{2}{p}&=  \norm{\iint_{\X^2}\Phi(s,x) \Phi(s,y)K\pr{\mathrm dy,\mathrm dx}}_{L^\frac{p}{2}_\Omega}\\
			&\le \iint_{\X^2} \norm{\Phi(s,x)\Phi(s,y)}_{L^\frac{p}{2}_\Omega} K\pr{\mathrm dy,\mathrm dx}\\
			&= \iint_{\X^2}\mathbb E \left[\abs{\Phi(s,x)}^\frac{p}{2}\abs{\Phi(s,y)}^\frac{p}{2} \right]^\frac{2}{p}K(\mathrm dy,\mathrm dx)\\
			&\le \iint_{\X^2} \norm{\Phi(s,x)}_{L^p_\Omega} \norm{\Phi(s,y)}_{L^p_\Omega} K(\mathrm dy,\mathrm dx)\\
			&=\norm{\norm{\Phi(s,\cdot)}_{L^p_\Omega}}_{\mathcal H_K}^2.
		\end{align}
		Inserting this into \eqref{ineq:bdg_gen_mtg_bdg}, we see the following inequality holds uniformly in $T>0$,
		\begin{align}
			\norm{\sup_{t \in [0,T]}\left|\iint_0^t \Phi(s,y)W(\mathrm ds,\mathrm dy)\right|}_{L^p_\Omega}\lesssim \pr{\int_0^T \norm{\norm{\Phi(s,\cdot)}_{L^p_\Omega}}_{\mathcal H_K}^2\mathrm ds}^\frac{1}{2} = \norm{\Phi}_{L^2_{[0,T]}\mathcal H_K L^p_\Omega},
		\end{align}
		which is the desired result.
		
	\end{proof}

	\begin{proof}[Proof of Lemma \ref{lem:hk longtime}]
		Suppose that \eqref{cond:exp_hk_lamda} holds for the given $\lambda$ with $\alpha$ replaced by $\tilde\alpha$, and let $t>0$.
		Then $t = n+a$ for some $n \in \N\cup\{0\}$ and $a \in [0,1]$.
		Hence, there is $C>1$ such that for all $x \in \X$ and $a \in [0,1]$, we have
		\begin{align}
			\left(P_a^* h^\lambda\right)(x) \le  C h(x)^\lambda.
		\end{align}
		Therefore, by the semigroup property, we see 
		\begin{equation}\label{ineq:weighted semigroup bound long time temporary}
			\begin{aligned}
				\left(P_t^* h^\lambda\right)(x) &= \left(\left(P_{1}^* \right)^n P_a^* h^\lambda\right) (x)\\
				&\le C^{n+1}h(x)^\lambda\\
				&\le b e^{\kappa t}h(x)^\lambda
			\end{aligned}
		\end{equation}
		where $b = \sup_{a \in [0,1]} C^{1-a}$, and $\kappa = \ln(\max\{1,C\})>0$, which proves \eqref{ineq:hk_m long time}.
		
		By Lemma \ref{lem:H0} and the assumed bound with exponent $\tilde\alpha$, the following inequality holds uniformly in $\delta \in (0,1]$ and $y_1,y_2 \in \X$,
		\begin{align}\label{ineq:weighted cs temporary}
			\left \langle G_\delta(\cdot,y_1) h^\lambda,G_\delta(\cdot,y_2) h^\lambda \right \rangle_{ \mathcal H_K } &\le \|G_\delta(\cdot, y_1)h^\lambda\|_{\mathcal H_K}\|G_\delta(\cdot, y_2)h^\lambda\|_{\mathcal H_K}\\
			&\lesssim\delta^{-\tilde\alpha} h(y_1)^{\lambda} h(y_2)^{\lambda}.
		\end{align}
		For any $t >1$, fix $\delta \in [1/2,1]$. Then by the semigroup property and Fubini's theorem, we see that 
		\begin{align}
			\left\|G_t(\cdot,x) h(\cdot)^\lambda\right\|_{ \mathcal H_K}^2 & = \iint_{\X^2}G_t(z_1,x) h(z_1)^\lambda G_t(z_2,x) h(z_2)^\lambda K(\mathrm dz_1,\mathrm dz_2)\\
			&= \iint_{\X^2} \pr{\int_{\X} G_\delta(z_1,y_1)G_{t-\delta}(y_1,x) m(\mathrm dy_1)} h(z_1)^\lambda  \\
			&\qquad  \qquad\qquad \times \pr{\int_\X  G_{\delta}(z_2,y_2) G_{t-\delta}(y_2,x) m(\mathrm dy_2)} h(z_2)^\lambda K(\mathrm dz_1,\mathrm dz_2) \\
			&= \int_\X \int_\X G_{t - \delta}(y_1,x)G_{t-\delta}(y_2,x)\\
			&\quad \quad \quad  \times \iint_{\X^2} G_\delta(z_1,y_1) h(z_1)^\lambda G_\delta(z_2,y_2) h(z_2)^\lambda K(\mathrm dz_1,\mathrm dz_2) m(\mathrm d y_1) m(\mathrm dy_2)\\
			&=\int_\X \int_\X G_{t-\delta}(y_1,x) G_{t-\delta}(y_2,x)\left \langle G_\delta(\cdot,y_1)h^\lambda, G_{\delta}(\cdot,y_2) h^\lambda\right\rangle _{\mathcal H_K} m(\mathrm dy_1) m(\mathrm dy_2)\\
			&\overset{\eqref{ineq:weighted cs temporary}}{\lesssim} \delta^{-\tilde\alpha}\pr{P^*_{t-\delta}\pr{h^\lambda}(x)}^2\\
			&\overset{\eqref{ineq:weighted semigroup bound long time temporary}}{\lesssim} \delta^{-\tilde\alpha}b^2e^{-2\kappa \delta}e^{2\kappa t}h(x)^{2\lambda}\\
			&\lesssim t^{-\tilde\alpha}e^{\tilde\kappa t}h(x)^{2\lambda},
		\end{align}
		where $\tilde\kappa:=2\kappa+1$ and we used $\delta\ge1/2$ and $t^{\tilde\alpha}\le e^t$ for $t>1$.
		Together with the assumed bound for $t\in(0,1]$, this proves \eqref{ineq:hk_k long time}.
		
		To show \eqref{ineq:time_holder_long_time}, we first note that for each $T>0$, it holds uniformly for $T\ge t' \ge t>0$ with $t'-t \le \frac{1}{2}$ and $x \in \X$ that 
		\begin{align}\label{ineq:time_holder_proof_half}
			\norm{G_t(\cdot,x) - G_{t'}(\cdot,x)}_{L^1_\X} \lesssim \frac{\abs{t-t'}^{\xi_1}}{t^\beta}.
		\end{align}
		To see this, consider the case where $0 < t \le \frac{1}{2}$, then $t' \le 1$ and the estimate is exactly \eqref{cond:time_holder}. Now suppose $t > \frac{1}{2}$ then by the semi-group property, we see uniformly in $T \ge t' >t \ge \frac{1}{2}$ and $x \in \X$, 
		\begin{align}
			&\norm{\int_\X G_{t - \frac{1}{2}}(\cdot, y)\pr{G_{\frac{1}{2}}(y,x) - G_{t'-t+ \frac{1}{2}}(y,x)  } m(\mathrm dy)    }_{L^1_\X}\\
			\le & \int_\X \norm{G_{t - \frac{1}{2}}(\cdot,y)}_{L^1_\X} \abs{G_{\frac{1}{2}}(y,x) - G_{t'-t+ \frac{1}{2}}(y,x) } m(\mathrm dy) \\
			\overset{\eqref{ineq:hk_m long time} }\lesssim & \int_\X \abs{G_{\frac{1}{2}}(y,x) - G_{t'-t+ \frac{1}{2}}(y,x) } m(\mathrm dy)\\
			\overset{\eqref{cond:time_holder} }\lesssim& \frac{\abs{t-t'}^{\xi_1}}{(1/2)^\beta } \lesssim \frac{\abs{t-t'}^{\xi_1}}{t^\beta },
		\end{align}
		where we used the fact that $t'-t+\frac{1}{2} \in (0,1]$.
		This proves \eqref{ineq:time_holder_proof_half}.
		Suppose $0< t< t' \le T$, and $\abs{t-t'} \ge \frac{1}{2}$.
		Let $n_1  \le n_N\in \N$ be
		\begin{align}
			n_1 := \min\left\{n \in \N: t \le \frac{n}{2}\right\},\qquad n_N:= \max\left\{n \in \N: \frac{n}{2} \le t' \right\},
		\end{align} 		
		Then by \eqref{ineq:time_holder_proof_half}, it holds uniformly in $x \in \X$ that 
		\begin{equation}\label{ineq:time_holder_end_case}
			\begin{aligned}
				\norm{G_t(\cdot,x) - G_{\frac{n_1}{2}}(\cdot,x)}_{L^1_\X} &\lesssim \frac{\abs{t - \frac{n_1}{2}}^{\xi_1\wedge 1}}{t^\beta}\\
				\norm{G_{t'}(\cdot,x) - G_{\frac{n_N}{2}}(\cdot,x)}_{L^1_\X} &\lesssim \frac{\abs{t' - \frac{n_N}{2}}^{\xi_1\wedge 1}}{(n_N/2)^\beta}.
			\end{aligned}
		\end{equation}
		By the triangle inequality,
		\begin{align}
			\norm{G_t(\cdot,x) - G_{t'}(\cdot,x)}_{L^1_\X}
			&\le \norm{G_t(\cdot,x) - G_{\frac{n_1}{2}}(\cdot,x)}_{L^1_\X}
			+\norm{G_{t'}(\cdot,x) - G_{\frac{n_N}{2}}(\cdot,x)}_{L^1_\X} \\
			&\qquad + \sum_{j=n_1+1}^{n_N}\norm{G_{\frac{j}{2}}(\cdot,x)-G_{\frac{j-1}{2}}(\cdot,x)}_{L^1_\X} \\
			&\overset{\eqref{ineq:time_holder_end_case},\,\eqref{ineq:time_holder_proof_half}}{\lesssim}
			\frac{\abs{t-\frac{n_1}{2}}^{\xi_1\wedge1}}{t^\beta}
			+\frac{\abs{t'-\frac{n_N}{2}}^{\xi_1\wedge1}}{(n_N/2)^\beta}
			+\sum_{j=n_1+1}^{n_N}\frac{\pr{\frac12}^{\xi_1}}{\pr{\frac{j-1}{2}}^\beta} \\
			&\lesssim \frac{\abs{t-\frac{n_1}{2}}^{\xi_1\wedge1}}{t^\beta}
			+\frac{\abs{t'-\frac{n_N}{2}}^{\xi_1\wedge1}}{t^\beta}
			+\frac{n_N-n_1}{t^\beta} \\
			&\le t^{-\beta}\pr{2\abs{t-t'}^{\xi_1\wedge1}+2|t-t'|} \\
			&\lesssim \frac{\abs{t-t'}^{\xi_1\wedge1}}{t^\beta}.
		\end{align}
		Here we used that $(j-1)/2\ge n_1/2\ge t$ in the third inequality. 
		This proves the desired result.
		To prove \eqref{ineq:spatial_holder_long_time}, fix $T>0$. For $t\in(0,1]$, the estimate is exactly \eqref{cond:time_holder}. If $1<t\le T$, then the semigroup property, \eqref{ineq:hk_m long time} with $\lambda=0$ and \eqref{cond:time_holder} at  $t=1$ give
		\begin{align}
			\norm{G_t(\cdot,x)-G_t(\cdot,x')}_{L^1_\X}
			&\le \int_\X \norm{G_{t-1}(\cdot,y)}_{L^1_\X}
			\abs{G_1(y,x)-G_1(y,x')}m(\mathrm dy)\\
			&\lesssim d(x,x')^{\xi_2}\\
			&\lesssim \frac{d(x,x')^{\xi_2}}{t^\beta},
		\end{align}
		uniformly in $x,x' \in \X$ with $d(x,x') \le 1$, and $t \in [1,T]$. This proves \eqref{ineq:spatial_holder_long_time}.
	\end{proof}
	
	\begin{proof}[Proof of Lemma \ref{lem:hk_abs}]
		We first show \eqref{ineq:hk abs holder m} holds for $\lambda =0$.  Let $x,x' \in \X$  with $d(x,x') \le 1$, consider first for $t > 1$ that
		\begin{align}
			\int_0^{t} \int_\X \left|G_{t - s}(y,x) - G_{t - s}(y,x') \right|m(\mathrm dy)\mathrm ds & = I_1+I_2,
		\end{align}
		where 
		\begin{align}
			I_1 &{=}\int_{t-1}^t\int_\X\abs{G_{t - s}(y,x) - G_{t-s}(y,x')}m(\mathrm dy)\mathrm ds  \\
			&{=}\int_0^1 \int_\X \left|G_{1-s}(y,x) - G_{1-s}(y,x') \right|m(\mathrm dy) \mathrm ds \\
			&{\overset{\eqref{cond:time_holder}}{\lesssim}} d(x,x')^{\xi_2},
		\end{align}
		uniformly in $x,x' \in \X$ with $d(x,x') \le 1$ and $t>1$, and 
		\begin{align}
			I_2 &= \int_0^{t - 1}\int_\X\left| G_{t - s}(y,x) - G_{t - s}(y,x')\right| m(\mathrm dy) \mathrm ds.
		\end{align}
		By the semigroup property and Tonelli's theorem, it holds uniformly in $x,x' \in \X$ with $d(x,x') \le 1$ and $t>1$ that
		\begin{align}
			I_2 
			&\,\,\,= \int_0^{t-1} \int_\X \left|\int_\X G_{t - s -1} (y,z) \left(G_{1}(z,x) - G_1(z,x') \right) m(\mathrm dz) \right| m(\mathrm dy) \mathrm ds\\
			&\,\,\,\le \int_0^{t-1}\int_\X \left(\int_\X G_{t-s-1}(y,z)m(\mathrm dy)\right)
			\left|G_1(z,x)-G_1(z,x')\right|m(\mathrm dz)\mathrm ds\\
			&\overset{\eqref{ineq:hk_m long time}}{\lesssim} \int_0^{t-1} e^{\kappa(t-s-1)}\mathrm ds
			\int_\X \left|G_1(z,x)-G_1(z,x')\right|m(\mathrm dz)\\
			&\overset{\eqref{cond:time_holder}}{\lesssim} d(x,x')^{\xi_2}\int_0^{t-1}e^{\kappa(t-s-1)}\mathrm ds,
		\end{align}
		where \eqref{ineq:hk_m long time} was used with $\lambda=0$.
		Note that the time integral is increasing in $t$ and bounded for each $t \ge 1$.
		For $t\in(0,1]$, the same spatial bound follows by integrating \eqref{cond:time_holder}, since $\beta<1$.
		Hence we see for each $T> 0$, it holds uniformly for $d(x,x') \le 1$ and $t \in [0,T]$ that
		\begin{align}\label{ineq:hk spatial temporary}
			\int_0^t \int_\X \left|G_{t - s}(y,x) - G_{t - s}(y,x') \right| m(\mathrm dy) \mathrm ds \lesssim d(x,x')^{\xi_2}.
		\end{align}
		
		Fix $T>0$ and consider $\delta\in(0,1]$ and $t\in[0,T]$. Then
		\begin{align}
			\int_0^{t+\delta} \int_\X \left| G_{t+\delta - s}(y,x) - G_{t-s}(y,x) \right| m(\mathrm dy) \mathrm ds = I_1+ I_2,
		\end{align}
		where 
		\begin{align}
			I_1 = \int_{t}^{t+\delta} \int_\X G_{t+\delta-s}(y,x) m(\mathrm dy) \mathrm ds\overset{\eqref{cond:exp_hk_lamda}}{\lesssim} \delta,
		\end{align}
		uniformly in $(t,x) \in [0,T] \times \X$ and $\delta \in (0,1]$, and 
		\begin{align}
			I_2 = \int_0^t \int_\X \left| G_{t+\delta-s}(y,x) - G_{t-s} (y,x)\right| m(\mathrm dy) \mathrm ds.
		\end{align}
		
		By the change of variables  and \eqref{ineq:time_holder_long_time}, applied with time horizon $T+1$, we see it holds uniformly in $(t,x) \in [0,T] \times \X$ and $\delta \in (0,1]$ that
		\begin{align}
			I_2
			=\int_0^t\norm{G_{r+\delta}(\cdot,x)-G_r(\cdot,x)}_{L^1_\X}\mathrm dr
			\lesssim \delta^{\xi_1\wedge1}\int_0^t r^{-\beta}\mathrm dr
			\lesssim \delta^{\xi_1\wedge1},
		\end{align}
		where the endpoint $r=0$ is immaterial and the last inequality follows from $\beta\in[0,1)$. Since $\delta\in(0,1]$, we also have $\delta\le\delta^{\xi_1\wedge1}$.
		
		Therefore, for each $T>0$,  it holds uniformly for $(t,x) \in [0,T] \times \X$ and $\delta \in (0,1]$  that
		\begin{align}\label{ineq:hk temporal temporary}
			\int_0^{t+\delta} \int_\X \left|G_{t +\delta- s}(y,x) - G_{t-s}(y,x) \right| m(\mathrm dy) \mathrm ds \lesssim \delta^{\xi_1\wedge 1}.
		\end{align}
		Therefore, we see by \eqref{ineq:hk temporal temporary}, \eqref{ineq:hk spatial temporary} and the triangular inequality that, for each $T>0$, it holds uniformly for $x,x' \in \X$ with $d(x,x') \le 1$,  $t \in [0,T]$ and $\delta \in (0,1]$ that
		\begin{align}
			\int_0^{t+\delta} \int_\X \left|G_{t+\delta-s}(y,x') - G_{t-s}(y,x) \right|m(\mathrm dy) \mathrm ds \lesssim \delta^{\xi_1\wedge 1} + d(x,x')^{\xi_2}.
		\end{align}
		For $t,t'\in[0,T]$ with $|t-t'|>1$, \eqref{eq:GC}, the triangle inequality and \eqref{ineq:hk_m long time} with $\lambda=0$ give
		\begin{align*}
			\int_0^T\norm{G_{t-s}(\cdot,x)-G_{t'-s}(\cdot,x')}_{L^1_\X}\,\mathrm ds
			&\le \int_0^t\norm{G_{t-s}(\cdot,x)}_{L^1_\X}\,\mathrm ds
			+\int_0^{t'}\norm{G_{t'-s}(\cdot,x')}_{L^1_\X}\,\mathrm ds\\
			&\lesssim \int_0^t e^{\kappa(t-s)}\,\mathrm ds
			+\int_0^{t'}e^{\kappa(t'-s)}\,\mathrm ds\\
			&\le 2T e^{\kappa T}
			\lesssim |t-t'|^{\xi_1\wedge1}.
		\end{align*}
		The last inequality uses $|t-t'|^{\xi_1\wedge1}\ge1$.
		Thus, by the triangle inequality, uniformly for $t,t'\in[0,T]$ and $x,x'\in\X$ with $d(x,x')\le1$,
		\begin{align}\label{ineq:hk abs holder m lambda = 0}
			\norm{G(\cdot,\cdot;t,x) - G(\cdot,\cdot; t',x')}_{L^1_{[0,T)} L^1_\X} \lesssim  \abs{t - t'}^{\xi_1\wedge 1} + d(x,x')^{\xi_2},
		\end{align}
		which implies \eqref{ineq:hk abs holder m} when $\lambda=0$. To show \eqref{ineq:hk abs holder m} for $\lambda \in [0,1/2]$, we see by H\"{o}lder's inequality that
		\begin{align}
			\norm{h(\cdot)^\lambda \pr{G(\cdot,\cdot;t,x) - G(\cdot,\cdot; t',x')}}_{L^1_{[0,T)} L^1_\X} &\lesssim \norm{ \abs{G(\cdot,\cdot;t,x) - G(\cdot,\cdot;t',x')}^{1/2} h(\cdot)^\lambda }_{L^2_{[0,T]\times\X}}
			\\& \qquad \times \left\| \left|G(\cdot,\cdot;t,x) - G(\cdot,\cdot;t',x')\right|^{1/2} \right\|_{L^2_{[0,T]\times\X}}. \label{eq:Phi1U}
		\end{align}
		
		Noticing $2\lambda \in [0,1]$, uniformly for every $t,t'\in [0,T]$ and $x,x'\in \X$ with $d(x, x')\leq 1$,
		\begin{align}
			& \norm{\left|G(\cdot,\cdot;t,x) - G(\cdot,\cdot;t',x')\right|^{1/2}h(\cdot)^\lambda }^2_{L^2_{[0,T]\times\X}}
			\\& = \int_0^T \mathrm ds \int_\X \abs{G(s,y;t,x) - G(s,y;t',x')} h(y)^{2\lambda} m(\mathrm dy)
			\\&\leq \int_0^T \left(\left\| G_{t-s}(\cdot,x) h(\cdot)^{2\lambda} \right\|_{L^1_\X} + \left\| G_{t'-s}(\cdot,x') h(\cdot)^{2\lambda} \right\|_{L^1_\X}\right) \mathrm ds 
			\\&\label{eq:Phi1Ua}
			\overset{\text{Lemma \ref{lem:hk longtime}}} \lesssim 
			h(x)^{2\lambda} + h(x')^{2\lambda}
			\overset{\eqref{ineq:ref_regular}} \lesssim 
			h(x)^{2\lambda}. 
		\end{align}
		On the other hand, uniformly for every $t,t'$ in $[0,T]$ and $x,x'\in \X$ with $d(x,x')\leq 1$,
		\begin{align}
			&\left\| \left|G(\cdot,\cdot;t,x) - G(\cdot,\cdot;t',x')\right|^{1/2} \right\|_{L^2_{[0,T]\times\X}} = \left\| G(\cdot,\cdot;t,x) - G(\cdot,\cdot;t',x') \right\|^{1/2}_{L^1_{[0,T]\times\X}}
			\\
			& \overset{\eqref{ineq:hk abs holder m lambda = 0}} \lesssim 
			\left(|t-t'|^{\xi_1 \wedge 1} + d(x,x')^{\xi_2}\right)^{1/2} \leq |t-t'|^{(\xi_1\wedge 1)/2}+ d(x,x')^{\xi_2/2}
			\\
			\label{eq:Phi1Ub} 
			&\lesssim |t-t'|^{(\xi_1\wedge 1)\over 2}+ d(x,x')^{{\xi_2 \over 2}\wedge 1}.
		\end{align}
		Putting \eqref{eq:Phi1Ua} and \eqref{eq:Phi1Ub} back to \eqref{eq:Phi1U}, we obtain that uniformly for every $t,t'\in [0,T]$ and $x,x'\in \X$ with $d(x, x')\leq 1$,
		\begin{align}
			\norm{h(\cdot)^\lambda \pr{G(\cdot,\cdot;t,x) - G(\cdot,\cdot; t',x')}}_{L^1_{[0,T)} L^1_\X}
			\lesssim \left(|t-t'|^{{{\xi_1 \wedge 1}\over 2}} + d(x,x')^{{\xi_2 \over 2}\wedge 1}\right) h(x)^{\lambda},
		\end{align}
		which is exactly \eqref{ineq:hk abs holder m}.
		
		To prove \eqref{ineq:hk abs holder k}, recall that $\alpha\in(0,1)$ in \eqref{cond:exp_hk_lamda}, and that $\beta\in[0,1)$ and $\xi_1,\xi_2>0$ are given in \eqref{cond:time_holder}. Take $q\ge2$ large enough so that $0<\bar\alpha:=\alpha/2+\beta/q<1/2$ and $\xi_2/q<1$.
		Take $q' \geq 1$ so that $\frac{1}{q} + \frac{1}{q'} = 1$.
		Define $\tilde \xi_1:=\min(\frac{\xi_1\wedge 1}{q}, \frac{1-\alpha}{2}) \in (0,1)$ and $\tilde \xi_2 = \frac{\xi_2}{q} \in (0,1)$. Let $\lambda_1\in[0,\frac{1}{2}]\subset[0,1/q']$.
		Pay attention that \eqref{cond:exp_hk_lamda} holds with $\lambda$ being replaced by either $\lambda_1 q'$ or $\lambda_1$.
		
		Hence, it holds uniformly for  every $0< t \leq t'$ in $[0,T]$, and $x,x' \in \X$ with $d(x,x') \le 1$ that,
		\begin{align}
			&\qquad
			\norm{\abs{G_t(\cdot,x) - G_{t'}(\cdot,x')} 
				h(\cdot)^{\lambda_1}}_{\mathcal H_K}
			\\ &\overset{\eqref{eq:G}}=\norm{\abs{\int_\X G_{\frac{t}{2}}(\cdot,y)\pr{G_{\frac{t}{2}}(y, x)- G_{t'-\frac{t}{2}}(y,x')} m(\mathrm dy) }
				h(\cdot)^{\lambda_1}
			}_{\mathcal H_K}
			\\  &\leq 
			\norm{\pr{\int_\X G_{\frac{t}{2}}(\cdot,y)\abs{G_{\frac{t}{2}}(y, x)- G_{t'-\frac{t}{2}}(y,x')} m(\mathrm dy)}h(\cdot)^{\lambda_1}}_{\mathcal H_K} 
			\\ &\overset{\text{Lemma \ref{lem:Min_H}}}\le \int_\X \norm{G_{\frac{t}{2}}(\cdot, y) 
				h(\cdot)^{\lambda_1}
			}_{\mathcal H_K} \abs{G_{\frac{t}{2}}(y, x)- G_{t'-\frac{t}{2}}(y,x')} m(\mathrm dy)
			\\ & \overset{\text{Lemma \ref{lem:hk longtime}}}{\lesssim } t^{-\alpha/2}\int_\X h(y)^{\lambda_1} \abs{G_{\frac{t}{2}}( y, x)- G_{t'-\frac{t}{2}}(y,x')} m(\mathrm dy). \label{ineq:continuity stochastic part middle step 1} 
		\end{align}
		
		Note uniformly in  $t, t'$ in $[0,T]$ and $x,x'\in \X$
		with $d(x,x') \le 1$ that
		\begin{align}
			&\int_\X h(y)^{\lambda_1}  \left|G_t(y,x) -G_{t'}(y,x') \right| m(\mathrm dy)
			\\&=\int_\X h(y)^{\lambda_1}  \left|G_t(y,x) -G_{t'}(y,x') \right|^\frac{1}{q'} \left|G_t(y,x) -G_{t'}(y,x') \right|^\frac{1}{q} m(\mathrm dy)
			\\&\overset{\text{H\"older}}{\leq} \norm{h(\cdot)^{\lambda_1} \abs{G_t(\cdot,x)- G_{t'}(\cdot,x')}^{\frac{1}{q'}}}_{L^{q'}_{\X}} \norm{\abs{G_t(\cdot,x)- G_{t'}(\cdot,x')}^{\frac{1}{q}}}_{L^{q}_{\X}}
			\\&=\norm{h(\cdot)^{\lambda_1 q'} \abs{G_t(\cdot,x)- G_{t'}(\cdot,x')}}^{1\over q'}_{L^1_{\X}} \norm{G_t(\cdot,x)- G_{t'}(\cdot,x')}^{1\over q}_{L^1_{\X}}
			\\&\le \norm{h(\cdot)^{\lambda_1 q'} \pr{ G_t(\cdot,x)+G_{t'}(\cdot,x')}}^{1\over q'}_{L^1_\X} \norm{G_t(\cdot,x) - G_{t'}(\cdot,x')}^{1\over q}_{L^1_\X}
			\\&\overset{\text{Minkowski}}\le \pr{\norm{h(\cdot)^{\lambda_1 q'}  G_t(\cdot,x)}_{L^1_\X} + \norm{h(\cdot)^{\lambda_1 q'} G_{t'}(\cdot,x')}_{L^1_\X}}^{1\over q'} \norm{G_t(\cdot,x) - G_{t'}(\cdot,x')}^{1\over q}_{L^1_\X}
			\\&\overset{\text{Lemma \ref{lem:hk longtime}}}\lesssim  \pr{h(x)^{\lambda_1 q'} +h(x')^{\lambda_1 q'}}^{\frac{1}{q'}} \norm{G_t(\cdot,x) - G_{t'}(\cdot,x')}^{1\over q}_{L^1_\X}
			\\&\overset{\eqref{ineq:ref_regular}}\lesssim  h(x)^{\lambda_1} \norm{G_t(\cdot,x) - G_{t'}(\cdot,x')}^{1\over q}_{L^1_\X}. \label{eq:Gtt}
		\end{align}
		Therefore, by the triangular inequality, we have uniformly for every 
		$0< t \leq t'$ in $[0,T]$, and $x,x' \in \X$ with $d(x,x') \le 1$ that, 

		\begin{align}
			&\norm{\abs{G_t(\cdot,x) - G_{t'}(\cdot,x')} 
				h(\cdot)^{\lambda_1}
			}_{\mathcal H_K}
			\\ & {\overset{\eqref{ineq:continuity stochastic part middle step 1}}{\lesssim}} t^{-\alpha/2}\int_\X h(y)^{\lambda_1} \left|G_{\frac{t}{2}}( y, x)- G_{t'-\frac{t}{2}}(y,x') \right|m(\mathrm dy)
			\\ & {\overset{\eqref{eq:Gtt}}{\lesssim} }t^{-\alpha/2} h(x)^{\lambda_1} \norm{ 
				G_{\frac{t}{2}}(\cdot,x) - 
				G_{t'-\frac{t}{2}}(\cdot,x')
			}^{1\over q}_{L^1_\X} 
			\\& {\overset{\text{
					Minkowski
			}}{\leq}} t^{-\alpha/2} h(x)^{\lambda_1} \pr{\norm{ G_{\frac{t}{2}}(\cdot,x) - G_{\frac{t}{2}}(\cdot,x')}_{L^1_\X} + \norm{G_{\frac{t}{2}}(\cdot,x') - G_{t'-\frac{t}{2}}(\cdot,x') }_{L^1_\X} }^{1\over q}
			\\& {\overset{
			\eqref{ineq:spatial_holder_long_time},\ \eqref{ineq:time_holder_long_time}}{\lesssim}}t^{-\alpha/2} h(x)^{\lambda_1} \pr{\frac{d(x,x')^{\xi_2}}{\pr{t/2}^\beta} + \frac{\abs{t'-t}^{\xi_1 \wedge 1}}{\pr{t/2}^\beta}}^{\frac{1}{q}}
			\\& \lesssim t^{-\bar\alpha} h(x)^{\lambda_1} \pr{d(x,x')^{\xi_2} + \abs{t'-t}^{\xi_1 \wedge 1}}^{\frac{1}{q}}. \label{eq:G-GS}
		\end{align}
		On one hand, uniformly for $0<t\leq t'$ in $[0,T]$, and $x,x'\in\X$ with $d(x,x')\leq1$, since $2\bar\alpha\in(0,1)$,
		\begin{align}
			&\int_0^t \norm{\abs{G_{t-s}(\cdot,x) - G_{t'-s}(\cdot,x')} 
				h(\cdot)^{\lambda_1}
			}_{\mathcal H_K}^2\mathrm ds 
			\\& \overset{\eqref{eq:G-GS}}\lesssim \int_0^t (t-s)^{-2\bar\alpha}h(x)^{2\lambda_1}\pr{d(x,x')^{\xi_2}+\abs{t'-t}^{\xi_1\wedge1}}^{\frac{2}{q}}\,\mathrm ds
			\\& \lesssim t^{1-2\bar\alpha}h(x)^{2\lambda_1}\pr{d(x,x')^{\xi_2}+\abs{t'-t}^{\xi_1\wedge1}}^{\frac{2}{q}}. \label{eq:Gtq}
		\end{align}
		On the other hand, uniformly in $0<t\leq t'$ in $[0,T],$ and $x,x' \in \X$ with $d(x,x')\leq 1$, 
		\begin{align}
			\int_{t}^{t'} \norm{G_{t'-s}(\cdot,x') h(\cdot)^{\lambda_1}}_{\mathcal H_K}^2\mathrm ds 
			& \overset{\text{Lemma \ref{lem:hk longtime}}}\lesssim \int_{t}^{t'} (t'-s)^{-\alpha} h(x')^{2\lambda_1} \mathrm ds \\
			&= h(x')^{2\lambda_1} \int_0^{t'-t} r^{-\alpha} \mathrm dr \\
			&\lesssim h(x')^{2\lambda_1} |t'-t|^{1-\alpha}\\
			&\overset{\eqref{ineq:ref_regular}}{\lesssim} h(x)^{2\lambda_1} |t'-t|^{1-\alpha}. \label{eq:Gt'q}
		\end{align}
		Now, we see uniformly in $0 <t\leq t'$ in $[0,T]$, and $x,x' \in \X$ with $d(x,x') \le 1$ that   
		\begin{align}
			&\int_0^T \norm{\abs{G_{t-s}(\cdot,x) - G_{t'-s}(\cdot,x')} h(\cdot)^{\lambda_1}}_{\mathcal H_K}^2\mathrm ds
			\\&\overset{\eqref{eq:GC}}=\int_0^t \norm{\abs{G_{t-s}(\cdot,x) - G_{t'-s}(\cdot,x')} h(\cdot)^{\lambda_1}}_{\mathcal H_K}^2\mathrm ds +  \int_{t}^{t'} \norm{G_{t'-s}(\cdot,x')h(\cdot)^{\lambda_1}}_{\mathcal H_K}^2\mathrm ds
			\\&\overset{\eqref{eq:Gtq} \& \eqref{eq:Gt'q}}\lesssim h(x)^{2\lambda_1} \pr{d(x,x')^{\xi_2} + \abs{t'-t}^{\xi_1\wedge1}}^{\frac{2}{q}} + h(x)^{2\lambda_1} |t'-t|^{1-\alpha}
			\\&\lesssim h(x)^{2\lambda_1} \pr{d(x,x')^{\frac{2\xi_2}{q}}+ |t'-t|^{\frac{2\pr{\xi_1\wedge1}}{q}} + |t'-t|^{1-\alpha}}
			\\& \lesssim h(x)^{2\lambda_1} \pr{d(x,x')^{2\tilde \xi_2}+ |t'-t|^{2\tilde \xi_1}}.
		\end{align}
		This proves \eqref{ineq:hk abs holder k}.

		It remains to prove \eqref{ineq:hk abs holder k weighted}. Let us fix $T>0$. Since $q\ge2$, its conjugate exponent satisfies $q'\le2$, and hence $\lambda q'\le1$ for every $\lambda\in[0,1/2]$. Thus, by \eqref{eq:Gtq} and \eqref{eq:Gt'q}, with $\lambda_1$ replaced by $\lambda$, it holds uniformly in $0<t\leq t'$ in $[0,T]$, and $x,x' \in \X$ with $d(x,x') \le 1$ that
		\begin{align*}
			&\int_0^{t'}\left\|\left|G_{t-s}(\cdot,x)-G_{t'-s}(\cdot,x')\right|h(\cdot)^\lambda\right\|_{\mathcal H_K}^2\,\mathrm ds\\
			\lesssim &h(x)^{2\lambda}\left[t^{1-2\bar\alpha}\left(d(x,x')^{2\xi_2/q}+|t-t'|^{2(\xi_1\wedge1)/q}\right)+|t-t'|^{1-\alpha}\right].
		\end{align*}
		Set
		\[
			\alpha_0:=\frac{1+2\bar\alpha}{2},\qquad
			\xi:=1\wedge\frac{2\xi_2}{q}\wedge\frac{2(\xi_1\wedge1)}{q}\wedge(\alpha_0-\alpha).
		\]
		Since $t\le t'$ and $|t-t'|\le t'$, this proves \eqref{ineq:hk abs holder k weighted}; the same estimate also covers $t=0$.
	\end{proof}
	\noindent  The following corollary follows from the proof of Lemma \ref{lem:hk_abs} (inequality \eqref{eq:G-GS}).
	\begin{corollary}\label{cor:hk_abs}
		For each $T>0$, there
		exist $\xi\in (0,1]$ and $\alpha_0\in(0,1)$ so that the following inequality holds uniformly in $t,t' \in (0,T]$ and $x,y \in \X$ with $d(x,y) \le 1$,
		\begin{align*}
			\norm{
			\abs{G_t(\cdot,x) - G_{t'}(\cdot,y)}
			}_{\mathcal H_K}^2 \lesssim \frac{\abs{t-t'}^\xi + d(x,y)^\xi}{\pr{t\wedge t'}^{
			\alpha_0}},
		\end{align*}
	\end{corollary}

	\subsubsection{Proof of Lemma \ref{lem:tight_tem_moment}  }
	
	We will now work towards the proof of Lemma \ref{lem:tight_tem_moment}. 
    	
	\begin{proof}[Proof of Lemma \ref{lem:tight_tem_moment}]
		Let $T\geq1$ and let $j\leq T$ be a non-negative integer. Then, for any $\lambda>0$ and sufficiently large $p\geq1$, the following inequality holds uniformly in $x\in\X$:
		\begin{align}
			\E\left[ |X(j,x)|^p\right] &\lesssim \E\left[|X(0,x)|^p \right] + \sum_{k = 1}^{j} \E\left[\left|X(k - 1,x)- X(k,x) \right|^p \right]\\
			&\lesssim h(x)^{\lambda}.
		\end{align}
		Let $d_\mathcal S$ be the metric on $\R_+\times\X$ defined by
		\begin{align}
			d_\mathcal S((t,y),(t',y')):=|t-t'|^{\tilde \xi_1}+d(y,y')^{\tilde \xi_2},
		\end{align}
		and, for each non-negative integer $j$ and $x\in\X$, set $\mathcal S_{j,x}:=[j,j+1]\times B_1(x)$. Let $b>0$ be as in \eqref{ineq:entropy bound} for $R=1$ and set $\kappa:=1/\tilde \xi_1+b/\tilde \xi_2$. By Lemma \ref{lem:product metric entropy}, each $\mathcal S_{j,x}$ is totally bounded under $d_\mathcal S$ and
		\begin{align}
			D(\mathcal S_{j,x},d_\mathcal S;\varepsilon)\lesssim\varepsilon^{-\kappa}
		\end{align}
		uniformly in $j$, $x$ and $\varepsilon\in(0,1]$. Hence, using \eqref{cond:tight_moment} with time horizon $T+1$, we may apply \cite[Theorem~1.1]{kratschmer2022kolmogorov} to $\mathcal S_{j,x}$ for each $j$ and $x$.
		Patching the resulting modifications over the unit time intervals and the countable spatial cover in \eqref{ineq:faster_than_ms}, we obtain a single locally H\"older-continuous modification, which we still denote by $X$, such that, for any $\lambda>0$, sufficiently large $p\geq1$ and $\varrho\in(0,1-\kappa/p)$,
		\begin{align}\label{ineq:holder norm bound temp}
			\E\left[\sup_{(t,y),(t',y')\in\mathcal S_{j,x}\atop(t,y)\neq(t',y')}\frac{\abs{X(t,y)-X(t',y')}^p}{\pr{\abs{t-t'}^{\tilde \xi_1}+d(y,y')^{\tilde \xi_2}}^{\varrho p}}\right]\lesssim h(x)^\lambda
		\end{align}
		uniformly in $x\in\X$ and non-negative integers $j\leq T$.
		From now on, we work with this modification. Therefore, uniformly in $x\in\X$ and non-negative integers $j\leq T$,
		\begin{align}
			\E\left[ \sup_{y \in B_1(x)\atop s \in [j,j+1]} \left| X(s,y)\right|^p \right] & \lesssim  \E\left[\left|X(j,x) \right|^p \right] + \E \left[\sup_{y \in B_1(x)\atop s \in [j,j+1]} \left| X(s,y) - X(j,x) \right|^p \right]\nonumber\\
			&\lesssim h(x)^\lambda\label{ineq:tight_x0_bound}.
		\end{align}
		Let $\{x_i\}_{i \in \mathcal I}\subset \X$ and $\lambda_0 >0$ be as in condition \eqref{ineq:faster_than_ms}. By Markov's inequality we see for each $\lambda >0$, it holds uniformly in $A>0$ that
		\begin{align}
			\bbP \left(\sup_{s \le T} \|X(s)\|_{(\lambda)} > A \right)&\le   \sum_{0 \le j \le T}\sum_{i\in\mathcal I} \bbP \left( \sup_{t\in [j,j+1]\atop x \in B_1(x_i)}  \left|X(t,x) \right|^p h(x)^{-\lambda p} \ge  A^p \right)\\
			&\le  \sum_{0 \le j \le T}\sum_{i\in\mathcal I} \bbP \left( \sup_{t \in [j,j+1]\atop x \in B_1(x_i)}  \left|X(t,x) \right|^p  \ge  A^p \inf_{x \in B_1(x_i)}h(x)^{\lambda p}\right)\\
			&\overset{\eqref{ineq:tight_x0_bound}}{\lesssim} \sum_{i\in\mathcal I}\frac{1}{A^p}h(x_i)^{\lambda}  \sup_{x \in B_1(x_i)}h(x)^{-\lambda p} \\
			&\overset{\eqref{ineq:ref_regular}}{\lesssim} \frac{1}{A^p}\sum_{i\in\mathcal I} h(x_i)^{-\lambda p + \lambda}.
		\end{align}
		By \eqref{ineq:faster_than_ms}, we may choose $p$ to be large enough so that $-\lambda p+\lambda <- \lambda_0$, which makes the last sum finite. Therefore, we see by the continuity of the probability measure that
		\begin{align*}
			\mathbb P \pr{\sup_{t \in [0,T]}\norm{X(t)}_{(\lambda)} < \infty} = 1.
		\end{align*}
		Fix $T\geq1$, $\lambda>0$ and $x_0\in\X$, and work on a probability-one event on which
		\begin{align*}
			M:=\sup_{s\in[0,T]}\norm{X(s)}_{(\lambda/2)}<\infty.
		\end{align*}
		For $t,t'\in[0,T]$ and $R\geq1$,
		\begin{align*}
			\norm{X(t)-X(t')}_{(\lambda)}
			&\leq\sup_{x\in\overline{B_R(x_0)}}\abs{X(t,x)-X(t',x)}
			+2M\sup_{x\in B_R(x_0)^c}h(x)^{-\lambda/2}.
		\end{align*}
		The first term converges to zero as $t' \to t$ by joint continuity and compactness of $\overline{B_R(x_0)}$. If $\X$ has finite diameter, the second term vanishes for all sufficiently large $R$. Otherwise, we see by Lemma \ref{lem:ref_grow} that the second term can be made arbitrarily small by choosing $R>0$ large enough. 
		Therefore, $t \mapsto X(t,\cdot)$ is continuous under $\norm{\cdot}_{(\lambda)}$ on $[0,T]$. Finally, we may take a countable intersection over $T\in\N$ and $\lambda=n^{-1}$ for $n\in\N$ to conclude the desired result. 
	\end{proof}

	\subsection{Proof of Lemmas \ref{lem:PsiM}, \ref{lem:CMS} \texorpdfstring{and}{and} \ref{lem:b_contract_combined}}
	
	\begin{proof}[Proof of Lemma \ref{lem:PsiM}]
			Suppose that both $(B,\Sigma)$ and $(B',\Sigma')$ are Duhamel-integrable pairs and that $B'$ is a predictable modification of $B$.
			Then, for each $y\in\X$, $B(s,y,\omega)=B'(s,y,\omega)$ for $\Lambda\otimes\mathbb P$-almost every $(s,\omega)\in[0,T]\times\Omega$.
			Therefore, for each $(t,x)\in[0,T]\times\X$, we see by Tonelli's theorem that
			\begin{align*}
				&\mathbb E\left[\int_0^t\mathrm ds\int_\X G(s,y;t,x)\abs{B(s,y)-B'(s,y)}m(\mathrm dy)\right]\\
				&\quad=\int_0^t\mathrm ds\int_\X G(s,y;t,x)\mathbb E\left[\abs{B(s,y)-B'(s,y)}\right]m(\mathrm dy)=0,
			\end{align*}
			which implies  $\Psi_1(B')$ is a modification of $\Psi_1(B)$.
			
			Now suppose that $\Sigma'$ is a predictable modification of $\Sigma$, and let $\widetilde\Sigma(s,y):=\Sigma(s,y)-\Sigma'(s,y)$ for $(s,y)\in[0,T]\times\X$.
			For each $y\in\X$, Fubini's theorem gives
			\begin{align*}
				\norm{\widetilde\Sigma(r,y)}_{L^2_\Omega}=0,\qquad \text{for $\Lambda$-almost every $r\in[0,T]$}.
			\end{align*}
			Fix $(t,x)\in[0,T]\times\X$, and for $y,z\in\X$, define
			\begin{align}
				I_1(y,z;t,x)&:=\pr{\int_0^tG(r,y;t,x)G(r,z;t,x)\norm{\widetilde\Sigma(r,y)}_{L^2_\Omega}^2\mathrm dr}^{1/2},\\
				I_2(y,z;t,x)&:=\pr{\int_0^tG(r,y;t,x)G(r,z;t,x)\norm{\widetilde\Sigma(r,z)}_{L^2_\Omega}^2\mathrm dr}^{1/2}.
			\end{align}
			It follows that $I_1(y,z;t,x)=I_2(y,z;t,x)=0$ for every $y,z\in\X$.
			By Tonelli's theorem and the Cauchy--Schwarz inequality, we see
			\begin{align}
				&\int_0^t\norm{G(s,\cdot;t,x)\norm{\widetilde\Sigma(s,\cdot)}_{L^2_\Omega}}_{\mathcal H_K}^2\mathrm ds\\
				&=\int_0^t\iint_{\X^2}G(s,y;t,x)G(s,z;t,x)\norm{\widetilde\Sigma(s,y)}_{L^2_\Omega}\norm{\widetilde\Sigma(s,z)}_{L^2_\Omega}K(\mathrm dy,\mathrm dz)\mathrm ds\\
				&=\iint_{\X^2}\int_0^tG(s,y;t,x)G(s,z;t,x)\norm{\widetilde\Sigma(s,y)}_{L^2_\Omega}\norm{\widetilde\Sigma(s,z)}_{L^2_\Omega}\mathrm dsK(\mathrm dy,\mathrm dz)\\
				&\leq\iint_{\X^2}I_1(y,z;t,x)I_2(y,z;t,x)K(\mathrm dy,\mathrm dz)=0.
			\end{align}
			Therefore, by the linearity of the stochastic integral and Lemma \ref{cor:mtg+mink_cs},
			\begin{align*}
				\norm{\Psi_2(\Sigma)(t,x)-\Psi_2(\Sigma')(t,x)}_{L^2_\Omega}^2
				&=\norm{\iint_0^tG(s,y;t,x)\widetilde\Sigma(s,y)W(\mathrm ds,\mathrm dy)}_{L^2_\Omega}^2\\
				&\lesssim\int_0^t\norm{G(s,\cdot;t,x)\norm{\widetilde\Sigma(s,\cdot)}_{L^2_\Omega}}_{\mathcal H_K}^2\mathrm ds=0.
			\end{align*}
			Hence, $\Psi_2(\Sigma')$ is a modification of $\Psi_2(\Sigma)$.
		\end{proof}

		\begin{proof}[Proof of Lemma \ref{lem:CMS}]
			Fix arbitrary $\lambda \in \R$ and $p \ge 1$, and observe that we may assume without loss of generality that $A = 0$. Clearly, $\norm{\cdot}_{\mathbb M^\sharp_{\lambda,p}}$ is a norm on $\mathbb M^\sharp_{\lambda,p}$, hence it is enough to show the space $(\mathbb M^\sharp_{\lambda,p}, \norm{\cdot}_{\mathbb M^\sharp_{\lambda,p}})$ is complete.
			
			\firststep
			Suppose now $\{X^{(n)}\}_{n \in \mathbb N}$ is Cauchy in $(\mathbb M^\sharp_{\lambda,p}, \norm{\cdot}_{\mathbb M^\sharp_{\lambda,p}})$. Then there exists $\{X_n\}_{n \in \mathbb N}$ in $\mathbb M_{\lambda,p}$ so that 
			\begin{align}
				[X_n]_{\mathrm p} = X^{(n)} \qquad \text{ in }\mathbb M^\sharp_{\lambda,p},
			\end{align}
			and $\{X_n\}_{n \in \N}$ is Cauchy in $(\mathbb M_{\lambda,p},\norm{\cdot}_{\mathbb M_{\lambda,p}})$. We will first show the sequence $\{X_n\}_{n \in \mathbb N}$ has a predictable limit $X$ so that $\lim_{n \to \infty}\norm{X_n - X}_{\mathbb M_{\lambda, p}} = 0$.

			Let $\mathcal P_T:=\mathcal P|_{\Omega\times[0,T]}$, and equip $\widehat\Omega:=\Omega\times[0,T]$ with $\mathcal P_T$ and the probability measure $\widehat{\mathbb P}:=T^{-1}(\mathbb P\otimes\Lambda)|_{\mathcal P_T}$.
			For every $x\in\X$ and $n,m\in\N$, we have
			\begin{align*}
				\norm{X_n(\cdot,x)-X_m(\cdot,x)}_{L^p(\widehat{\mathbb P})}
				\leq h(x)^\lambda\norm{X_n-X_m}_{\mathbb M_{\lambda,p}}.
			\end{align*}
			Hence, for every fixed $x\in\X$, the sequence $\{X_n(\cdot,x)\}_{n\in\N}$ is Cauchy in $L^p(\widehat{\mathbb P})$ and therefore Cauchy in $\widehat{\mathbb P}$-probability.
			Regarding $x\in\X$ as a parameter, it follows from \cite[Lemma~5.6 and Corollary~5.33]{MR4226142} that there exists a $\mathcal P_T\otimes\mathscr B_\X$-measurable random field $X$ such that $X_n(\cdot,x)$ converges to $X(\cdot,x)$ in $\widehat{\mathbb P}$-probability for every $x\in\X$.
			In particular, $X$ is jointly predictable.
			
			\nextstep We now prove that $X_n$ converges to $X$ in the $\norm{\cdot}_{\mathbb M_{\lambda,p}}$ norm.
			For $N\in\N$, set
			\begin{align*}
				r_N:=\sup_{n,m\geq N}\norm{X_n-X_m}_{\mathbb M_{\lambda,p}},
			\end{align*}
			so that $r_N\downarrow0$ as $N\to\infty$.
			Fix $x\in\X$ and $n\geq N$.
			By \cite[Lemma~5.2]{MR4226142}, there exists a subsequence $\{X_{m_k}(\cdot,x)\}_{k\in\N}$, with $m_k\geq N$, which converges to $X(\cdot,x)$ $\widehat{\mathbb P}$-almost surely.
			Therefore, Fubini's theorem and Fatou's lemma imply that, for $\Lambda$-almost every $t\in[0,T]$,
			\begin{align*}
				h(x)^{-p\lambda}\mathbb E\left[\abs{X_n(t,x)-X(t,x)}^p\right]
				&\leq\liminf_{k\to\infty}h(x)^{-p\lambda}\mathbb E\left[\abs{X_n(t,x)-X_{m_k}(t,x)}^p\right]\\
				&\leq r_N^p.
			\end{align*}
			Taking the $p$-th root, the essential supremum in $t$, and then the supremum in $x$, we obtain
			\begin{align*}
				\norm{X_n-X}_{\mathbb M_{\lambda,p}}\leq r_N,\qquad n\geq N.
			\end{align*}
			Since $r_N\to0$, it follows that $X\in\mathbb M_{\lambda,p}$ and $\lim_{n\to\infty}\norm{X_n-X}_{\mathbb M_{\lambda,p}}=0$.
			
			\nextstep We see $[X]_{\mathrm p} \in \mathbb M_{\lambda,p}^\sharp$, and 
			\begin{align}
				\lim_{n \to \infty} \norm{X^{(n)}-[X]_{\mathrm p}}_{\mathbb M_{\lambda,p}^\sharp} &= \lim_{n \to \infty} \norm{[X_n]_{\mathrm p}-[X]_{\mathrm p}}_{\mathbb M_{\lambda,p}^\sharp} \\
				&=\lim_{n \to \infty}\norm{X_n - X}_{\mathbb M_{\lambda,p}}\\
				&=0,
			\end{align}
			which concludes the proof.
		\end{proof}
	
	\begin{proof}[Proof of Lemma \ref{lem:b_contract_combined}]
		We first show that $(B_U, \Sigma_U)$ is a Duhamel-integrable pair.
		Note that the predictability of $B_U$ and $\Sigma_U$ is given by the fact that the coefficients $\sigma(t,x,u),\, b(t,x,u)$ are predictable functions on $\Omega\times\mathbb R_+\times\X\times\mathbb R$ by \eqref{eq:DC}--\eqref{eq:NC}, and $U(t,x)$ itself is a  predictable random field.
		Hence, it is enough to show for some $p \ge 2$ that
		\begin{align}\label{ineq:bcc D pair}
			\norm{\norm{G(\cdot,\cdot; t,x) B_U}_{L^1_{[0,T]}L_\X^1}}_{L^p_\Omega} + \norm{ \norm{G(\cdot,\cdot;t,x)\Sigma_U}_{L^2_{[0,T]}\mathcal H_K} }_{L^p_{\Omega}}<\infty.
		\end{align}
		Observe by Minkowski's inequality that it holds uniformly in $(t,x) \in [0,T]\times \X$ and $A \ge 1$ that
		\begin{align*}
			\norm{\norm{G(\cdot,\cdot; t,x) B_U}_{L^1_{[0,T]}L_\X^1}}_{L^p_\Omega}&\,\,=\norm{\int_0^t \mathrm ds \int_\X  G(s,y;t,x) |B_U(s,y)|m(\mathrm dy)}_{L^p_\Omega}\\
			&\,\,\le \int_0^t \mathrm ds  \int_\X  G(s,y;t,x) \norm{B_U(s,y)}_{L^p_\Omega}m(\mathrm dy)\\
			&\overset{\eqref{eq:LiG}}{\lesssim}\int_0^t \mathrm ds \int_\X G(s,y;t,x) \pr{1+\norm{U(s,y)}_{L^p_\Omega}} m(\mathrm dy)\\
			&\,\,\le \int_0^t\mathrm ds \int_\X e^{sA}G(s,y;t,x) h(y)^{\lambda} m(\mathrm dy) \pr{1+\norm{U}_{\mathbb M_{\lambda, p,A}}}\\
			&\overset{\text{Lemma~\ref{lem:hk longtime}, \eqref{ineq:hk_m long time}}}{\lesssim}h(x)^\lambda \pr{1+\norm{U}_{\mathbb M_{\lambda,p,A}}}\int_0^t \mathrm ds \,e^{sA} \\
			&\,\,\le A^{-1} e^{At} h(x)^\lambda \pr{1+\norm{U}_{\mathbb M_{\lambda,p,A}}}.
		\end{align*}

		Following the proof of Lemma \ref{cor:mtg+mink_cs}, we see it holds uniformly in $(t,x)\in [0,T]\times \X$ and $A \ge 0$ that 
		\begin{align*}
			\norm{\norm{G(\cdot,\cdot; t,x) \Sigma_U}_{L^2_{[0,T]}\mathcal H_K}}_{L^p_\Omega}^2 
			&\,\,\le \int_0^t \mathrm ds \norm{G(s,\cdot; t,x) \norm{\Sigma_U(s,\cdot)}_{L^{p}_\Omega} }_{\mathcal H_K}^2\\
			&\,\,=\int_0^t \mathrm ds \norm{G(s,\cdot; t,x) h(\cdot)^\lambda h(\cdot)^{-\lambda} \norm{\Sigma_U(s,\cdot)}_{L^{p}_\Omega} }_{\mathcal H_K}^2\\
			&\,\,\le \norm{\Sigma_U}_{\mathbb M_{\lambda, p,A}}^2\int_0^t\mathrm ds\, e^{2As}\norm{G(s,\cdot; t,x) h(\cdot)^\lambda}_{\mathcal H_K}^2  \\
			&\overset{\text{Lemma~\ref{lem:hk longtime}, \eqref{ineq:hk_k long time}}}{\lesssim} \norm{\Sigma_U}_{\mathbb M_{\lambda,p,A}}^2 h(x)^{2\lambda}\int_0^t\mathrm ds \, \pr{e^{2As} (t-s)^{-\widetilde\alpha}}.
		\end{align*}
		Note for every $\widetilde\alpha \in [0,1)$, it holds uniformly in $A,t >0$ that 
		\begin{align*}
			\int_0^t e^{2As}(t-s)^{-\widetilde\alpha} \mathrm ds  \lesssim A^{\widetilde\alpha - 1} e^{2At}.
		\end{align*}
		Hence it holds uniformly in $(t,x) \in [0,T]\times \X$ and $A \ge 1$ that 
		\begin{align*}
			\norm{\norm{G(\cdot,\cdot;t,x) \Sigma_U}_{L^2_{[0,T]}\mathcal H_K}}_{L^p_\Omega} \lesssim  A^\frac{\widetilde\alpha -1}{2}e^{A t}\norm{\Sigma_U}_{\mathbb M_{\lambda,p,A}} h(x)^\lambda \lesssim A^\frac{\widetilde\alpha - 1}{2} e^{At} h(x)^\lambda\pr{1+\norm{U}_{\mathbb M_{\lambda,p,A}}}.
		\end{align*}
		Hence \eqref{ineq:bcc D pair} holds. More precisely, for each $p \ge 2$, the following inequality holds uniformly in $A \ge 1$ and $(t,x) \in [0,T]\times \X$ 
		\begin{align}\label{ineq:bcc D pair explicit}
			&\norm{\norm{G(\cdot,\cdot; t,x) B_U}_{L^1_{[0,T]}L_\X^1}}_{L^p_\Omega} + \norm{ \norm{G(\cdot,\cdot;t,x)\Sigma_U}_{L^2_{[0,T]}\mathcal H_K} }_{L^p_{\Omega}}\\
			\lesssim & A^{\frac{\widetilde\alpha - 1}{2}}h(x)^\lambda e^{A t} \pr{1+\norm{U}_{\mathbb M_{\lambda, p,A}}}.
		\end{align}

		We next show that $\Psi(g,B_U,\Sigma_U)$ admits a predictable modification in $\mathbb M_{\lambda,p}$.
		
		We first verify predictability. By \eqref{cond:ctem continuous}, $\Psi_0(g)$ is deterministic and jointly continuous.
		For $n\in\N$, let $\delta_n:=2^{-n}T$ and define $\Psi^{(n)}(0,x):=g(x)$. For $t\in(j\delta_n,(j+1)\delta_n]$, where $j=0,\ldots,2^n-1$, set
		\begin{align}\label{eq:predictable_approximation}
			\Psi^{(n)}(t,x)
			&:=(P_t^*g)(x)+\int_0^{j\delta_n}\mathrm ds\int_\X G(s,y;t,x)B_U(s,y)m(\mathrm dy)\\
			&\qquad+\iint_0^{j\delta_n}G(s,y;t,x)\Sigma_U(s,y)W(\mathrm ds,\mathrm dy).
		\end{align}

		Set $S:=(j\delta_n,(j+1)\delta_n]\times\X$. The drift integral in the definition of $\Psi^{(n)}(t,x)$ has an $\mathscr F_{j\delta_n}\otimes\mathscr B(S)$-measurable version. The stochastic integral vanishes when $j=0$, so suppose that $j\ge1$.

		By Lemma \ref{lem:H0}, the standard elementary predictable approximations of $G(\cdot,\cdot;t,x)\Sigma_U$ on $[0,j\delta_n]$, using  finite-dimensional projections, can be chosen jointly measurable in $(t,x)\in S$. 
		Hence, by \cite[Corollary~5.33]{MR4226142}, the third term on the right hand side of \eqref{eq:predictable_approximation} 
		has an $\mathscr F_{j\delta_n}\otimes\mathscr B(S)$-measurable version.
		Since $\mathscr F_{j\delta_n}\otimes\mathscr B((j\delta_n,(j+1)\delta_n])\subset\mathcal P_T$ on $(j\delta_n,(j+1)\delta_n]$, patching the measurable versions of the drift and stochastic integrals over these time intervals, together with the deterministic term $(P_t^*g)(x)$ and the value $g(x)$ at $t=0$, gives a $\mathcal P_T\otimes\mathscr B_\X$-measurable version of $\Psi^{(n)}$, where $\mathcal P_T:=\mathcal P|_{\Omega\times[0,T]}$. We use this predictable version of $\Psi^{(n)}$ below.

		By Minkowski's inequality and Lemma \ref{cor:mtg+mink_cs}, using \eqref{eq:LiG} and \eqref{ineq:hk_m long time}--\eqref{ineq:hk_k long time}, we have, uniformly in $n\in\N$,
		\begin{align*}
			\norm{\Psi^{(n)}}_{\mathbb M_{\lambda,p}}
			\lesssim \norm{g}_{(\lambda)}+1+\norm{U}_{\mathbb M_{\lambda,p}}<\infty.
		\end{align*}
		Together with the predictability established above, this proves that $\Psi^{(n)}\in\mathbb M_{\lambda,p}$.

		For $t\in(j\delta_n,(j+1)\delta_n]$, \eqref{eq:Psi} and \eqref{eq:predictable_approximation} give
		\begin{align*}
			\Psi(g,B_U,\Sigma_U)(t,x)-\Psi^{(n)}(t,x)
			&=\int_{j\delta_n}^t\mathrm ds\int_\X G(s,y;t,x)B_U(s,y)m(\mathrm dy)\\
			&\qquad\quad+\iint_{j\delta_n}^tG(s,y;t,x)\Sigma_U(s,y)W(\mathrm ds,\mathrm dy).
		\end{align*}
		By \eqref{eq:LiG}, Minkowski's inequality, Lemma \ref{cor:mtg+mink_cs}, and \eqref{ineq:hk_m long time}--\eqref{ineq:hk_k long time}, it follows uniformly in $n\in\N$, $A\geq1$, and $(t,x)\in(0,T]\times\X$ that
		\begin{align*}
			&e^{-At}h(x)^{-\lambda}\norm{\Psi^{(n)}(t,x)-\Psi(g,B_U,\Sigma_U)(t,x)}_{L^p_\Omega}\\
			&\lesssim\pr{1+\norm{U}_{\mathbb M_{\lambda,p,A}}}
			\left[\int_{j\delta_n}^t e^{-A(t-s)}\mathrm ds
			+\left(\int_{j\delta_n}^t e^{-2A(t-s)}(t-s)^{-\widetilde\alpha}\mathrm ds\right)^{1/2}\right]\\
			&\lesssim\pr{1+\norm{U}_{\mathbb M_{\lambda,p,A}}}\left(\delta_n+\delta_n^{(1-\widetilde\alpha)/2}\right).
		\end{align*}
		The last inequality uses $0<t-j\delta_n\le\delta_n$ and $\widetilde\alpha<1$. The estimate also holds at $t=0$, since both random fields equal $g(x)$. 
		Since $\sum_{n\in\N}(\delta_n+\delta_n^{(1-\widetilde\alpha)/2})<\infty$, we see $\Psi^{(n)}(t,x)$ converges almost surely to $\Psi(g,B_U,\Sigma_U)(t,x)$. Taking the pointwise limit of $\Psi^{(n)}$ wherever it exists and is finite, and setting the field equal to zero elsewhere, yields a $\mathcal P_T\otimes\mathscr B_\X$-measurable modification of $\Psi(g,B_U,\Sigma_U)$ whose value at $t=0$ is $g(x)$. We use this modification below without changing notation.

		It remains to bound $\norm{\Psi(g,B_U,\Sigma_U)}_{\mathbb M_{\lambda,p}}$. For $p\ge2$, the following estimate holds uniformly in $(t,x)\in[0,T]\times\X$ and $A\ge1$:
		\begin{align*}
			&\norm{\Psi(g,B_U,\Sigma_U)(t,x)}_{L^p_\Omega}\\
			&\quad\lesssim \abs{P^*_t g(x)} +\norm{\norm{G(\cdot,\cdot; t,x) B_U}_{L^1_{[0,T]}L_\X^1}}_{L^p_\Omega}+ \norm{ \norm{G(\cdot,\cdot;t,x)\Sigma_U}_{L^2_{[0,T]}\mathcal H_K} }_{L^p_{\Omega}}\\
			&\quad\overset{\text{Lemma~\ref{lem:hk longtime}, \eqref{ineq:hk_m long time}},\eqref{ineq:bcc D pair explicit}}{\lesssim} h(x)^\lambda + \pr{A^\frac{\widetilde\alpha-1}{2}e^{At}} h(x)^\lambda  \pr{1+\norm{U}_{\mathbb M_{\lambda,p,A}}}.
		\end{align*}
		Taking $A=1$ in this estimate gives $\norm{\Psi(g,B_U,\Sigma_U)}_{\mathbb M_{\lambda,p}}<\infty$.
		
		Finally, to prove the difference estimate, we write
		\begin{align*}
			\Psi(g,B_U,\Sigma_U) - \Psi(g,B_V,\Sigma_V) &= \int_0^t\mathrm ds  \int_\X  G(s,y; t,x) \widetilde B_{U - V}(s,y) m(\mathrm dy)\\
			&\qquad + \iint_0^t G(s,y; t,x) \widetilde \Sigma_{U - V}(s,y) W(\mathrm ds,\mathrm dy),
		\end{align*}
		where 
		\begin{gather*}
			\widetilde B_{X}(s,y):= \widetilde b(s,y,X(s,y)):=b(s,y,X(s,y)+V(s,y)) - b(s,y,V(s,y)),\\ \widetilde \Sigma _{X}(s,y):= \widetilde\sigma(s,y,X(s,y)):=\sigma(s,y,X(s,y)+V(s,y)) - \sigma(s,y,V(s,y)).
		\end{gather*}
		We see $\widetilde b$ and $\widetilde \sigma$ satisfy \eqref{ineq:g_lip} and 
		\begin{align*}
			|\widetilde b(t,x,u)| + |\widetilde \sigma(t,x,u)| \lesssim |u|,\qquad \text{ uniformly in } (t,x,u) \in [0,T]\times \X \times  \R.
		\end{align*}
		Hence, by repeating the above argument, we see it holds uniformly in $(t,x) \in [0,T]\times \X$ and $A \ge 1$ that
		\begin{align*}
			\norm{\Psi(g,B_U,\Sigma_U)(t,x) - \Psi(g,B_V,\Sigma_V)(t,x)}_{L^p_\Omega} \lesssim \pr{A^\frac{\widetilde\alpha - 1}{2} e^{At} } h(x)^\lambda \norm{U - V}_{\mathbb M_{\lambda,p,A}},
		\end{align*}
		which implies the desired inequality.
	\end{proof}

	\subsection{Proof of Lemma \ref{lem:approx_estimates}}

	We recall the following standard estimate for centered moments of Poisson random variables; we include a proof for completeness.
	\begin{lemma}\label{lem:poisson moment bounds}
		Let \(Z\sim \operatorname{Pois}(\lambda)\), \(\lambda>0\). For every positive even integer \(p=2m\), \(m\in\N\),
		there exists \(C_p<\infty\) such that
		\[
		\mathbb E|Z-\lambda|^p
		\le C_p\bigl(\lambda+\lambda^{p/2}\bigr).
		\]
	\end{lemma}
	
	\begin{proof}
		Let \(p=2m\) be a positive even integer. By \cite[Lemma 3.1]{privault2011generalized},
		\[
		\mathbb E[(Z-\lambda)^{2m}]
		=
		\sum_{a=0}^{2m}\lambda^a S_2(2m,a),
		\]
		where \(S_2(n,a)\) denotes the number of partitions of an \(n\)-element set into \(a\) blocks
		of size at least \(2\). Hence \(S_2(2m,a)=0\) for \(a>m\). Indeed, if there are $a$ number of sets of size at least $2$, then the total elements is at least $2a$. Hence, for each $m \in \N$, there exists $C_m>0$, independent of $\lambda>0$, so that 
		\[
		\mathbb E|Z-\lambda|^{2m}
		=
		\sum_{a=1}^{m}\lambda^a S_2(2m,a)
		\le C_m(\lambda+\lambda^m).
		\]
		This proves the lemma.
	\end{proof}
	
	\begin{proof}[Proof of Lemma \ref{lem:approx_estimates}]
		(1) Let $t>0$, $\varepsilon \in (0,1]$ and $x \in \X$. We see by Lemma \ref{lem:H0}, Lemma \ref{lem:hk longtime} and {Jensen's} inequality that uniformly in $x \in \X$, $t >0 $ and $\varepsilon \in (0,1]$,
		\begin{align}
			\|r_t^\varepsilon(\cdot,x)\|_{\mathcal H_K}^2 &\le \left(e^{-t/\varepsilon} \sum_{n = 1}^\infty \frac{(t/\varepsilon)^n}{n!}\|G(n\varepsilon,\cdot,x)\|_{\mathcal H_K} \right)^2\\
			&\,\,\le e^{-t/\varepsilon} \sum_{n = 1}^\infty \frac{(t/\varepsilon)^n}{n!} \|G(n\varepsilon,\cdot,x)\|_{\mathcal H_K}^2\\
			&\overset{\eqref{ineq:hk_k long time}}{\lesssim} e^{-t/\varepsilon} \sum_{n = 1}^\infty \frac{(t/\varepsilon)^n}{n!} (n\varepsilon)^{-\alpha} e^{\kappa \varepsilon n}\\
			&\,\, =\exp \left(t\left(e^{\kappa \varepsilon} - 1 \right) /\varepsilon \right) e^{-te^{\kappa \varepsilon}/\varepsilon} \sum_{n = 1}^\infty \frac{\left(e^{\kappa \varepsilon}t/ \varepsilon\right)^n}{n!} (n\varepsilon)^{-\alpha},
		\end{align}
		for some $\alpha \in (0,1)$. Hence, by {Jensen's} inequality,
		\begin{align}
			e^{-te^{\kappa \varepsilon}/\varepsilon} \sum_{n = 1}^\infty \frac{\left(e^{\kappa \varepsilon}t/ \varepsilon\right)^n}{n!} (n\varepsilon)^{-\alpha} &\le \left (e^{-t e^{\kappa \varepsilon}/\varepsilon} \sum_{n = 1}^\infty \frac{(e^{\kappa \varepsilon}t/\varepsilon)^{n}}{n!} (n\varepsilon)^{-1} \right)^{\alpha}\\
			&{\le}\left(e^{-te^{\kappa \varepsilon}/\varepsilon}\sum_{n=1}^\infty \frac{\left(e^{\kappa \varepsilon}t/\varepsilon \right)^{n+1}}{(n+1)!} \frac{n+1}{n}t^{-1}\right)^\alpha \\
			& \le 2^\alpha t^{-\alpha}.
		\end{align}
		Finally, observe that there is $\kappa' \in \R$ so that for $t \ge 0$,
		\begin{align}
			\sup_{ \varepsilon \in (0,1]} \exp \left(t \left(e^{\kappa \varepsilon} - 1 \right) /\varepsilon\right) \le  e^{\kappa ' t}.
		\end{align}
		Collecting all terms, we see it holds uniformly in $x \in \X$, $t >0$ and $\varepsilon \in (0,1]$ that
		\begin{align}
			\|r_t^\varepsilon(\cdot,x)\|_{\mathcal H_K }^2 \lesssim e^{\kappa' t} t^{-\alpha}
		\end{align}
		for some $\kappa ' \in \R$. Moreover, by Lemma \ref{lem:hk longtime}, there exists $\kappa'' \in \R$ so that for $t>0$, $\varepsilon \in (0,1]$ and $x \in \X$, it holds uniformly in $x \in \X$, $t >0 $ and $\varepsilon \in (0,1]$ that
		\begin{align}
			\|r_t^\varepsilon(\cdot,x)\|_{L^1_\X}
			&\le e^{-t/\varepsilon}\sum_{n=1}^\infty \frac{(t/\varepsilon)^n}{n!}
			\|G_{n\varepsilon}(\cdot,x)\|_{L^1_\X}\\
			&\lesssim e^{-t/\varepsilon}\sum_{n=1}^\infty \frac{(t/\varepsilon)^n}{n!}e^{\kappa n\varepsilon}
			\le \exp\left(t\frac{e^{\kappa\varepsilon}-1}{\varepsilon}\right)
			\lesssim e^{\kappa''t},
		\end{align}
		Since $1\lesssim t^{-\alpha}e^t$ for $t>0$, enlarging the exponential constant if necessary proves both bounds in \eqref{ineq:R_bound} with a single $\kappa$ independent of $\varepsilon$.
		
		(2) The objective is to show
		\begin{align}\label{lim:sup_cov}
			\lim_{\varepsilon \downarrow 0} \sup_{x\in \X} \int_0^t \left\|{\abs{\int_\X r_{t-s}^\varepsilon(y,x)G_\varepsilon(\cdot,y)m(\mathrm dy) - G(t-s,\cdot,x)}}\right\|^2_{\mathcal H_K}\,\mathrm ds {=} 0.
		\end{align}
		For $t>0$, by Lemma \ref{lem:H0} and {Jensen's} inequality, it holds uniformly in $x \in \X$ that
		\begin{align}
			&\int_0^t \left\|{\abs{\int_\X r_{t-s}^\varepsilon({y,x}) G_\varepsilon({\cdot,y}) m(\mathrm dy) - G(t-s,\cdot,x)}}\right\|^2_{\mathcal H_K} \mathrm ds\\
			\overset{\text{Lemma~\ref{lem:hk longtime}, \eqref{ineq:hk_k long time}}}{\lesssim}& \int_0^t  \left\|e^{-(t-s)/\varepsilon}\sum_{n = 1}^\infty \frac{((t-s)/\varepsilon)^n}{n!} {\abs{G((n+1)\varepsilon,\cdot,x) - G(t-s, \cdot,x)}}\right\|^2_{\mathcal H_K } \mathrm ds\\
			&\qquad + \int_0^t e^{-2(t-s)/\varepsilon} (t-s)^{-\alpha} \mathrm ds\\
			\le \,\,\,&\int_0^t e^{-(t-s)/\varepsilon} \sum_{n = 1}^\infty \frac{\left(\frac{t-s}{\varepsilon}\right)^n}{n!} \left\|{\abs{G(\varepsilon(n+1),\cdot,x) - G(t-s,\cdot,x)}}\right\|^2_{\mathcal H_K } \mathrm ds +  \int_0^t e^{-2s/\varepsilon} s^{-\alpha} \mathrm ds\\
			=: \,\,&\int_0^t {\left(I_{1,T}(t-s,x,\varepsilon)+I_{2,T}(t-s,x,\varepsilon)\right)} \mathrm ds +\int_0^t e^{-2s/\varepsilon} s^{-\alpha} \mathrm ds,
		\end{align}
		where we used {Jensen's} inequality and Lemma \ref{lem:H0} in the second inequality and
		\begin{gather*}
			I_{1,T}(t-s,x,\varepsilon) :=e^{-(t-s)/\varepsilon}{\sum_{n\varepsilon > T+1}} \frac{\left(\frac{t-s}{\varepsilon}\right)^n}{n!} \left\|{\abs{G(\varepsilon(n+1),\cdot,x) - G(t-s,\cdot,x)}}\right\|^2_{\mathcal H_K }{,}\\
			I_{2,T}(t-s,x,\varepsilon):= e^{-(t-s)/\varepsilon}\sum_{\substack{\\n \varepsilon \le T+1}}\frac{\left(\frac{t-s}{\varepsilon}\right)^n}{n!} \left\|{\abs{G(\varepsilon(n+1),\cdot,x) - G(t-s,\cdot,x)}}\right\|^2_{\mathcal H_K }{.}
		\end{gather*}

		First, for an arbitrary but fixed $T>0$, it holds uniformly in $0 < s< t$  and $x \in \X$ that
		\begin{align}
			&e^{-(t-s)/\varepsilon}\sum_{n \varepsilon > T+1} \frac{\left( \frac{t-s}{\varepsilon}\right)^n}{n!} \left\|{\abs{G\left(\varepsilon(n+1), \cdot, x \right)  - G(t-s, \cdot,x)}}\right\|_{\mathcal H_K }^2\\
			{\lesssim} &e^{-(t-s)/\varepsilon}\sum_{n \varepsilon > T+1} \frac{\left( \frac{t-s}{\varepsilon}\right)^n}{n!} \left(\left\|G\left(\varepsilon(n+1), \cdot, x \right)\right\|_{\mathcal H_K }^2 + \left\| G(t-s, \cdot,x)\right\|_{\mathcal H_K }^2\right)\\
			\lesssim& e^{-(t-s)/\varepsilon}\sum_{n\varepsilon>T+1} \frac{\left( \frac{t-s}{\varepsilon}\right)^n}{n!} \left((\varepsilon (n+1))^{-\alpha}e^{\kappa n\varepsilon} + (t-s)^{-\alpha}\right)\\
			\le & e^{-(t-s)/\varepsilon} \sum_{n \varepsilon >T+1} \frac{\left( \frac{t-s}{\varepsilon}e^{\kappa \varepsilon}\right)^n}{n!} T^{-\alpha} + e^{-(t-s)/\varepsilon} \sum_{n \varepsilon >T+1} \frac{\left( \frac{t-s}{\varepsilon}e^{\kappa \varepsilon}\right)^n}{n!}(t-s)^{-\alpha}\\
			= &{e^{(e^{\kappa \varepsilon}-1)(t-s)/\varepsilon}} \left(T^{-\alpha}+(t-s)^{-\alpha} \right)e^{- \frac{t-s}{\varepsilon}e^{\varepsilon \kappa }}\sum_{n\varepsilon>T+1}  \frac{\left( \frac{t-s}{\varepsilon}e^{\kappa \varepsilon}\right)^n}{n!}.
		\end{align}
		
		Let $\psi(u):= (1+u) \ln(u+1)- u$ for $u \ge 0$ and note that $\psi$ is an increasing function on $\R_+$. Hence, by {Bennett's inequality} for Poisson random variables, we see for sufficiently small $\varepsilon>0$ and $0<s<t\le T$,
		\begin{align}
			e^{- \frac{t-s}{\varepsilon}e^{\varepsilon \kappa }}\sum_{n\varepsilon>T+1}  \frac{\left( \frac{t-s}{\varepsilon}e^{\kappa \varepsilon}\right)^n}{n!} &\le \exp \left( - e^{\kappa \varepsilon }\frac{t-s}{\varepsilon}\psi\left( \frac{T+1}{t-s}e^{-\kappa \varepsilon}-1\right)\right)\\
			&\le \exp \left( - e^{\kappa \varepsilon }\frac{t-s}{\varepsilon}\psi\left( \frac{T+1}{T}e^{-\kappa \varepsilon}-1\right)\right),
		\end{align}
		which is uniformly bounded by $1$ and converges to zero as $\varepsilon \downarrow 0$. Hence, by the dominated convergence theorem,
		\begin{align}
			\lim_{\varepsilon \downarrow 0}\sup_{x \in \X}\int_0^t I_{1,T}(t-s,x,\varepsilon) \mathrm ds = 0 \qquad \text{ for each }T\ge t >0.
		\end{align}
		{On the other hand, since  $\varepsilon(n+1)\le T+2$, Corollary~\ref{cor:hk_abs}, applied with time horizon $T+2$ implies that, there exist $\xi \in (0,1]$ and $\beta \in (0,1)$, so that the following estimate uniformly in $\varepsilon\in(0,1)$, $0\le s<t\le T+1$, and $x\in\X$,}
		\begin{align}
			I_{2,T}(t-s,x,\varepsilon)&\lesssim e^{-(t-s)/\varepsilon} \sum_{n \varepsilon \le T+1} \frac{\left( \frac{t-s}{\varepsilon}\right)^n}{n!} \frac{{\abs{\varepsilon(n+1) - (t-s)}^\xi}}{\varepsilon^\beta(n+1)^\beta \wedge (t-s)^\beta}\nonumber\\
			&\le e^{-(t-s)/\varepsilon} \sum_{n\varepsilon {\le} T+1} \frac{\left(\frac{t-s}{\varepsilon}\right)^n}{n!} \frac{{\abs{\varepsilon(n+1) - (t-s)}^\xi}}{\varepsilon^\beta (n+1)^\beta} \nonumber\\
			&\qquad +e^{-(t-s)/\varepsilon} \sum_{n\varepsilon {\le} T+1} \frac{\left(\frac{t-s}{\varepsilon}\right)^n}{n!}\frac{{\abs{\varepsilon(n+1) - (t-s)}^\xi}}{(t-s)^\beta}\nonumber\\
			&:=J_{1,T}(t-s,\varepsilon) + J_{2,T}(t-s,\varepsilon),\label{ineq:j1j2 composition approx poisson}
		\end{align}
		
		For each $t,\,\varepsilon>0$, let $N_\varepsilon(t)\sim {\operatorname{Pois}}(t/\varepsilon)$. Using Jensen's inequality and the fact that $\xi \in (0,1]$, we see for each $\varepsilon \in (0,1]$ and $t>s \ge 0$,
		\begin{equation}\label{ineq:j2 decomposition poisson moments}
			\begin{aligned}
				J_{2,T}(t-s,\varepsilon)&\le \frac{\E\left[ \abs{{\varepsilon\bigl(N_\varepsilon(t-s)+1\bigr)} -(t-s) }^\xi\mathbf{1}_{\{{\varepsilon N_\varepsilon(t-s)\le T+1}\}}\right]}{(t-s)^\beta}\\
				&\le \frac{\E\left[ \abs{{\varepsilon\bigl(N_\varepsilon(t-s)+1\bigr)} -(t-s) }^\xi\right]}{(t-s)^\beta}\\
				&\le \frac{\E\left[ \abs{{\varepsilon\bigl(N_\varepsilon(t-s)+1\bigr)} -(t-s) }^2\right]^\frac{\xi}{2}}{(t-s)^\beta}\\
				&\le \frac{\pr{2\varepsilon(t-s) +\varepsilon^2}^\frac{\xi}{2}}{(t-s)^\beta},
			\end{aligned}
		\end{equation}
		where we used the variance of the Poisson distribution in the last line.
		
		For the first term, let $p \ge {2}$ be a sufficiently large {even integer} so that $\beta \frac{p}{p-1}<1$, then by H\"{o}lder's inequality, we see
		\begin{equation}\label{ineq:j1 decomposition approx poisson}
			\begin{aligned}
				J_{1,T}(t-s,\varepsilon) & \le \left(e^{-(t-s)/\varepsilon} \sum_{n\varepsilon {\le} T+1} \frac{\left(\frac{t-s}{\varepsilon}\right)^n}{n!}  {\abs{\varepsilon(n+1) - (t-s)}^{\xi p}}\right)^\frac{1}{p}\\
				&\qquad \times\left( e^{-(t-s)/\varepsilon} \sum_{n\varepsilon {\le} T+1} \frac{\left(\frac{t-s}{\varepsilon}\right)^n}{n!} \left(\varepsilon(n+1) \right)^{-\beta \frac{p}{p-1}} \right)^\frac{p-1}{p}
			\end{aligned}
		\end{equation}
		Since $\xi \in (0,1]$, we see by Jensen's inequality and  Lemma \ref{lem:poisson moment bounds} that the first term is bounded by
		\begin{equation}\label{ineq:j1 decomposition poisson moments}
			\begin{aligned}
				\mathbb E\left[ \abs{{\varepsilon\bigl(N_\varepsilon(t-s)+1\bigr)} - (t-s)}^{p\xi} \right]^\frac{1}{p}&\le \varepsilon^\xi\E\left[\abs{N_\varepsilon(t-s) + 1 - \frac{t-s}{\varepsilon}}^p \right]^\frac{\xi}{p}\\
				&\lesssim \varepsilon^\xi +\varepsilon^\xi \E\left[\abs{N_\varepsilon(t-s)  - \frac{t-s}{\varepsilon}}^p \right]^\frac{\xi}{p} \\
				&\lesssim \varepsilon^\xi \pr{1+ \pr{\frac{t-s}{\varepsilon}}^\frac{\xi}{2} +\pr{\frac{t-s}{\varepsilon}}^\frac{\xi}{p} }\\
				&\le \varepsilon^\xi + (t-s)^\frac{\xi}{2}\varepsilon^\frac{\xi}{2} + {(t-s)^\frac{\xi}{p}}\varepsilon^\frac{\xi(p-1)}{p}.
			\end{aligned}
		\end{equation}
		For the second term, apply {Jensen's} inequality to see
		\begin{equation}\label{ineq:j11 bound approx poisson}
			\begin{aligned}
				\left( e^{-(t-s)/\varepsilon} \sum_{n\varepsilon {\le} T+1} \frac{\left(\frac{t-s}{\varepsilon}\right)^n}{n!} \left(\varepsilon(n+1) \right)^{-\beta \frac{p}{p-1}} \right)^\frac{p-1}{p} &\le \left(e^{-(t-s)/\varepsilon}\sum_{n\varepsilon {\le} T+1} \frac{\left(\frac{t-s}{\varepsilon}\right)^n}{n!} \left(\varepsilon(n+1) \right)^{-1}\right)^{\beta }\\
				& \le \left(e^{-(t-s)/\varepsilon} \sum_{n = 0}^\infty \frac{\left(\frac{t-s}{\varepsilon} \right)^n}{n!} \frac{1}{\varepsilon(n+1)}\right)^{\beta}\\
				&= \left(e^{-(t-s)/\varepsilon} \sum_{n = 0}^\infty  \frac{1}{t-s}\frac{\left( \frac{t-s}{\varepsilon}\right)^{n+1}}{(n+1)!}\right)^{\beta}\\
				&= \left(\frac{1}{t-s} e^{-(t-s)/\varepsilon} \sum_{n = 1}^\infty \frac{\left( \frac{t-s}{\varepsilon}\right)^n}{n!}\right)^{\beta}\\
				&\le (t-s)^{-\beta}
			\end{aligned}
		\end{equation}
		Inserting \eqref{ineq:j11 bound approx poisson} and \eqref{ineq:j1 decomposition poisson moments} into \eqref{ineq:j1 decomposition approx poisson}, we see it holds uniformly in $T+1\ge t>s\ge 0$ and $\varepsilon \in (0,1]$ that
		\begin{align}\label{ineq:j1 final bound}
			J_{1,T}(t-s,\varepsilon) \lesssim \pr{\varepsilon^\xi + (t-s)^\frac{\xi}{2}\varepsilon^\frac{\xi}{2} + {(t-s)^\frac{\xi}{p}}\varepsilon^\frac{\xi(p-1)}{p}} (t-s)^{-\beta}.
		\end{align}
		Inserting this and \eqref{ineq:j2 decomposition poisson moments} into \eqref{ineq:j1j2 composition approx poisson}, we see
		\begin{align*}
			\lim_{\varepsilon \downarrow 0} \sup_{t \in (0,T+1]\atop x \in \X} \int_0^tI_{2,T}(t-s,x,\varepsilon) \mathrm ds = 0.
		\end{align*}
		Collecting all terms, we see \eqref{lim:sup_cov} holds.

		For (3), {let $\beta$ be as in \eqref{cond:time_holder}.} By Lemma \ref{lem:hk longtime}, we see for each $T>0$, it holds uniformly in $t \in (0,T]$ that
		\begin{align}
			&\int_\X \left|r_t^\varepsilon({y,x}) - G_t({y,x}) \right| m(\mathrm dy)\\
			\lesssim& e^{-\frac{t}{\varepsilon}} \left( 1+\sum_{n \varepsilon \le T + 1}  \frac{(t/\varepsilon)^n}{n!} \int_\X\left|G_{n\varepsilon}({y,x}) - G_t({y,x}) \right|m(\mathrm dy) +  2\sum_{n \varepsilon > T + 1}  \frac{(e^{\kappa \varepsilon}t/\varepsilon)^n}{n!}\right).
		\end{align}
		Let $\psi(\lambda) = (\lambda+1)\ln(\lambda+1) - \lambda$ and {$c_T = \min\left\{\frac12\psi\left(\frac{1}{2T}\right),1\right\}$}. By Bennett's inequality for Poisson random variables, it holds uniformly in sufficiently small $\varepsilon>0$ and all $t \in (0,T]$ that 
		\begin{align}
			e^{-t/\varepsilon}\sum_{n\varepsilon > T+1} \frac{(e^{\kappa \varepsilon}t/\varepsilon)^n}{n!}
			&= {e^{(e^{\kappa\varepsilon}-1)t/\varepsilon}}e^{-e^{\kappa\varepsilon}t/\varepsilon}
			\sum_{n\varepsilon>T+1}\frac{(e^{\kappa\varepsilon}t/\varepsilon)^n}{n!}\\
			&\le {e^{(e^{\kappa\varepsilon}-1)t/\varepsilon}}
			\exp\left(-\frac{e^{\kappa\varepsilon}t}{\varepsilon}
			\psi\left({\frac{T+1}{t}e^{-\kappa\varepsilon}-1}\right)\right)\\
			&\lesssim {\exp\left(-c_T\frac{t}{\varepsilon}\right)}.
		\end{align}
		Here we used that $(e^{\kappa\varepsilon}-1)/\varepsilon$ is bounded for $\varepsilon\in(0,1]$, and, for all sufficiently small $\varepsilon>0$, $e^{\kappa\varepsilon}\ge\frac12$ and $\frac{T+1}{t}e^{-\kappa\varepsilon}-1\ge\frac1{2T}$.
		This gives the required exponential term in \eqref{ineq:p_approx_m}.

		Let $\xi:=\xi_1\wedge1$, choose $k\in\N$ sufficiently large, and set $p:=2k/\xi$ and $q:=p/(p-1)$ so that $q\beta<1$. By \eqref{ineq:time_holder_long_time} and H\"{o}lder's inequality, the following inequality holds uniformly in $(t,x) \in [0,T]\times \X$ and $\varepsilon \in (0,1]$:
		\begin{align}
			&e^{-t/\varepsilon}\sum_{n\varepsilon\le T + 1}  \frac{(t/\varepsilon)^n}{n!} \int_\X\left|G_{n\varepsilon}({y,x}) - G_t({y,x}) \right|m(\mathrm dy)\\
			\lesssim&  e^{-t/\varepsilon} \sum_{n\varepsilon\le T+1} \frac{(t/\varepsilon)^n}{n!}\left(\frac{1}{(n\varepsilon)^\beta} + \frac{1}{t^\beta} \right)  |n\varepsilon - t|^{\xi}\\
			\le& \eta_\varepsilon(t) \left( t^{-\beta} + I(t) \right)
		\end{align}
		where {$\eta_\varepsilon (t):= \left(e^{-t/\varepsilon} \sum_{n\varepsilon\le T+1} \frac{(t/\varepsilon)^n}{n!} |n\varepsilon - t|^{p\xi} \right)^\frac{1}{p}$} and
		\begin{align}
			I(t) &:=\left( e^{-t/\varepsilon}\sum_{n\varepsilon\le T+1} \frac{(t/\varepsilon)^n}{n!} \frac{1}{\left(n\varepsilon \right)^{q\beta}}\right)^\frac{1}{q}\\
			&\le \left( e^{-t/\varepsilon}\sum_{n = 1}^\infty \frac{(t/\varepsilon)^n}{n!} \frac{1}{n\varepsilon}\right)^\beta\\
			&\le \frac{2}{t^\beta} \left(e^{-t/\varepsilon} \sum_{n = 0}^\infty \frac{(t/\varepsilon)^n}{n!} \right)^\beta\\
			&= \frac{2}{t^\beta},
		\end{align}
		where we used Jensen's inequality in the second line. Combining all the above, we see uniformly in $(t,x) \in {(0,T]} \times \X$ and sufficiently small $\varepsilon >0$, it holds that
		\begin{align}
			\int_\X \left|r_t^\varepsilon(y,x) - G_t(y,x) \right| {m(\mathrm dy)} \lesssim e^{-c_T \frac{t}{\varepsilon}} + \eta_\varepsilon(t) t^{-\beta}.
		\end{align}
		{Finally, by the even-moment case of Lemma \ref{lem:poisson moment bounds}, it holds uniformly in $t \in (0,T]$ and $\varepsilon \in (0,1]$ that}
		\begin{align*}
			\eta_\varepsilon(t)
			&\le {\varepsilon^\xi\E\left[\abs{N_\varepsilon(t)-\frac{t}{\varepsilon}}^{p\xi}\right]^\frac1p}\\
			&\lesssim {\varepsilon^{\xi-\frac1p}t^{\frac1p}
			+\varepsilon^{\frac\xi2}t^{\frac\xi2}}\\
			&\lesssim {\varepsilon^{\xi-\frac1p}+\varepsilon^{\frac\xi2}}
			=: \gamma_\varepsilon,
		\end{align*}
		Since $p\xi\ge2$, {$\lim_{\varepsilon\downarrow0}\gamma_\varepsilon=0$}, which proves the statement.
	\end{proof}

	\subsection{Proof of Proposition \ref{prop:approximation of coefficients}}
	The following lemma will be useful.
	\begin{lemma}\label{lem:approx of identity of heat kernel}
		Denote for $x,y \in \X$ that $d_1(x,y):= d(x,y) \wedge 1$. Then for any $\xi >0$ and $\lambda >0$, 
		\begin{align}
			\lim_{\varepsilon \downarrow 0} \sup_{ x \in \X} h(x)^{-\lambda}\int_\X G_\varepsilon(y,x) d_1(x,y)^\xi m(\mathrm dy) = 0.
		\end{align}
	\end{lemma}
	\begin{proof}
		Fix $\xi,\lambda>0$ and restrict $\varepsilon$ to $(0,1]$. By \eqref{cond:exp_hk_lamda} with $\lambda=0$,
		\begin{align*}
			M:=\sup_{0<\varepsilon\le1}\sup_{x\in\X}(P_\varepsilon^*1)(x)<\infty.
		\end{align*}
		Let $r\in(0,1)$. By Lemma \ref{lem:ref_grow} and compactness of closed bounded balls, choose a compact set $K\subset\X$ such that $h(x)^{-\lambda}\le r^\xi$ for $x\notin K$; if $\X$ is compact, take $K=\X$. Choose $x_1,\ldots,x_N\in K$ with $K\subset\bigcup_{i=1}^N B_r(x_i)$, and set $f_i(y):=d_1(y,x_i)^\xi$. Each $f_i$ is bounded and continuous, so \eqref{cond:ctem continuous} gives
		\begin{align}\label{eq:approx identity of hk uniform convergence K}
			E_\varepsilon:=\max_{1\le i\le N}\sup_{x\in K}\abs{(P_\varepsilon^*f_i)(x)-f_i(x)}\longrightarrow0
			\qquad\text{as }\varepsilon\downarrow0.
		\end{align}
		For $x\in K$, choose $i$ with $d(x,x_i)<r$. The triangle inequality for $d_1$ gives, with $C:=\max\{1,2^{\xi-1}\}$,
		\begin{align*}
			\int_\X G_\varepsilon(y,x)d_1(y,x)^\xi m(\mathrm dy)
			&\le C\int_\X G_\varepsilon(y,x)\pr{d_1(y,x_i)^\xi+d_1(x_i,x)^\xi}m(\mathrm dy)\\
			&\le C\pr{(P_\varepsilon^*f_i)(x)+r^\xi(P_\varepsilon^*1)(x)}\\
			&\le C\pr{E_\varepsilon+(1+M)r^\xi}.
		\end{align*}
		For $x\notin K$,
		\begin{align*}
			h(x)^{-\lambda}\int_\X G_\varepsilon(y,x)d_1(y,x)^\xi m(\mathrm dy)
			\le h(x)^{-\lambda}(P_\varepsilon^*1)(x)
			\le Mh(x)^{-\lambda}\le Mr^\xi.
		\end{align*}
		Combining the estimates above and using $h\ge1$, we obtain
		\begin{align*}
			\limsup_{\varepsilon\downarrow0}\sup_{x\in\X}h(x)^{-\lambda}
			\int_\X G_\varepsilon(y,x)d_1(y,x)^\xi m(\mathrm dy)
			\le C(1+M)r^\xi.
		\end{align*}
		Letting $r\downarrow0$ proves the assertion.
	\end{proof}
	\begin{proof}[Proof of Proposition \ref{prop:approximation of coefficients}]
		\emph{Step 1.} We first construct the regularized
		coefficients and verify the coefficient properties required in
		\eqref{eq:FineNConditions}.
		After redefining $b$ and $\sigma$ to be zero on the common exceptional null set, if necessary, we may assume that the standing coefficient conditions hold for every $\omega\in\Omega$.
		Let us define for each $(t,\omega,u,x) \in \R_+\times \Omega \times \R\times \X$ and $\varepsilon \in (0,1]$ that
		\begin{gather}
			b_\varepsilon(\omega,t,x,u):= \int_\X G_\varepsilon(y,x) b(\omega,t,y,u) m(\mathrm dy),\\
			\sigma_\varepsilon(\omega,t,x,u):= \int_\X G_\varepsilon(y,x) \sigma(\omega,t,y,u) m(\mathrm dy).
		\end{gather}
		Since $G_\varepsilon$ is deterministic and non-negative, both
		functions are predictable, and  \eqref{ineq:non_negativity_coef_bound_at_zero} gives  
		\begin{align}
			b_\varepsilon(\omega,t,x,0)&\geq0,
			&\sigma_\varepsilon(\omega,t,x,0)&=0.
		\end{align}
		Moreover, by \eqref{eq:LiG}, \eqref{ineq:g_lip},
		\eqref{ineq:non_negativity_coef_bound_at_zero}, and
		\eqref{cond:exp_hk_lamda} with $\lambda=0$, for each $T>0$, uniformly
		in $\omega\in\Omega$, $(t,x)\in[0,T]\times\X$,
		$\varepsilon\in(0,1]$, and $u,v\in\R$,
		\begin{equation}\label{ineq:approx_coef_uniform_bounds}
			\begin{aligned}
			\abs{b_\varepsilon(\omega,t,x,u)}\lesssim1&+|u|,
			\qquad
			\abs{\sigma_\varepsilon(\omega,t,x,u)}\lesssim|u|,\\
			\abs{b_\varepsilon(\omega,t,x,u)-b_\varepsilon(\omega,t,x,v)}
			&+\abs{\sigma_\varepsilon(\omega,t,x,u)-\sigma_\varepsilon(\omega,t,x,v)}
			\lesssim|u-v|.
			\end{aligned}
		\end{equation}
		In addition, by
		\eqref{cond:time_holder} when $d(x,y)\le1$, and by \eqref{ineq:approx_coef_uniform_bounds} together with $\varepsilon^{-\beta}d(x,y)^{\xi_2}\ge1$ when $d(x,y)>1$, it holds uniformly in $\omega\in\Omega$, $(t,x,y)\in[0,T]\times\X^2$, $u\in\R$, and
		$\varepsilon\in(0,1]$ that 
		\begin{equation}\label{ineq:approx_coef_joint_holder}
			\begin{aligned}
				\abs{b_\varepsilon(\omega,t,x,u)-b_\varepsilon(\omega,t,y,u)}
				&\lesssim \varepsilon^{-\beta}d(x,y)^{\xi_2}(1+|u|),\\
				\abs{\sigma_\varepsilon(\omega,t,x,u)-\sigma_\varepsilon(\omega,t,y,u)}
				&\lesssim \varepsilon^{-\beta}d(x,y)^{\xi_2}|u|,
			\end{aligned}
		\end{equation}
		where $\xi_2 \in (0,1]$ is as in \eqref{cond:time_holder}.
		Thus \eqref{ineq:non_negative_approx_coef_ineq} holds with $\eta=\xi_2$.
		Together with the Lipschitz bounds in $u$, the spatial estimates also
		imply \eqref{eq:con} for $b_\varepsilon$ and $\sigma_\varepsilon$.
		Consequently, all coefficient properties required in
		\eqref{eq:FineNConditions} hold at the $\varepsilon$-level. Therefore,
		for each $\varepsilon\in(0,1]$ and
		$f\in\mathcal C_\X^{\mathrm b}\subset\mathcal C^{\mathrm{tem}}_{\X,h}$, there exists a
		unique solution $u_\varepsilon$ to
		\begin{align}
			u_\varepsilon(t,x) = P^*_tf(x)&+ \int_0^t \mathrm ds \int_\X G(t-s,y,x) b_\varepsilon(s,y,u_\varepsilon(s,y)) m(\mathrm dy)\\
			&+\iint_0^t G(t-s,y,x) \sigma_\varepsilon(s,y,u_\varepsilon(s,y)) W(\mathrm dy,\mathrm ds).
		\end{align}
		\begin{samepage}
		\emph{Step 2.} It remains to prove
		\eqref{eq:Approximation}. Let $u$ be the solution to \eqref{eq:spde}, fix arbitrary $T>0$ and
		$\lambda\in(0,1]$, and set for $\varepsilon \in (0,1]$ and $(t,x)\in[0,T]\times\X$ that
		\begin{align}
			D_\varepsilon(t,x)&:=u_\varepsilon(t,x)-u(t,x),&
			V_\varepsilon(t)&:=\sup_{x\in\X}h(x)^{-2\lambda}
			\E\big[|D_\varepsilon(t,x)|^2\big].
		\end{align}
		\end{samepage}
		To separate the coefficient differences into the solution error, a
		spatial-averaging error, and an approximation-of-the-identity residual,
		write
		\[
			B_s(x):=b(s,x,u(s,x)),\qquad
			\Sigma_s(x):=\sigma(s,x,u(s,x)),
		\]
		and set
		\[
			Q_\varepsilon(s,x):=\int_\X G_\varepsilon(y,x)
			\norm{u(s,y)-u(s,x)}_{L^2_\Omega}m(\mathrm dy).
		\]
		By the definition of $b_\varepsilon$, Minkowski's inequality and \eqref{ineq:g_lip}, it
		holds uniformly in $(s,x)\in(0,T]\times\X$ and
		$\varepsilon \in (0,1]$ that
		\begin{align*}
			\norm{b_\varepsilon(s,x,u_\varepsilon(s,x))-B_s(x)}_{L^2_\Omega}
			&\leq
			\norm{b_\varepsilon(s,x,u_\varepsilon(s,x))
			-b_\varepsilon(s,x,u(s,x))}_{L^2_\Omega}\\
			&\quad+\norm{\int_\X G_\varepsilon(y,x)
			\left[b(s,y,u(s,x))-b(s,y,u(s,y))\right]m(\mathrm dy)}_{L^2_\Omega}\\
			&\quad+\norm{P_\varepsilon^*B_s(x)-B_s(x)}_{L^2_\Omega}\\
			&{\overset{\eqref{ineq:g_lip}}{\lesssim}}
			(P_\varepsilon^*\mathbf1)(x)
			\norm{D_\varepsilon(s,x)}_{L^2_\Omega}+Q_\varepsilon(s,x)
			+\norm{P_\varepsilon^*B_s(x)-B_s(x)}_{L^2_\Omega}\\
			&{\overset{\eqref{cond:exp_hk_lamda}}{\lesssim}}
			\norm{D_\varepsilon(s,x)}_{L^2_\Omega}+Q_\varepsilon(s,x)
			+\norm{P_\varepsilon^*B_s(x)-B_s(x)}_{L^2_\Omega}.
		\end{align*}
		Similarly, it holds uniformly in $(s,x)\in(0,T]\times\X$ and $\varepsilon \in (0,1]$ that
		\begin{align*}
			\norm{\sigma_\varepsilon(s,x,u_\varepsilon(s,x))-\Sigma_s(x)}_{L^2_\Omega}
			{\lesssim}
			\norm{D_\varepsilon(s,x)}_{L^2_\Omega}+Q_\varepsilon(s,x)
			+\norm{P_\varepsilon^*\Sigma_s(x)-\Sigma_s(x)}_{L^2_\Omega}.
		\end{align*}
		Hence, it holds uniformly in $\varepsilon \in (0,1]$ and $t\in[0,T]$ that 
		\begin{align*}
			&\sup_{x\in\X}h(x)^{-2\lambda}
			\norm{\int_0^t\!\int_\X G(t-s,y,x)
			\left[b_\varepsilon(s,y,u_\varepsilon(s,y))-B_s(y)\right]
			m(\mathrm dy)\mathrm ds}_{L^2_\Omega}^2\\
			{\overset{\eqref{ineq:hk_m long time}}{\lesssim}}&
			\left(\int_0^t\left[
			V_\varepsilon(s)^{1/2}
			+\sup_{y\in\X}h(y)^{-\lambda}Q_\varepsilon(s,y)
			+\sup_{y\in\X}h(y)^{-\lambda}
			\norm{P_\varepsilon^*B_s(y)-B_s(y)}_{L^2_\Omega}
			\right]\mathrm ds\right)^2\\
			{\lesssim}&
			\int_0^tV_\varepsilon(s)\,\mathrm ds
			+\left(\int_0^t\sup_{y\in\X}h(y)^{-\lambda}
			Q_\varepsilon(s,y)\,\mathrm ds\right)^2+
			\left(\int_0^t\sup_{y\in\X}h(y)^{-\lambda}
			\norm{P_\varepsilon^*B_s(y)-B_s(y)}_{L^2_\Omega}
			\,\mathrm ds\right)^2.
		\end{align*}
		In addition, it holds uniformly in $\varepsilon \in (0,1]$ and $t\in[0,T]$ that
		\begin{align*}
			&\sup_{x\in\X}h(x)^{-2\lambda}
			\norm{\iint_0^t G(t-s,y,x)
			\left[\sigma_\varepsilon(s,y,u_\varepsilon(s,y))-\Sigma_s(y)\right]
			W(\mathrm ds,\mathrm dy)}_{L^2_\Omega}^2\\
			{\overset{\eqref{ineq:mtg+mink+cs}}{\lesssim}}&
			\sup_{x\in\X}h(x)^{-2\lambda}\int_0^t
			\norm{G(t-s,\cdot,x)
			\norm{\sigma_\varepsilon(s,\cdot,u_\varepsilon(s,\cdot))
			-\Sigma_s(\cdot)}_{L^2_\Omega}}_{\mathcal H_K}^2\,\mathrm ds\\
			{\overset{\eqref{ineq:hk_k long time}}{\lesssim}}&
			\int_0^t(t-s)^{-\alpha}\left[
			V_\varepsilon(s)^{1/2}
			+\sup_{y\in\X}h(y)^{-\lambda}Q_\varepsilon(s,y)
			+\sup_{y\in\X}h(y)^{-\lambda}
			\norm{P_\varepsilon^*\Sigma_s(y)-\Sigma_s(y)}_{L^2_\Omega}
			\right]^2\mathrm ds\\
			{\lesssim}&
			\int_0^t(t-s)^{-\alpha}V_\varepsilon(s)\,\mathrm ds
			+\int_0^t(t-s)^{-\alpha}
			\left[\sup_{y\in\X}h(y)^{-\lambda}Q_\varepsilon(s,y)\right]^2
			\mathrm ds\\
			&\qquad+
			\int_0^t(t-s)^{-\alpha}\sup_{y\in\X}h(y)^{-2\lambda}
			\norm{P_\varepsilon^*\Sigma_s(y)-\Sigma_s(y)}_{L^2_\Omega}^2
			\,\mathrm ds.
		\end{align*}
		Therefore, it holds uniformly in $\varepsilon \in (0,1]$ and $t\in[0,T]$ that
		\begin{align}\label{ineq:approximation_master}
			V_\varepsilon(t)\lesssim{}&
			\int_0^t\bigl(1+(t-s)^{-\alpha}\bigr)
			V_\varepsilon(s)\,\mathrm ds
			+\left(\int_0^t\sup_{x\in\X}h(x)^{-\lambda}
			Q_\varepsilon(s,x)\,\mathrm ds\right)^2\nonumber\\
			&+\int_0^t(t-s)^{-\alpha}
			\left[\sup_{x\in\X}h(x)^{-\lambda}
			Q_\varepsilon(s,x)\right]^2\mathrm ds\nonumber\\
			&+\left(\int_0^t\sup_{x\in\X}h(x)^{-\lambda}
			\norm{P_\varepsilon^*B_s(x)-B_s(x)}_{L^2_\Omega}
			\,\mathrm ds\right)^2\nonumber\\
			&+\int_0^t(t-s)^{-\alpha}\sup_{x\in\X}h(x)^{-2\lambda}
			\norm{P_\varepsilon^*\Sigma_s(x)-\Sigma_s(x)}_{L^2_\Omega}^2
			\,\mathrm ds.
		\end{align}
		\emph{Step 3.} We first control the spatial-averaging term
		$Q_\varepsilon$, which gives the second and third terms on the
		right-hand side of \eqref{ineq:approximation_master}. Similar to the
		$J_7$ term in the proof of
		Proposition~\ref{thm:non_negtive}, we see \eqref{ineq:u_bounded_spatial_increment} holds. Now, let
		$\xi'\in(0,1)$ be as in \eqref{ineq:u_bounded_spatial_increment}, and set
		\[
			q_\varepsilon:=\sup_{x\in\X}h(x)^{-\lambda}
			\int_\X G_\varepsilon(y,x)d_1(x,y)^{\xi'}m(\mathrm dy),
		\]
		where $d_1$ is as in Lemma~\ref{lem:approx of identity of heat kernel}. Hence, $\lim_{\varepsilon \downarrow 0 }q_\varepsilon=0$.
		Thus, by \eqref{ineq:u_bounded_spatial_increment} and
		\eqref{cond:exp_hk_lamda}, it holds uniformly in $\varepsilon \in (0,1]$ and $s\in(0,T]$ that
		\begin{align}\label{ineq:approximation_Qepsilon}
			\sup_{x\in\X}h(x)^{-\lambda}Q_\varepsilon(s,x)
			\lesssim 1\wedge\big(q_\varepsilon s^{-\beta}\big),
		\end{align}
		where $\beta \in [0,1)$ is as in \eqref{ineq:u_bounded_spatial_increment}. For the drift term, Lemma~\ref{lem:approx of identity of heat kernel} and dominated convergence give
		\begin{align}\label{ineq:approximation_2nd_term}
			\lim_{\varepsilon \downarrow 0 }\int_0^T\sup_{x\in\X}h(x)^{-\lambda}
			Q_\varepsilon(s,x)\,\mathrm ds  =0.
		\end{align}
		\begin{samepage}
		For the third term on the right hand side of \eqref{ineq:approximation_master}, we may split the time integral to see that it holds uniformly in $\delta \in (0,T]$ and $\varepsilon \in (0,1]$ that
		\begin{align*}
			\sup_{t\leq T}\int_0^t(t-s)^{-\alpha}
			\left[\sup_{x\in\X}h(x)^{-\lambda}
			Q_\varepsilon(s,x)\right]^2\mathrm ds
			&\overset{\eqref{ineq:approximation_Qepsilon}}{\lesssim} \sup_{t \le T} \int_0^t (t-s)^{-\alpha} \pr{1\wedge \pr{q_\varepsilon s^{-\beta}}}^2\mathrm ds\\
			&\lesssim
			\sup_{t\leq T}\int_0^{t\wedge\delta}(t-s)^{-\alpha}\mathrm ds
			+q_\varepsilon^2\delta^{-2\beta}
			\sup_{t\leq T}\int_{t\wedge\delta}^t(t-s)^{-\alpha}\mathrm ds\\
			&\lesssim \delta^{1-\alpha}
			+q_\varepsilon^2\delta^{-2\beta}.
		\end{align*}
		\end{samepage}
		Letting first $\varepsilon\downarrow0$ and then $\delta\downarrow0$ yields
		\begin{align}\label{ineq:approximation_3rd_term}
			\lim_{\varepsilon\downarrow 0}\sup_{t\leq T}\int_0^t(t-s)^{-\alpha}
			\left[\sup_{x\in\X}h(x)^{-\lambda}
			Q_\varepsilon(s,x)\right]^2\mathrm ds=0.
		\end{align}
		\emph{Step 4.} It remains to control the last two terms in \eqref{ineq:approximation_master}. Fix $p>2$ and $R>\|f\|_{(\lambda/2)}$, and set
		\[
			A_R:=\left\{\sup_{0\leq r\leq T}
			\norm{u(r,\cdot)}_{(\lambda/2)}\leq R\right\}.
		\]
		By \eqref{eq:LiG},\eqref{ineq:g_lip},\eqref{eq:con} and the fact that $u\in\mathcal C_{\R_+}\mathcal C^{\mathrm{tem}}_{\X,h}$, we see  $B_s,\Sigma_s\in\mathcal C^{\mathrm{tem}}_{\X,h}$ for every $s\in[0,T]$ almost surely. Hence, by \eqref{cond:ctem continuous}, for every
		$s\in[0,T]$, it holds almost surely that
		\begin{equation}\label{eq:approximation_good_terms}
			\begin{aligned}
				\lim_{\varepsilon \downarrow 0}\mathbf1_{A_R}\sup_{x\in\X}h(x)^{-\lambda}
				\left(\left|P_\varepsilon^*B_s(x)-B_s(x)\right|
				+\left|P_\varepsilon^*\Sigma_s(x)-\Sigma_s(x)\right|\right)
				=0.
			\end{aligned}
		\end{equation}
		By the dominated convergence theorem, we see it holds for every fixed $R$ that 
		\begin{align*}
			&\lim_{\varepsilon \downarrow 0}\int_0^T\sup_{x\in\X}h(x)^{-\lambda}
			\norm{\mathbf1_{A_R}\big(P_\varepsilon^*B_s(x)-B_s(x)\big)}_{L^2_\Omega}
			\,\mathrm ds=0,\\
			&\lim_{\varepsilon \downarrow 0}\int_0^T\sup_{x\in\X}h(x)^{-2\lambda}
			\norm{\mathbf1_{A_R}\big(P_\varepsilon^*\Sigma_s(x)-\Sigma_s(x)\big)}_{L^2_\Omega}^2
			\,\mathrm ds=0.
		\end{align*}
		By H\"older's inequality, Corollary~\ref{cor:moment},
		\eqref{eq:LiG} and \eqref{ineq:hk_m long time}, it holds uniformly in $R$ that
		\begin{equation}\label{ineq:approximation_bad_term}
			\begin{aligned}
				&\sup_{\varepsilon\in(0,1]}\sup_{s\leq T}\sup_{x\in\X}
				h(x)^{-\lambda}
				\norm{\mathbf1_{A_R^c}\big(P_\varepsilon^*B_s(x)-B_s(x)\big)}_{L^2_\Omega}\lesssim
				\mathbb P(A_R^c)^{1/2-1/p}\\
				&\sup_{\varepsilon\in(0,1]}\sup_{s\leq T}\sup_{x\in\X}
				h(x)^{-\lambda}
				\norm{\mathbf1_{A_R^c}\big(P_\varepsilon^*\Sigma_s(x)-\Sigma_s(x)\big)}_{L^2_\Omega}\lesssim
				\mathbb P(A_R^c)^{1/2-1/p}.
			\end{aligned}
		\end{equation}
		Since $u\in\mathcal C_{\R_+}\mathcal C^{\mathrm{tem}}_{\X,h}$,
		$\mathbb P(A_R^c)\to0$ as $R\to\infty$. Thus, letting first
		$\varepsilon\downarrow0$ and then $R\to\infty$, we obtain that 
		\begin{align}\label{eq:approximation_4th_term}
			\lim_{\varepsilon \downarrow 0 }\int_0^T\sup_{x\in\X}h(x)^{-\lambda}
			\norm{P_\varepsilon^*B_s(x)-B_s(x)}_{L^2_\Omega}\,\mathrm ds
			=0.
		\end{align}
		This controls the fourth term in \eqref{ineq:approximation_master}. Since $\alpha \in (0,1)$, we see by \eqref{eq:approximation_good_terms} and dominated convergence that,
		\begin{align*}
			\lim_{\varepsilon\downarrow 0}&\sup_{t\leq T}\int_0^t(t-s)^{-\alpha}
			\sup_{x\in\X}h(x)^{-2\lambda}
			\norm{\mathbf1_{A_R}\big(P_\varepsilon^*\Sigma_s(x)-\Sigma_s(x)\big)}_{L^2_\Omega}^2\,\mathrm ds=0.
		\end{align*}
		Furthermore, it holds uniformly in $R$ that 
		\begin{align*}
			&\sup_{t\leq T}\int_0^t(t-s)^{-\alpha}
			\sup_{x\in\X}h(x)^{-2\lambda}
			\norm{\mathbf1_{A_R^c}\big(P_\varepsilon^*\Sigma_s(x)-\Sigma_s(x)\big)}_{L^2_\Omega}^2\,\mathrm ds
			\overset{\eqref{ineq:approximation_bad_term}}{\lesssim}\mathbb P(A_R^c)^{1-2/p}.
		\end{align*}
		Letting first $\varepsilon\downarrow0$ and then $R\to\infty$ gives
		\begin{align}\label{eq:apprximation_5th_term}
			\lim_{\varepsilon \downarrow 0}\sup_{t\leq T}\int_0^t(t-s)^{-\alpha}
			\sup_{x\in\X}h(x)^{-2\lambda}
			\norm{P_\varepsilon^*\Sigma_s(x)-\Sigma_s(x)}_{L^2_\Omega}^2\,\mathrm ds
			=0,
		\end{align}
		which controls the last term in \eqref{ineq:approximation_master}. 

		\emph{Step 5.} We now return to \eqref{ineq:approximation_master}. By inserting \eqref{ineq:approximation_2nd_term}, \eqref{ineq:approximation_3rd_term}, \eqref{eq:approximation_4th_term} and \eqref{eq:apprximation_5th_term} into \eqref{ineq:approximation_master}, we see there exists $H_\varepsilon$ with $\lim_{\varepsilon\downarrow0}H_\varepsilon=0$ such that the following inequality holds uniformly in $t\in[0,T]$ and sufficiently small $\varepsilon>0$:
		\[
			V_\varepsilon(t)\lesssim \int_0^t
			\bigl(1+(t-s)^{-\alpha}\bigr)V_\varepsilon(s)\,\mathrm ds+H_\varepsilon.
		\]
		We see by Corollary~\ref{cor:moment}, applied to $u$ and to $u_\varepsilon$ (the latter with $(b,\sigma)$ replaced by $(b_\varepsilon,\sigma_\varepsilon)$, see Step~1), that $\sup_{t\in[0,T]}V_\varepsilon(t)<\infty$. Hence we may use Gronwall's lemma to see that 
		\[
			\lim_{\varepsilon\downarrow0}\sup_{t\in[0,T]}V_\varepsilon(t)=0,
		\]
		which is precisely the weighted
		$L^2_\Omega$ convergence in \eqref{eq:Approximation}.
	\end{proof}

	\subsection{Proof of Lemmas \ref{lem:product metric entropy} and \ref{lem:exp_ineq_elementary}}
	\begin{proof}[Proof of Lemma \ref{lem:product metric entropy}]
		Let $\varepsilon >0$, then there are $N_\mathcal X = N_\mathcal  X(\varepsilon),\, N_\mathcal Y  = N_\mathcal Y (\varepsilon)\in \N$ with $N_\mathcal X \lesssim \varepsilon^{-\frac{\beta_\mathcal X}{\xi_\mathcal X}}$ and $N_\mathcal Y \lesssim \varepsilon^{-\frac{\beta_\mathcal Y}{\xi_\mathcal Y}}$ uniformly in $\varepsilon \in (0,1]$, so that 
		\begin{align}
			\mathcal X \subset \bigcup_{k = 1}^{N_\mathcal X} B_{\pr{\varepsilon/2}^{1/\xi_\mathcal X}}(x_k),\qquad \mathcal Y \subset \bigcup_{\ell = 1}^{N_\mathcal Y}B_{\pr{\varepsilon/2}^{1/\xi_\mathcal Y}}(y_\ell)
		\end{align}
		for some $\{x_k\}_{k = 1}^{N_\mathcal X} \subset \mathcal X$ and $\{y_\ell\}_{\ell = 1}^{N_\mathcal Y}\subset \mathcal Y$. Recall that 
		\begin{align}
				d_\mathcal S\pr{(x,y), (x',y')}:= d_\mathcal X\pr{x,x'}^{\xi_\mathcal X} + d_\mathcal Y\pr{y,y'}^{\xi_\mathcal Y}.
		\end{align}
		Denote the ball of radius $r>0$ in $\mathcal S$ under the metric $d_\mathcal S$ by $B_\mathcal S((x,y),r)$ for $(x,y) \in \mathcal S$. Then for $1 \le k  \le N_\mathcal X$ and $1 \le \ell \le N_\mathcal Y$, we have
		\begin{align}
			B_{\pr{\varepsilon/2}^\frac{1}{\xi_\mathcal X}}(x_k)\times B_{\pr{\varepsilon/2}^\frac{1}{\xi_\mathcal Y}}(y_\ell) \subset B_\mathcal S\pr{(x_k,y_\ell),\varepsilon}.
		\end{align}
		Indeed, suppose $(x,y)\in \mathcal S$ is a point in the set on the left hand side, then 
		\begin{align}
			d_\mathcal S\pr{\pr{x,y},\pr{x_k,y_\ell}} =d_\mathcal X\pr{x,x_k}^{\xi_\mathcal X} +d_\mathcal Y\pr{y,y_\ell}^{\xi_\mathcal Y}<\varepsilon.
		\end{align}
		On the other hand, we see 
		\begin{align}
			\mathcal S = \mathcal X \times \mathcal Y \subset \bigcup_{1 \le k \le N_\mathcal X \atop 1 \le \ell \le N_\mathcal Y} B_{\pr{\varepsilon/2}^\frac{1}{\xi_\mathcal X}}(x_k)\times B_{\pr{\varepsilon/2}^\frac{1}{\xi_\mathcal Y}}(y_\ell) \subset \bigcup_{1 \le k \le N_\mathcal X \atop 1 \le \ell \le N_\mathcal Y} B_\mathcal S\pr{(x_k,y_\ell),\varepsilon}.
		\end{align}
		Hence, uniformly in $\varepsilon\in (0,1]$, there exists $N_\mathcal S(\varepsilon) : =N_\mathcal X N_\mathcal Y \lesssim \varepsilon^{-\pr{\frac{\beta_\mathcal X}{\xi_\mathcal X} + \frac{\beta_\mathcal Y}{\xi_\mathcal Y}}}$ number of balls with radius $\varepsilon >0$ in $\pr{\mathcal S,d_\mathcal S}$  that cover $\mathcal S$, which is the desired result.
	\end{proof}
	
	\begin{proof}[Proof of Lemma \ref{lem:exp_ineq_elementary}]
		We first show $\ln(xy)-y(x-1) <0 $ for all $y >0$ and $x \ge 2$. Observe for each $x \ge 2$, the function $f_x(y) := \ln (xy) - y(x-1)$ is increasing for $0 < y \le \frac{1}{x-1}$ and decreasing when $y > \frac{1}{x-1}$.
		Hence
		\begin{gather*}
			\sup_{x \ge 2}\sup_{y>0} \left(\ln(xy) - y(x-1)\right) = \sup_{x \ge 2}\ln\left(\frac{x}{x-1} \right) - 1\le \ln(2) -1 < 0.
		\end{gather*}
		Therefore,
		$xy\exp(-xy+y)= \exp\left(\ln(xy) - y(x-1) \right)<1, \text{ for all } y>0,\, x \ge 2.$
		Multiply both sides by $\exp(-y)$ to get 
		\begin{gather*}
			xy \exp(-xy) < \exp(-y) ,\quad \text{ for all } y>0,\, x \ge 2.
		\end{gather*}
		This implies
		\begin{gather*}
			xy \exp\left( - \frac{y}{2} 2x\right) \le \exp \left( - \frac{y}{2}\right),\quad \text{ for all } y >0,x \ge 1.
		\end{gather*}
		Multiply both sides by $\frac{1}{xy}$ to see the desired inequality.
	\end{proof}

\bibliographystyle{plain}
\bibliography{CCC}
\end{document}